\documentclass[11pt]{amsart}
\usepackage[headings]{fullpage}
\usepackage{amssymb,epic,eepic,epsfig,amsbsy,amsmath,amscd,color}
\usepackage[all]{xy}
\usepackage{graphicx}
\usepackage{texdraw}
\usepackage{url}
\usepackage{bbm}

\newlength{\standardunitlength}
\catcode`\@=11
\long\def\@makecaption#1#2{%
     \vskip 10pt

\setbox\@tempboxa\hbox{%
       \small\sf{\bfcaptionfont #1. }\ignorespaces #2}%
     \ifdim \wd\@tempboxa >\captionwidth {%
         \rightskip=\@captionmargin\leftskip=\@captionmargin
         \unhbox\@tempboxa\par}%
       \else
         \hbox to\hsize{\hfil\box\@tempboxa\hfil}%
     \fi}
\font\bfcaptionfont=cmssbx10 scaled \magstephalf
\newdimen\@captionmargin\@captionmargin=2\parindent
\newdimen\captionwidth\captionwidth=\hsize
\catcode`\@=12

\usepackage{mathrsfs}

                        \theoremstyle{plain}

\newtheorem{theorem}{Theorem}[section]

\newtheorem{lemma}[theorem]{Lemma}
\newtheorem{corollary}[theorem]{Corollary}
\newtheorem{proposition}[theorem]{Proposition}

\newtheorem{definition}{Definition}

\theoremstyle{definition}

\newtheorem{remark}[theorem]{Remark}
\newtheorem{example}[theorem]{Example}

\newcommand{\Span}{\operatorname{Span}}
\newcommand{\tr}{\operatorname{tr}}

\newcommand{\id}{\operatorname{id}}

\newcommand{\qbinom}[2]{\begin{bmatrix}#1\\ #2\end{bmatrix}}

\newcommand\FIG[3]{\begin{figure}
    \includegraphics[#3]{#1.eps}
    \caption{#2}
    \label{fig:#1}
    \end{figure}}

\newcommand\incl[2]{{\includegraphics[height=#1]{#2}}}

\def\BN{\mathbb N}
\def\BZ{\mathbb Z}

\def\BQ{\mathbb Q}
\def\BR{\mathbb R}
\def\BC{\mathbb C}

\def\cT{\mathcal T}
\def\Xh{{\mathbf X}_h}

\def\cA{{\mathcal A}}
\def\cR{\mathcal R}

\def\D{\Delta}

\def\cU{\mathcal U}

\def\cL{\mathcal L}

\def\V{\mathcal V}

\def\P{\mathcal P}

\def\a{\alpha}

\def\Ga{\Gamma}

\def\modA {{\mathcal A}}
\def\modZ {{\mathbb Z}}

\def\la{\langle}
\def\ra{\rangle}

\def\g{\gamma}

\def\ve{\varepsilon}
\def\Ga{\Gamma}
\def\d{\delta}

\def\sub{\subset}

\def\sgn{\operatorname{sgn}}

\def\tT{\tilde \cT}

\def\Aut{\mathrm{Aut}}

\def\al{\alpha}

\def\bi{\mathbf i}

\def\fg{\mathfrak{g}}
\def\hh{\mathfrak{h}}

\def\bk{\mathbf{k}}
\def\UU{{\mathbf U_h}}

\def\tr{\mathrm{tr}}

\def\ev{{\mathrm{ev}}}
\def\evxi{\ev_\xi}

\newcommand{\ad}{\mathrm{ad}}
\newcommand{\tri}{\triangleright}

\newcommand{\br}{\mathbf r}

\def\cF{{\mathcal F}}

\def\be{\begin{equation} }
\def\ee{\end{equation} }
\def\cB{{\mathcal B}}
\def\ord{\mathrm{ord}}

\newcommand{\inv}{{\mathrm{inv}}}
\def\cD{{\mathcal D}}
\def\tcD{\tilde{\mathcal T}}
\def\sh{\operatorname{sh}}
\def\bc{\mathbf c}

\def\cF{{\mathcal F}}

\def\cI{\mathcal I}

\newcommand{\X}{X}

\newcommand{\cQ}{{\mathcal Q}}

\newcommand{\cH} {{\mathcal H}}

\def\Hopf{\mathfrak H}

\def\eqz{\overset{(\zeta)}{=}}
\def\neqz {\overset{(\zeta)}{\neq}}

\newcommand\suppress[1]{[[suppressed]]}
\newcommand\xyzxyz{x\otimes y\otimes z\mapsto xyz}

\newcommand\modK {\mathcal K}

\renewcommand\g{\mathfrak{g}}
\newcommand\modk {{\BC [[h]]}}
\newcommand\modg {\mathbf g}
\newcommand\modi {\mathbf i}
\newcommand\modj {\mathbf j}
\newcommand\modm {\mathbf m}
\newcommand\modn {\mathbf n}

\newcommand\modu {\mathbf u}

\newcommand\modQ {\mathbb Q}
\newcommand\ho{\hat\otimes}
\newcommand\col {\colon\thinspace}
\newcommand\Uq{\mathbf U_q}

\newcommand\Uqv{\Uq^\ev}

\newcommand\Uh{\mathbf U_h}

\newcommand\Uhx[1]{\mathbf U_h^{\ho #1}}

\newcommand\Uhg{\Uh(\g)}

\newcommand\Uhh{\Uh[h^{-1}]}
\newcommand\modr {\mathbf r}
\newcommand\modc {\mathbf c}

\newcommand\modb {\mathbf b}
\newcommand\Zqh{\widehat{\modZ [q]}}

\newcommand\trr{\triangleright}

\newcommand\hide[1]{.}

\newcommand\cl{\operatorname{cl}}
\newcommand\trq{\tr_q}
\newcommand\bH{\underline{\mathscr H}}
\newcommand\bS{\underline{S}}
\newcommand\bD{\underline{\Delta }}

\newcommand\dv{{\dot v}}
\newcommand\de{{\dot e}}
\newcommand\dK{{\dot K}}
\newcommand\dS{{\dot S}}
\newcommand\dT{{\dot T}}

\newcommand\zzzvert {\,|\,}
\newcommand\hf{\frac12}
\newcommand\dth{{\dot\theta }}

\newcommand\rad{\operatorname{ad^\mathrm{r}}}
\newcommand\radb{\underline{\rad}}

\newcommand\Zt{{\BN^t}}

\newcommand\np{\newpage}

\newcommand\no[1]{}

\newcommand\bbb[2]{\left[\begin{matrix}#1\\#2\end{matrix}\right]}

\newcommand\Gv{G^\ev}
\newcommand\cUq{\check{\mathbf U}_q}

\newcommand\uqq[1]{[\Uq]_{#1}}

\newcommand\uqvq[1]{[\Uqv]_{#1}}

\newcommand{\UZ}{{\mathbf U}_\BZ}

\newcommand{\bm}{{\mathbf m}}

\newcommand\bn{\mathbf{n}}

\newcommand\dx{\dot x}
\newcommand\dy{\dot y}
\newcommand\dmu{\dot\mu }
\newcommand\lm{{\lambda _\modm }}
\newcommand\elm{\de_\lm}
\newcommand\klm{\dK_\lm}

\newcommand\Zvv{\modZ [v,v^{-1}]}
\newcommand\ul{\underline}
\newcommand\congto{\overset{\cong}{\longrightarrow}}
\newcommand\Vev{\cV^{\ev}}
\newcommand\bfk{\mathbf{k}}
\newcommand\modh {\hh}
\def\Uq{\mathbf U_q}

\newcommand\calZ{{\mathcal Z}}
\newcommand\CZ{{\mathcal Z}}
\newcommand\mods {\operatorname{ev}}
\newcommand\Qqh{\widehat{\modQ [q]}}
\newcommand\ZZ{\mathcal Z}

\def\sH{\mathscr H}

\def\sX{\mathscr X}

\def\XA{{\mathbf X}_\cA}
\def\brQ{\breve Q}

\def\br{\mathbf r}

\def\be{ \begin{equation} }
\def\ee{ \end{equation} }
\def\Zq{\BZ[q^{\pm 1}]}
\def\Ht{\operatorname{ht}}
\def\Gal{\mathrm{Gal}}
\newcommand\Zxi{\modZ[\xi]}
\newcommand\Pg{{P\g}}
\newcommand{\ZZg}{\ZZ_\fg}
\newcommand{\ZZPg}{\ZZ_{P\fg}}
\newcommand\Qab{\mathbb{Q}^{\mathrm{ab}}}
\newcommand\bbQ{\mathbb{Q}}
\newcommand\bbC{\mathbb{C}}
\newcommand\bbZ{\mathbb{Z}}
\newcommand\hon{{\ho n}}
\newcommand\Ch{\bbC[[h]]}

\newcommand\ot{\otimes}
\def\sR{\Ch}

\def\bu{\mathbf u}
\def\scH{\mathscr H}
\def\boldmu{\boldsymbol \mu}
\def\boldep{\boldsymbol \epsilon}

\def\boldpsi{\boldsymbol \psi}

\def\scH{\mathscr H}
\def\sL{\mathscr L}

\def\oX{\overline {\sX}}
\def\oXtwo{\overline{\sX \otimes \sX}}
\def\btau{\kappa}
\def\ibar{\iota_{\mathrm{bar}}}

\def\brU{\breve{\mathbf U}}

\def\tphi{\varphi}
\def\bral{\breve \al}
\def\brK{\breve K}
\def\fW{{\mathfrak W}}
\def\iprod{\operatorname{\overleftarrow{\prod}  }}

\def\brH{\breve H}
\def\uS{\underline S}

\def\tA{{\tilde {\cA}}}
\def\sXZ{{\mathbf X}_\BZ}

\def\Uhalf{{\mathbf U}_{\sqrt h}\ }
\def\UUH{\Uhalf\ }

\def\UUh{\sX}
\def\oX{\overline{\sX}}

\def\Chh{\Chalf}
\def\bbe{\mathbf e}
\def\Vh{{\mathbf V}_h}

\def\bo{\bar \ot}
\def\tcR{\tilde{\cR}}
\def\Oo{\overline{\otimes}}
\def\uD{\underline \Delta}
\def\UA{\UZ}
\def\VA{{\mathbf V}_\BZ}
\def\brK{\breve K}

\def\bode{\boldsymbol{\delta}}
\def\VAev{\VA^\ev}
\def\BC{\BQ}
\def\BQvK{\BQ(v)[x^{\pm 1}]}
\def\leu{\prec}
\def\BC{{\mathbb C}}
\def\briota{\breve \iota}
\def\brbe{\breve{\bbe}}
\def\brUA{\breve{\mathbf U}_\BZ}
\def\Vev{\VA^\ev}

\def\sXZev{\sXZ^\ev}

\def\sXZe{\sXZ^\ev}
\def\sXZX{\breve{\mathbf X}_\BZ}
\def\brVA{\breve{\mathbf V}_\BZ}

\def\cI{\mathcal I}
\def\btri{\blacktriangleright}

\def\cUq{\tilde {\mathbf U}_q}
\def\dva{\dot \varphi}
\def\dtphi{\dot {\tphi}}
\def\UUh{{\mathbf U}_{\sqrt h}}
\def\UUH{{\mathbf U}_{\sqrt h}}
\def\Chh{{\BC[[\sqrt h]]}}

\def\tK{\widetilde {\modK}}
\def\UZe{\UZ^\ev}
\def\Ur{\mathcal U}
\def\oX{\overline {\sX}}
\def\CtK{\widetilde\modK'}
\newcommand{\CK}{\modK'}

\def\ooo{(q;q)}
\def\UZe{\UZ^\ev}
\def\boldpsi{\boldsymbol \psi}
\def\dad{\dot {\mathrm {ad}}}
\def\dad{\dot {\mathrm{ad}}}
\def\bu{\modu}

\def\fuD{f^ {\uD}}

\def\fbS{f^ {\bS}}

\def\fad{f^{\ad}}
\def\tfad{\tilde f^{\ad}}
\def\fmu{f^{\mu}} \def\tfmu{\tilde f ^\mu}
\def\boldmu{\boldsymbol \mu}
\def\uad{\underline{\mathrm{ad}}}
\def\btri{\blacktriangleright}
\def\ez{ \ev_{v^{1/\cD}=\zeta}}
\def\Cv{{\BC[v^{\pm1}]}}
\def\Cvh{{\widehat{\BC[v]}}}

\def\CXh{  \Xh}

\def\tB{\tilde {\cB}}

\newcommand{\Zc}{{\mathfrak Z}}
\def\cH{\chi}
 \def\Hopf{L}
 
 \def\DDD{\mathbb D}
\def\ddd{\mathbbmtt d}
\def\eqzeta{\, \overset{(\zeta)}{=} \, }
\def\CsX{\Xh}
\def\Cv{\BC[v^{\pm 1}]}
 \def\hCv{\widehat{\BC[v]}}
  \def\Cv{{\BC[v^{\pm 1}]}}
  \def\Ztwo{\BZ_{(2)}}
\def\vol{\operatorname{vol}}
\def\XZ{{\mathbf X}_\BZ}
\def\VZ{{\mathbf V}_\BZ}
\def\brUZ{\breve{\mathbf U}_\BZ}
\def\brVZ{\breve{\mathbf V}_\BZ}

\begin{document}

\title[Unified quantum invariants]{Unified quantum invariants for integral homology spheres associated with simple Lie algebras}

\author[Kazuo Habiro]{Kazuo Habiro}
\address{Research Institute for Mathematical Sciences, Kyoto
University, Kyoto 606-8502, Japan}

\email{habiro@kurims.kyoto-u.ac.jp}
\author[Thang  T. Q. L\^e]{Thang  T. Q. L\^e}
\address{School of Mathematics, 686 Cherry Street,
 Georgia Tech, Atlanta, GA 30332, USA}
\email{letu@math.gatech.edu}

\date{\today}

\thanks{K.H.  was partially supported by the Japan Society for
  the Promotion of Science, T. T. Q. L. was partially supported in part by National Science Foundation. \\
2010 {\em Mathematics Classification:} Primary 57M27. Secondary 57M25.}

\begin{abstract}
  For each finite dimensional, simple, complex Lie algebra $\fg$ and
  each root of unity $\xi$ (with some mild restriction on the order)
  one can define the Witten-Reshetikhin-Turaev (WRT) quantum  invariant
  $\tau_M^\fg(\xi)\in \BC$ of oriented 3-manifolds $M$. In the present paper we construct an
  invariant $J_M$ of {\em integral homology spheres} $M$ with values in the
  cyclotomic completion $\Zqh$ of the polynomial ring $\modZ [q]$,
  such that the evaluation of $J_M$ at each root of unity gives the
  WRT quantum invariant of $M$ at that root of unity.  This result
  generalizes the case $\fg=sl_2$ proved by the first author.  It
  follows that $J_M$ unifies all the quantum invariants of $M$
  associated with $\fg$, and represents the quantum invariants as a
  kind of ``analytic function'' defined on the set of roots of unity.
  For example, $\tau_M(\xi)$ for all roots of unity are determined by
  a ``Taylor expansion'' at any root of unity, and also by the values
  at infinitely many roots of unity of prime power orders.  It follows
  that WRT quantum invariants $\tau _M(\xi)$ for all roots of unity are
  determined by the Ohtsuki series, which can be regarded as the
  Taylor expansion at $q=1$, and hence by the Le-Murakami-Ohtsuki
  invariant.  Another consequence is that the WRT quantum invariants
  $\tau_M^\fg(\xi)$ are algebraic integers.  The construction of the
  invariant $J_M$ is done on the level of quantum group, and does not
  involve any finite dimensional representation, unlike the definition
  of the WRT quantum invariant. Thus, our
  construction gives a unified, ``representation-free'' definition of
  the quantum invariants of integral homology spheres.
\end{abstract}

\maketitle

{\small \tableofcontents}

\parskip .5em

\np
\section{Introduction} \label{sec:intro}
The main goal of the paper is to construct an invariant $J^\fg_M$ of
integral homology spheres $M$ associated to each finite dimensional
simple Lie algebra $\g$, which unifies the Witten-Reshetikhin-Turaev
quantum invariants at various roots of unity.  The invariant $J^\fg_M$
takes values in the completion
$\Zqh=\varprojlim_n\modZ[q]/((1-q)(1-q^2)\cdots(1-q^n))$ of the
polynomial ring $\modZ[q]$, which may be regarded as a ring of
analytic functions on roots of unity.  This invariant unifies the
quantum invariants at various roots of unity in the sense that for
each root of unity $\xi$, the evaluation $\ev_\xi(J^\fg_M)$ at
$q=\xi$ of $J^\fg_M$ is equal to the WRT quantum invariant
$\tau^\fg_M(\xi)$ of $M$ at $\xi$ whenever $\tau^\fg_M(\xi)$ is
defined.  This invariant is a generalization of the $sl_2$ case
constructed in \cite{H:unified}.

\subsection{The WRT invariant}\label{sec:intro-wrt}
Witten \cite{Witten}, using non-mathematically rigorous path integrals
in quantum field theory, gave a physics interpretation of the Jones
polynomial \cite{Jones} and predicted the existence of $3$-manifold invariants
associated to every simple Lie algebra and certain integer, called
level. Using the quantum group $U_q(sl_2)$ at roots of unity,
Reshetikhin and Turaev \cite{RT2} gave a rigorous construction of
$3$-manifold invariants, which are believed to coincide with the Witten
invariants. These invariants are called the Witten-Reshetikhin-Turaev
(WRT) quantum invariants. Later the machinery of quantum groups helped to
generalize the WRT invariant $\tau_M^\fg(\xi)$ to the case when $\fg$
is an arbitrary simple Lie algebra, and $\xi$ is a root of
unity.

In this paper we will focus on the quantum invariants of an integral
homology $3$-sphere, i.e. a closed oriented $3$-manifold $M$ such that
$H_*(M,\BZ)= H_*(S^3,\BZ)$.

Let $\ZZ\subset \mathbb{C}$ denote the set of all roots of unity.  For
each simple Lie algebra $\fg$, there is a subset $\ZZg\subset \ZZ$ and
the $\fg$ WRT invariant of an integral homology sphere $M$ gives a
function
\begin{gather}
  \label{e3}
  \tau^\fg_M\colon \ZZg \to \mathbb{C},
\end{gather}
(We recall the definition of $\tau_M^\fg(\xi)$ in Section \ref{sec.WRT}. The definition of $\tau_M^\fg(\xi)$ for closed $3$-manifolds involves
 a choice of a certain root of $\xi$, but it turns out that for
integral homology spheres this choice is irrelevant.)

We are interested in the behavior of the WRT function \eqref{e3}
associated to each Lie algebra $\fg$.
It is natural to raise the following questions.
\begin{itemize}
\item Is it possible to extend the domain of the map $\tau^\fg_M$ to
 $\ZZ$ in a natural way?
\item How strongly are the values at different roots of unity
  $\xi,\xi'\in \ZZg$ related?
\item Is there some restriction on the range of the function?  In
  particular, is $\tau_M^\fg(\xi)$ an algebraic integer for all $\fg$
  and $\xi$?
\item How are the quantum invariants related to finite type invariants
of $3$-manifolds \cite{Ohtsuki,H:claspers,Goussarov}? In particular, is
there any relation between the quantum invariants and the
Le-Murakami-Ohtsuki  invariant \cite{LMO}?
\end{itemize}

\newcommand{\ZQ}{\modZ[q]}

\subsection{The ring $\Zqh$ of analytic functions on roots of unity}
\label{sec.Habiro}
Define a completion $\Zqh$ of the polynomial ring $\ZQ$ by
\begin{equation*}
  \Zqh = \varprojlim_{n}\modZ [q]/((q;q)_n),
\end{equation*}
where as usual
\begin{equation*}
  (x;q)_n := \prod_{j=1}^n (1-x q^{j-1}).
\end{equation*}
The ring $\Zqh$ may be regarded as the ring of ``analytic functions
defined on the set $\calZ$ of roots of unity''
\cite{H:cyclotomic,H:unified}.  This statement is justified by the
following facts.  For more details, see Section 1.2 of
\cite{H:unified}.

For a root of unity $\xi\in\CZ$ of order $r$, we have
$(\xi;\xi)_n=0$ for $n\ge r$.
Hence the evaluation map
\begin{gather*}
  \ev_\xi\colon \modZ[q]\to \modZ[\xi],\quad f(q)\mapsto f(\xi)
\end{gather*}
induces a ring homomorphism
\begin{gather*}
  \ev_\xi\colon \Zqh\to \modZ[\xi].
\end{gather*}
We write $f(\xi)=\ev_\xi(f(q))$.

Each element $f(q)\in\Zqh$ defines a function from $\CZ$ to $\BC$.
Thus we have a ring homomorphism
\begin{gather}
  \label{e4}
  \ev\colon\Zqh \to \mathbb{C}^{\CZ}
\end{gather}
defined by $\ev(f(q))(\xi)=\ev_\xi(f(q))$.
This homomorphism is injective \cite{H:cyclotomic}, i.e., $f(q)$ is
determined by the values $f(\xi)$ for $\xi\in\CZ$.
Therefore, we may regard $f(q)$ as a function on the set $\CZ$.

In fact, a function $f(q)\in\Zqh$ can be determined by values on a
subset $\CZ'$ of $\CZ$ if $\calZ'$ has a limit point $\xi_0\in\CZ$ with respect
to a certain topology of $\calZ$, see \cite[Theorem
6.3]{H:cyclotomic}.
In this topology, an element $\xi \in \calZ$ is a limit point of a
subset $\calZ'\subset \calZ$ if and only if there are infinitely many
$\xi' \in \calZ'$ such that the orders (as roots of unity) of   $\xi'\xi^{-1}$ are  prime powers.  For
example, each $f(q)\in\Zqh$ is determined by the values at infinitely
many roots of unity of prime  orders.

For $\xi\in\CZ$, there is a ring homomorphism
\begin{gather*}
  T_\xi\colon \Zqh \to \modZ[\xi][[q-\xi]]
\end{gather*}
induced by the inclusion $\modZ[q]\subset\modZ[\xi][q]$, since, for
$n\ge0$, the element $(q;q)_{n\,\ord(\xi)}$ is divisible by $(q-\xi)^n$
in $\modZ[\xi][q]$.  The image $T_\xi(f(q))$ of $f(q)\in\Zqh$ may be
regarded as the ``Taylor expansion'' of $f(q)$ at $\xi$.  The
homomorphism $T_\xi$ is injective \cite[Theorem 5.2]{H:cyclotomic}.
Hence a function $f(q)\in\Zqh$ is determined by its Taylor expansion
at a point $\xi\in\CZ$.  Injectivity of $T_\xi$ implies that $\Zqh$
is an integral domain.

The above-explained properties of $\Zqh$  depend on the ground ring
 $\BZ$ of integers in an essential way.  In fact, the similar completion $\Qqh =
\varprojlim_{n}\modQ [q]/((q;q)_n)$ is radically different. For example,
$\Qqh$ is not an integral domain, and quite opposite to the case over
$\BZ$, the Taylor expansion map $T_\xi \colon \Qqh\to \BQ[\xi][[q-\xi]]$ is
surjective but not injective, see \cite[Section 7.5]{H:cyclotomic}.

Recently, Manin \cite{Manin} and Marcolli \cite{Mar} have promoted the ring
$\Zqh$ as a candidate for the ring of analytic functions on the
non-existing ``field of one element''.

\subsection{Main result and consequences}
\label{sec:main-result-cons}

The following is the main result of the
present paper.

\begin{theorem}
  \label{r40a} For each simple Lie algebra $\fg$,
  there is a unique invariant $J_M=J^\g_M\in \Zqh$ of an integral homology sphere $M$
  such that for all $\xi\in\ZZg$ we have
  \begin{equation*}
    \mods_\xi (J_M) = \tau^\fg _M(\xi).
  \end{equation*}
\end{theorem}

Theorem \ref{r40a} is proved in Section \ref{sec:r40a}. It follows from Theorems  \ref{r.JMdef}, \ref{thm.coresub},
 \ref{r.Jvalues}, and \ref{thm030}.

The case $\g=sl_2$ of Theorem \ref{r40a} is announced in
\cite{H:rims2001} and proved in \cite{H:unified}.  For $\g=sl_2$, the
invariant $J_M$ has been generalized to invariants of rational
homology spheres with values in modifications of $\Zqh$ in
\cite{BBL,Le_Strong_Integrality,Beliakova-Le,Beliakova-Buehler-Le}.

The theorem implies that for integral homology $3$-spheres,
$\tau^\g _M(\xi)$ does not depend on the choice of a root of $\xi$ which
is used in the definition of $\tau^\g_M(\xi)$.

We list here a few consequences of Theorem \ref{r40a}.  For the
results stated without proof and with the $sl_2$ case proved in
\cite{H:unified}, the proof is the same as the proof of the
corresponding result in \cite{H:unified}.

\subsubsection{Analytic continuation of $\tau^\g_M$ to all roots of unity}

Even if a root of unity $\xi\in\CZ$ is not contained in $\ZZg$, the
domain of definition of the WRT function $\tau^\g_M$, we have a
well-defined value $\evxi(J_M)\in\Zxi$.  By the uniqueness of $J_M$,
it would be natural to {\em define} the $\g$ WRT invariant
$\tau^\g_M(\xi)$ at $\xi\in\CZ\setminus\ZZg$ as $\evxi(J_M)$.  We may
regard it as an analytic continuation of
$\tau^\g_M\colon\ZZg\to\mathbb{C}$.

The specializations $\evxi(J_M)$ are compatible also with the projective
version of the $\g$ WRT invariant
\begin{gather}
  \tau^{P\g}_M\colon \ZZPg \to \mathbb{C},
\end{gather}
where $\ZZPg$ is another subset of $\ZZ$. See Section \ref{sec.WRT}.

\begin{proposition}
  \label{r3}
  For an integral homology sphere $M$ and for $\xi\in\ZZPg$, we have
  \begin{gather*}
    \evxi (J_M) = \tau^{P\fg}_M(\xi).
  \end{gather*}
  As a consequence, for $\xi\in\ZZg\cap\ZZPg$, we have
  \begin{gather}
    \label{e5}
    \tau^\g_M(\xi)=\tau^\Pg_M(\xi).
  \end{gather}
\end{proposition}

\begin{remark}
  \label{r8}
  For a closed $3$-manifold $M$ which is not necessarily an
  integral homology sphere we do not have \eqref{e5} but for some values
  of $\xi$ we have identities of the form
  \begin{gather*}
    \tau^\g_M(\xi)=\tau^\Pg_M(\xi) \tilde\tau^\g_M(\xi)
  \end{gather*}
  where $\tilde\tau^\g_M(\xi)$ is an invariant of $M$ satisfying
  $\tilde\tau^\g_M(\xi)=1$ for $M$ an integral homology sphere.  For details, see
  e.g. \cite{Blanchet,KM,KT,Le:quantum}.
\end{remark}

\subsubsection{Integrality of quantum invariants}
An immediate consequence of Theorem \ref{r40a} is the following
integrality result.

\begin{corollary}%
  \label{r14a}
  For any integral homology sphere $M$ and for $\xi \in \calZ_\g$, we
  have $\tau^\g_M (\xi)\in \modZ [\xi]$.  In particular,
  $\tau^\g_M (\xi)$ is an algebraic integer.
\end{corollary}

Here we list related integrality results for quantum invariants for
closed $3$-manifolds, which are not necessarily integral homology
spheres.

H. Murakami \cite{Murakami} (see also \cite{MR}) proved that the
$Psl_2$ WRT invariant, also known as the quantum $SO(3)$ invariant
\cite{KM}, of a closed $3$-manifold at $\xi\in\ZZ$ of {\em prime order}
is contained in $\Zxi$.  This result, for roots of unity of prime
orders, has been generalized to $sl_n$ by Masbaum and Wenzl \cite{MW}
and independently by Takata and Yokota \cite{TY}, and to all simple Lie algebras by the second
author \cite{Le:quantum}.

The case of roots of {\em non-prime orders}, conjectured by Lawrence
\cite{Lawrence:integrality} in the $sl_2$ case, has
been developed later.  The case $\g=sl_2$ of Corollary \ref{r14a} is
obtained by the first author in \cite{H:unified}.  Beliakova, Chen and
the second author \cite{BCL} proved that for any root of unity $\xi$,
$\tau^{sl_2}_M(\xi)$ (which depends on a fourth root of $\xi$) is
an algebraic integer.  For general Lie algebras, however, the proof in
\cite{BCL} does not work.  Corollary \ref{r14a} is the first
integrality result, for general Lie algebras, in the case of non-prime
orders.

\subsubsection{Relationships between quantum invariants at different
  roots of unity}

One can obtain from Theorem \ref{r40a} results about the values of the
WRT invariants, more refined than integrality.

Let $\Qab\subset\bbC$ denote the maximal abelian extension of $\bbQ$,
which is the smallest extension of $\bbQ$ containing $\calZ$.  The
image of the WRT function $\tau^\g_M$ is contained in the integer ring
$\mathcal{O}(\Qab)$ of $\Qab$, which is the subring of $\Qab$
generated by $\calZ$.  Note that an automorphism
$\alpha\in\Gal(\Qab/\bbQ)$ maps each root of unity $\xi$ to a root of
unity $\alpha(\xi)$ of the same order as $\xi$.  There is a canonical
isomorphism
\begin{gather*}
   \Gal(\Qab/\bbQ) \cong \Aut_{\mathrm{Grp}}(\calZ),
\end{gather*}
which maps $\alpha\in\Gal(\Qab/\bbQ)$ to its restriction to $\calZ$. Here $\Aut_{\mathrm{Grp}}(\calZ)$ is the group of automorphisms of $\calZ$, considered as a subgroup of the multiplicative group $\BC\setminus\{0\}$.

\begin{proposition}
  \label{r9}
  For every integral homology sphere $M$, the $\g$ WRT function
  $\tau^\g_M\colon\CZ\to\mathbb{Q}^{\mathrm{ab}}$ is
  Galois-equivariant in the sense that for each automorphism
  $\alpha\in\Gal(\mathbb{Q}^{\mathrm{ab}}/\mathbb{Q})$ we have
  \begin{gather*}
    \tau^\g_M(\alpha(\xi))=\alpha(\tau^\g_M(\xi))
  \end{gather*}
\end{proposition}

The $sl_2$ case of Proposition \ref{r9} is mentioned in
\cite{H:unified}.

Proposition \ref{r10} below is proved in Section \ref{sec:r10}.

\begin{proposition}[\cite{H:unified} for $\g=sl_2$]
  \label{r10}
We have $\ev_1(J_M)=1$  for every integral homology sphere $M$.
\end{proposition}

\begin{proposition}[\cite{H:unified} for $\g=sl_2$]
  \label{r15}
  For $\xi,\xi'\in\ZZ$ with $\ord(\xi'\xi^{-1})$ a prime power, we
  have
  \begin{gather*}
    \tau^\g_M(\xi)\equiv \tau^\g_M(\xi') \pmod{\xi'-\xi}
  \end{gather*}
  in $\BZ[\xi,\xi']$.
\end{proposition}

Proposition \ref{r15} holds also when $\ord(\xi'\xi^{-1})$ is not a
prime power, but in this case the statement is trivial since
$\xi'-\xi$ is a unit in $\BZ[\xi,\xi']$.

\begin{corollary}
  \label{r14}
  For every integral homology sphere $M$ and for every root of unity
  $\xi\in\ZZ$ of prime power order, we have
  \begin{gather*}
    \tau^\g_M(\xi)-1 \in (1-\xi)\Zxi.
  \end{gather*}
  Consequently, we have $\tau^\g_M(\xi)\neq 0$.
\end{corollary}

For $\g=sl_2$, a refined version of Corollary \ref{r14} is given
in \cite[Corollary 12.10]{H:unified}.

\subsubsection{Integrality of the Ohtsuki series}
When $M$ is a rational homology sphere, Ohtsuki \cite{Ohtsuki1}
extracted a power series invariant, $\tau^{sl_2}_\infty(M) \in
\BQ[[q-1]]$, from the values of $\tau_M^{Psl_2}(\xi)$ at roots of
unity of prime orders.  The Ohtsuki series is characterized by certain
congruence relations modulo odd primes.  The existence of the Ohtsuki
series invariant for other Lie algebras was proved in
\cite{Le:PSUn,Le:quantum}, see also \cite{Rozansky:Ohtsuki}.

The Ohtsuki series $\tau^\g_\infty(M)\in\bbQ[[q-1]]$ and the unified
WRT invariant $J_M$ are related as follows.

\begin{proposition}[\cite{H:unified} for $sl_2$]
  \label{r16}
  For every integral homology sphere $M$, we have
  \begin{gather*}
    \tau^\g_\infty(M)=T_1(J_M)\in\bbZ[[q-1]].
  \end{gather*}
  In other words, the Ohtsuki series is equal to the Taylor expansion of the
  unified WRT invariant at $q=1$.
  Moreover, all the coefficients in the Ohtsuki series are integers.
\end{proposition}

The fact $\tau^{\fg}_\infty(M)\in\bbZ[[q-1]]$, for $\fg= sl_2$, was conjectured by
R. Lawrence \cite{Lawrence:integrality},
and first proved by Rozansky \cite{Rozansky:integrality}. Here we have general results for all simple Lie algebras.

\subsubsection{Relation to the Le-Murakami-Ohtsuki invariant}
The Le-Murakami-Ohtsuki (LMO) invariant \cite{LMO} is a counterpart of
the Kontsevich integral for homology $3$-spheres; it is a universal
invariant for finite type invariants of integral homology $3$-spheres
\cite{Le:universal}.  The LMO invariant $\tau^{LMO}(M)$ of a closed
$3$-manifold takes values in an algebra $\cA(\emptyset)$ of the
so-called Jacobi diagrams, which are certain types of trivalent
graphs. For each simple Lie algebra $\fg$, there is a ring
homomorphism (the weight map)
\begin{gather*}
  W_\fg\colon \cA(\emptyset) \to\BQ[[h]].
\end{gather*}
It was proved in \cite{KLO} that
$$ W_\fg(\tau^{LMO}(M) )= \tau^\g_\infty(M) |_{q= e^h}.$$
Hence, we have the following.

\begin{corollary}[\cite{H:unified} for $sl_2$]
  For an integral homology $3$-sphere $M$, the LMO invariant totally
  determines the WRT invariant $\tau_M^\fg(\xi)$ for every simple Lie
  algebra and every root of unity $\xi\in\ZZg$.
\end{corollary}

It is still an open question whether the LMO invariant determines the WRT invariant for {\em rational} homology $3$-spheres. %

\subsubsection{Determination of the quantum invariants}

\begin{corollary}[\cite{H:unified} for $sl_2$]
  \label{r48a} For an integral homology $3$-sphere,
 $J_M$ is determined by the WRT function $\tau_M^\fg$.  (Thus $J_M$
  and $\tau^\g_M$ have the same strength in distinguishing two
  integral homology $3$-spheres.)  Moreover, both $J_M$ and $\tau^\g_M$
  are determined by the values of $\tau^\fg_M(\xi)$ for $\xi \in
  \calZ'$, where $\calZ'\subset \calZ$ is any infinite subset with at
 least one
  limit point in $\ZZ$ in the sense explained in Section \ref{sec.Habiro}.
\end{corollary}

For example,  the value of $\tau^\fg_M(\xi)$ at any root of unity
$\xi$ is determined by the values $\tau^\fg_M(\xi_k)$ at $\xi_k =
\exp(2\pi i/ 2^k)$ for infinitely many integers $k\ge0$.

\subsection{Formal construction of the unified invariant}

Here we outline the proof of Theorem \ref{r40a}. Since we are not able to directly generalize
the proof  of the case $\fg=sl_2$ in  \cite{H:unified}, we use another approach which involves deep
results in quantized enveloping algebras (quantum groups).
The conceptual definition of the unified invariant presented here is also different.

\subsubsection{First step: construction of $J_M$}
The first step is to construct an invariant $J_M\in\Zqh$ using the
quantum group $\Uh(\g)$ of $\g$.  Here we use neither the definition
of $\tau^\g_M(\xi)$ nor the quantum link invariants associated to
finite-dimensional representations of $\Uh(\g)$.  Instead, we use the
universal quantum invariant of bottom tangles and the {\em full twist
forms}, which are partially defined functionals $\cT_{\pm}$ on the
quantum group $\Uh(\g)$ and play a role of $\pm1$-framed surgery on
link components.

Every integral homology $3$-sphere $M$ can be obtained as the result
$S^3_L$ of surgery on $S^3$ along an algebraically split link $L$ with
framing $\pm 1$ on each component.  Here a link is said to be
algebraically split if the linking number between any two distinct
components is $0$.  Surgery on two algebraically split, $\pm 1$
framing links $L$ and $L'$ give the orientation-preserving
homeomorphic integral homology $3$-spheres if and only if $L$ and $L'$
are related by a sequence of Hoste moves (see Figure \ref{fig:FR})
\cite{H:kirby}.  Hence, in order to construct an invariant of integral homology
$3$-spheres, it suffices to construct an invariant of algebraically
split, $\pm 1$ framing links which is invariant under the Hoste moves.

To construct such a link invariant, we use the universal quantum
invariant of bottom tangles associated to the quantum group $\Uh(\g)$.
Here a bottom tangle is a tangle in a cube
consisting of arc components whose endpoints are on the bottom square
in such a way that the two endpoints of each component is placed side
by side (see Section \ref{sec:bottom-tangles}).  For an
$n$-component bottom tangle $T$, the universal $\g$ quantum invariant
$J_T=J^\g_T$ of $T$ is defined by using the universal $R$-matrix and
the ribbon element for the ribbon Hopf algebra structure of $\Uh(\g)$,
and takes values in the $n$-fold completed tensor power $\Uhg^\hon$.

The invariant $J_M\in\Zqh$ is defined as follows.  As above, let $L$ be
an $n$-component algebraically split framed link with framings
$\epsilon_1,\ldots,\epsilon_n\in\{\pm1\}$, and assume that $S^3_L\cong M$.
Let $T$ be an $n$-component bottom tangle whose closure is isotopic to
$L$, where the framings of $T$ are switched to $0$.  Define
\begin{gather}
  \label{e8}
  J_M:= (\cT_{\epsilon_1}\otimes\cdots\otimes\cT_{\epsilon_n})(J_T).
\end{gather}
Here $\cT_{\pm}\colon \Uhg\dashrightarrow \bbC[[h]]$ are {\em partial maps} (i.e.  maps defined on a submodule of $\Uh$)
defined formally by
\begin{gather*}
  \cT_{\pm}(x)=\langle x, \mathbf{r}^{\pm1}\rangle ,
\end{gather*}
where $\mathbf{r}\in\Uhg$ is the ribbon element, and
\begin{gather*}
\langle,\rangle  \colon \Uhg\ho\Uhg\dashrightarrow \Ch
\end{gather*}
is the quantum Killing form, which is a partial map.  The tensor
product $\cT_{\epsilon_1}\otimes\cdots\otimes\cT_{\epsilon_n}$ is not
well defined on the whole $\Uhg^\hon$, but is well defined on a
$\Zqh$-submodule $\tK_n\subset \Uhg^\hon$ and we have a $\Zqh$-module
homomorphism
\begin{gather*}
  \cT_{\epsilon_1}\otimes\cdots\otimes\cT_{\epsilon_n}\colon
  \tK_n\rightarrow \Zqh.
\end{gather*}
Here we regard $\Zqh$ as a subring of
$\bbC[[h]]$ by setting $q=\exp h$.
The module $\tK_n$ contains $J_T$ for all $n$-component, algebraically
split $0$-framed links $T$.    We will prove that $J_M$ as
defined in \eqref{e8} does not depend on the choice of $T$ and is
invariant under the Hoste moves. Hence $J_M\in\Zqh$ is an invariant of
an integral homology sphere.

One step in the construction of $J_M$ is to construct a certain integral form of the quantum group $\Uh(\fg)$ which is sandwiched between the Lusztig integral form and the De Concini-Procesi integral form.

\subsubsection{Second step: specialization to the WRT invariant at
  roots of unity}

The next step is to prove the specialization property
$\evxi(J_M)=\tau^\g_M(\xi)$ for each $\xi\in\ZZg$.  Once we have
proved this identity, uniqueness of $J_M$ follows since
every element of $\Zqh$ is determined by the values at infinitely many
$\xi\in\ZZ$ of prime power order, see Section \ref{sec.Habiro}.

\subsection{Organization of the paper}
In Section~\ref{sec:univ-invar-bott} we give a general construction of an invariant of {\em integral homology 3-spheres} from what we call a {\em core subalgebra} of a ribbon Hopf algebra. Section~\ref{sec:UhUq} introduces the quantized enveloping algebra $\Uh(\fg)$ and its subalgebras. In Section~\ref{sec:corealgebra} we  construct  a core subalgebra of the ribbon Hopf algebra $\UUh$, which is $\Uh$ with a slightly bigger ground ring. From the core subalgebra we get invariant $J_M$ of integral homology 3-spheres. In Section~\ref{sec:integcore} we construct an integral version of the core algebra. Section \ref{sec:nonc-grad} a (generically non-commutative) grading of the quantum group is introduced. In Section \ref{sec:integralM} we prove that $J_M\in \Zqh$. In Section \ref{sec.WRT} we show that the WRT invariant can be recovered from $J_M$, proving the main results. In Appendices we give an independent proof of a duality result of Drinfel'd and Gavarini and provide proofs of a couple of technical results used in the main body of the paper.

\subsection{Acknowledgements} The authors would like to thank H. Andersen, A.~Beliakova, A.~Brugui\`eres, C.~Kassel, G.~ Masbaum, Y.~Soibelman, S.~Suzuki, T.~Tanisaki, and N.~Xi for helpful discussions. Part of this paper was written while the second author was visiting Aarhus University, University of Zurich, RIMS Kyoto University, ETH Zurich, University of Toulouse, and he would like to thank these institutions for the support.

\np

\section{Invariants of integral homology $3$-spheres derived from ribbon Hopf algebras}
\label{sec:univ-invar-bott}

In this section, we give the part of the proofs of our main results,
which can be stated without giving the details of the structure of
the quantized enveloping algebra $\Uh=\Uh(\mathfrak g)$. We introduce the notion of a {\em core subalgebra} of a ribbon Hopf algebra and show that every core subalgebra gives rise to an invariant of {\em integral homology 3-spheres}.

\subsection{Modules over $\Ch$} \label{sec:ChModules}

Let $\Ch$ be the ring of formal power $\BC$-series in the variable $h$.

Note that $\Ch$ is a local ring, with maximal ideal $(h)=h\Ch$. An element $x = \sum x_k h^k\in \Ch$ is invertible if and only if the constant term $x_0$ is non-zero.

\subsubsection{$h$-adic topology, separation and completeness}
Let $V$ be a $\Ch$-module.  Then $V$ is equipped with the $h$-adic
topology given by the filtration $h^kV$, $k\ge0$. Any $\Ch$-module
homomorphism $f: V \to W$ is automatically continuous.  In general,
the $h$-adic topology of a $\Ch$-submodule $W$ of a $\Ch$-module $V$
is different from the topology of $W$ induced by the $h$-adic topology
of $V$.

Suppose $I$ is an index set. Let $V^I$ be the set of all collections
$(x_i)_{i\in I}$, $ x_i \in V$.  We say that a collection $(x_i)_{i\in
I} \in V^I$ is {\em $0$-convergent} in $V$ if for every positive
integer $k$, $x_i \in h^k V$ except for a finite number of $i\in I$.
In this case, the sum $\sum_{i\in I} x_i$ is convergent in the
$h$-adic topology of $V$.  If $I$ is finite, then any collection
$(x_i)_{i\in I}$ is $0$-convergent.

The {\em $h$-adic completion} $\hat{V}$ of $V$ is defined by
\begin{gather*}
  \hat{V}=\varprojlim_{k} V/h^kV.
\end{gather*}

A $\Ch$-module $V$ is {\em separated} if the natural map $V\to \hat{V}$ is injective,  which is equivalent to  $\cap_k h^k V = \{ 0\}$.
If $V$ is separated, we  identify $V$ with the image of the embedding $ V \hookrightarrow \hat V$.

A $\Ch$-module $V$ is {\em  complete} if the natural map
$V\to \hat{V}$ is surjective.

For a $\Ch$-submodule $W$ of a completed $\Ch$-module $V$, the {\em
  topological completion of $W$ in $V$} is the image of $\widehat W$
under the natural map $\widehat W \to \widehat V=V$. One should not
confuse the topological completion of $W$ and the topological closure
of $W$, the latter being the smallest closed (in the $h$-adic
topology) subset containing $W$. See Example \ref{exa:1} below.

\subsubsection{Topologically free modules}  \label{sec:top1}

For a vector space $A$ over $\BC$, let $A[[h]]$ denote the
$\Ch$-module of formal power series $\sum_{n\ge0} a_nh^n$, $a_n\in
A$.  Then $A[[h]]$ is naturally isomorphic to the $h$-adic completion
of $A\otimes_{\BC}\Ch$.

A $\Ch$-module $V$ is said to be {\em topologically free} if $V$ is
isomorphic to $A[[h]]$ for some vector space $A$.  A {\em topological
basis} of $V$ is the image by an isomorphism $A[[h]]\cong V$ of a
basis of $A\subset A[[h]]$.  The cardinality of a topological basis of
$V$ is called the {\em topological rank} of $V$.

It is known that a  $\Ch$-module is topologically free if and only if it is separated, complete, and torsion-free, see e.g. \cite[Proposition XVI.2.4]{Kassel}.

Let $I$ be a set.
Let $\sR^I=\prod_{i\in I}\Ch$ be the set of all collections $(x_i)_{i\in
I}$, $x_i \in \sR$.  Let $(\sR^I)_0\subset\sR^I$ be the $\Ch$-submodule
consisting of the $0$-convergent collections.
Then $(\sR^I)_0\cong(\BC I)[[h]]$ is topologically free, where $\BC I$
is the vector space generated by $I$.

Note that $\sR^I$ also is topologically free.  In fact, we have a $\Ch$-module isomorphism $\Ch^I
\cong \BC^I[[h]]$.  If $I$ is infinite, then the topological rank of $\sR^I$ is uncountable.

For $j\in I$, define a collection $\delta_j=((\delta_j)_i)_{i\in I}\in 
(\sR^I)_0
$ by
\be
\label{eq.deltai}
(\delta_j)_i =\delta_{j,i}=\begin{cases}
1 \quad & \text{if } \ i= j \\
0 &\text{if } \ i\neq j.
\end{cases}
\ee

Suppose $V$ is a topologically free $\Ch$-module with the isomorphism $f:
(\sR^I)_0 \to V$. Let $e(i)= f(\delta_i)\in V$. For $x\in V$, the
collection $(x_i)_{i\in I} = f^{-1}(x)$ is called the {\em coordinates} of
$x$ in the topological basis $\{ e(i) \}$. We have then \be
\label{eq.coor1} x = \sum_{i \in I} x_i e(i), \ee where the sum on the
right hand side converges to $x$ in the $h$-adic topology of $V$.

\subsubsection{Formal series modules}

A $\sR$--module $V$ is a {\em formal series $\sR$-module} if there is
a $\sR$-module isomorphism $f: \sR^I \to V$ for a {\em countable} set
$I$.

\begin{remark} Besides the $h$-adic topology, another natural topology on
 $\sR^I=\prod_{i\in I}\Ch$ is  the {\em product topology}.  (Recall
that the product topology of $\prod_{i\in I}\Ch$ is the coarsest
topology with all the projections $p_i:\prod_{i\in I}\Ch\to \Ch$ being
continuous.)
\end{remark}

Suppose $V$ is a formal series module, with an isomorphism $f: \sR^I
\to V$.  Let $e(i)= f(\delta_i)$, where $\delta_i$ is defined as in
\eqref{eq.deltai}. The set $\{e(i) \mid i \in I\}$ is called a {\em
formal basis} of $V$.

For $x\in V$ the collection $f^{-1}(x) \in \sR^I$ is called the {\em
coordinates} of $x$ in the formal basis $\{ e(i) \mid i \in
I\}$. Unlike the case of topological bases, in general the sum $
\sum_{i\in I} x_i e(i) $ does not converge in the $h$-adic topology of
$V$ (but does converge to $x$ in the product topology). However, it is
often the case that $V$ is a $\sR$-submodule of a bigger $\sR$-module
$V'$ in which $\{ e(i) \mid i\in I\}$ is 0-convergent. Then the sum $
\sum_{i\in I} x_i e(i) $, though not convergent in the $h$-adic
topology of $V$, does converge (to $x$) in the $h$-adic topology of
$V'$.

\begin{example}\label{exa:1} The following example is important for us.

Suppose $V$ is a topologically free
  $\sR$-module with a countable topological basis $\{e(i) \mid i\in I\}$. Assume that $a:I \to \Ch$ is a function such that $a(i) \neq 0$ for every $i\in I$ and
$(a(i))_{i\in I}$ is 0-convergent. Let $V(a)$ be the topological completion in $V$ of the $\Ch$-span of
$\{ a(i) e(i) \mid i \in I\}$. Then $V(a)$ is topologically free with
  $\{ a(i) e(i) \mid i \in I\}$ as a topological basis.

The submodule $V(a)$ is not closed in the $h$-adic topology of $V$.  The closure
$\overline{V(a)}$ of $V(a)$ in the $h$-adic topology is a formal
series $\Ch$-module, with an isomorphism
\begin{gather*}
  f\col \Ch^I \to V,\quad \delta_i\mapsto a(i)e(i).
\end{gather*}

The topology of $\overline{V(a)}$ induced by
the $h$-adic topology of $V$ is the product topology.

If $x\in V(a)$, then we have a unique presentation
\be
\label{eq.z1z}
x = \sum_{i\in I} x_i (a(i) e(i))
\ee
where $(x_i)_{i\in I} \in (\sR^I)_0$.

If $x\in \overline{ V(a)}$, then $x$ also has a unique presentation \eqref{eq.z1z}, with $(x_i)_{i\in I} \in \sR^I$.

\end{example}

\subsubsection{Completed tensor products}

For two complete $\Ch$-modules $V$ and $V'$, the {\em completed tensor
product} $V\ho V'$ of $V$ and $V'$ is the
$h$-adic completion of $V\otimes V'$, i.e.
\begin{gather*}
  V\ho V' = \varprojlim_{n} (V\otimes V')/h^n(V\otimes V').
\end{gather*}
Suppose  both $V$ and $V'$ are topologically free with topological bases $\{ b(i) \mid i\in I\}$
and $\{ b'(j) \mid j \in J\}$ respectively.
Then $V\ho V'$ is topologically free with a
topological basis naturally identified with $\{ b(i) \otimes b'(j) \mid i\in I, j \in J\}$.

\begin{proposition} \label{s.a2}
 Suppose $W_1,V_1,W_2,V_2$ are topologically free $\Ch$-modules, where $W_j$ is a submodule of $V_j$ for $j=1,2$.

  Then the natural maps $ W_1 \otimes W_2 \to V_1 \otimes V_2$  and
 $ W_1 \ho W_2 \to V_1 \ho V_2$ are injective.
\end{proposition}

\begin{proof}
The map  $W_1 \ot W_2 \to V_1 \otimes V_2$ is the composition of two maps  $W_1 \ot W_2 \to W_1 \otimes V_2$ and $W_1 \ot V_2 \to V_1 \otimes V_2$.  This reduces the proposition to the case  $W_2=V_2$, which we will assume.

 Let $\iota: W_1 \hookrightarrow V_1$ be the inclusion map. We need to show that $\iota \ot \id:W_1 \ot V_2
 \to V_1 \ot V_2$ and $\iota \ho \id:W_1 \ho V_2 \to
 V_1 \ho V_2$ are injective.

 Since $W_1\ot V_2$ is separated, we can consider  $W_1\ot V_2$ as a submodule of $W_1\ho V_2$. Then $\iota\ot \id$ is the restriction of $\iota \ho \id$. Thus, it is enough to show that $\iota \ho \id$
 is injective.

Suppose $x\in W_1 \ho V_2$ such that $(\iota \ho \id)(x) =0$. We have to show that $x=0$.

Let $\{ b(i), i \in I\}$ be a topological basis of $V_2$. Using a topological basis of $W_1$ one sees that $x$ has a unique presentation
\be
 x = \sum_{i\in I} x_i \otimes b(i),
 \label{e.dd2}
 \ee
where $x_i \in W_1$, and the
  collection $(x_i)_{i\in I}$ is
0-convergent in $V_1$. Then we have
$$ 0= (\iota \ho \id)(x)= \sum_{i\in I} \iota (x_i) \otimes b(i) \in V_1\ho V_2.$$
The uniqueness of the presentation of the form  \eqref{e.dd2} for elements in $V_1\ho V_2$ implies that $\iota(x_i)=0$ for every $i\in I$.
Because $\iota$ is injective, we have $x_i=0$ for every $i$. This means $x=0$, and hence $\iota\ho \id$ is injective.
\end{proof}

\subsection{Topological ribbon Hopf algebra}
\label{sec:ribbon-hopf-algebra}

In this paper, by a {\em topological Hopf algebra} $\mathscr H
=(\mathscr H ,{\boldsymbol \mu} ,{\boldsymbol \eta} ,\Delta
,{\boldsymbol \epsilon} ,S)$ we mean a topologically free $\Ch$-module
$\mathscr H$
of countable topological rank,
 together with $\Ch$-module homomorphisms
\begin{gather*}
  {\boldsymbol \mu} \col \mathscr H \ho \mathscr H \rightarrow \mathscr H ,\quad {\boldsymbol \eta} \col \modk \rightarrow \mathscr H ,\quad \Delta \col \mathscr H \rightarrow \mathscr H \ho \mathscr H ,\quad
  {\boldsymbol \epsilon} \col \mathscr H \rightarrow \modk ,\quad S\col \mathscr H \rightarrow \mathscr H
\end{gather*} which are the
multiplication, unit, comultiplication, counit and antipode of
$\mathscr H $, respectively, satisfying the usual axioms of a Hopf
algebra.
For simplicity, we include invertibility of the antipode in the axioms
of Hopf algebra. We denote $\eta(1)$ by $1\in \mathscr H$.

Note that $\sH$ is a $\Ch$-algebra in the usual (non-complete) sense,
although $\sH$ is not a $\Ch$-coalgebra in general.  A (left)
$\sH$-module $V$ (in the usual sense) is said to be {\em topologically
free} if $V$ is topologically free as a $\Ch$-module. In that case, by
continuity the left action $\sH\otimes V\to V$ induces a $\Ch$-module
homomorphism
\begin{gather*}
  \sH\ho V\to V.
\end{gather*}

For details on topological Hopf algebras and topologically free modules, see e.g. \cite[Section XVI.4]{Kassel}.

Let $\boldmu^{[n]}: \scH^{\ho n} \to \scH$ and
$ \Delta^{[n]} : \scH \to \scH^{\ho n}$ be respectively the multi-product and the multi-coproduct defined by
 \begin{align*}
 \boldmu^{[n]} &=   \boldmu   ( \id \ho \boldmu)  \dots ( \id ^{\ho (n-3)} \ho \boldmu) ( \id ^{\ho (n-2)} \ho \boldmu)\\
 \Delta^{[n]} &= ( \id ^{\ho (n-2)} \ho \Delta)  ( \id ^{\ho (n-3)} \ho \Delta)  \dots  ( \id \ho \Delta) \Delta
 \end{align*}
with the convention that $\Delta^{[1]}= \boldmu^{[1]}=\id$, $\Delta^{[0]}=\boldsymbol \epsilon$, and $ \boldmu^{[0]}=\boldsymbol \eta$.

A {\em universal $R$-matrix} \cite{Drinfeld} for $\mathscr H $ is an invertible element
$\cR=\sum\alpha \otimes \beta \in \mathscr H \ho \mathscr H $ satisfying
\begin{gather*}
  \cR\Delta (x)\cR^{-1}=\sum x_{(2)}\otimes x_{(1)}\quad \text{for $x\in \mathscr H $},\\
  (\Delta \otimes \id)(\cR)=\cR_{13}\cR_{23},\quad (\id\otimes \Delta )(\cR)=\cR_{13}\cR_{12},
\end{gather*}
where $\Delta (x)=\sum x_{(1)}\otimes x_{(2)}$ (Sweedler's notation), and
$\cR_{12}=\sum \alpha \otimes \beta \otimes 1$, $\cR_{13}=\sum \alpha \otimes 1\otimes \beta $,
$\cR_{23}=\sum1\otimes \alpha \otimes \beta $.
A Hopf algebra with a universal $R$-matrix is called a {\em
  quasitriangular Hopf algebra}.
The universal $R$-matrix satisfies
\begin{gather}
  \cR^{-1}= (S\otimes \id)(\cR)=(\id\otimes S^{-1})(\cR), \quad
  ({\boldsymbol \epsilon} \otimes \id)(\cR)=(\id\otimes {\boldsymbol \epsilon} )(\cR)=1 \notag\\
  (S\ot S) (\cR) = \cR    \label{eq.SR}.
\end{gather}

A quasitriangular Hopf algebra $(\mathscr H ,\cR)$ is called a {\em ribbon Hopf
  algebra} \cite{RT1} if it is equipped with a {\em
  ribbon element}, which is defined to be an invertible, central
  element $\modr \in \mathscr H $ satisfying
\begin{gather*}
  \modr ^2 =\bu S(\bu ),\quad S(\modr )=\modr ,\quad {\boldsymbol \epsilon} (\modr )=1, \quad
  \Delta (\modr )=(\modr \otimes \modr )(\cR_{21}\cR)^{-1},
\end{gather*}
where $\bu=\sum S(\beta )\alpha \in \mathscr H $ and $\cR_{21}=\sum \beta \otimes \alpha \in \mathscr H ^{\ho2}$.

The element $\modg :=\bu\modr ^{-1}\in \mathscr H $, called the {\em balanced element}, satisfies
\begin{gather*}
\Delta(\modg)= \modg \ot \modg,\quad  S(\modg)= \modg^{-1},\quad   \modg x\modg ^{-1} = S^2(x)\quad \text{for $x\in \mathscr H $}.
\end{gather*}

See \cite{Kassel,Ohtsuki,Turaev} for more details on quasitriangular
and ribbon Hopf algebras.

\subsection{Topologically free $\scH$-modules}The ground ring $\Ch$ is
considered as a topologically free $\scH$-module, called the {\em trivial module}, by the action of the co-unit:
$$ a \cdot x = \boldep(a) \, x.$$

Suppose $V,W$ are topologically free $\scH$-modules. Then $V \ho W$
has the structure of $\scH \ho \scH$-module, given by
$$ (a \ot b) \cdot (x \ot y) = (a \cdot x) \ot (b \cdot y).$$
Using the comultiplication, $ V \ho W$ has an
$\scH$-module structure given by
$$ a \cdot ( x \ot y) := \Delta(a) \cdot (x \ot y)
=\sum a_{(1)}x\ot a_{(2)}y.$$

An element $x\in V$ is called {\em invariant} (or {\em $\scH$-invariant}) if for every $a \in
\sH$,
$$ a \cdot x = \boldsymbol \epsilon (a) \, x.$$
The set of invariant elements of $V$ is denoted by $V^\inv$.
The following is standard and well-known.

\begin{proposition} \label{r.inva}
Suppose $V$ and $W$ are topologically free $\scH$-modules, and $f: V\ho W \to \Ch$ is a $\Ch$-module homomorphism.

(a) An element $x \in V \ho W$ is invariant if and only if for every $a\in \scH$,
\be
\label{eq.inva}
(S(a)  \otimes 1)\cdot x = (1 \ot \, a) \cdot x.
\ee

(b) Dually, $f$ is an $\scH$-module homomorphism if and only if
for every $a\in \scH$ and $x \in V\ho W$,
$$ f[ (a \ot 1 ) \cdot (x) ] = f [(1 \ot S(a)) \cdot (x)].$$

(c)
Suppose $f$ is an $\scH$-module homomorphism, and $x\in V$ is invariant. Then the $\Ch$-module homomorphism
$$ f_x: W \to \Ch, \quad y \mapsto f(x \ot y),$$
is an $\scH$-module homomorphism.

(d)
 Suppose $g: V \to \Ch$ is an $\scH$-module homomorphism. Then for every $i$ with $1\le i \le n$,
 $$ (\id^{\ho (i-1)} \ho \,  g \,  \ho \id^{\ho (n-i)} )\left( (V^{\ho n})^\inv \right) \subset ( V^{\ho (n-1)})^\inv.$$

\end{proposition}

\begin{proof} (a) Suppose one has \eqref{eq.inva}. Let $ a\in \scH$ with $\Delta(a) = \sum  a_{(1)} \otimes a_{(2)}$. Assume $x = \sum x' \ot x''$. By definition,
\begin{align*}
a \cdot x & = \sum ( a_{(1)}\ot a_{(2)})  \cdot (x'  \ot   x'') = \sum ( a_{(1)}\ot 1) (1 \ot  a_{(2)})  \cdot (x'  \ot   x'') \\
&= \sum ( a_{(1)}\ot 1)     \cdot (x'  \ot  a_{(2)} \cdot x'')   \\
&=   \sum ( a_{(1)}\ot 1)  \left( S( a_{(2)}) \cdot x'  \ot  x''  \right)\\
&=   \sum ( a_{(1)}  S( a_{(2)}) \ot 1) \cdot   ( x' \ot x'') = \boldep (a) \, x,
\end{align*}
which show that $x$ is invariant.

Conversely, suppose $x$ is invariant.  From axioms of a Hopf algebra,
$$
 1 \ot a = \sum  (S( a_{(1)}) \ot 1)\,  \Delta  (a_{(2)}).
$$
Applying both sides to $x$,  we have
\begin{align*}
(1 \otimes \, a) \cdot x & =  \sum  (S(a_{(1)})   \ot 1 )\cdot \left( a_{(2)} \cdot  x  \right) \\
&=  \sum (S(a_{(1)})   \ot 1 )\cdot \left( \boldep (a_{(2)}) \, x  \right)   \quad \text{by invariance}\\
&=  (S( a)  \ot 1 )\cdot  x,
\end{align*}
which proves \eqref{eq.inva}.

(b) The proof of (b) is similar and is left for the reader. Statement (b) is mentioned in textbooks \cite[Section 6.20]{Jantzen} and \cite[Section 6.3.2]{KlS}.

(c) Let $a \in \scH$ and $y \in W$. One has
\begin{align*}
f_x ( a \cdot y) &= f ( x \ot ( a \cdot y))  \\
&= f( S^{-1} (a) \cdot x \ot y) \quad \text{by part (b)} \\
&= \boldep \left ( S^{-1} (a)   \right) f(x \ot y) \quad \text{by invariance of $x$} \\
&= \boldep (a) f_x(y) \quad \text{since }\boldep \left ( S^{-1} (a)   \right)  = \boldep (a).
\end{align*}
This proves $f_x$ is an $\scH$-module homomorphism.

(d) The map $\tilde g:= \id^{\ho (i-1)} \ho \,  g \,  \ho \id^{\ho (n-i)}$ is also an $\scH$-module homomorphism. Hence for every $a\in \scH$ and $x \in (V^{\ho n})^\inv$,
$$ a \cdot \tilde g(x) = \tilde g (a \cdot x)= \tilde g (\boldsymbol{\ve}(a) x) = \boldsymbol{\ve}(a)\tilde g ( x).$$
This shows $\tilde g(x)$ is invariant.
\end{proof}

\subsection{Left image of an element} \label{sec.leftmodule}
Let $V$ and $W$ be topologically free $\Ch$-modules.

Suppose $x\in V \ho W$. Choose a topological basis $ \{ e(i) \mid i
\in I\}$ of $W$. Then $x$ can be uniquely presented as an $h$-adically
convergent sum
\be \label{eq.sd8} x = \sum_{i\in I} x_i \otimes e(i),
\ee where  $\{x_i \in V \mid i \in I\}$ is 0-convergent.  The {\em left image}
$V_x$ of
$x\in V\ho W$ is the topological closure (in the $h$-adic topology of $V$) of the $\Ch$-span of $ \{ x_i
\mid i\in I\}$.
It is easy to show that $V_x$ does not depend on the
choice of the topological basis $ \{ e(i) \mid i \in I\}$ of $W$.

\begin{proposition}
\label{r.stable}
Suppose $V,W$ are topologically free $\scH$-modules.  Let $x \in V \ho
W$ and let $V_x\subset V$  be the left image of $x$.

 (a) If $x$ is $\scH$-invariant, then $V_x$ is $\scH$-stable, i.e. $\scH \cdot V_x \subset V_x$.

 (b) If $(f\ho g)(x)=x$, where $f:V\to V$ and $g: W \to W$ are $\Ch$-module isomorphisms, then $f(V_x)=V_x$.
\end{proposition}
\begin{proof} Let $ \{ e(i) \mid i \in I\}$ be a topological basis of $W$, and $x_i$ be as in \eqref{eq.sd8}.

(a)
By Proposition \ref{r.inva}(a), the $\scH$-invariance of $x$ implies that for every $a \in \scH$,
\be
\label{eq.sd8a}
\sum_{i\in I} a \cdot x_i \ot e(i)  = \sum_{j\in I}  x_j \ot S^{-1}(a) \cdot e(j).
\ee
Using the topological basis $\{ e(i)\}$, we have the structure constants
$$ S^{-1}(a) \cdot e(j) = \sum_{i\in I}  a_j^i \, e(i),$$
where $a_j^i \in \Ch$. Using this expression in \eqref{eq.sd8a},
$$ \sum_{i\in I} a \cdot x_i \ot e(i) = \sum_{i\in I} \sum_{j\in I} a_j^i x_j \ot e(i).$$
The uniqueness of expression of the form \eqref{eq.sd8} shows that
$$ a \cdot x_i = \sum_{j\in I} a_j^i x_j \in V_x.$$
Since the $\Ch$-span of $x_i$ is dense in $V_x$ and the action of $a$ is continuous in the $h$-adic topology of $V$, we have $a \cdot V_x \subset V_x$.

(b) Using $x= (f\ho g)(x)$, we have
$$ x= \sum_i f(x_i) \ot g(e_i).$$
Since $g$ is a $\Ch$-module isomorphism, $\{ g(e_i)\}$ is a
topological basis of $W$. It follows that $V_x$ is the closure of the $\Ch$-span of $\{f(x_i) \mid i\in I\}$. At the same time $V_x$ is the closure of the $\Ch$-span of $\{x_i \mid i\in I\}$.
Hence, we have $f(V_x)= V_x$.
\end{proof}

\subsection{Adjoint action and ad-invariance}\label{sec:adjaction}
 Suppose $\mathscr H$ is a  topological ribbon Hopf algebra.
The  (left) adjoint action
\begin{gather*}
  \ad\col \mathscr H \ho \mathscr H \rightarrow \mathscr H
\end{gather*}
of $\mathscr H $ on itself is defined by
\begin{gather*}
  \ad(x\otimes y) = \sum x_{(1)} y S(x_{(2)}).
\end{gather*}
It is convenient to use an infix notation for $\ad$:
\begin{gather*}
  x\trr y = \ad(x\otimes y).
\end{gather*}
We regard $\mathscr H $ as a (topologically free) $\mathscr H $-module  via
the  adjoint action, unless otherwise stated. Then $\mathscr H ^{\ho
  n}$ becomes a topologically free $\scH$-module, for every $n \ge 0$. The action of $x\in \scH$ on $y\in \scH^{\ho n}$ is denoted by $x\trr_n y$.

To emphasize the adjoint action, we say that a $\Ch$-submodule $V \subset \scH^{\ho n}$ is {\em ad-stable} if $V$ is an $\scH$-submodule of $\scH^{\ho n}$.
An element $x\in \scH^{\ho n}$ is {\em ad-invariant} if it is an
invariant element of $\scH^{\ho n}$ under the adjoint actions. For
example, an element of $\scH$ is ad-invariant if and only if it is central.

For ad-stable submodules $V  \subset \scH^{\ho n}$ and $W \subset \scH^{\ho m}$, a $\Ch$-module homomorphism $f : V \to W$ is {\em ad-invariant} if $f$ is an $\scH$-module homomorphism.

In particular, a linear functional $f: V\to \Ch$, where $V \subset \scH^{\ho n}$, is ad-invariant if $V$ is ad-stable and for $x\in\sH$, $y\in V$,
\begin{gather*}
  f(x\trr_n  y) = \boldsymbol\epsilon(x)f(y).
\end{gather*}

The main source of ad-invariant linear functionals comes from quantum traces.
Here the quantum trace
$\trq^V\col \mathscr H \rightarrow \modk $ for a finite-dimensional representation $V$
(i.e. a topologically free $\mathscr H $-module  of finite
topological rank)
is
defined by
\begin{gather*}
  \trq^V(x) =\tr^V(\modg x)\quad \text{for $x\in \mathscr H $},
\end{gather*}
where $\tr^V$ denotes the trace in $V$. It is known that $\trq^V\colon \sH \to \BC[[h]]$ is ad-invariant.

\subsection{Bottom tangles}
\label{sec:bottom-tangles}
Here we recall the definition of bottom tangles from
\cite[Section 7.3]{H:bottom}.

An $n$-component {\em bottom tangle} $T=T_1\cup \cdots \cup T_n$ is a framed
tangle in a cube consisting of $n$ arc components $T_1,\ldots ,T_n$ such that
all the endpoints of the $T_i$ are in a bottom line and that for each
$i$, the component $T_i$ runs from the $2i$th endpoint to the $(2i-1)$th endpoint,
where the endpoints are counted from the left.  See Figure \ref{fig:bt} (a)
\FIG{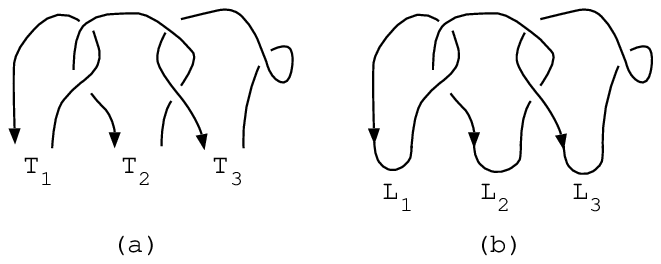}{(a) A $3$-component bottom tangle $T=T_1\cup T_2\cup T_3$.  (b) Its
closure $\cl(T)=L_1\cup L_2\cup L_3$.}{height=30mm} for an
example.  In figures, framings are specified by the blackboard framing
convention.

The {\em closure} $\cl(T)$ of $T$ is the $n$-component, oriented,
ordered framed link in $S^3$, obtained from $T$ by pasting a
``$\cup $-shaped tangle'' to each component of $L$, as depicted in Figure
\ref{fig:bt} (b).  For any oriented, ordered framed link $L$, there is
a bottom tangle whose closure is isotopic to $L$.

The {\em linking matrix} of a bottom tangle $T=T_1\cup \cdots\cup T_n$ is
defined as that of the closure $T$.  Thus the linking number of $T_i$
and $T_j$, $i\neq j$, is defined as the linking number of the
corresponding components in $\cl(T)$, and the framing of $T_i$ is
defined as the framing of the closure of $T_i$.

A link or a bottom tangle is called {\em algebraically-split} if the
linking matrix is diagonal.

\subsection{Universal invariant and quantum link invariants}
\label{sec:uni11}
Reshetikhin and Turaev \cite{RT1} constructed  quantum invariants of
framed links colored by finite dimensional representations of a ribbon
Hopf algebra, e.g. the quantum group $\Uh(\mathfrak g)$.  Lawrence,
Reshetikhin, Ohtsuki and Kauffman
\cite{Lawrence:90,Reshetikhin,Ohtsuki:colored,Kauffman} constructed
``universal quantum invariants'' of links and tangles with values
in (quotients of) tensor powers of the ribbon Hopf algebra, where the
links and tangles are not colored by representations.  We recall here
construction of link invariants via the universal invariant of bottom
tangles. We refer the readers to \cite{H:bottom} for details.

Fix a ribbon Hopf algebra $\sH$.
Let $T$ be a bottom tangle with $n$ components. We choose a
  diagram for $T$, which is obtained from copies of {\em fundamental
  tangles}, see Figure \ref{fig:fundamental}, by pasting horizontally
  and vertically.
\FIG{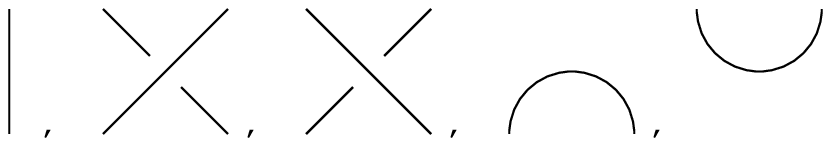}{Fundamental tangles: vertical line, positive and
  negative crossings, local minimum and local maximum.  Here the
  orientations are arbitrary.}{height=12mm}
For each copy of fundamental tangle in the diagram of $T$, we put
elements of $\sH$ with the rule described in Figure
\ref{fundamental2}.
\begin{figure}
    \begin{center}\input{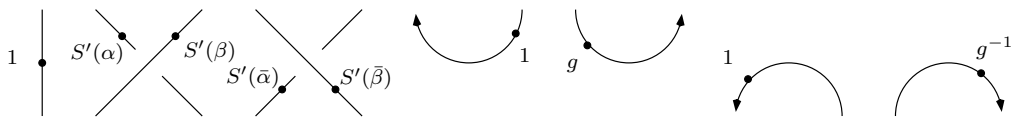}\end{center}
    \caption{How to put elements of $\sH$
on the strings.  Here $R= \sum \al \ot \beta$ and $R^{-1} = \sum \bar \al \ot \bar \beta$. For each string in the positive and the negative
crossings, ``$S'$'' should be replaced by $\id$ if the string is
oriented downward, and by $S$ otherwise.}
    \label{fundamental2}
  \end{figure}

We set
\begin{equation*}
   J_T := \sum x_1\otimes \cdots \otimes x_n\in \mathscr H ^{\ho n},
\end{equation*}
where each $x_i$ is the product of the elements put on the $i$-th
component $T_i$, with product taken in the order reversing the order of the orientation.  The (generally infinite) sum comes from the decompositions of  $\cR^{\pm 1}$ as
(infinite) sums of tensor products.  It is known that $J_T$
gives an isotopy invariant of
bottom tangles, called the universal invariant of $T$.
Moreover, $J_T$ is ad-invariant
 (\cite{Kerler}, see also \cite{H:bottom}).

Let
$\chi _1,\ldots ,\chi _n \col \mathscr H \rightarrow \modk $ be ad-invariant. In other words, $\chi _1,\ldots ,\chi _n $ are $\mathscr H $-module homomorphisms.
As explained
in \cite{H:bottom}, the quantity
\begin{gather*}
  (\chi _1\ho \cdots \ho \chi _n)(J_T)\in \modk
\end{gather*}
is a {\em link invariant} of the closure link $\cl(T)$ of $T$.

In particular, if $\chi _1,\ldots ,\chi _n$ are the quantum traces
$\trq^{V_1},\ldots ,\trq^{V_n}$ in finite-dimensional representations
$V_1,\ldots ,V_n$, respectively, then
\begin{gather*}
  (\trq^{V_1}\ho \cdots \ho \trq^{V_n})(J_T) \in \modk
\end{gather*}
is the quantum link invariant for $\cl(T)$ colored by the
representations $V_1,\ldots ,V_n$.

\subsection{Mirror image of bottom tangles} \label{sec.mirror}
\begin{definition}\label{d.mirror}
  A {\em mirror homomorphism} of a topological ribbon Hopf algebra $\scH$ is
an $h$-adically continuous $\BC$-algebra homomorphism $\tphi: \scH \to \scH$ satisfying
\begin{align}
(\tphi \ho \tphi) \cR  &= \cR_{21}^{-1} \label{eq.mi1} \\
(\tphi \ho \tphi)^2 \cR &= \cR  \label{eq.mi2}\\
\tphi (\mathbf g)&= \mathbf g.  \label{eq.mi3}
\end{align}

\end{definition}
In general, such a $\tphi$ is not a $\Ch$-algebra homomorphism. In fact, what we will have in the future is $\tphi(h)=-h$.

For a bottom tangle $T$ with diagram $D$ let the {\em mirror image of\ } $T$ be the bottom tangle whose diagram is obtained from $D$ by switch over/under crossing at every crossing.

\begin{proposition}\label{r.mirror1}
Suppose $\tphi$ is a mirror homomorphism of a ribbon Hopf algebra $\scH$. If $T'$ is the mirror image of an $n$-component bottom tangle $T$, then
$$ J_{T'} = \tphi^{\ho n} (J_T).$$
\end{proposition}
\begin{proof} Let $D$ be a diagram of $T$.
By rotations if necessary at  crossings, we can assume that the two strands at each crossing of $D$ are oriented downwards. Then at each  crossing, we assign $\al$ and $\beta$ to the strands if the  it is positive, and we assign $\bar\beta$ and $\bar \al$ to the strands if it
is negative,  at the same spots where we would assign $\al$ and $\beta$ if the  crossing were positive, see Figure \ref{fig:mirrorproof}.
\FIG{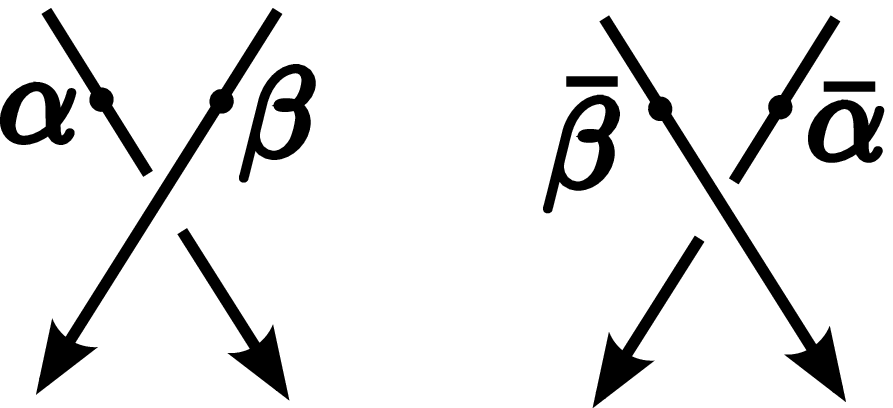}{Assignments on positive and negative crossings.}{height=12mm}
Here $\cR= \sum \al \ot \beta$ and $\cR^{-1} = \sum   \bar \al \ot \bar \beta$.
Conditions \eqref{eq.mi1}, \eqref{eq.mi2} implies that
$$ \sum \bar \beta \ot \bar \al = \sum \tphi( \al) \ot \tphi(\beta), \quad \sum  \al \ot \beta = \sum \tphi( \bar \beta) \ot \tphi(\bar \al).$$
Together with \eqref{eq.mi3} this shows that the assignments to strands of diagram $D'$ of $T'$ can be obtained by applying $\tphi$ to the corresponding assignments to strands of $D$. Since $\tphi$ is a $\BC$-algebra homomorphism, we get $ J_{T'} = \tphi^{\ho n} (J_T).$
\end{proof}

\subsection{Braiding and transmutation}
\label{transmutation}
Let $\cR=\sum \al\otimes \beta$ be the $R$-matrix. The {\em braiding} for $\mathscr H $ and its inverse
\begin{gather*}
  {\boldsymbol{\psi}} ^{\pm 1}\col \mathscr H \ho \mathscr H \rightarrow \mathscr H \ho \mathscr H
\end{gather*}
are given by
\begin{gather}
  \label{e12}
  {\boldsymbol{\psi}} (x\otimes y) = \sum (\beta \trr y) \otimes (\alpha \trr x), \quad
  {\boldsymbol{\psi}}^{-1} (x\otimes y) = \sum (S(\alpha )\trr y)\otimes (\beta \trr x).
\end{gather}

The maps ${\boldsymbol \mu} ,{\boldsymbol \eta} ,{\boldsymbol \epsilon} $ are $\mathscr H $-module homomorphisms.  In particular,
we have
\begin{gather}
  \label{e7}
  x\trr yz= \sum (x_{(1)}\trr y)(x_{(2)}\trr z)\quad \text{for $x,y,z\in \mathscr H $}.
\end{gather}
 In general, $\Delta $ and $S$ are not  $\mathscr H $-module homomorphisms, but
the following twisted versions of $\Delta $ and $S$ introduced
by Majid (see \cite{Majid1,Majid2})
\begin{gather*}
  \bD\col \mathscr H \rightarrow \mathscr H \ho \mathscr H ,\quad \bS\col \mathscr H \rightarrow \mathscr H
\end{gather*}
defined by
\begin{gather}
  \label{e10}
  \bD(x)
  = \sum x_{(1)}S(\beta )\otimes  (\alpha \trr x_{(2)})
  = \sum (\beta \trr x_{(2)})\otimes  \alpha  x_{(1)},\\
  \label{e11}
  \bS(x)
  = \sum \beta S(\alpha \trr x)
  = \sum S^{-1}(\beta \trr x)S(\alpha ),
\end{gather}
for $x\in \mathscr H $, are $\mathscr H $-module homomorphisms.  Geometric interpretations of
$\bD$ and $\bS$ are given in \cite{H:bottom}.

\begin{remark}
$\bH:=(\mathscr H ,{\boldsymbol \mu} ,{\boldsymbol \eta} ,\bD,{\boldsymbol \epsilon} ,\bS)$ forms a braided
Hopf algebra in the braided category of topologically free $\sH$-modules, called the {\em
  transmutation} of $\mathscr H $ \cite{Majid1,Majid2}.
\end{remark}

\subsection{Clasp bottom tangle}
\label{sec:clasp-bottom-tangle}

Let $C^+$ be the {\em clasp tangle} depicted in Figure
\ref{fig:c+}.
\FIG{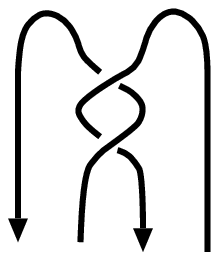}{The clasp tangle $C^+$.}{height=16mm}
We call $\modc =J_{C^+}\in \mathscr H ^{\ho2}$ the {\em clasp element} for $\mathscr H $.  With $\cR=\sum \al \otimes \beta= \sum \al' \otimes \beta'$, we have
\begin{gather}
  \label{e43}
  \modc =(S\ho \id)(\cR_{21}\cR)=\sum S(\alpha )S(\beta ')\otimes \alpha '\beta .
\end{gather}

Let $C^-$ be the mirror image of $C^+$, see Figure \ref{fig:c-}, and $\modc ^-=J_{C^-}\in \mathscr H ^{\ho2}$.

Let $(C^+)'$ be the tangle obtained by reversing the orientation of the second component of $C^+$, and $(C^+)''$ be the result of putting $(C^+)'$ on top of
the tangle $\ \raisebox{-2mm}{\incl{6mm}{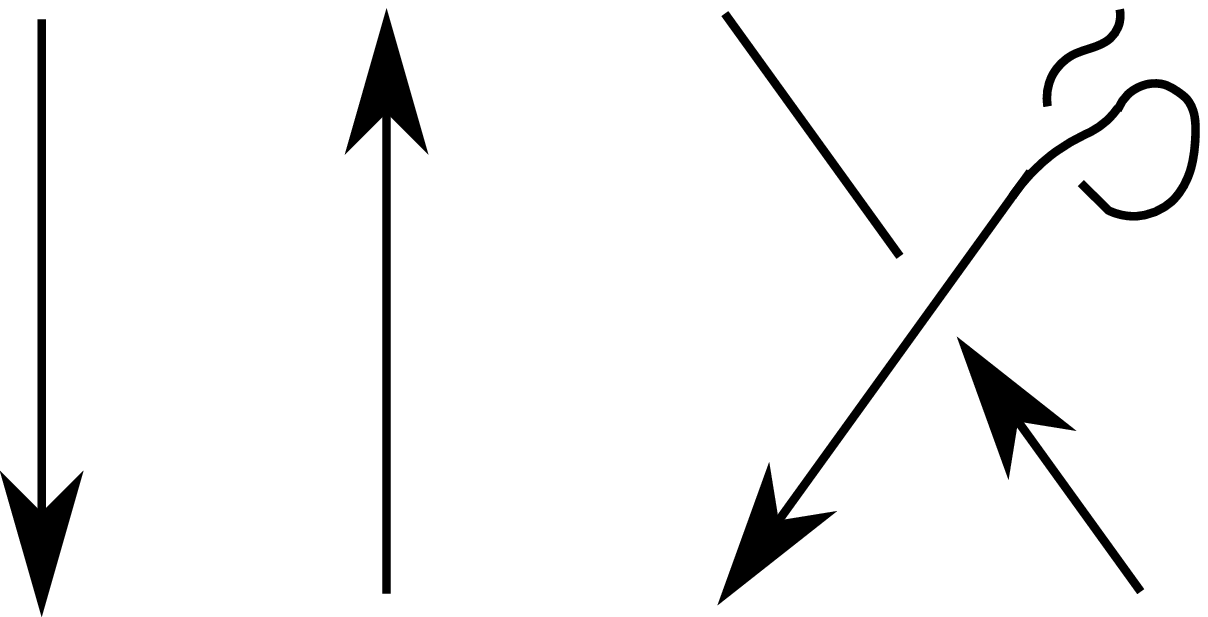}}$ , see Figure
\ref{fig:c-}.
\FIG{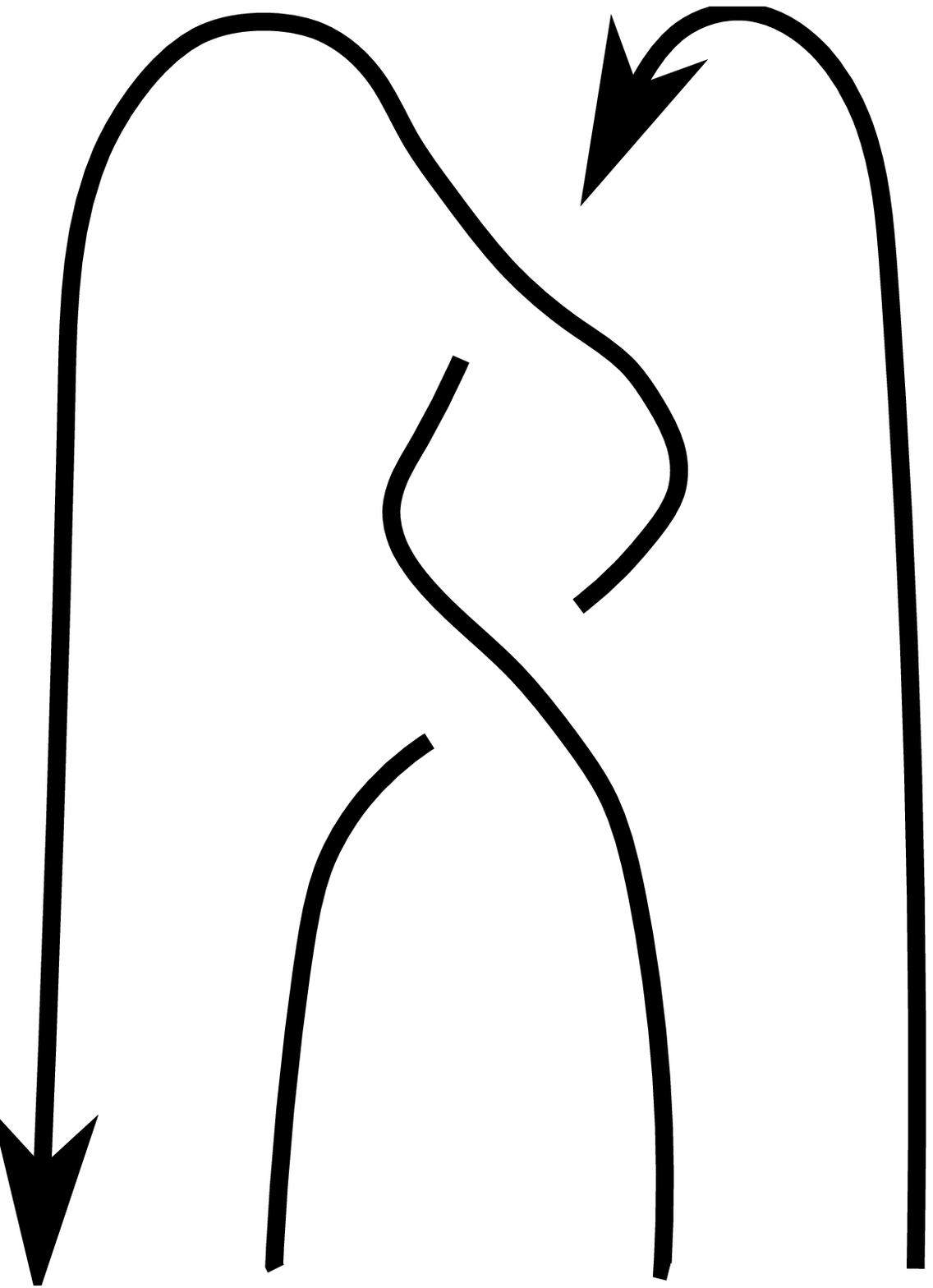}{ The negative clasp $C^-$ (on the left) and $(C^+)''$ are
 isotopic.}{height=20mm}
By the geometric interpretation of $\bS$, see \cite[Formula 8--10]{H:bottom}, we have
$$ J_{(C^+)''} = (\id \ho \bS) J_{C^+}.$$

Since $(C^+)''$ is isotopic to $C^-$, we have
\begin{gather}
  \label{e77}
  \modc ^-=
  (\id\ho \bS)(\modc ).
\end{gather}

\subsection{Hoste moves} \label{sec.hoste}
It is known that every integral homology $3$-sphere can be obtained by surgery on $S^3$ along
an algebraically split link with $\pm 1$ framings.

The following refinement of the Kirby--Fenn--Rourke theorem on framed
links was first essentially conjectured by Hoste \cite{Hoste}.  (Hoste
stated it in a more general form related to Rolfsen's calculus for
rationally framed links.)

\begin{theorem}[\cite{H:kirby}]
  \label{r2}
  Let $L$ and $L'$ be two (non-oriented, non-ordered)
  algebraically-split $\pm 1$-framed links in $S^3$.  Then $L$ and
  $L'$ give orientation-preserving homeomorphic results of surgery if
  and only if $L$ and $L'$ are related by a sequence of ambient
  isotopy and {\em Hoste moves}.  Here a Hoste move is a Fenn--Rourke (FR)
  move between two algebraically-split, $\pm 1$-framed links, see
  Figure \ref{fig:FR}.
\FIG{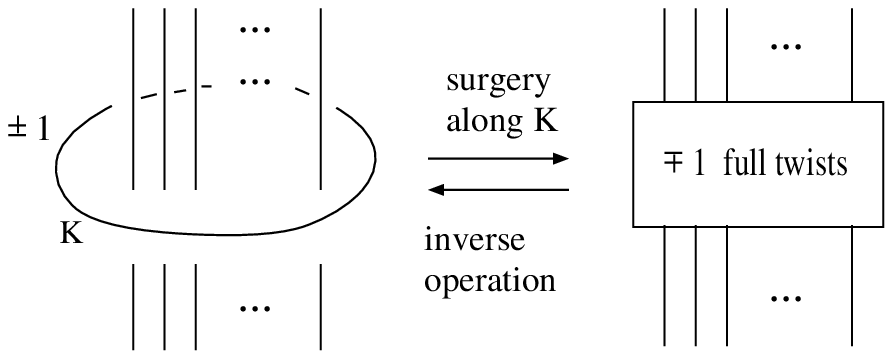}{A Hoste move (including the case when there are no vertical
  strands).  Here both these two framed links are algebraically-split
  and $\pm1$-framed.}{height=30mm}
\end{theorem}

 Theorem \ref{r2} implies that, to construct an invariant of integral
homology spheres, it suffices to construct an invariant of
algebraically-split, $\pm 1$-framed links which is invariant under the
Hoste moves.

\begin{lemma} \label{r2a}
Suppose $f$ is an invariant of oriented, unordered, algebraically-split $\pm 1$-framed links which is invariant under Hoste moves. Then $f(L)$ does not depend on the orientation of the link $L$. Consequently,
 $f$ descends to an invariant of integral homology 3-spheres, i.e. if the results of surgery along two oriented, unordered, algebraically-split $\pm 1$-framed links $L$ and $L'$ are homeomorphic integral homology 3-spheres, then $f(L)= f(L')$.
\end{lemma}
\begin{proof}
Suppose $K$ is a component of $L$ so that $L=  L_1 \cup K$. We will show that $f$ does not depend on the orientation of $K$ by induction on the unknotting number of $K$.

First assume that $K$ is an unknot. We first apply the Hoste move to $K$, then apply the
Hoste move in the reverse way, obtaining  $L_1\cup (-K)$, where $-K$ is the orientation-reversal of $K$.
This shows $f(L_1 \cup K) = f (L_1 \cup (-K))$.

Suppose $K$ is an arbitrary knot with positive unknotting number.
We can use a Hoste move to realize a self-crossing change of $K$, reducing the unknotting number. Induction on the unknotting number shows that $f$ does not depend on the orientation of $K$.
\end{proof}

\subsection{Definition of the invariant $J_M$ for the case when the ground ring is a field} \label{sec.simple}
In this subsection, we explain a construction of an invariant of
integral homology spheres associated to a ribbon Hopf algebra over a
field $\bfk$, equipped with ``full twist forms''. In this section (Section \ref{sec.simple}), and only in this section, we will assume that $\sH$ is
 a ribbon Hopf algebra over a field
$\bfk$.  This assumption simplifies the definition of the invariant.

\subsubsection{Full twist forms}
\label{sec:twist-forms}
Recall that $\modc =J_{C^+}\in \sH \ot \sH$ is the universal invariant of the clasp bottom tangle, and $\br$ is the ribbon element.

A pair
of ad-invariant linear functionals $\cT_+,\cT_-\col \mathscr H
\rightarrow \bfk$ are called {\em full twist forms} for $\mathscr H $
if
\begin{gather}
  \label{e15}
  (\cT_\pm \ot \id)  (\modc) = \modr ^{\pm 1}.
\end{gather}

 The following lemma essentially shows how the universal link invariant behaves under the Hoste move, if there are full twist forms.
\begin{lemma}  \label{r47}
 Suppose that a ribbon Hopf algebra $\mathscr H $ admits full twist forms $(\cT_+, \cT_-)$.
  Let $T=T_1\cup \cdots \cup T_n$ be an $n$-component bottom tangle ($n\ge 1$)
  such that the first component
  $T_1$ of $T$ is a $1$-component trivial bottom tangle (see Figure
  \ref{trivial}).
\begin{figure}
    \begin{center}\input{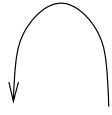}\end{center}
    \caption{The trivial bottom tangle.}
    \label{trivial}
  \end{figure}
  Let $T'=T'_2\cup \cdots \cup T'_n$ be the
  $(n-1)$-component bottom tangle obtained from $T\setminus T_1=T_2\cup \cdots \cup T_n$ by
  surgery along the closure of $T_1$ with framing $\pm 1$ (see Figure
  \ref{twist1}).
\begin{figure}
    \begin{center}\input{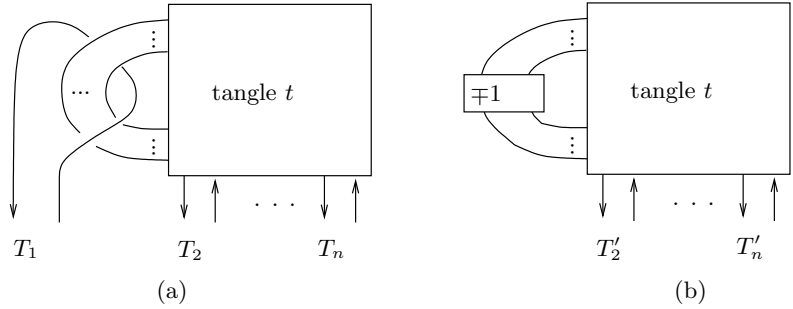}\end{center}
    \caption{(a) The tangle $T$.  (b) The tangle $T'$.}
    \label{twist1}
  \end{figure}
  Then we have
  \begin{gather}
    \label{e17}
    J_{T'} = (\cT_\pm \ot \id^{\ot (n-1)})(J_T).
  \end{gather}

\end{lemma}

\begin{proof}
  In this proof we use the universal invariant for tangles that are not bottom tangles.  For details, see \cite[Section 7.3]{H:bottom}.
  
 If $T_{p,q}$ is a $(p+q+1)$-component tangle as depicted in
  Figure~\ref{twist2}(a) with $p,q\ge 0$, then we have
  \begin{gather*}
    J_{T_{p,q}} = (\id\ot \id^{\ot p}\ot S^{\ot
    q})(\id\ot \Delta ^{[p+q]})(\modc ).
  \end{gather*}
\begin{figure}
    \begin{center}\input{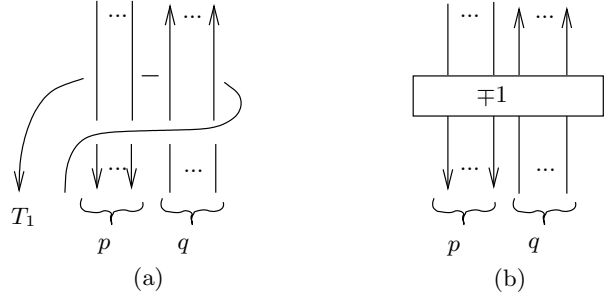}\end{center}
    \caption{(a) The tangle $T_{p,q}$.  (b) The tangle
  $T'_{p,q;\pm1}$.}
    \label{twist2}
  \end{figure}

  The tangle $T'_{p,q;\pm 1}$ obtained from $T_{p,q}\setminus T_1$ by surgery
  along the closure of the first component $T_1$ of $T_{p,q}$ with
  framing $\pm 1$  (see Figure \ref{twist2}(b)) has the universal invariant
  \begin{gather*}
    J_{T'_{p,q;\pm 1}} = (\id^{\ot p}\ot S^{\ot q})\Delta ^{[p+q]}(\modr ^{\pm 1}).
  \end{gather*}
  Since $\cT_\pm $ is a full twist form, it follows that
  \begin{gather*}
    (\cT_\pm \ot \id^{\ot (p+q)})(J_{T_{p,q}})=J_{T'_{p,q}}.
  \end{gather*}

  The general case follows from the above case and functoriality of
  the universal invariant, since $T$ can be obtained from some
  $T_{p,q}$ by tensoring and composing appropriate tangles.
\end{proof}

\subsubsection{Invariant of integral homology 3-spheres} We will show here that
a ribbon Hopf algebra $\mathscr H $ with full twist forms $\cT_\pm$ gives rise to an
invariant of integral homology spheres.

Suppose $T$ is an $n$-component bottom tangle with zero linking matrix
and $\ve _1,\ldots ,\ve _n\in \{1,-1\}$. Let $M=M(T;\ve _1,\ldots ,\ve
_n)$ be the oriented 3-manifold obtained by surgery on $S^3$ along the
framed link $L=L(T;\ve _1,\ldots ,\ve _n)$, which is the closure link
of $T$ with the framing on the $i$-th component switched to
$\ve_i$. Since $L$ is an algebraically split link with $\pm 1$ framing
on each component, $M$ is an integral homology 3-sphere. Every
integral homology 3-sphere can be obtained in this way.

\begin{proposition} \label{r44}
 Suppose $\scH$ is a ribbon Hopf algebra with full twist forms $\cT_\pm$, and $M=M(T;\ve _1,\ldots ,\ve _n)$ is an integral homology 3-sphere. Then
 $$ J_M:=(\cT_{\epsilon _1}\ot \dots \ot \cT_{\epsilon _n})(J_T)\in \bk$$
 is an invariant of $M$. In other words, if $M(T;\ve _1,\ldots ,\ve _n)\cong M(T';\ve'_1,\dots,\ve'_{n'})$, then
 $$ (\cT_{\epsilon _1}\ot \dots \ot \cT_{\epsilon _n})(J_T) = (\cT_{\epsilon '_1}\ot \dots \ot \cT_{\epsilon' _{n'}})(J_{T'}).$$
\end{proposition}

\begin{proof}
  Since $\cT_\pm $ are ad-invariant, $(\cT_{\epsilon _1}\ot \dots
  \ot \cT_{\epsilon _n})(J_T)$ depends only on
  $\epsilon_1,\ldots,\epsilon_n$ and the oriented, ordered framed link $\cl(T)$, but not on the choice of $T$, see
  e.g. \cite[Section 11.1]{H:bottom}. This shows $(\cT_{\epsilon _1}\ot \dots \ot \cT_{\epsilon _n})(J_T)$ is an invariant of framed link $L(T;\ve_1,\dots,\ve_n)$.

  We now show that $(\cT_{\epsilon _1}\ot \dots \ot \cT_{\epsilon _n})(J_T)$ does not depend on the order of the components of $L$. Suppose $L=L(T;\ve_1,\dots,\ve_n)$ and $L'$ is the same $L$, with the orders of the $(i+1)$-th and $(i+2)$-th switched. Then $L'= L(T';\ve'_1,\dots,\ve'_{n})$, where $T'$ is  $T$ on top of a simple braid of bands which switches the $i+1$ and $i+2$ components, see Figure \ref{fig:braid1}.
  Also, $\ve'_j= \ve_j$ for $j\neq i+1, i+2$, and $\ve_{i+1}'=\ve_{i+2}$, $\ve_{i+2}'= \ve_{i+1}$.
\FIG{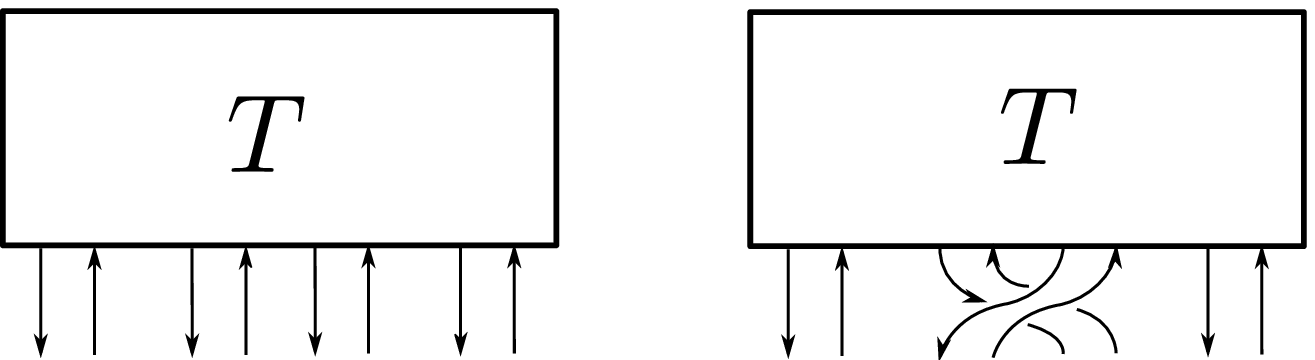}{Modification of a bottom tangle with a braid of bands.}{height=20mm}

According to the geometric interpretation of the braiding \cite[Proposition 8.1]{H:bottom},
$$ J_{T'} = (\id^{\ot i} \ot \boldpsi \ot \id^{\ot n-i-2}) (J_T).$$

 By \eqref{e12},
 $$\boldpsi(x \otimes y) = \sum  (\beta \tri y) \ot (\al \tri x),\quad   \text{ where  }\cR= \sum \al \ot \beta.$$
    Since $(\boldep \ot  \boldep)(\cR)= 1$ and $\cT_\pm$ are  ad-invariant,
    $$(\cT_{\epsilon _1}\ot \dots \ot \cT_{\epsilon _n})(J_T) = (\cT_{\epsilon _1'}\ot \dots \ot \cT_{\epsilon _n'})(J_{T'}).$$
     Thus, $(\cT_{\epsilon _1}\ot \dots \ot \cT_{\epsilon _n})(J_T)$ is an invariant of
 oriented, unordered framed links.

  By Proposition
  \ref{r47}, $(\cT_{\epsilon _1}\ot \dots \ot \cT_{\epsilon _n})(J_T)$ is invariant under the Hoste moves. Lemma \ref{r2a} implies that $(\cT_{\epsilon _1}\ot \dots \ot \cT_{\epsilon _n})(J_T)$ descends to an invariant of integral homology 3-spheres.
\end{proof}

\subsubsection{Examples of full twist forms: Factorizable case} \label{sec:factorizable-case}
A finite-dimensional, quasitriangular Hopf algebra over a field $\bfk$
is said to be  {\em factorizable} if the clasp element $\bc \in  \scH \otimes_\bk \scH$ is non-degenerate in the sense that there exist bases $\{\bc'(i), i \in I \}$ and $\{ \bc''(i), i \in I\}$ of $\scH$ such that
$$ \bc = \sum _{i \in I} \bc'(i) \otimes \bc''(i).$$
 This
definition of factorizability is
equivalent to the original definition by Reshetikhin and
Semenov-Tian-Shansky \cite{Reshetikhin-SemenovTianShansky}.

Suppose $\scH$ is a factorizable ribbon Hopf algebra. The non-degeneracy condition shows that
there is a unique bilinear form, called the {\em clasp form},
$$\sL \col \mathscr H \otimes
\mathscr H \rightarrow \bfk $$
such that for every $x\in \scH$,
\begin{align}
 (\sL \ot \id) (x \ot \bc)  = x  \label{e15a}, \quad  (\id \ot \sL) (\bc \ot x) = x.
\end{align}
    Using the ad-invariance of
$\modc$, one
can show that  $\sL \col \mathscr H \otimes
\mathscr H \rightarrow \bfk $ is ad-invariant. Since $\br^{\pm 1}$ are ad-invariant, the
form $\cT_\pm\col\scH\to \bk$ defined by
 \begin{gather}
  \label{e16}
  \cT_\pm (x) :=  \sL( \modr ^{\pm 1} \ot x),
\end{gather}
is ad-invariant and satisfies \eqref{e15} due to \eqref{e15a}.  Hence
 $\cT_+$ and $\cT_-$ are full twist forms for $\scH$,
 and defines an invariant of integral homology 3-sphere according to Proposition  \ref{r44}.

\begin{remark}
  \label{r38}
  Given a finite-dimensional, factorizable, ribbon Hopf algebra
  $\mathscr H $, one can construct the {\em Hennings invariant} for
  closed $3$-manifolds
  \cite{Hennings,Kauffman-Radford,Ohtsuki:3mfds,Kerler,Lyubashenko,Sawin,Virelizier,H:bottom}.
  The invariant given in Proposition  \ref{r44} constructed from the
  full twist forms in \eqref{e16} is equal to the Hennings invariant.
\end{remark}

\subsection{Partially defined twist forms and invariant $J_M$}\label{sec.partii}
 Let us return to the case when $\scH$ is a ribbon Hopf algebra over
 $\Ch$. Recall that $\scH$ is a topologically free $\scH$-module with
 the adjoint action. In general $\scH$ does not admit full twist forms.

In the construction of the invariant of integral homology 3-spheres in Proposition \ref{r44}, one first constructs the universal invariant of algebraically split tangles $J_T$, then feeds the result to the functionals $\cT_{\epsilon _1}\otimes \cdots \otimes \cT_{\epsilon _n}$
 which come from the twist forms $\cT_\pm$. We will show that the
 conclusion of Proposition \ref{r44} holds true if the twist forms $\cT_\pm$ are defined on a submodule large enough so that the domain of $\cT_{\epsilon _1}\otimes \cdots \otimes \cT_{\epsilon _n}$ contains all the values of $J_T$, with $T$ algebraically split bottom tangles.

 \subsubsection{Partially defined twist forms} Suppose $\sX\subset \scH$ is a topologically free $\Ch$-submodule. By Proposition  \ref{s.a2} all the natural maps  $\sX^{\ot n} \to \sX^{\ho n}\to \scH^{\ho n}$ and
 $ \sX \ho \scH^{\ho (n-1)} \to \scH^{\ho n}$ are injective. Hence we will consider $\sX^{\ot n}$, $\sX^{\ho n}$, and $\sX \ho \scH^{\ho (n-1)}$ as submodules of $\scH^{\ho n}$. This will explain the meaning of statements like ``$ \bc \in \sX \ho \scH$''.

 \begin{definition}
 A twist system $\cT=(\cT_\pm, \sX)$ of a topological ribbon Hopf algebra $\scH$ consists of a topologically free $\Ch$-submodule %
  $\sX \subset \scH$ and a pair of $\Ch$-linear functionals $\cT_\pm : \sX \to \Ch$
 satisfying the following conditions.

 (i) $\sX$ is ad-stable (i.e. $\sX$ is stable under the adjoint action of $\scH$) and $\cT_\pm$ are ad-invariant.

  (ii)  $\bc \in \sX \ho \scH$.

 (iii) One has
 $$ (\cT_\pm \ho \id)(\bc) = \br^{\pm 1}.$$

 \end{definition}

Recall that for an $n$-component bottom tangle $T$ with zero linking matrix and $\ve _1,\ldots ,\ve _n\in \{1,-1\}$, $M(T;\ve _1,\ldots ,\ve _n)$ is the integral homology sphere obtained by surgery on $S^3$ along the framed link $L(T;\ve _1,\ldots ,\ve _n)$, which is the closure link of $T$ with the framing on the $i$-th component switched to $\ve_i$.

 \begin{proposition}\label{rah1}
 Suppose $\cT=(\cT_\pm, \sX)$ is a twist system of a topological ribbon Hopf algebra $\scH$ such that
 $J_T \in \sX^{\ho n}$ for any $n$-component algebraically split 0-framed bottom tangle $T$. Let $M=M(T;\ve _1,\ldots ,\ve _n)$ be an integral homology 3-sphere. Then
 $$ J_M:=(\cT_{\epsilon _1}\ho \dots \ho \cT_{\epsilon _n})(J_T)\in \Ch$$
 is an invariant of $M$. In other words, if $M(T;\ve _1,\ldots ,\ve _n)= M(T';\ve'_1,\dots,\ve'_{n'})$, then
 $$ (\cT_{\epsilon _1}\ho \dots \ho \cT_{\epsilon _n})(J_T) = (\cT_{\epsilon '_1}\ho \dots \ho \cT_{\epsilon' _{n'}})(J_{T'}).$$
 \end{proposition}

 \begin{proof} First we show the following claim, which is a refinement of
 Lemma \ref{r47}.

 {\em Claim.}
   \label{r4}
    Let $T$ and $T'$ be tangles as in Lemma \ref{r47}.   Then $J_T \in {\sX \ho \mathscr H^{\ho (n-1)}}$, and
   \begin{gather}
     \label{e18}
     J_{T'} = (\cT_\pm \ho \id^{\ho (n-1)})(J_T)\in \mathscr H ^{\ho(n-1)}.
   \end{gather}

 \begin{proof}[Proof of Claim] If $T_{p,q}$ is a $(p+q+1)$-component tangle as depicted in
  Figure~\ref{twist2}(a) with $p,q\ge 0$, then we have
  \begin{gather*}
    J_{T_{p,q}} = (\id\ho \id^{\ho p}\ho S^{\ho
      q})(\id\ho \Delta ^{[p+q]})(\modc ).
  \end{gather*}
 Since $\bc \in \sX \ho \scH$, we have
 $$ J_{T_{p,q}} \in \sX \ho \scH^{\ho p+q}.$$

 Since $T$ is obtained from $T_{p,q}$  by tensoring and composing appropriate tangles which do not involve
 the first component, we also have
 $$ J_T \in \sX \ho \scH^{\ho n}.$$

  The remaining part of the proof follows exactly the proof of
  Proposition \ref{r47}.  One first verifies
    the case of $T_{p,q}$ using conditions (ii) and (iii) in the definition of twist system, from which the general case follows. This proves the claim.
 \end{proof}

 Using the ad-invariance of $\cT_\pm$ and \eqref{e18}, one can repeat verbatim the proof of Proposition \ref{r44}, replacing $\ot$ by $\ho$ everywhere, to get  Proposition \ref{rah1}.
\end{proof}

\subsubsection{Values of the universal invariant of algebraically split tangles} In Proposition \ref{rah1}, we need  $J_T\in \sX^{\ho n}$ for an $n$-component bottom tangle $T$ with zero linking matrix. To help proving statement like that, we use the following result.

Let $\modK _n\subset \mathscr \scH^{\ho n}$, $n\ge 0$, be a family of subsets. A $\BC [[h]]$-module homomorphism $f\col \Uhx a\rightarrow \Uhx b$,
$a,b\ge 0$, is said to be {\em $(\modK _n)$-admissible} if we have
\begin{gather}
  \label{e34}
  f_{(i,j)}(\modK_{i+j+a})\subset \modK_{i+j+b}.
\end{gather}
for all $i,j\ge 0$. Here $f_{(i,j)}:= \id^{\ho i} \ho f \ho \id^{\ho j}$.

\begin{proposition}[Cf. Corollary 9.15 of \cite{H:bottom}]
  \label{r26}
  Let $\modK _n\subset \mathscr H ^{\ho n}$, $n\ge 0$, be a family of subsets such that

 (i) $1_\modk \in \modK _0$, $1_{\scH}\in \modK_1$, $\modb \in \modK_3$,

 (ii) for $x \in \modK_n$  and $y\in \modK_m$ one has $ x\ot y \in\modK_{n+m}$, and

 (iii)
 each of ${\boldsymbol \mu}
  ,{\boldsymbol{\psi}} ^{\pm 1},\bD,\bS$
 is $(\modK_n)$-admissible.

Then, we have $J_T\in \modK _n$ for any $n$-component algebraically split, $0$-framed bottom tangle $T$.
\end{proposition}
Here $\modb\in \Uh^{\ho 3}$ is the universal invariant of the Borromean bottom tangle depicted in Figure~\ref{fig:borromean}\FIG{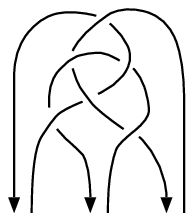}{The Borromean tangle
    }{height=15mm}.

\subsection{Core subalgebra} \label{sec:coresub}
We define here  a  {\em core subalgebra of a topological ribbon Hopf algebra}, and show that every core subalgebra gives rise to an invariant of integral homology 3-spheres.

In the following we use overline to denote the closure in the $h$-adic topology of $\scH^{\ho n}$.

A {\em topological Hopf subalgebra} of a topological Hopf
algebra $\scH$ is a $\Ch$-subalgebra $\scH'\subset\scH$ such that
$\scH'$ is topologically free as a $\Ch$-module and
\begin{gather*}
  \Delta(\scH')\subset \scH'\ho\scH',\quad
  S^{\pm1}(\scH')\subset \scH'.
\end{gather*}
In general, $\scH'$ is not closed in $\scH$.

\begin{definition}
A topological Hopf subalgebra $\sX \subset \scH$ of a topological
ribbon Hopf algebra $\scH$ is called a {\em core subalgebra} of $\scH$
if

(i) $\sX$ is $\scH$-ad-stable, i.e. it is an $\scH$-submodule of $\scH$,

(ii) $\cR \in \oXtwo$ and $\mathbf g \in \oX$,  and

(iii) The clasp element $\bc$, which is contained in $\oXtwo$ by (ii)
(see below),
has a presentation
\be
 \bc = \sum_{i\in I} \bc'(i) \otimes \bc''(i),
 \label{eq.clasp5}
 \ee
where each  of the two sets  $\{\bc'(i) \mid i \in I\}$ and $\{\bc''(i) \mid i \in I \}$  is
\begin{itemize}
\item 0-convergent in $\scH$, and
\item a topological basis of $\sX$.
\end{itemize}
\end{definition}

Some clarifications are in order. As
 a topological Hopf subalgebra, $\sX$ is topologically free as a $\Ch$-module. By Proposition \ref{s.a2}, all the natural maps $\sX^{\ot n}\to \sX^{\ho n} \to \scH^{\ho n} $ are injective. We will  consider $\sX^{\ot n}$ as a $\Ch$-submodule of $\scH^{\ho n}$ in (ii) above when we take its closure in the $h$-adic topology of $\scH^{\ho n}$. Furthermore, since $\cR^{-1} =(S \ho \id)(\cR)$ and $\mathbf g^{-1}= S(\mathbf g)$, condition (ii) implies that $\cR^{\pm 1} \in \oXtwo$ and $\mathbf g^{\pm1} \in \oX$.
  Since $J_T$, the universal invariant of an $n$-component bottom tangle $T$, is built from $\cR^{\pm1}$ and $\mathbf g^{\pm1}$, condition (ii) implies that $J_T \in \overline{\sX^{\ot n}}$. In particular, $\bc \in \oXtwo$.

\begin{remark} A core subalgebra has properties similar to, but still
 different from,  those of both a minimal Hopf algebra \cite{Radford}
 and a factorizable Hopf algebra
 \cite{Reshetikhin-SemenovTianShansky}. Note that the notions of a
 minimal algebra and a factorizable algebra were introduced only for
 the case when the ground ring is a field. Over $\Ch$ the picture is
 much more complicated. For example, in \cite{Radford} it was shown
 that a minimal algebra over a field is always
 finite-dimensional. Here our core algebras are of infinite rank over $\Ch$.
\end{remark}

From now on we fix a core subalgebra $\sX$ of a topological ribbon
Hopf algebra $\scH$.

\begin{lemma}\label{r.inclusion} Suppose
$f: \scH\to \scH$ is a $\Ch$-module homomorphism such that $f(\sX)\subset \sX$. Then $f(\oX)\subset \oX$.
In particular, $\oX$ is ad-stable.
\end{lemma}
\begin{proof}
Since $f$ is continuous in the topology of $\scH$, we have $f(\oX)\subset \oX$.
\end{proof}

\subsubsection{Clasp form associated to a core subalgebra}
Suppose $\sX \subset \scH$ is a core subalgebra with the  presentation \eqref{eq.clasp5} for $\bc$.
Since  $\{ \bc'(i)\}$ is a topological basis of  $\sX$,
every $y\in \sX$ has its coordinates $y'_i\in \Ch$ such that
$$ y= \sum_{i\in I} y'_i \bc'(i),$$
where $(y'_i)_{i\in I}$ is 0-convergent, i.e. $(y'_i)_{i\in I} \in
(\Ch^I)_0$. The map $y \mapsto (y'_i)$ is a $\Ch$-module isomorphism from $\sX$ to $(\Ch^I)_0$.

 The set $\{ \bc''(i)\}$ is a formal basis of  $\oX$, which is a formal series $\Ch$-module.  Every $x \in \oX$  has its coordinates $x''_i\in \Ch$ such that in the $h$-adic topology of $\scH$,
\be
\label{eq.xx}
 x= \sum_{i\in I} x''_i \bc''(i),
 \ee
where $(x''_i)_{i\in I} \in \Ch^I$. The map $x \mapsto (x_i'')$ is an
$\Ch$-module isomorphism from $\oX$ to $\Ch^I$.

Define a bilinear form $ \sL=\la ., . \ra : \oX \ot \sX \to \Ch$, called the {\em clasp form}, by
$$ \la x, y \ra := \sum_{i\in I} x''_i \, y'_i.$$
The sum on the right hand side is convergent since $(y'_i)_{i\in I}$ is 0-convergent.
The bilinear form is defined so that $\{ \bc''(i) \}$ and $\{ \bc'(i)\}$ are dual to each other:
\be
\label{eq.bracket}
\la \bc''(i) , \bc'(j) \ra = \delta_{ij}.
\ee

By continuity (in the $h$-adic topology), $\sL$ extends to a
$\Ch$-module map, also denoted by $\sL$,
$$ \sL: \oX \ho \sX \to \Ch.$$

The following lemma says that the above  bilinear form is dual to $\bc$.
\begin{lemma} \label{r.ccc}
(a) One has $ \bc \in (\sX \ho \scH) \cap (\scH \ho \sX)$.

(b) For every $x\in \oX$ and $ y\in \sX$ one has
\begin{align}
(\sL \ho \id)( x \ot \bc ) & = x   \label{eq.du1}\\
( \id \ho \sL)(  \bc  \ot y ) & = y
\label{eq.du2}
\end{align}
\end{lemma}
\begin{remark}
By part (a), $ \bc \in \sX \ho \scH$, hence $ x\otimes \bc \in \oX \ho\sX \ho \scH$. This is the reason why the left hand side of \eqref{eq.du1} is well-defined as an element of $\scH$. Similarly the left hand side of \eqref{eq.du2} is well-defined. With this well-definedness, all the proofs will be the same as in the case of finite-dimensional vector spaces over a field.
\end{remark}
\begin{proof} (a) Since $\{ \bc''(i)\}$ is 0-convergent in $\scH$, $\bc = \sum_i \bc'(i) \ot \bc''(i) \in \sX \ho \scH$. Similarly, $\bc \in \scH \ho \sX$.

(b)
Suppose $x$ has the presentation \eqref{eq.xx}. By \eqref{eq.bracket}, we have
\be
 \la x, \bc'(i) \ra = x''(i).
 \ee
 Thus, we have
 \be
 \label{eq.xx1}
 x = \sum_i \la x, \bc'(i) \ra \, \bc''(i),
 \ee
which  is \eqref{eq.du1}. The  identity \eqref{eq.du2} is proved similarly.
\end{proof}

Because $\br^{\pm 1} \in \oX$, one can define the
$\Ch$-module homomorphisms
\be
\label{eq.cTdef}
\cT_\pm: \sX \to \Ch \qquad \text{ by} \quad \cT_\pm(y) = \la \br^{\pm 1}, y \ra.
\ee

Since $\bc$ is ad-invariant, one can expect the following.
\begin{lemma} \label{r.invar}
(a) The clasp form  $\sL: \oX \ho \sX \to \Ch$ is ad-invariant, i.e. it  is an $\scH$-module homomorphism.

(b) The maps $\cT_\pm: \sX \to \Ch$ are ad-invariant.
\end{lemma}
\begin{proof} (a) By Proposition \ref{r.inva}(b), $\sL$ is ad-invariant if and only if for every $a \in \scH$, $x\in \oX$, and $y\in \sX$,
\be
\label{eq.da11}
\la a \tri x, y \ra = \la x, S(a) \tri y \ra,
\ee
which we will prove now.

Since $\bc= \sum_i \bc'(i) \ot \bc''(i)$ is ad-invariant, by Proposition \ref{r.inva}(a),
$$ \sum_i S(a) \tri \bc'(i) \ot \bc''(i) = \sum_i  \bc'(i) \ot a \tri \bc''(i).$$
Tensoring with $x$ on the left, and applying $\sL \otimes \id$,
\begin{align*}
\sum_i \la x, S(a) \tri \bc'(i) \ra \, \bc''(i)& = \sum_i \la x, \bc'(i) \ra \, a \tri \bc''(i)\\
&= \sum_i x''(i) \, a\tri \bc''(i) \\
&= a \tri \left(\sum_i x''(i)\,  \bc''(i)   \right) = a\tri x.
\end{align*}
Tensoring on the right with $\bc'(j)$ then applying $\sL$, one has
$$ \la x, S(a) \tri \bc'(j) \ra = \la a \tri x,  \bc'(j) \ra,$$
which is \eqref{eq.da11} with $y= \bc'(j)$. Since $\{ \bc'(j)\}$ is a topological basis of $\sX$, \eqref{eq.da11} holds for any $y \in \sX$.

(b) follows from Proposition \ref{r.inva}(c).
\end{proof}

\begin{proposition} \label{r.vv1}
Suppose $f: \scH \to \scH$ and $g: \scH \to \scH$ are $\Ch$-module isomorphisms such that $f(\sX)= \sX$, $g(\sX)=\sX$, and
$(f\ho g)(\bc) = \bc$. Then $g(\oX)=\oX$, and
for every $x \in \oX$, $y\in \sX$, one has
\be
\label{eq.oper}
 \la g(x), f(y) \ra = \la x, y \ra.
 \ee
\end{proposition}

\begin{proof}By Lemma \ref{r.inclusion},  $g^{\pm1}(\oX)\subset \oX$. It follows that $g(\oX)= \oX$.
One has
$$ \bc = \sum_i \bc'(i) \ot \bc''(i) = \sum_i f(\bc'(i)) \ot g(\bc''(i)).$$
Since $g(x) \in \oX$, one can replace $x$ by $g(x)$ in \eqref{eq.du1},
\begin{align*}
g(x) & = (\sL \ho \id)( g(x) \ot \bc ) = \sum_i (\sL \ho \id)\left ( g(x) \ot f(\bc'(i)) \ot g(\bc''(i)) \right) \\
&= \sum_i \la g(x), f(\bc'(i)) \ra\,  g(\bc''(i))  \\
&= g \left( \sum_i \la g(x), f(\bc'(i)) \ra\,  g(\bc''(i)) \right)
\end{align*}
Injectivity of $g$ implies
$$ x =  \sum_i \la g(x), f(\bc'(i)) \ra\,  \bc''(i).$$
Comparing with \eqref{eq.xx1} we have, for every $i\in I$,
$$ \la g(x), f(\bc'(i)) \ra =   \la x, \bc'(i) \ra, $$
which shows that \eqref{eq.oper} holds for $y=\bc'(i)$, $i\in I$. Hence, \eqref{eq.oper} holds for every $y\in \sX$  since $\{\bc'(i) \}$ is a topological basis of $\sX$.
\end{proof}

\subsubsection{Twist system from core subalgebra}
\begin{proposition} \label{r.twistcore}
The collection $\cT=(\cT_\pm, \sX)$ is a twist system for $\scH$.
\end{proposition}
\begin{proof}
By definition, $\sX$ is ad-stable. By Lemma \ref{r.invar}, $\cT_\pm$ are ad-invariant.
By Lemma \ref{r.ccc}(a), $\bc\in \sX \ho \scH$. Finally, Identity \eqref{eq.du1} with $x= \br^{\pm1}$ gives
$$ (\cT_\pm \ho \id) \bc = \br^{\pm1}.$$
This shows $\cT=(\cT_\pm, \sX)$ is a twist system.
\end{proof}

\subsection{From core subalgebra to invariant of integral homology 3-spheres}
\label{sec.JM}

\begin{theorem}
\label{r.JMdef}
Let $\sX$ be a core subalgebra of a topological ribbon Hopf algebra $\scH$, with its associated $\scH$-module homomorphisms $\cT_\pm:\sX \to \Ch$. Assume $T$ is an $n$-component bottom tangle with 0 linking matrix, $\ve_i\in \{\pm 1\}$ for $i=1,\dots, n$, and
$M=M(T;\ve_1,\dots,\ve_n)$ is the integral homology 3-sphere obtained from $S^3$ by surgery along $\cl(T)$, with framing of the $i$-th component changed to $\ve_i$.

Then  $J_T\in \sX^{\ho n}$, and
$$ J_M:= (\cT_{\ve_1} \ho \dots \ho \cT_{\ve_n}) (J_T) \in \Ch$$
 defines an invariant of integral homology 3-spheres.
\end{theorem}

By Propositions \ref{rah1} and \ref{r.twistcore}, to prove Theorem \ref{r.JMdef},  it is sufficient to show the following
\begin{proposition}\label{r.Jvalue0}
Suppose $\sX$ is a core subalgebra of a topological ribbon Hopf algebra $\scH$ and $T$ is  an $n$-component bottom tangle with 0 linking matrix. Then $J_T \in \sX^{\ho n}$.
\end{proposition}
The rest of this section is devoted for a proof of this proposition, based on Proposition \ref{r26}.

\subsubsection{$(\sX^{\ho n})$-admissibility} The following lemma follows easily from the definition.
\begin{lemma}\label{r.admiss1}
Suppose $f: \scH^{\ho a} \to \scH^{\ho b}$ is a $\Ch$-module
homomorphism having a presentation as an $h$-adically convergent sum
$f= \sum_{p \in P}f_p$ such that for each $p$, $f_p(\sX^{\ho a}) \subset
\sX^{\ho b}$, where $P$ is a countable set.
(Here, the sum $f$ being $h$-adically convergent means that for each $j\ge0$
we have $f_p(\scH^{\ho a})\subset h^j\scH^{\ho b}$ for all but
finitely many $p\in P$.)
Then $f$
is $(\sX^{\ho n})$-admissible.
\end{lemma}
\begin{proposition} \label{r.admiss5}
Each of $\boldmu,\boldpsi^{\pm1}, \uD, \uS$ is $(\sX^{\ho n})$-admissible.
\end{proposition} \begin{proof}
(a) Because $\boldmu (\sX \ho \sX) \subset \sX$, by Lemma \ref{r.admiss1}, $\boldmu$ is $(\sX^{\ho n})$-admissible.

(b) Because $\cR\in \overline{\sX \ot \sX}$, $\cR$ has a presentation
$$ \cR = \sum_{p\in P} \cR_1(p) \ot \cR_2(p), \quad \cR_1(p), \cR_2(p) \in \sX,$$
where the sum is convergent in the $h$-adic topology of $\scH \ho \scH$.
Using the definitions \eqref{e12}--\eqref{e11}, we the following presentations as  $h$-adically convergent sums
\begin{align*}
\boldpsi &= \sum_{p \in P} \boldpsi^+_{p} , \quad &\text{where } & &  \boldpsi^+_{p}(x\ot y) &=\sum\cR_2(p) \tri y \ot \cR_1(p) \tri x \\
\boldpsi^{-1} &= \sum_{p \in P} \boldpsi^-_{p}  &\text{where }& &  \boldpsi^-_{p}(x\ot y) &=\sum S(\cR_1(p)) \tri y \ot \cR_2(p) \tri x \\
 \uD&= \sum_{p \in P} \uD_p  &\text{where }& &  \uD_p (x) &= \sum \cR_2(p) \tri x_{(2)} \ot \cR_1 x_{(1)}\\
  \uS&= \sum_{p \in P} \uS_p  &\text{where }& &  \uS_p (x) &= \sum \cR_2(p)\,  S(  \cR_1\tri  x).
\end{align*}
Since $\cR_1(p), \cR_2(p)\in \sX$, which is a topological
  Hopf algebra, we see that $\boldpsi^\pm_p(\sX \ho \sX) \subset \sX
  \ho \sX$, $\uD_p(\sX) \subset \sX \ho \sX$, and $ \uS_p(\sX) \subset
  \sX$. By Lemma \eqref{r.admiss1}, all $\boldpsi^{\pm1}, \uD, \uS$
  are $(\sX^{\ho n})$-admissible.
\end{proof}

\subsubsection{Braided commutator and Borromean tangle}
\label{sec:borr-tangle-braid}
We recall from \cite{H:claspers,H:bottom} the definitions and properties of
the braided commutator for a braided Hopf algebra and
a formula for  universal invariant of the Borromean tangle.

Define the {\em braided commutator} ${\Upsilon}\col \scH \ho \scH \rightarrow \scH $ (for the braided
Hopf algebra $\underline \scH$) by
\begin{gather*}
  {\Upsilon} = {\boldsymbol \mu} ^{[4]}(\id \ho {\boldsymbol{\psi}} \ho \id)(\id\ho \bS\ho \bS\ho \id)(\uD\ho \uD).
\end{gather*}
As noted in \cite[Section 9.5]{H:bottom}, with $\bc = \sum_i \bc'(i) \ot \bc''(i)$, we have
\begin{gather}
  \label{e42}
  \modb  =  \sum_{i,j\in I} (\id^{\ho 2} \ho \Upsilon)\left( \bc'(i)  \ot \bc'(j) \ot \bc''(j) \ot \bc''(i)\right)  =\sum_{i,j\in I}  \bc'(i) \ot \bc'(j)  \ot \Upsilon (  \bc''(j)\ot \bc''(i)).
\end{gather}
Let $\modb_{i,j}$ be the $(i,j)$-summand of the right hand side. Then each $\modb_{i,j}\in \sX^{\ho 3}$ and
$\modb= \sum_{i,j} \modb_{i,j}$, with the sum converging in the $h$-adic topology of $\scH^{\ho 3}$. We want to show that the sum $\sum_{i,j} \modb_{i,j}$ is convergent in the $h$-adic topology of $\sX^{\ho 3}$.

\subsubsection{Two definitions of braided commutator}
From \cite[Section 9.3]{H:bottom}, we have
\begin{align}
  \label{e41}
  {\Upsilon} &= {\boldsymbol \mu} (\ad\ho \id)(\id\ho (\bS\ho \id)\uD) \\
    &= {\boldsymbol \mu} (\id\ho \radb)((\id\ho \bS)\uD\ho \id), \label{e41a}
\end{align}
where $ \radb$ is the right adjoint action (of the braided
Hopf algebra $\underline \scH$) defined by
\begin{gather*}
  \radb:={\boldsymbol \mu} ^{[3]}(\bS\ho \id \ho \id )({\boldsymbol{\psi}}\ho \id )(\id \ho \uD).
\end{gather*}

\begin{lemma}
  \label{rr3}
    For $x,y\in \scH $, we have
  \begin{gather}
    \label{e40}
    \radb(x\otimes y)= S^{-1}(y)\trr x  .
  \end{gather}
\end{lemma}

\begin{proof} 
In what follows we use $\cR= \sum \cR_1 \otimes \cR_2 = \sum \cR_1' \otimes \cR_2' =\sum \cR_1'' \otimes \cR_2''=\sum \cR_1''' \otimes \cR_2'''$.
  One can easily verify
  \begin{gather}
    \label{e37}
    {\boldsymbol{\psi}} (x\otimes y):=\sum( \cR_2   \cR_2' \trr y)\otimes  \cR_1\ x\, S( \cR_1').
  \end{gather}
  We have
  \begin{gather*}
    \begin{split}
      \radb(x\otimes y)
      &={\boldsymbol \mu} ^{[3]}(\bS\ho \id \ho \id )({\boldsymbol{\psi}} \ho \id )(x\otimes \uD(y))\\
      &=\sum{\boldsymbol \mu} ^{[3]}(\bS\ho \id \ho \id )({\boldsymbol{\psi}} \ho \id )(x\otimes y_{\ul{(1)}}\otimes y_{\ul{(2)}})\\
      &=\sum{\boldsymbol \mu} ^{[3]}(\bS\ho \id \ho \id )({\boldsymbol{\psi}} (x\otimes y_{\ul{(1)}})\otimes y_{\ul{(2)}})\\
      &=\sum {\boldsymbol \mu} ^{[3]}(\bS\ho \id \ho \id )( \cR_2   \cR_2' \trr y_{\ul{(1)}}\otimes  \cR_1xS( \cR_1')\otimes y_{\ul{(2)}})
      \quad \text{by \eqref{e37}}\\
      &=\sum {\boldsymbol \mu} ^{[3]}(\bS( \cR_2   \cR_2' \trr y_{\ul{(1)}})\otimes  \cR_1xS( \cR_1')\otimes y_{\ul{(2)}})\\
      &=\sum \bS( \cR_2   \cR_2' \trr y_{\ul{(1)}}) \cR_1xS( \cR_1')y_{\ul{(2)}},
    \end{split}
  \end{gather*}
  where $\uD(y)=\sum y_{\ul{(1)}}\otimes y_{\ul{(2)}}$.  Using
  \begin{gather*}
    \sum y_{\ul{(1)}}\otimes y_{\ul{(2)}}=\sum ( \cR_2'' \trr y_{(2)})\otimes  \cR_1''y_{(1)},\\
    \bS(w)=\sum S^{-1}( \cR_2''' \trr w)S( \cR_1''' ),
  \end{gather*}
  we obtain
  \begin{gather*}
    \begin{split}
      \radb(x\otimes y)
      &=\sum \bS( \cR_2   \cR_2' \trr y_{\ul{(1)}}) \cR_1xS( \cR_1')y_{\ul{(2)}}\\
      &=\sum S^{-1}( \cR_2''' \trr ( \cR_2   \cR_2' \trr ( \cR_2'' \trr y_{(2)})))S( \cR_1''' ) \cR_1xS( \cR_1') \cR_1''y_{(1)}\\
      &=\sum S^{-1}( \cR_2'''  \cR_2   \cR_2'  \cR_2'' \trr y_{(2)})S( \cR_1''' ) \cR_1xS( \cR_1') \cR_1''y_{(1)}.
    \end{split}
  \end{gather*}
  Since $\sum \cR_2'''  \cR_2  \otimes S( \cR_1''' ) \cR_1=\sum \cR_2'  \cR_2'' \otimes S( \cR_1') \cR_1''=\cR_{21}^{-1}\cR_{21}=1\otimes 1$, we obtain
  \begin{gather*}
    \begin{split}
      \radb(x\otimes y)
      &=\sum S^{-1}(y_{(2)})xy_{(1)}= S^{-1}(y)\trr x.
    \end{split}
  \end{gather*}
  This completes the proof of the lemma.
\end{proof}

By Lemma \ref{rr3} and $\ad(\scH\ho\sX)\subset\sX$ we easily obtain
\begin{gather}
  \label{ee1}
  \radb(\sX\ho\scH)\subset \sX.
\end{gather}

\begin{lemma}
  \label{upslionHX}
  We have
\begin{gather}
  \label{ee2}
  \Upsilon(\scH\ho\sX)\subset\sX,\\
  \label{ee3}
  \Upsilon(\sX\ho\scH)\subset\sX.
\end{gather}

\end{lemma}

\def\uone{(\underline 1)}
\def\utwo{(\underline 2)}

\begin{proof}
  Using \eqref{e41} and $\ad(\scH\ho\sX)\subset\sX$, we have
  \begin{gather*}
    \begin{split}
    \Upsilon(\scH\ho \sX)
    &={\boldsymbol \mu} (\ad\ho \id)(\id\ho \bS\ho \id)(\id\ho\uD)(\scH\ho \sX)\\
    &\subset{\boldsymbol \mu} (\ad\ho \id)(\id\ho \bS\ho \id)(\scH\ho\sX\ho \sX)\\
    &\subset{\boldsymbol \mu} (\ad\ho \id)(\scH\ho\sX\ho \sX)\\
    &\subset{\boldsymbol \mu} (\sX\ho \sX)\\
    &\subset \sX.
    \end{split}
  \end{gather*}
  Using \eqref{e41a} and \eqref{ee1}, we can similarly check that $\Upsilon(\sX\ho\scH)\subset\sX$.
\end{proof}

\subsubsection{Borromean tangle}
 \begin{lemma} \label{r.Borro0}
 One has $\modb \in \sX^{\ho 3}$.
 \end{lemma}

\begin{proof}

Since $\{ \bc''(i) \mid i \in I \}$ is 0-convergent in $\scH$, we have
$\bc''(i) = h^{k_i} \tilde \bc''(i)$, where $\tilde\bc''(i)\in\scH$
and for any $N\ge0$ we have $k_i\ge N$ for all but finitely many $i$.

Recall that, by \eqref{e42}, we have
\begin{gather}
  \label{ee42}
  \begin{split}
    \modb
    &=\sum_{i,j\in I} \modb_{i,j},\quad \text{where}\quad
    \modb_{i,j} =\bc'(i) \ot \bc'(j) \ot \Upsilon (\bc''(i) \ot \bc''(j)).
  \end{split}
\end{gather}
By \eqref{ee2} and \eqref{ee3}, we have
\begin{gather*}
  \Upsilon (\bc''(i) \ot \bc''(j))=h^{k_i}\Upsilon (\tilde\bc''(i) \ot
  \bc''(j))\in h^{k_i}\sX,\\
  \Upsilon (\bc''(i) \ot \bc''(j))=h^{k_j}\Upsilon (\bc''(i) \ot
  \tilde\bc''(j))\in h^{k_j}\sX,
\end{gather*}
respectively.  Hence
\begin{gather}
  \Upsilon (\bc''(i) \ot \bc''(j))\in h^{\max(k_i,k_j)}\sX.
\end{gather}
Since $\bc'(i),\bc'(j)\in\sX$, the sum \eqref{ee42} defines an element
of $\sX^{\ho3}$.
 \end{proof}

 \subsubsection{Proof of Proposition \ref{r.Jvalue0}} It is clear that $1 \in \sX^ 0= \Ch$, $1\in \sX$, and $\sX^{\ho n } \ot \sX^{\ho m} \subset \sX^{\ho n+m}$. By Proposition \ref{r.admiss5}, each of $\boldmu, \boldpsi^{\pm1}, \uD,\uS$ is $(\sX^{\ho n})$-admissible. By Lemma \ref{r.Borro0}, $\modb \in \sX^{\ho 3}$. Hence by Proposition \ref{r26}, $J_T \in \sX^{\ho n}$.

 This completes the proof of Proposition \ref{r.Jvalue0} and also the proof of Theorem \ref{r.JMdef}.

\subsection{Integrality of $J_M$}
\begin{theorem} \label{thm.construction}
 Suppose $\sX$ is a core subalgebra of a topological ribbon Hopf algebra $\scH$ with the associate  twist system $\cT_\pm: \sX \to \Ch$.
 Assume that there is a family of subsets
 $\tK _n\subset \mathscr \sX ^{\ho n}$, $n \ge 0$, such that
  \begin{itemize}
  \item[(AL1)] $1_\modk \in \tK _0$, $1_\scH \in \tK _1$, $\modb \in \tK_3$, each of
 ${\boldsymbol{\psi}} ^{\pm 1},{\boldsymbol \mu}
  ,\bD,\bS$ is $(\tK_n)$-admissible, and $ x\ot y\in \tK_{n+m}$ for any $x\in \tK_n, y \in \tK_m$.

  \item[(AL2)] For any $\ve_1,\dots, \ve_n \in \{ \pm \}$,
 $$ (\cT_{\ve_1} \ho \dots \ho \cT_{\ve_n}) (\tK_n) \subset \tK_0.$$
  \end{itemize}
  Then the invariant $J_M$ of integral homology 3-spheres has values in $ \tK_0$.
\end{theorem}

\begin{proof}  Suppose $T$ is an $n$-component bottom tangle $T$ with zero linking matrix.
By Proposition~\ref{r26}, Condition (AL1) implies that $J_T \in \tK_n$. Condition (AL2) implies that
$$ J_M = (\cT_{\ve_1 }\ho \dots \ho\cT_{\ve_n})\left(J_T \right) \in \tK_0,$$
where $M= M(T;\ve_1,\dots,\ve_n)$.
\end{proof}
In  the paper we will  construct a core subalgebra $\sX$  and a
 sequence of $\widehat{\BZ[q]}$-submodules $\tK_n \subset \sX^{\ho  n}$ satisfying the assumptions (AL1) and (AL2) of Theorem \ref{thm.construction}
 for the quantized universal enveloping algebra (of a simple Lie algebra) with $\tK_0= \widehat {\BZ[q]}$. By Theorem \ref{thm.construction}, the corresponding invariant of
 integral homology 3-spheres takes values in $\Zqh$.
 We then show that this invariant specializes to the Witten-Reshetikhin-Turaev invariant at
 roots of unity. In a sense, the $(\tK_n)$ form an integral version
 of the $(\sX^{\ho n})$. The construction of the integral objects $\tK_n$ is much more complicated than that of $\sX$.

\np
\section{Quantized enveloping algebras}\label{sec:UhUq}
 In this section we  present  basic facts about  the quantized enveloping algebras associated to a simple Lie algebra $\fg$: the $h$-adic version $\Uh(\fg)$, the $q$-version $\Uq(\fg)$ and its simply-connected version $\brU_q(\fg)$.  We discuss the well-known braid group actions, various automorphisms of $\Uq$,
the universal $R$-matrix and ribbon structure, and Poincar\'e-Birkhoff-Witt bases. New materials include
gradings on the quantized enveloping algebras in Section \ref{even-grading}, the mirror automorphism $ \varphi$, and a calculation of the clasp element.

\subsection{Quantized enveloping algebras $\Uh$, $\Uq$, and $\brU_q$} \label{sec.UhUq}
\subsubsection{Simple Lie algebra} \label{sec.UhUq1}
Suppose  $\g$ is a finite-dimensional, simple Lie algebra over
$\BC$ of rank $\ell$.  Fix a Cartan subalgebra $\modh $ of $\g$ and a basis
$\Pi=\{\alpha _1,\ldots ,\alpha _\ell\}$ of simple roots in the dual space $\modh ^*$.
Set $\modh ^*_\BR=\BR\Pi \subset \modh ^*$.  Let $Y=\modZ \Pi\subset \modh ^*_\BR$ denote the root
lattice,  $\Phi\subset Y$  the set of all roots, and $\Phi_+\subset \Phi $
the set of all positive roots.  Denote by $t$ the number of positive roots, $t= |\Phi_+|$.
Let $(\cdot,\cdot)$ denote the
invariant inner product on $\modh ^*_\BR$ such that $(\alpha ,\alpha )=2$ for every
short root $\alpha $. For $\alpha \in \Phi $, set $d_\alpha =(\alpha ,\alpha )/2\in \{1,2,3\}$. %
  Let $X$ be the {\em weight lattice}, i.e. $X \subset \modh ^*_\BR$ is the $\BZ$-span of
the {\em fundamental weights}
$\bral _1,\ldots ,\bral _\ell\in \modh ^*_\BR$, which are defined by
$(\bral _i,\alpha _j) = \delta_{ij} d_{\al_i}$.

For $\gamma =\sum_{i=1}^\ell k_i \al_i\in Y$, let $\Ht(\gamma)= \sum_i k_i$.
Let $\rho$ be the half-sum of positive roots,
$\rho =\frac12\sum_{\alpha \in \Phi _+}\alpha $.
It
is known that $ \rho =\sum_{i=1}^\ell\bral _i $.

We list all simple Lie algebras and their constants in Table \ref{tab:1}.
\begin{table}[ht]
\begin{center}
\begin{tabular}{|c|c|c|c|c|c|c|c|c|c|}
\hline
     & $A_\ell$ & $B_\ell$     & $C_\ell$   & $D_\ell$   & $E_6$  &  $E_7$  &  $E_8$  & $F_4$  & $G_2$ \\ \hline
     $d$  & $1$      & $2$                   &  2                &   1        &   1    &    1    &    1    &   2    &   3   \\ \hline
    $D$  & $\ell+1$ & 2                 &    2      &   4              &   3    &   2     &   1     &   1    &   1   \\ \hline
    $ h^\vee$ & $\ell +1$& $2\ell-1$    & $\ell+1$   &$2\ell-2$    & 12     &  18     & 30      &   9    &   4  \\ \hline
\end{tabular}
\end{center}
\caption{Constants $d,D,h^\vee$ of simple Lie algebras}
\label{tab:1}
\end{table}

\subsubsection{Base rings} \label{sec.UhUq2}
Let $v$ be an indeterminate, and set $\modA := \modZ [v^{\pm1}]\subset \BC (v)$.  We
regard $\modA $ also as a subring of $\BC [[h]]$, with $v= \exp(h/2)$.  Set
$q=v^2$.
\begin{remark}
  \label{r33}
  We will follow mostly Janzen's book \cite{Jantzen}.  However, our
  $v$, $q$ and $h$ are equal to ``$q$'', ``$q^2$'' and ``$-h$'',
  respectively, of \cite{Jantzen}.  Since $q=v^2$, one could avoid
  using either $q$ or $v$.  We will use both $q$ and $v$ because on
  the one hand the use of half-integer powers of $q$ would be cumbersome, and
  on the other hand we would like to stress that many constructions in
  quantized enveloping algebras can be done over $\BZ[q^{\pm 1}]$.
\end{remark}

For $\alpha \in \Phi$ and integers $n,k\ge0$, set
\begin{gather*}
v_\alpha := v^{d_\alpha }, \quad q_\alpha :=q^{d_\alpha } = v_\alpha ^2,\\
 [n]_\alpha := \frac{v^n_\alpha -v^{-n}_\alpha }
{v_\alpha -v^{-1}_\alpha }, \quad [n]_\alpha ! := \prod_{i=1}^n [i]_\alpha , \quad
\qbinom{n}{k}_\alpha := \prod_{i=1}^k \frac{[n-i+1]_\alpha }{[i]_\alpha },\\
\{n\}_\alpha := v^n_\alpha -v^{-n}_\alpha ,  \quad \{n\}_\alpha ! := \prod_{i=1}^n
\{i\}_\alpha .
\end{gather*}
When $\alpha $ is a short root, we sometimes suppress the subscript $\alpha $ in
these expressions.

Recall that for $n\ge 0$ and for any element $x$ in a $\modZ [q]$-algebra,
\begin{gather*}
  (x;q)_n := \prod_{j=0}^{n-1} (1-x q^j).
\end{gather*}

\subsubsection{The algebra  $\Uh$}
\label{subsub.UU}

The quantized enveloping algebra $\Uh=\Uh(\g)$ is defined as the
$h$-adically  complete $\BC [[h]]$-algebra, topologically generated by
$E_\alpha ,F_\alpha ,H_\alpha $ for $\alpha \in \Pi$, subject to the relations
\begin{gather}
  \label{e45}
  H_\alpha H_\beta   =  H_\beta  H_\alpha ,\\
  \label{e46}
  H_\alpha E_\beta   - E_\beta  H_\alpha   = (\alpha ,\beta ) E_\beta , \qquad
  H_\alpha F_\beta   - F_\beta  H_\alpha   = - (\alpha ,\beta ) F_\beta , \\
  \label{e47}
  E_\alpha F_\beta  - F_\beta  E_\alpha  = \delta _{\alpha \beta }\frac{K_\alpha -K_\alpha ^{-1}}{v_\alpha
    - v_\alpha ^{-1}}, \quad \text{where } K_\alpha  = \exp(hH_\alpha /2),
\end{gather}
\begin{align}
  \label{e48}
  \sum_{s=0}^{r} (-1)^s \qbinom{r}{s}_\alpha \,  E_\alpha ^{r-s} E_\beta  E_\alpha ^s & =0, \quad \text{where } r=1- (\beta ,\alpha)/d_\al,\\
  \label{e49}  \sum_{s=0}^{r} (-1)^s \qbinom{r}{s}_\alpha \,  F_\alpha ^{r-s} F_\beta  F_\alpha ^s & =0, \quad \text{where } r=1- (\beta ,\alpha)/d_\al.
\end{align}

We also write $E_i,F_i,K_i$ respectively for $E_{\al_i}, F_{\al_i}, K_{\al_i}$, for $i = 1,\ldots, \ell$.

For every $\lambda = \sum_{\alpha \in \Pi } k_\alpha \alpha \in \modh ^*_\BR$, define $H_\lambda = \sum_\alpha
{k_\alpha }\, H_\al$ and
$K_\lambda  = \exp( \frac{h}{2}H_\lambda)$. In particular, one can define $\brK_\al:= K_{\bral}$, for $\al\in \Pi$.
\subsubsection{Hopf algebra structure}\label{sec.a1}
 The algebra $\Uh$  has
a structure of a complete Hopf algebra over $\BC [[h]]$, where the
comultiplication, counit and antipode are given by:
\begin{align*}
  \Delta (E_\alpha )& = E_\alpha \otimes 1 + K_\alpha \otimes E_\alpha ,
   & {\boldsymbol \epsilon} (E_\alpha )& =0,  & S(E_\alpha )& = -K_\alpha ^{-1}E_\alpha ,\\
    \Delta (F_\alpha )& = F_\alpha \otimes K_\alpha ^{-1} + 1 \otimes F_\alpha ,
    &{\boldsymbol \epsilon} (F_\alpha )& =0,  &  S(F_\alpha )& = -F_\alpha K_\alpha ,\\
    \Delta (H_\alpha )& = H_\alpha \otimes 1 + 1 \otimes H_\alpha ,
    & {\boldsymbol \epsilon} (H_\alpha )& =0,  & S(H_\alpha )& = -H_\alpha .
\end{align*}

\subsubsection{The algebra $\Uq$ and its simply-connected version $\brU_q$}\label{sec.UU2}
Let $\Uq$ denote the $\BC (v)$-subalgebra of
$\Uhh=\Uh\otimes _{\BC [h]}\BC [h,h^{-1}]$ generated by $E_\alpha , F_\alpha $, and
$K_\alpha ^{\pm 1}$ for all $\alpha \in \Pi$.  Alternatively, $\Uq$ is defined to be
the $\BC (v)$-subalgebra generated by the elements $K_\alpha $, $K_\alpha ^{-1}$,
$E_\alpha $, $F_\alpha $ ($\alpha \in \Pi $), with relations \eqref{e47}--\eqref{e49} and
\begin{gather}
  \label{e50}
  K_\alpha K_\alpha ^{-1}=K_\alpha ^{-1}K_\alpha =1,\\
  \label{e51}
  K_\beta E_\alpha =v^{(\beta ,\alpha )}E_\alpha K_\beta ,\quad
  K_\beta F_\alpha =v^{-(\beta ,\alpha )}E_\alpha K_\beta
\end{gather}
for $\alpha ,\beta \in  \Pi $.

The algebra $\Uq$ inherits a Hopf algebra structure from  $\Uh[h^{-1}]$, where
\begin{gather*}
  \Delta(K_\alpha )= K_\alpha \otimes K_\alpha ,\quad
  {\boldsymbol \epsilon} (K_\alpha )=1,\quad
  S(K_\alpha ) = K_{\alpha }^{-1}.
\end{gather*}

Similarly, the simply-connected version $\brU_q$ is the $\BC (v)$-subalgebra of
$\Uhh$ generated by $E_\alpha , F_\alpha $, and
$\brK_\alpha ^{\pm 1}$ for all $\alpha \in \Pi$. Again $\brU_q$ is a $\BC(v)$-Hopf algebra, which contains $\Uq$ as a Hopf subalgebra. Let $\brU_q^0$ be the $\BC(v)$-algebra generated by $\brK_\al^{\pm1}, \al \in \Pi$. Then
\be
\label{eq.brUq}
\brU_q= \brU_q^0 \Uq.
\ee

The simply-connected version $\brU_q$ has been studied in \cite{DKP,Gavarini,Caldero} in connection with quantum adjoint action and various duality results. We need the simply connected version $\brU_q(\fg)$ for a duality result, and also for the description of the $R$-matrix.

\subsection{Automorphisms} \label{sec.8147}
There are  unique $h$-adically continuous  $\BC$-algebra automorphisms
$\ibar, \tphi,\omega$ of $\Uh$ defined by
 \begin{align*}
 \ibar(h)&= -h, \quad &\ibar(H_\al)&= H_\al, \quad &\ibar(E_\al)&= E_\al , \quad  &\ibar(F_\al)&=  F_\al\\
 \omega(h)&=h, \quad &\omega(H_\al) &= - H_\al, \quad &\omega(E_\al) &= F_\al, \quad  &\omega(F_\al)&= E_\al \\
 \tphi(h)&= -h, \quad &\tphi(H_\al)& = -H_\al, \quad &\tphi(E_\al)&= -F_\al K_\al, \quad  &\tphi(F_\al)&= -K_\al^{-1} E_\al,
 \end{align*}
 and a unique $h$-adically continuous  $\BC$-algebra anti-automorphism $\tau$ defined by
$$  \tau(h)=h, \quad \tau(H_\al)= -H_\al, \quad \tau(E_\al)=E_\al, \quad \tau(F_\al)=F_\al.$$

The map $\ibar$ is the bar operator of \cite{Lusztig}, and $\tau, \omega$ are the same $\tau,\omega$ in \cite{Jantzen}. All three are involutive, i.e. $
\tau ^2=
\ibar^2= \omega^2=\id$.
The restrictions of $
\ibar, \tphi,\tau,\omega$ to $\Uh\cap\Uq$ naturally extend to maps from
$ \Uq$ to $ \Uq$, and we have
\begin{align*}
 \tau(v)&= \omega(v)=v, \quad &\tau(K_\al) &= \omega(K_\al)=K_\al^{-1}, \\
 \ibar(v)&= v^{-1},\quad  &\ibar(K_\al)&= K_\al^{-1},\\
 \tphi(v)&= v^{-1},\quad  &\tphi(K_\al)&= K_\al.
 \end{align*}
Unlike $\ibar,\tau,\omega$, the map $\tphi$ is a $\BC$-Hopf algebra homomorphism:
\begin{proposition}
\label{r.phi5}
The $\BC$-algebra automorphism $\tphi$ commutes with $S$ and $\Delta$, i.e.
$$ \tphi S = S \tphi, \quad (\tphi \ho \tphi) \Delta = \Delta \tphi.$$
Besides $\varphi= \ibar \tau \omega S =  \ibar \omega \tau  S= S \ibar \tau \omega $, and
\be
\label{eq.tphi2}
\tphi^2(x) = S^2(x) = K_{-2\rho} x K_{2\rho}.
\ee
\end{proposition}
\begin{proof}
All the statements can be easily checked on generators $h, H_\al, E_\al, F_\al$.
\end{proof}

\subsection{Gradings by root lattice} \label{sec:Ygrading}
\subsubsection{$Y$-grading}
\label{sec:y-grading}

There are $Y$-gradings on $\Uh$ and $\Uq$ defined
by
\begin{gather*}
  |E_\alpha |=\alpha ,\quad  |F_\alpha |=-\alpha ,\quad |H_\alpha |=|K_\alpha |=0.
\end{gather*}
For a subset $A \subset \Uh$, denote by $A_\mu, \mu \in Y$, the set of all elements of $Y$-grading $\mu$ in $A$.

We frequently use the following simple fact: If $x$ is $Y$-homogeneous and $\beta \in Y$, then
\begin{gather}
  K_\beta \,  x = v^{(\beta , |x|)} x K_\beta .
  \label{eq.8157}
\end{gather}
 In the language of representation theory, $x\in \Uh$ has $Y$-grading $\beta\in Y$ if and only if it is
 an element of weight $\beta$ in the adjoint representation of $\Uh$.

\subsubsection{$(Y/2Y)$-grading and the even part of $\Uq$}
\label{even-grading}

\begin{proposition}
  \label{r48}
  There is a unique $(Y/2Y)$-grading on the $\BC (v)$-algebra $\Uq$ satisfying
  \begin{gather*}
    \deg(K_\alpha ) \equiv \alpha ,\quad
    \deg(E_\alpha ) \equiv 0,\quad
    \deg(F_\alpha ) \equiv \alpha \pmod{2Y}
  \end{gather*}
  for $\alpha \in \Pi $.
\end{proposition}

\begin{proof}
  Using the defining relations \eqref{e47}--\eqref{e51} for $\Uq$, one
   checks that the $(Y/2Y)$-grading is well defined.
\end{proof}

The degree $0$ part of $\Uq$ in the $(Y/2Y)$-grading, which is
generated by $K_\alpha ^{\pm 2}$, $E_\alpha $ and $F_\alpha K_\alpha $
for $\alpha \in \Pi $, is called the {\em even part} of $\Uq$ and
denoted by $\Uqv$.
Elements of $\Uqv$ are said to be {\em even}.

For each $\alpha \in Y$, the degree $(\alpha \mod 2Y)$ part of $\Uq$ is $K_\alpha \Uqv$.

\begin{lemma}\label{eq.evenD3}
(a) Suppose $\mu \in Y$. Let $(\Uq^\ev)_\mu$ be the grading $\mu$ part of $\Uq^\ev$. Then
\begin{align*}
S\left ((\Uq^\ev)_\mu \right) & \subset   K_{\mu}\, \Uq^\ev,
\\ \Delta \left ( (\Uq^\ev)_\mu \right)  & \subset \bigoplus_{\lambda \in Y}  K_\lambda\,  (\Uq^\ev)_{\mu -\lambda }\otimes (\Uq^\ev)_\lambda.
\end{align*}
In particular, $\Delta(\Uq^\ev) \subset \Uq \otimes \Uq^\ev$.

(b) The adjoint action preserves the even part, i.e.
 $\Uq \tri \Uq^\ev \subset \Uq^\ev$.

(c) Each of $\ibar,\tau$, and $\tphi$ leaves $\Uq^\ev$ stable, i.e. $f(\Uq^\ev) \subset \Uq^\ev$ for $f=\ibar,\tau,\tphi$.
\end{lemma}
\begin{proof} (a) Suppose $x \in (\Uq^\ev)_\mu$. We have to show that
\begin{align*}
S(x) & \in K_{\mu}\, \Uq^\ev, \\
\Delta(x) &\in \bigoplus_{\lambda \in Y}  K_\lambda\,  (\Uq^\ev)_{\mu -\lambda }\otimes (\Uq^\ev)_\lambda.
\end{align*}
If the statements hold for $x=x_1\in (\Uq^\ev)_{\mu_1}$ and $x= x_2 \in (\Uq^\ev)_{\mu_2}$, then they hold for $ x= x_1 x_2  \in (\Uq^\ev)_{\mu_1+\mu_2}$. Since $\Uq^\ev$ is generated as an algebra by $K_\al^{\pm2} \in (\Uq^\ev)_0$, $E_\al \in (\Uq^\ev)_\al$, and $F_\al K_\al \in (\Uq^\ev)_{-\al}$, it is enough to prove the statements when $x$ is one of $K_\al^{\pm 2}, E_\al$, or $F_\al K_\al$. For these special values of $x$, the explicit formulas of $S(x)$ and $\Delta(x)$ are given in subsection \ref{sec.a1}, from which the statements follow immediately.

(b) For  $x\in \Uq$, we have the following explicit formula for the adjoint actions
\begin{align}
K_\al \tri x &= K_\al \, x \, K_\al^{-1}  \notag \\
E_\al \tri x &= E_\al \,x  - K_\al \, x \, K_\al^{-1}\, E_\al  \label{eq.8161}\\
F_\al \tri x &= (F_\al \, x - x\, F_\al)\, K_\al.  \notag
\end{align}
If $x$ is even, then  all the right hand sides of the above are even. Since $\Uq$ is generated by $K_\al, E_\al, F_\al$, we have  $\Uq \tri \Uq^\ev \subset \Uq^\ev$.

(c) One can check directly that each of $\ibar,\tau$, and $\tphi$
maps any of the generators $K_\al^{\pm 2}, E_\al, F_\al K_\al$ of $\Uq^\ev$ to an element of $\Uq^\ev$.
\end{proof}
\begin{remark}
In Section \ref{sec:nonc-grad}, we refine the $Y/2Y$-grading of the $\BC(v)$-algebra $\Uq$
to a grading of the $\BC(v)$-algebra $\Uq$
 by a noncommutative $\modZ /2\modZ $-extension of $Y/2Y$.
\end{remark}

 By \eqref{eq.brUq}, $\brU_q= \brU_q^0 \Uq$, where $\brU_q^0= \BC(v)[\brK_1^{\pm 1}, \dots, \brK_\ell^{\pm1}]$.
Here we set $\brK_i=\brK_{\al_i}$ for $i=1,\dots,\ell$.
Let $\brU_q^{\ev,0}= \BC(v)[\brK_1^{\pm 2}, \dots, \brK_\ell^{\pm2}]$ and
 $$ \brU_q^\ev:= \brU_q^{\ev,0}\, \Uq^\ev.$$
 \begin{lemma}\label{eq.evenD3a}  One has
  $\brU_q \tri \brU_q^\ev  \subset \brU_q^\ev$ and
  $\brU_q \tri \Uq^\ev  \subset \Uq^\ev$.
 \end{lemma}

 \begin{proof} The proof is similar to that of Lemma \ref{eq.evenD3}(b).
 \end{proof}

\subsection{Triangular decompositions and their even versions} \label{sec.tri1}

Let $\Uh^+$ (resp. $\Uh^-$, $\Uh^0$) be the $h$-adically closed
$\BC [[h]]$-subalgebra of $\Uh$ topologically generated by $E_\alpha $
(resp. $F_\alpha $, $H_\alpha $) for $\alpha \in \Pi$.

Let $\Uq^+$
(resp. $\Uq^-$, $\Uq^0$) denote the $\BC (v)$-subalgebra of $\Uq$ generated by
$E_\alpha $ (resp. $F_\alpha $, $K_\alpha ^{\pm 1}$) for $\alpha \in \Pi$.

It is known that the multiplication map
\begin{gather*}
  \label{e52}
  \Uq^-\otimes \Uq^0\otimes \Uq^+\longrightarrow\Uq,\quad x\otimes x'\otimes x''\mapsto xx'x''
\end{gather*}
is an isomorphism of $\BC (v)$-vector spaces.  This fact is called the
triangular decomposition of $\Uq$.  Similarly,
\begin{gather*}
  \Uh^-\ho\Uh^0\ho\Uh^+\longrightarrow \Uh,
  \quad x\otimes x'\otimes x''\mapsto xx'x'',
\end{gather*}
is an
isomorphism of $\BC [[h]]$-modules. These triangular decompositions
descend to various subalgebras of $\Uq$ and $\Uh$ which we will
introduce later.

We need also an even version of triangular decomposition for $\Uqv$. Although $\Uq^+\subset \Uqv$, the negative part $\Uq^-$ is not even.

   Let $\Uq^{\ev,-}:= \tphi (\Uq^+)$,  which is the $\BC (v)$-subalgebra of $\Uqv$ generated by $F_\alpha K_\alpha = -\tphi(E_\al)$, $\alpha \in \Pi $. Then $\Uq^{\ev,-} \subset \Uq^\ev$.
Let
$\Uq^{\ev,0}$ be the even part of $\Uq^{0}$, i.e.
\begin{gather*}
  \Uq^{\ev,0}:=\Uqv\cap \Uq^0
  =\BC (v)[K_1^{\pm 2}, \dots, K_\ell^{\pm 2}].
\end{gather*}

Using \eqref{eq.8157}, we obtain the following
isomorphisms of vector spaces
\begin{gather}
\label{e54a} \Uq^{\ev,-} \otimes  \Uq^{0} \otimes  \Uq^+ \congto \Uq,\quad x\otimes y\otimes z\mapsto xyz.\\
  \label{e54}
  \Uq^{\ev,-} \otimes  \Uq^{\ev,0} \otimes  \Uq^+ \congto \Uqv,\quad x\otimes y\otimes z\mapsto xyz.\\
   \label{e54b}\Uh^{\ev,-} \otimes  \Uh^{0} \otimes  \Uh^+ \congto \Uh^\ev,\quad x\otimes y\otimes z\mapsto xyz.
  \end{gather}
  where we set $\Uh^{\ev,-}= \tphi(\Uh^+)$, which is the $h$-adically closed $\BC[[h]]$-subalgebra of $\Uh$ topologically generated by $F_\al K_\al$, $\al\in\Pi$. We call \eqref{e54a}, \eqref{e54}, and \eqref{e54b} respectively the {\em even triangular decomposition} of $\Uq, \Uq^\ev$, and $\Uh$.

\subsection{Braid group action} \label{503}

\subsubsection{Braid group and Weyl group}

The {\em braid group} for the root system $\Phi $ has the presentation
with generators $T_\alpha $ for $\alpha \in \Pi $ and with relations
\begin{gather*}
  T_\alpha T_\beta =T_\beta T_\alpha \quad \text{for $\alpha ,\beta \in \Pi $, $(\alpha ,\beta )=0$},\phantom{-}\\
  T_\alpha T_\beta T_\alpha =T_\beta T_\alpha T_\beta \quad \text{for $\alpha ,\beta \in \Pi $, $(\alpha ,\beta )=-1$},\\
  T_\alpha T_\beta T_\alpha T_\beta =T_\beta T_\alpha T_\beta T_\alpha \quad \text{for $\alpha ,\beta \in \Pi $, $(\alpha ,\beta )=-2$},\\
  T_\alpha T_\beta T_\alpha T_\beta T_\alpha T_\beta =T_\beta T_\alpha T_\beta T_\alpha T_\beta T_\alpha \quad \text{for $\alpha ,\beta \in \Pi $, $(\alpha ,\beta )=-3$}.
\end{gather*}

The Weyl group $\fW $ of $\Phi $ is the quotient of braid group by the relations
$T_\alpha ^2=1$ for $\alpha \in \Pi $.  We denote the generator in $\fW $ corresponding
to $T_\alpha$ by $s_\alpha $.  We set $T_i=T_{\alpha _i}$, $s_i=s_{\alpha _i}$ for
$i=1,\ldots ,\ell$.

Suppose  $\mathbf i = (i_1,\ldots ,i_k)$ with
$i_j\in \{1,2,\ldots ,\ell\}$. Let
$
  w(\mathbf i)=s_{i_1}s_{i_2}\cdots s_{i_k}\in \fW $.
If there is no shorter
sequence $\modj $ such that $w(\modi )=w(\modj )$, then we say that the sequence $\modi $ is  {\em reduced}, and  {\em $w(\bi)$ has length $k$}.
It is known that the length of any reduced sequence is less than or equal to $t:=|\Phi_+|$, the number of positive roots of $\g$. A sequence $\modi $ is called {\em longest reduced} if $\modi $ is reduced
and has length $t$. There is a unique  element $w_0\in \fW $ such that for any longest reduced sequence $\bi$ one has $w(\bi)=w_0$.

\subsubsection{Braid group action}
As described in \cite[Chapter 8]{Jantzen}, there is an action of the
braid group on the $\BC(v)$-algebra $\Uq$.  For $\alpha \in
\Pi $, $T_\alpha\colon \Uq\rightarrow \Uq$ is
 $\BC(v)$-algebra automorphism
defined by
\begin{gather*}
  T_\alpha (K_\gamma ) = K_{s_\alpha (\gamma )},\quad
  T_\alpha (E_\alpha ) = -F_\alpha K_\alpha , \quad
  T_\alpha (F_\alpha ) = -K_\alpha ^{-1}E_\alpha ,
  \end{gather*}
  \begin{align*}
  T_\alpha (E_\beta )
  &= \sum_{i=0}^{r }(-1)^iv_\alpha ^{-i}E_\alpha ^{(r-i)}E_\beta E_\alpha ^{(i)}, \quad \text{with } r = -(\beta,\al)/d_\al,\\
  T_\alpha (F_\beta )
 & = \sum_{i=0}^{r }(-1)^iv_\alpha ^iF_\alpha ^{(i)}F_\beta F_\alpha ^{(r-i)}, \quad \text{with } r = -(\beta,\al)/d_\al,
  \end{align*}
  where $\gamma \in Y$, $\beta \in \Pi \setminus \{\alpha \}$.
The restriction of $T_\alpha$ to $\Uq\cap\Uh$ extends to a
continuous $\BC[[h]]$-algebra automorphism $T_\alpha$ of $\Uh$ by
setting
\begin{gather*}
T_\alpha (H_\gamma )= H_{s_\alpha (\gamma )} \quad \text{ for $\gamma \in Y$.}
\end{gather*}

\begin{remark} Our $T_\al$ is the same as $T_\al$ of \cite{Jantzen}. Our $T_i= T_{\al_i}$ is $T_{i,1}''$ of \cite{Lusztig}, or $\tilde T_i^{-1}$ of \cite{Lusztig02}.
\end{remark}

One can easily check that
\begin{gather}
  T_\alpha ^{\pm 1}(K_\beta \Uqv)\subset K_{s_\alpha (\beta )}\Uqv.
\end{gather}
for $\alpha \in  \Pi $, $\beta \in Y$.  In particular, the even
part $\Uqv$ is stable under $T_\alpha ^{\pm 1}$. Thus, we have
\begin{proposition}
The even part $\Uqv$ is stable under the action of the braid group.
\end{proposition}

\subsection{PBW type bases} \label{sec.PBW}
\subsubsection{Root vectors}
Suppose $\modi =(i_1,\ldots ,i_t)$ is a longest reduced sequence.  For
$j\in \{1,\ldots ,t\}$, set
\begin{gather*}
  \gamma _j=\gamma _j(\modi ):=s_{i_1} s_{i_2} \cdots  s_{i_{j-1}}(\alpha _{i_j}).
\end{gather*}
 It is known that $\gamma _1, \ldots , \gamma _t$ are distinct positive roots and $\{\gamma _1, \ldots , \gamma _t\}=\Phi _+$.   The elements
\begin{eqnarray*}
E_{\gamma_j} (\modi ): = T_{\alpha _{i_1}} T_{\alpha _{i_2}}\cdots T_{\alpha _{i_{j-1}}}(E_{\alpha _{i_j}}),\quad
F_{\gamma_j} (\modi ) := T_{\alpha _{i_1}} T_{\alpha _{i_2}}\cdots T_{\alpha _{i_{j-1}}}(F_{\alpha _{i_j}})
\end{eqnarray*}
are called {\em root vectors corresponding to $\bi$}.
 The $Y$-grading of the root vectors are $|E_{\gamma_j} (\modi )|= {\gamma_j} =-|F_{\gamma_j} (\modi )|$. It is known that $E_{\gamma_j} (\modi )\in \Uq^+$ and $F_{\gamma_j} (\modi )  \in \Uq^-$.

  In general, $E_{\gamma_j}(\bi)$ and $ F_{\gamma_j}(\bi)$ depend on $\bi$, but if ${\gamma_j}$ is a simple root, i.e.  ${\gamma_j}=\al \in \Pi$, then we have
$E_{\gamma_j} (\modi )=E_\al $, $F_{\gamma_j} (\modi )=F_\al $.

\subsubsection{PBW type bases}
\label{s.PBW}
 Fix a longest reduced sequence $\bi$. In what follows,  we often suppress $\modi $ and write $E_\gamma =E_\gamma (\modi )$, $F_\gamma =F_\gamma (\modi )$ for all $\gamma \in \Phi _+$.

The divided powers $E_\gamma ^{(n)}$,  $F_\gamma ^{(n)}$  for $\gamma \in \Phi _+$, $n\ge 0$ are defined
by:
\begin{gather*}
  E_\gamma ^{(n)} :=E_\gamma ^n/[n]_\gamma !,\quad  F_\gamma ^{(n)}=F_\gamma ^n/[n]_\gamma !.
\end{gather*}

Following Bourbaki, we denote by $\BN$ the set of non-negative integers.
For $\bn\in \BN^t$, define
\begin{gather*}
  F^{(\bn)} =  \iprod_{\gamma_j\in \Phi_+}  F_{\gamma_j} ^{(n_j)}
 , \qquad E^{(\bn)} = \iprod_{\gamma_j\in \Phi_+}  E_{\gamma_j} ^{(n_j)}, \quad
\end{gather*}
Here $\iprod_{\gamma_j \in \Phi_+}$ means to take the product in the reverse order of $( \gamma_1,\gamma_2,\dots, \gamma_t)$. For example,
$$  F^{(\bn)} =  \iprod_{\gamma_j\in \Phi_+}  F_{\gamma_j} ^{(n_j)}= F_{\gamma _t}^{(n_t)}   \,
  F_{\gamma _{t-1}}^{(n_{t-1})}\cdots F_{\gamma _1}^{(n_1)}.$$

The set $\{ E^{(\modn) } \mid \modn \in \BN^t\}$ is a basis
of the $\BC(v)$-vector space $\Uq^+$, and a topological basis of $\Uh$.

Similarly, the set $\{ F^{(\modn )} \mid \modn \in \BN^t\} $ is a basis
of $\Uq^-$ and a topological basis of $\Uh^-$.

On the other hand, $\{K_\gamma\mid\gamma \in Y\}$ is a $\BC(v)$-basis of $\Uq^0$ and $\{ H^\bk \mid \bk \in \BN^\ell\}$, where  $H^\bk=\prod_{j=1}^\ell H_j^{k_j}$ for $\bk=(k_1,\dots,k_\ell)$, is a topological basis of $\Uh^0$.

Combining these bases and using the even triangular decompositions \eqref{e54a}--\eqref{e54b}, we get the  following proposition, which describes the Poincar\'e-Birkhoff-Witt bases of $\Uq$, $\Uqv$, $\Uh$:
\begin{proposition} \label{r.PBWbases}
For any longest reduced sequence $\bi$,
\begin{gather*}
 \{ F^{(\modm)} K_\bm K_\gamma E^{(\modn)} \mid  \modm ,\modn \in \BN^t, \gamma \in Y\}   \text{ is a $\BC(v)$-basis for } \Uq \\
 \{ F^{(\modm)} K_\modm K_\gamma ^2 E^{(\modn)}\mid \modm ,\modn \in \BN^t, \gamma \in Y\}  \text{ is a $\BC(v)$-basis for } \Uqv \\
 \{ F^{(\modm)} K_\modm H^{\bk} E^{(\modn)}\mid \modm ,\modn \in \BN^t, \bk \in \BN^\ell\} \text{ is a topological basis for } \Uh.
 \end{gather*}
where
\begin{gather}
\label{eq.Kn1}  K_\modn  := \prod_{j=1}^t K_{\gamma _j}^{n_j}=K_{-|F^{(\modn)}|} \quad \text{for
    $\modn =(n_1,\ldots ,n_t)\in \BN^t$.}
\end{gather}

\end{proposition}

\subsection{$R$-matrix} \label{sec:universal-r-matrix}
\subsubsection{Quasi-R-matrix} \label{Rmatrix}
Fix a longest reduced sequence $\bi$. Recall that $\{k\}_\al = v_\al^k - v_\al^{-k}$.

The {\em quasi-R-matrix} $\Theta \in \Uh^{\ho2}$ is defined by (see \cite{Jantzen,Lusztig})
\begin{equation}
  \label{e9}
  \Theta  = \sum_{\modn  \in \BN^t} F_{\modn}   \otimes  E_\modn,
\end{equation}
where for $\bn= (n_1,\dots, n_t) \in \BN^t$,
\begin{align}
\label{eq.En}
  E_\modn  &: =   E^{(\bn)} \prod_{j=1}^t  \{ n_j\}_{\gamma_j}!=  \iprod_{\gamma_j\in \Phi_+}\Big( (v_{\gamma_j} - v_{\gamma_j}^{-1}) E_{\gamma_j}\Big) ^{n_j}, \\
  \label{eq.Fn}
   F_\modn &:= \,  F^{(\bn)} \prod_{j=1}^t (-1)^{n_j} v_{\gamma _j}^{-n_j(n_j-1)/2}  = \iprod_{\gamma_j\in \Phi_+}  (-1)^{n_j} v_{\gamma _j}^{-n_j(n_j-1)/2}\,  F_{\gamma_j} ^{(n_j)}.
  \end{align}

It is known that $\Theta$ does not depend on $\bi$, and
\begin{gather}
  \label{e64}
  \Theta ^{-1}= (\ibar \ot \ibar) (\Theta) = \sum_{\modn  \in \BN^t}   F'_{\modn } \otimes  E'_\modn %
\end{gather}
where
$$ F'_{\modn } = \ibar( F_{\modn }), \ E'_{\modn } = \ibar( E_{\modn }).$$

\subsubsection{Universal $R$-matrix and ribbon element} \label{Rmatrix2}

Define an inner product on $\modh_\BR= \Span_\BR \{H_\alpha \zzzvert \alpha \in \Pi \}$ by
$(H_\alpha , H_\beta )=(\alpha ,\beta )$. Recall that $\bral$'s are the fundamental weights. Let $\brH_\al= H_{\bral}$.
Then $\{ \brH_\al/d_\al, \al \in \Pi\}$ is dual to $\{ H_\al, \al \in \Pi\}$
with respect to the inner product, i.e. $(H_\al, \brH_\beta/d_\beta)= \delta_{\al,\beta}$ for $\al,\beta \in \Pi$.
 Define the {\em diagonal part}, or the {\em Cartan part}, of the $R$-matrix by
\begin{gather}
\label{eq.diapart}  \cD=\exp\left (\frac{h}{2} \sum_{\al\in \Pi} (H_\al \otimes \brH_\al/d_\al) \right)\in (\Uh^0)^{\ho2}.
\end{gather}

We have $\cD=\cD_{21}$, where $\cD_{21}\in (\Uh^0)^{\ho2}$ is obtained from $\cD$ by permuting the first and the second tensorands.

A simple calculation shows that, for $Y$-homogeneous $x,y\in \Uh$, we have
\begin{gather}
 \label{eq.Dcommute}
  \cD(x\otimes y)\cD^{-1}= x \, K_{|y|}\otimes K_{|x|}\, y.
\end{gather}

The universal $R$-matrix and its inverse are given by
\begin{equation}
  \cR = \cD\Theta ^{-1},\quad  \cR^{-1} = \Theta  \, \cD^{-1},
  \label{501}
\end{equation}
Note that our $R$-matrix is the inverse of the $R$-matrix in
\cite{Jantzen}.

The quasitriangular Hopf algebra $(\UU,\cR)$ has a ribbon element $\modr $
whose corresponding balanced element (see Section
\ref{sec:ribbon-hopf-algebra}) is given by $\modg =K_{-2\rho }$. For $Y$-homogeneous $x\in \Uh$ we have
\begin{gather}
  S^2(x)=K_{-2\rho }xK_{2\rho }= q^{-(\rho ,|x|)}x.
  \label{e81}
\end{gather}

With $\cR=\sum \cR_1 \otimes \cR_2$, the
ribbon element and its inverse are given by
\begin{gather*}
\modr  = \sum S(\cR_1  )K_{-2\rho } \cR_2  ,\quad
\modr ^{-1} =  \sum \cR_1  K_{2\rho } \cR_2   = \sum \cR_2  K_{-2\rho }\cR_1.
\end{gather*}
One has $\br= J_T$ and $\br^{-1} = J_{T'}$, where $T$ and $T'$ are the bottom tangles in Figure \ref{fig:ribbon}.
\FIG{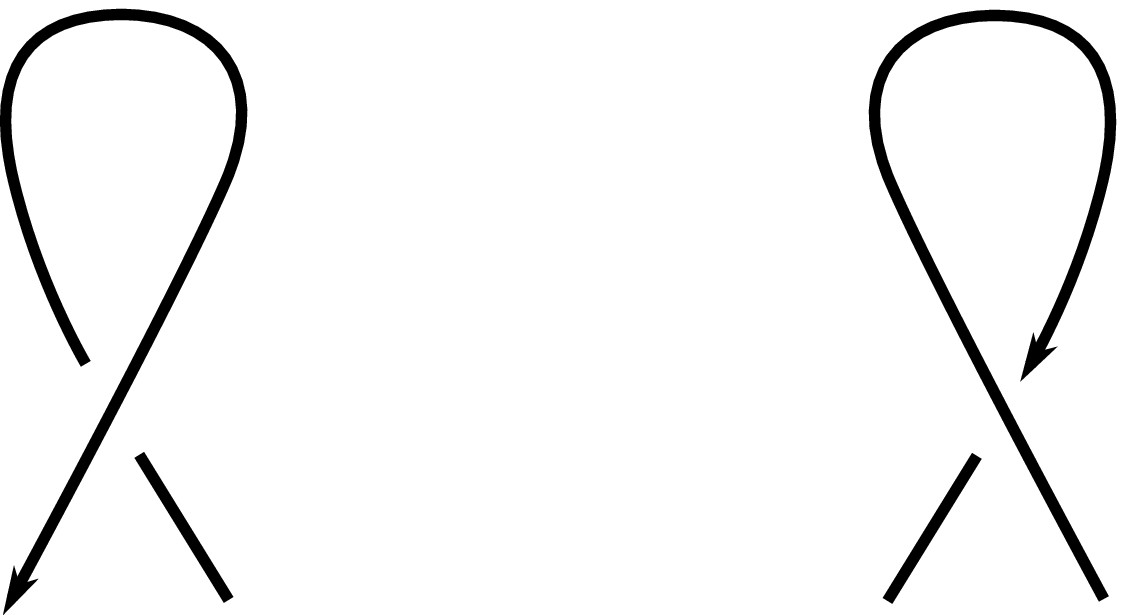}{Tangles $T$ (left) and $T'$ determining the ribbon element $\br$ and its inverse}{height=20mm}

Using \eqref{e9} and \eqref{501}, we obtain
\begin{gather}
  \label{e80r}
  \modr  = \sum_{\modn \in \BN^t} F_\modn  K_\modn \modr _0 E_\modn, \quad
  \modr ^{-1} = \sum_{\modn \in \BN^t} F_\modn ' K_\modn ^{-1} \modr _0^{-1}E_\modn '
\end{gather}
where $K_\modn$ is given by \eqref{eq.Kn1} and
\begin{equation}
\modr _0
:=K_{-2\rho }\, {\boldsymbol \mu} (\cD^{-1})
=K_{-2\rho }\, \exp (-\frac{h}{2}\sum_{\al\in \Pi} H_\al \,\brH_\al/d_\al ). \notag
\end{equation}
We also have
\be \label{eq.rin1}
S(\br) = \uS(\br) = \br .
\ee

\subsection{Mirror homomorphism $\tphi$ } We defined the $\BC$-algebra homomorphism $\tphi$ in Section \ref{sec.8147}.
\begin{proposition} \label{r.mirror}
 The $\BC$-automorphism $\tphi$ is a mirror homomorphism for $\Uh$, i.e.
\begin{align}
\label{eq.mr1}   \tphi(K_{2\rho})&= K_{2\rho} \\
\label{eq.mr2}  (\tphi \ho \tphi) (\cR)  &= (\cR^{-1})_{21}\\
\label{eq.mr3}  (\tphi \ho \tphi)^2 (\cR)&=\cR.
\end{align}
 Consequently, if $T'$ is the mirror image of an $n$-component bottom tangle $T$, then
$J_{T'} = \tphi^{\ho n}( J_T)$.
\end{proposition}
\begin{proof}
 Identity \eqref{eq.mr1} is part of the definition of $\tphi$. One could prove the other two  \eqref{eq.mr2} and \eqref{eq.mr3} by direct calculations. Here is an alternative proof using known identities.

By Proposition \ref{r.phi5},
$\tphi=  \ibar \omega \tau  S$.
Hence \eqref{eq.mr2} follows from the following four known identities:
\begin{align*}
(S\ho S) (\cR)& = \cR  \quad \text{by property of $R$-matrix, Equ. \eqref{eq.SR}} \\
(\tau \ho \tau) \cR&= (\tau \ho \tau) (\cD  \Theta^{-1})= \Theta^{-1} \cD   \quad \text{by \cite[7.1(2)]{Jantzen}}\\
(\omega \ho \omega) (\Theta^{-1} \cD ) &= \Theta_{21}^{-1} \cD   \quad \text{by \cite[7.1(3)]{Jantzen}}\\
(\ibar \ho \ibar)(\Theta_{21}^{-1} \cD )&= \Theta_{21} \cD ^{-1}  = \cR_{21}^{-1} \quad \text{by \eqref{e64}}.
\end{align*}

Identity \eqref{eq.mr3} follows from \eqref{eq.tphi2} and \eqref{eq.SR}:
$$ (\tphi^2 \ho \tphi^2)(\cR)= (S^2 \ot S^2) (\cR)=\cR.$$

 This shows $\tphi$ is a mirror homomorphism. By  Proposition \ref{r.mirror1},  $J_{T'} = \tphi^{\ho n}( J_T)$.
\end{proof}

Because the negative twist is the mirror image of the positive one, we have the following.
\begin{corollary} \label{r.phir}
One has $\tphi(\br)= \br^{-1}$.
\end{corollary}

\subsection{Clasp element and quasi-clasp element}
\label{Rosso-form}
Here we calculate explicitly the value of the clasp element $\bc=J_{C^+} \in \Uh \ho \Uh$, which is
 the universal invariant of the clasp tangle $C^+$ of Figure \ref{fig:c+}.
 Recall that we have defined $E_\bn$, $ F_\bn$, and $ \cD$  in Section
\ref{sec:universal-r-matrix}.  We call
\be
\label{eq.quasiclasp}
\Gamma : = \bc \cD^2
\ee
the {\em quasi-clasp element}. Like the quasi-$R$-matrix, the quasi-clasp element enjoy better integrality than the clasp element itself.

\begin{lemma}
  \label{r5}
  Fix a longest reduced sequence $\bi$. We have
\begin{align}  
\modc &= \sum_{\modm ,\modn \in \Zt}  q^{-(\rho ,|E_\modn |)} \big(F_\bm K_\bm \otimes F_\bn K_\bn  \big)\,  \big( \cD^{-2} \big) \, \big( E_\bn \otimes E_\bm \big) \label{e80}\\
 \Gamma &=   \sum_{\modm ,\modn \in \Zt}  q^{-(\rho ,|E_\modn |) +(|E_\bm|, |E_\bn|)} \big(F_\bm K_\bm^{-1} E_\bn \otimes F_\bn K_\bn^{-1}\, E_\bm  \big)   \label{e80a} \\
   \modc &= ( \tphi \otimes \underline S ^{-1}\,  \tphi) (\modc)  \, . \label{e80b}
\end{align}

\end{lemma}

\begin{proof}  Let $\cD^{-2}=\sum (\cD^{-2})_1\otimes (\cD^{-2})_2$,  and $\cR^{-1} = \sum\bar \cR_1  \otimes \bar \cR_2  =\sum\bar \cR_1  '\otimes \bar \cR_2  '$.
  By \eqref{e43}, we obtain
  \begin{gather*}
    \begin{split}
      \modc &=\sum S(\cR_1  )S(\cR_2  ')\otimes \cR_1  '\cR_2  \\
      &=\sum \bar \cR_1  S^2(\bar \cR_2  ')\otimes \bar \cR_1  '\bar \cR_2  .
    \end{split}
  \end{gather*}

  We have
  \begin{gather*}
      \cR^{-1}=\Theta \cD^{-1}
      =\sum_{\modm \in \Zt}F_\modm  (\cD^{-1})_1 \otimes E_\modm  (\cD^{-1})_2
      =\sum_{\modm \in \Zt}F_\modm K_\modm  (\cD^{-1})_1 \otimes (\cD^{-1})_2 E_\modm .
  \end{gather*}
  Substituting this into the formula for $\modc$,
  we obtain
  \begin{gather*}
    \begin{split}
      \modc &=\sum_{\modm ,\modn \in \Zt} F_\modm K_\modm (\cD^{-2})_1S^2(E_\modn )\otimes F_\modn K_\modn (\cD^{-2})_2E_\modm ,
    \end{split}
  \end{gather*}
   which using \eqref{e81} is \eqref{e80}. Identity \eqref{e80a} follows from \eqref{e80}, via \eqref{eq.Dcommute}.

  Since $C^-$ is the mirror image of $C^+$, by Proposition \ref{r.mirror},
  $
  \modc^- =(\tphi\otimes \tphi) (\modc),
$
  which, together with \eqref{e77}, gives \eqref{e80b}.
\end{proof}

 From  \cite[Proposition 2.1.14]{Majid2}, one has
\be
\label{eq.cflip}
 (\id \ot S^2)(\bc) = \bc_{21}.
 \ee

\np
\section{Core subalgebra of $\UUh$ and quantum Killing form}
\label{sec:corealgebra}
In this section we construct a core subalgebra $\Xh$ of the ribbon Hopf algebra
$$\UUh:= \Uh \ho_{\Ch}\Chh,$$
which is the extension of $\Uh$ when the ground ring is $\Chh$.
We will use the Drinfel'd dual $\Vh$ of $\Uh$ to construct $\Xh$. To show that $\Xh$ is a Hopf algebra we use a stability principle established in Section \ref{sec:topdil}, which also finds applications later. We then discuss the clasp form of $\Xh$ which turns out to coincide with the well-known quantum Killing form (or Rosso form) when restricted to $\Uq$. Thus, we get a geometric interpretation of the quantum Killing form.
\subsection{A dual of $\Uh$} \label{sec.Vh}
Fix a longest reduced sequence $\bi$.
For $\bn=(n_1,\dots,n_k) \in \BN^k$ let $$\|\bn\|= \sum_{j=1}^k n_j.$$

Let us recall the topological basis of $\Uh$ described in Proposition \ref{r.PBWbases}.
For  $\bn= (\bn_1, \bn_2, \bn_3)\in \BN^t \times \BN^\ell \times \BN^t$, let
$$ \bbe_h(\bn)= F^{(\bn_1)}K_{\bn_1}\,  H^{\bn_2} E^{(\bn_3)},$$
where $F^{(\bn_1)}, K_{\bn_1},  H^{\bn_2}, E^{(\bn_3)}$ are defined in Section \ref{s.PBW}.
By Proposition \ref{r.PBWbases},
$$ \{ \bbe_h(\bn) \mid \bn \in \BN^{t + \ell +t} \}$$
is a topological basis of $\Uh$.

Let $\Vh$ be closure (in the $h$-adic topology of $\Uh$)  of the $\Ch$-span of the set
\be
\label{eq.hbasis}
 \{h^{\| \bn\| } \bbe_h(\bn) \mid \bn \in \BN^{t + \ell +t} \}.
 \ee
 Then $\Vh$ is a formal series $\Ch$-module, having the above set \eqref{eq.hbasis} as a formal basis. (See Example~\ref{exa:1} of Section \ref{sec:ChModules}.)
Every  $x \in \Vh$ has a unique presentation of the form
$$
 x = \sum_{\bn \in \BN^{t+\ell+t}} x_\bn \, \left( h^{\| \bn\| } \bbe_h(\bn)\right)
$$
where $x_\bn \in \Ch$. The map $x \to (x_\bn)_{\bn \in I}$ is a $\Ch$-module isomorphism between $\Vh$ and $\Ch^I$, with $I = \BN^{t+\ell+t}$.

In the terminology of Drinfel'd \cite{Drinfeld}, $\Vh$ is a ``quantized
formal series Hopf algebra'' (QFSH-algebra), see also \cite{CP}. As part of his duality principle, Drinfel'd  associates a QFSH-algebra to every so-called ``quantum universal enveloping algebra'' (QUE-algebra).
Gavarini \cite{Gavarini} gave a detailed treatment of this duality, and showed that the above defined $\Vh$ is the QFSH-algebra associated to $\Uh$, which is a QUE-algebra.

For $n\ge 0$ let $\Vh^{\bo n}$ be the topological closure of $\Vh^{\ot n}$ in $\Uh^{\ho n}$.
Then $\Vh^{\bo n}$ is the $n$-th tensor power of $\Vh$ in the category of QFSH-algebras, see \cite[Section 3.5]{Gavarini}. The result of Drinfel'd, proved in details by Gavarini \cite{Gavarini}, says that
$\Vh$ is a Hopf algebra in the category of QFSH-algebras, where the Hopf algebra structure of $\Vh$
is the restriction of the Hopf algebra structure of $\Uh$. Thus, we have the following.

\begin{proposition}\label{r.41a}
 One has
$$\boldmu (\Vh^{\bo 2}) \subset \Vh, \quad \Delta(\Vh) \subset \Vh^{\bo 2}, \quad S(\Vh) \subset \Vh.$$
\end{proposition}

For completeness, we give an independent proof of Proposition \ref{r.41a} in  Appendix \ref{sec:r.41a}. Yet another proof can be obtained from Proposition \ref{r.VZVh}.

\begin{proposition}\label{r.phu1}
Fix a longest reduced sequence $\bi$.
Then  $\Vh$ is the topological closure (in the $h$-adic topology of $\Uh$) of the $\Ch$-algebra generated by $h H_\al, hF_\gamma(\bi), hE_\gamma(\bi)$ with $\al \in \Pi, \gamma\in \Phi_+$.
\end{proposition}
\begin{proof} Let $\Vh'$ be the topological closure (in the $h$-adic topology of $\Uh$) of the $\Ch$-algebra generated by $h H_\al, hF_\gamma, hE_\gamma$ with $\al \in \Pi, \gamma\in \Phi_+$.
One can easily check  $K_\gamma\in\Vh'$ for $\gamma\in Y$.
The set $ \{h^{\| \bn\| } \bbe_h(\bn) \mid \bn \in \BN^{t + \ell +t} \}$ is a formal basis of $\Vh$.

When $\bn=(n_1,\dots,n_{t+\ell+t}) \in \BN^{t+\ell+t}$ is such that all $n_j=0$ except for one which is equal to 1, then the basis element  $h^{\| \bn\| } \bbe_h(\bn)$ is one of $h H_\al, hF_\gamma K_\gamma, hE_\gamma$.
It follows that $h H_\al, hF_\gamma, hE_\gamma\in \Vh$, and hence $\Vh' \subset \Vh$.

From the definition of $\bbe_h(\bn)$, for any $\bn=(\bm,\bk,\bu)\in \BN^t \times \BN^\ell \times \BN^t$,
 \be
 \label{eq.vr1}
 h^{\| \bn \|} \bbe_h(\bn) = a \, \iprod_{\gamma_j \in \Phi_+} (hF_{\gamma_j})^{m_{j}}
\prod_{\gamma_j\in\Phi_+}K_{\gamma_j}^{m_j}
\prod_{j=1}^\ell (h H_j)^{k_j} \iprod_{\gamma_j \in \Phi_+} (hE_{\gamma_j})^{u_{j}},
 \ee
where $\bm=(m_1,\dots,m_t), \bk=(k_1,\dots,k_\ell), \bu=(u_1,\dots,u_t)$ and
$$a = \frac 1{\prod_{j=1}^t \left( [m_j]_{\gamma_j}! \, [u_j]_{\gamma_j}!\right)}$$
 is a unit in $\Ch$. Since the right hand side of \eqref{eq.vr1} is in $\Vh'$, we have
  $\Vh \subset \Vh'$. Thus, $\Vh=\Vh'$.
\end{proof}

\subsection{Ad-stability and $\tphi$-stability of $\Vh$} Recall that we defined the left image of an element $x \in \Uh \ho \Uh$ in Section \ref{sec.leftmodule}.
\begin{proposition}
\label{r.42a} The module  $\Vh$ is the left image of the clasp element $\bc$ in $\Uh \ho \Uh$. Moreover, $\Vh$ is ad-stable, i.e.
 $\Uh \tri \Vh \subset \Vh$.
\end{proposition}
\begin{proof}
For $\bn=(\bn_1,\bn_2, \bn_3)\in \BN^{t+\ell+t}$ let
$$ \bbe''_h(\bn) = F^{(\bn_3)}K_{\bn_3} \brH^{\bn_2} E^{(\bn_1)},$$
where for $\bk= (k_1,\dots,k_\ell) \in \BN^\ell$, $ \brH^\bk= \prod_{j=1}^\ell \brH_{\al_j}^{k_j}.$

Then $ \{ \bbe_h''(\bn) \mid \bn \in \BN^{t + \ell +t} \}$
is a topological basis of $\Uh$. From \eqref{e80},
\be
\label{eq.bch}
 \bc= \sum_{\bn \in \BN^{t+\ell +t}} u_h(\bn) \,h^{\| \bn\| } \bbe_h(\bn) \ot \bbe''_h(\bn),
 \ee
where $u_h(\bn)$ is a unit $  \Ch$  for each $\bn\in \BN^{t+\ell +t}$. The  exact value of $u_h(\bn)$ is as follows: For $\bn=(\bn_1,\bn_2,\bn_3) \in \BN^{t+\ell+t}$,
 \be
 \label{eq.unith}
  u_h{(\bn_1,\bn_2,\bn_3)}= q^{-(\rho, |E_{\bn_3}|)}\, u''_h(\bn_1)\,  u_h'(\bn_2)\,  u''_h(\bn_3)\,
  \ee
 where for $ \bk=(k_1,\dots,k_\ell) \in \BN^\ell$ and $ \bm =(m_1,\dots,m_t) \in \BN^t$,
  $$ u'_h(\bk) = \prod_{j=1}^\ell \frac{(-1)^{k_j}}{ k_j!\,  d_{\al_j}^{k_j}}, \quad u''_h({\bm})= \prod_{j=1}^t \frac{v_{\gamma_j}^{-m^2_j}(q_{\gamma_j};q_{\gamma_j})_{m_j}}{h^{m_j}}.$$

By definition, the left image of $\bc$ is the topological closure of the $\Ch$-span of $\{ u_h(\bn) \,h^{\Vert \bn \Vert} \bbe_h(\bn) \}$, which is the same as $\Vh$, since the $u_h(\bn)$ are units in $\Ch$.

Since $\bc$ is ad-invariant, by Proposition \ref{r.stable}, we have $\Uh\tri \Vh \subset \Vh$.
\end{proof}

\begin{remark}
Proposition \ref{r.42a}  shows that $\Vh$ does not depend on the choice of the longest reduced sequence $\bi$.
\end{remark}

\begin{proposition}
\label{r.tphiX}
One has $\tphi(\Vh) \subset \Vh$, i.e. $\Vh$ is $\tphi$-stable.
\end{proposition}
\begin{proof}
By Lemma \ref{r5}, $\bc= (\varphi \ot \uS^{-1} \varphi)(\bc)$. Note that  $\uS^{-1} \varphi$ is an $\Ch$-linear automorphism of $\Uh$. By Proposition \ref{r.stable}(b), $\varphi$ leaves stable the left image of $\bc$, i.e. $\varphi(\Vh) = \Vh$.
\end{proof}

\subsection{Extension of ground ring and stability principle}\label{sec:topdil}
 Let $\sqrt h$ be an intermediate such that $h= (\sqrt h)^2$. Then $\Ch\subset \Chh$. For a $\Ch$-module homomorphism $f:V \to V'$, we often use the same symbol $f$ to denote $f\ho \id: V\ho_{\Ch} \Chh \to V'\ho_{\Ch} \Chh$.

Suppose the following data are given

(i) a topologically free $\Ch$-module $V$ equipped with a topological base $\{ e(i) \mid i\in I \}$, and

(ii) a function $a: I \to \Ch$ such that $a(i) \neq 0$ and $\{ a(i), i\in I\}$ is 0-convergent.

Let $V(\sqrt a)$ be the topologically free $\Chh$-module with topological basis $\{ \sqrt {a(i)} \, e(i) \mid i \in I\}$, and $V(a)\subset V$ be the closure (in the $h$-adic topology of $V$) of the $\Ch$-span of
$\{ a(i) \, e(i) \mid i \in I\}$. We call $(V, V(\sqrt a), V(a))$ a {\em topological dilatation triple} defined by the data given in (i) and (ii).

\begin{proposition}[Stability principle]
\label{r.topdilat}
Suppose $(V, V(\sqrt a), V(a))$  and $(V', V'(\sqrt {a'}), V'(a'))$ are two topological dilatation triples and $f: V\to V'$ is
a $\Ch$-module homomorphism such that $f(V(a)) \subset V'(a')$. Then $f(V(\sqrt a)) \subset V'(\sqrt {a'})$.
\end{proposition}
\begin{proof}
{\it Claim 1.} If $x_1,x_2,x_3\in \Ch$, $x_3\neq0$, such that $x_1 x_2 /x_3  \in \Ch$ then $x_1 \sqrt {x_2/x_3} \in \Chh$.\\
{\em Proof of Claim 1.} Let $x_i= h^{k_i} y^i$, where $y_i$ is invertible in $\Ch$. Assumption $x_1 x_2 /x_3  \in \Ch$ means $k_1 + k_2\ge k_3$. Then $k_1+ k_2/2 \ge (k_1 +k_2)/2 \ge k_3/2$, which implies the claim.

Let us now prove the proposition. The $\Ch$-module $V(a)$ is a formal series $\Ch$-module with formal basis
$\{ a(i)  e(i) \mid i\in I\}$, see Example \ref{exa:1}.
Every  $x\in V(a)$ has a unique presentation as an $h$-adically convergent sum
\be
 x= \sum_{i\in I} x_i \, \left ( a(i) e(i)\right), \quad \text{where $(x_i)_{i\in I} \in \Ch^I$.} \notag
 \ee

Using the topological bases $\{e(i) \mid i\in I\}$ of $V$ and $\{ e'(i') \mid i'\in I'\}$ of $V'$, we have
$$ f(e(i)) = \sum_{j \in I'}f_{i}^j  e'(j),$$
where $ f_{i}^j\in \Ch$, and for a fixed $i$, $\{ f_{i}^j \mid j\in I'\}$ is 0-convergent. Multiplying by appropriate powers of $\sqrt{a(i)}$, we get
\begin{align}
\notag   f(a(i) e(i)) &= \sum_{j\in I'}{\tilde f}_{i}^j \,  \left( a'(j) e'(j)\right) \quad  & \text{where } \tilde f_{i}^j &  =   \frac{a(i)}{a'(j)}   f_{i}^j\\
\label{eq.sd3}   f( \sqrt { a(i)}\, e(i)) &= \sum_{j\in I'}{\tilde \tilde f}_{i}^j   \left(\sqrt { a'(j)}\, e'(j)\right) \quad  &\text{where }
{\Tilde \Tilde f}_{i}^j &  =    \frac{\sqrt {a(i)}}{\sqrt{a'(j)}}  f_{i}^j.
\end{align}
The assumption $f(V(a)) \subset V'(a')$ implies that
 $ \tilde f_{i}^j \in \Ch$, which, together with $ f_{i}^j\in \Ch$ and Claim 1, shows that $ {\Tilde \Tilde f}_{i}^j \in \Chh$.
  Equation \eqref{eq.sd3} shows that $f(V(\sqrt a)) \subset V'(\sqrt {a'})$.
 This proves the proposition.
\end{proof}

\subsection{Definition of $\Xh$} \label{sec:sX}
Fix a longest reduced sequence $\bi$.
Recall that $\{\bbe_h(\bn) \mid \bn \in \BN^{t+\ell+t}\}$ is a topological basis of $\Uh$, see Section \ref{sec.Vh}. Let $a: \BN^{t+\ell+t} \to \Ch$ be the function defined by $a(\bn)= h^{\|\bn\|}$, and consider the topological dilatation triple $(\Uh, \Uh(\sqrt a), \Uh(a))$.
Denote the middle one by $\Xh$, $\Uh(\sqrt a)= \Xh$. Later we show that $\Xh$ does not depend on $\bi$.

By definition, $\Uh(a)$ is the  closure (in the $h$-adic topology of $\Uh$) of the $\BC[[h]]$-span of $\{h^{\|\bn\|}\bbe_h(\bn) \mid \bn \in \BN^{t+\ell+t}\}$. Thus, $\Uh(a)=\Vh$.

Also by definition, $\Xh$ is the topologically free $\Chh$-module with the topological basis
\be
\label{eq.basissX}
\{h^{\|\bn\|/2}\bbe_h(\bn) \mid \bn \in \BN^{t+\ell+t}\}.
\ee
 Note that $\Xh$ is a submodule of
$\Uhalf= \Uh \ho_{\Ch} \BC[[\sqrt h]] $.

The topological closure $\overline{\Xh}$ of $\Xh$ in $\UUH$ is a formal series $\Chh$-module with \eqref{eq.basissX} as a formal basis.

\begin{theorem} \label{r.sXalg}
The $\Chh$-module $\Xh$ is a topological Hopf subalgebra of $\UUH$.
 Moreover $\UUH \tri \Xh \subset \Xh$, i.e. $\Xh$ is ad-stable, and $\tphi(\Xh)\subset \Xh$.
\end{theorem}
\begin{proof} We will show that $\Xh$ is closed under all the Hopf algebra operations of $\UUh$.

Let us first show that $\Xh$ is closed under the co-product. Both $(\Uh, \Xh, \Vh)$ and $(\Uh^{\ho 2}, \Xh^{\ho 2}, \Vh^{\Oo 2})$ are topological dilatation triples, and
$\Delta(\Uh) \subset \Uh^{\ho 2}$ and $\Delta(\Vh) \subset \Vh^{\Oo 2}$ (see Proposition \ref{r.41a}). Hence, by the stability principle (Proposition \ref{r.topdilat}), $\Delta(\Xh) \subset \Xh ^{\ho 2}$.

Similarly,
applying stability principle to all the operations of a Hopf algebra, namely $\boldsymbol \mu, \boldsymbol \eta, \boldsymbol \Delta, \boldsymbol \epsilon, S$ (using Proposition \ref{r.41a}), as well as  the adjoint actions (using Proposition \ref{r.42a}) and the map $\tphi$ (using Proposition \ref{r.tphiX})  we get the results.
\end{proof}

\begin{corollary} \label{r.phu2}
Fix a longest reduced sequence $\bi$.
The $\Chh$-algebra $\Xh$ is the topologically complete subalgebra of $\UUH$ generated by $\sqrt h H_\al, \sqrt h E_\gamma(\bi), \sqrt h F_\gamma(\bi)$, with $\al \in \Pi, \gamma \in \Phi_+$.
\end{corollary}
\begin{proof}
Using the fact that $\Xh$ is an algebra, the proof is the same as that of Proposition \ref{r.phu1}.
\end{proof}

\subsection{$\Xh$ is a core subalgebra of $\UUH$}
Recall that the definition of a core subalgebra is given in Section \ref{sec:coresub}.
\begin{theorem}\label{thm.coresub}
 The subalgebra $\Xh$ is a core subalgebra of the topological ribbon Hopf algebra $\UUH$.
\end{theorem}
\begin{proof} For the convenience of the reader, we recall the definition of a  core subalgebra: $\Xh$ is a core subalgebra of $\UUH$ means that $\Xh$ is a topological Hopf subalgebra of $\UUH$ and the following (i)--(iii) holds.

(i) $\Xh$ is $\UUH$-stable,

(ii) $\cR\in \overline{\Xh \ot \Xh}$ and $K_{2 \rho} \in \overline{\Xh}$, and

(iii) The clasp element $\bc$ has a presentation
$$ \bc = \sum_{i\in I} \bc'(i) \ot \bc''(i),$$
where each of  $\{ \bc'(i) \}$ and $ \{ \bc''(i) \}$ is 0 convergent in $\UUH$ and is a topological basis of $\Xh$.

Let us look at all three statements.

(i)
By Theorem  \ref{r.sXalg}, $\Xh$ is a topological Hopf subalgebra of $\UUH$, and (i) holds.

(ii) Since $\sqrt h H_\al\in \Xh$ (see Corollary \ref{r.phu2}),
 $ K_{\pm 2\rho} =\exp(\pm \sum_{\al \in \Phi_+}hH_\al)\in \Xh \subset \overline{\Xh}$.

By \eqref{501},   $ \cR ^{-1}= \Theta \cD^{-1}$,
 where
 $$\Theta =  \sum_{\bn \in \BN^t} F_\bn \ot E_\bn \quad \text{and} \quad  \cD^{-1}=\exp(-\frac h2 \sum _{\al \in \Pi} H_\al \ot \brH_{\al}/d_\al).$$

 As $\sqrt h H_\al, \sqrt h \brH_\al\in \Xh$, one has $\cD^{-1}\in \overline{\Xh \ot \Xh}$.

Using the definition \eqref{eq.En}, \eqref{eq.Fn} of $E_\bn,F_\bn$, and Corollary \ref{r.phu2}, we have
$$ F_\bn \ot E_\bn \sim  \iprod_{\gamma_j\in \Phi_+} (hF_{\gamma_j} \ot E_{\gamma_j})^{n_j} \in \Xh \ot \Xh,$$
where $a\sim b$ means $a=ub$, where $u$ is a unit in $\Ch$. Hence
$ \Theta = \sum F_\bn \ot E_\bn \in  \overline{\Xh \ot \Xh}$. It follows that $\cR^{-1}= \Theta \cD^{-1} \in \overline{\Xh \ot \Xh}$. Since $\cR = (\id \ho S) (\cR^{-1})$, we also have $\cR\in \overline{\Xh \ot \Xh}$.
  Thus (ii) holds.

(iii) Let $I= \BN^{t+\ell+t}$, and for $\bn \in I$,
 \be
 \label{eq.za1}
 \bc'(\bn) = h^{\| \bn\| /2} \bbe_h(\bn) , \quad \bc''(\bn) =  u_h(\bn)  h^{\| \bn\| /2} \bbe''_h(\bn),
 \ee
 where $u_h(\bn)$ is the unit of $\Ch$ in \eqref{eq.unith}.
 By \eqref{eq.bch},
 $$ \bc = \sum_{\bn} \bc'(\bn) \ot \bc''(\bn).$$
By definition, $\{ \bc'(\bn)\}$ is a topological basis of $\Xh$. Since $\{ \brH_\al \mid \al \in \Pi\}$ is
a basis of ${\mathfrak h}^*_\BR$, $\{ \bc''(\bn)\}$ is also a topological basis of $\Xh$. The factors $h^{\| \bn\| /2}$
in \eqref{eq.za1} shows that each set  $\{ \bc'(\bn)\}$ and $\{ \bc''(\bn)\}$ is 0-convergent. Hence (iii) holds. This completes the proof of the theorem.
\end{proof}

By Theorem \ref{r.JMdef}, the core subalgebra $\Xh$ gives rise to an invariant $J_M\in \Chh$ of integral homology 3-spheres $M$, via the twists $\cT_\pm$ which we will study in the next subsections.

\subsection{Quantum Killing form}
\label{sec:rosso-form}
Since $\Xh$ is a core subalgebra of $\UUH$, according to Section \ref{sec:coresub}, one has a clasp form, which is a $\UUH$-module homomorphism
\be
\label{eq.formLL}
\sL : \overline{\Xh} \ho \Xh  \to \Chh,
\ee
defined by
\be
\label{eq.defform}
 \sL (\bc''(\bn) \ot  \bc'(\bm) ) =\delta_{\bn,\bm}, \quad \text{ for }\bn, \bm \in \BN^{t+\ell+t   }
\ee
where $\bc''(\bn)$ and $ \bc'(\bm)$ are given by \eqref{eq.za1}.
We also denote $\sL(x \ot y)$ by $\la x, y \ra$.

Let us calculate explicitly the form $\sL$.
Recall that $F_\bn,E_\bn\in \Uq$ were defined by \eqref{eq.En} and \eqref{eq.Fn}, which depend on a longest reduced sequence.
\begin{proposition} Fix a longest reduced sequence $\bi$.
For $\bm,\bn,\bn',\bm'\in \BN^t, \bk,\bk' \in \BN^\ell, \al, \beta \in Y, k,l \in \BN$, one has
\begin{align}
  \langle F_\modm K_\modm \,  h^{k/2} H^k_\al\,   E_\modn ,F_{\modn '}K_{\modn '} \,  h^{l/2} H^l_\beta \,  E_{\modm '}\rangle
  &=\delta_{k,l} \delta _{\modm ,\modm '}\delta _{\modn ,\modn '} \, q^{(\rho ,|E_\modn |)}\, (-1)^k \, k! \, (\al,\beta)^k   \label{e57a1}\\
  \langle F_\modm K_\modm K_{\mu} E_\modn ,F_{\modn '}K_{\modn '}K_{\mu '}E_{\modm '}\rangle
  &=\delta _{\modm ,\modm '}\delta _{\modn ,\modn '}q^{(\rho ,|E_\modn |)}v^{-(\mu ,\mu ')/2} \label{e57a}
\end{align}
\end{proposition}
\begin{proof}
 Formula \eqref{e57a1} is obtained from
\eqref{eq.defform} by a simple calculation, using the definition \eqref{eq.za1}
of $\bc'(\bn)$ and $\bc''(\bn)$. Formula \eqref{e57a} is obtained from \eqref{e57a1} using the expansion $K_\mu= \exp(h H_\mu/2) = \sum_{k} h^k H^k_\mu/(2^kk!)$.
\end{proof}

Suppose  $x,y\in \Uq$. There are non-zero $a,b\in \Cv$ such that $ax,by \in \Xh$.
By \eqref{e57a}, $\la ax, by \ra \in \BC[v^{\pm 1/2}]$. Hence we can define $\la x,y \ra = \frac{\la ax, by \ra}{ab} \in \BC(v^{1/2})$. Thus, we have a $\BC(v)$-bilinear form
\be
\label{eq.qformUq}
\la., . \ra : \Uq \ot \Uq \to \BC(v^{1/2}).
\ee

\begin{remark} The form we construct is not new. On $\Uq$ the form
$\sL$ is exactly the quantum Killing form (or the Rosso form)  \cite{Rosso,Tanisaki} (see \cite{Jantzen}),  which was constructed via an elaborate process. For example, if one defines the quantum Killing form by \eqref{e57a}, then it not easy to check
the ad-invariance of the quantum Killing form.
Essentially here we give a geometric characterization of the quantum Killing form:
 it is the dual of the clasp element $\modc$. The
ad-invariance of the quantum Killing form then follows right away
from the ad-invariance of $\modc$. We also determine the space $\Xh$, which in a sense is the biggest space for which the quantum Killing form can be defined (with values in $\Ch$).
\end{remark}

\subsection{Properties of quantum Killing form}
We again emphasize that the form $\sL$ is ad-invariant, i.e. the map $\sL$ in \eqref{eq.formLL} is a $\UUH$-module homomorphism, see Lemma \ref{r.invar}. It follows that the form \eqref{eq.qformUq} is $\Uq$-ad-invariant.

Since each of $\{ \bc'(\bn)\}$ and $\{ \bc''(\bn) \}$ is a topological basis of $\Xh$ and they are dual to each other,  the bilinear form $\langle .,.\rangle $ is non-degenerate.

From \eqref{e57a1}, we see that the quantum Killing form is {\em triangular} in the following sense.  Let $x,x' \in \Xh \cap \Uh^{\ev,-}$, $y,y'\in \Xh \cap \Uh^0 $, and $z,z'\in \Xh \cap \Uh^+$, then
\begin{gather}
  \label{e89}
  \langle xyz,x'y'z'\rangle =\langle x,z'\rangle \langle y,y'\rangle \langle z,x'\rangle .
\end{gather}

The quantum Killing form is uniquely determined up to a scalar
 by the ad-invariant, non-degenerate, and triangular
properties, see \cite[Theorem 4.8]{JL2}.

The quantum Killing form is not symmetric. In fact, for $x,y\in \Xh$, we have
\begin{gather*}
  \langle y,x\rangle = \langle x,S^2(y)\rangle =\langle S^{-2}(x),y\rangle,
\end{gather*}
which follows from the identity $(\id\ot  S^2) (\bc) = \bc_{21}$. If $y$ is central, then $S^2(y)=K_{-2\rho} y K_{2\rho}=y$. Hence
\be
\label{eq.qkcentral}
\la x, y \ra = \la y ,x \ra \quad \text{if $y$ is central}.
\ee

The quantum Killing form extends to a multilinear form
$$ \la., .\ra : \overline{\Xh ^{\ot n}} \, \ho \Xh^{\ho n} \to \Chh,$$
where $\overline{\Xh ^{\ot n}}$ is the topological closure of $\Xh ^{\ot n}$,
by $$\la  x_1\ot \dots \ot x_n, y_1\ot \dots \ot y_n \ra = \prod_{j=1}^n \la x_j, y_j \ra.$$

\begin{lemma}
Suppose $x,y,z$ are element of  $\Xh^0=\Xh \cap \UUH^0$. Then
\be
\label{eq.sd10}
\la xy , z \ra = \la x \otimes y, \Delta(z)\ra
\ee
\end{lemma}

\begin{proof}
This follows from \eqref{e57a1}, with $\bn=\bm=0$.
\end{proof}

Note that \eqref{eq.sd10} does not hold for general $x,y,z\in \Xh$.

\subsection{Twist system associated to $\Xh$ and invariant of integral homology 3-spheres} \label{sec.twist1}
According to the result of Section \ref{sec.partii}, the core subalgebra  $\Xh$ gives rise to a twist system
 $\cT_\pm: \Xh \to \Chh$, defined by
$$ \cT_\pm (x) = \la \br^{\pm 1}, x\ra$$
and an invariant $J_M\in \Chh$ of integral homology 3-spheres $M$. Recall that $J_M$ is defined as follows.
Suppose $T$  is an $n$-component bottom tangle with 0 linking matrix and $\ve_i \in \{ -1,1\}$
 and $M$ is obtained from $S^3$ by surgery along the closure link $\cl(T)$ with the framing of the $i$-th component switched to $\ve_i$. Then
$$ J_M= (\cT_{\ve_1} \ho \dots \ho \cT_{\ve_n})(J_T).$$
In the next few sections we will show that $J_M \in \Zqh$.

Let us calculate the values of $\cT_\pm$ on basis elements. Recall that
$$\br_0= K_{-2\rho} \exp(-\frac 12 \sum_{\al \in \Pi} H_\al \brH_\al/d_\al).$$
\begin{proposition} \label{p.989}

 (a) Fix a longest reduced sequence $\bi$. For $\bm,\bn\in \BN^t, \gamma\in Y, x\in \Xh^0$, one has
\begin{align}
\label{eq.989}  \cT_+(F_\bm K_\bm \, x \,  E_\bn) &=  \delta _{\modm ,\modn }\, q^{(\rho ,|E_\modn |)}\, \la \br_0 , x \ra\\
\label{eq.989a} \la \br_0 , K_{ \gamma} \ra &=   v^{(\gamma ,\rho )-\frac{1}{4}(\gamma ,\gamma )} \in \BZ[v^{\pm 1/2}].
\end{align}

(b) For every $x \in \Xh$, one has
\be \cT_-(x) =  \cT_+( \tphi(x)).
\label{eq.cTminus}
\ee
\end{proposition}
\begin{proof}
(a) By \eqref{e80r},
$$ \br= \sum_{\bn \in \BN^t} F_\bn K_\bn \br_0 E_\bn.$$
Identity \eqref{eq.989} follows from the triangular property of the quantum Killing form.
Identities in  \eqref{eq.989a} follow from a calculation using \eqref{e57a1} and the explicit expression of $\br_0$.

(b) By \eqref{e80b}, $\bc = (\tphi \ho \bS^{-1} \tphi)(\bc)$. By Proposition \ref{r.vv1},
for $y\in \overline{\Xh}$ and $x\in \Xh$, one has
\be  \la y, x \ra = \la \underline S^{-1} \tphi (y), \tphi(x) \ra. \label{eq.za2}
\ee

 By Corollary \ref{r.phir} and \eqref{eq.rin1},
 $ \underline S^{-1} \, \tphi (\br^{-1})= \br$. Using \eqref{eq.za2} with $y=\br^{-1}$, we get \eqref{eq.cTminus}.
\end{proof}

\subsection{Twist forms on $\Uq$}\label{twistonUq}
 By construction we have twist forms $\cT_\pm: \Xh \to \Chh$, with domain $\Xh$ and codomain $\Chh$. We can  change  the domain to get a better image space.

By Proposition \ref{p.989}, for $\bm,\bn \in \BN^t$ and $\gamma \in Y$,
\be
\label{eq.ba11}
\cT_+ (F_\bm K_\bm K_{2\gamma}  \,  E_\bn)= \delta _{\modm ,\modn }\, q^{(\rho ,|E_\modn |)}\,v^{2(\gamma,\rho) - (\gamma,\gamma)} \in \Zq \subset \BZ[v^{\pm 1}].
\ee

Because $\{ F_\bm K_\bm K_{2\gamma}  \,  E_\bn \mid \bm,\bn \in \BN^t,\gamma \in Y\}$ is a $\BC(v)$-basis of
$\Uq^\ev$, we have
$$ \cT_+(\Uq^\ev \cap \Xh) \subset \BC(v)\cap \Chh.$$
Using $\cT_-(x)= \cT_+(\varphi(x))$ (see Proposition \ref{p.989}), and the fact both $\Uq^\ev$ and $\Xh$ are $\varphi$-stable, we also have
$$ \cT_-(\Uq^\ev \cap \Xh) \subset \BC(v)\cap \Chh.$$

Because $\Uq^\ev \cap \Xh$ spans $\Uq^\ev$ over $\BC(v)$,
we can extend the restriction of $\cT_\pm$ on  $\Uq^\ev \cap \Xh$ to  $\BC(v)$-linear maps, also denoted by $\cT_\pm$:
$$ \cT_\pm: \Uq^\ev \to \BC(v).$$
The values of $\cT_+$ on the basis elements are given by \eqref{eq.ba11}. It is clear that
\be
\label{eq.ba12}
\cT_\pm(\UZ^\ev) \subset \BQ(v).
\ee

\np
\section{Integral core subalgebra}
\label{sec:integcore}

In Section \ref{sec:corealgebra} we constructed a core subalgebra $\Xh$ of $\UUh$ which gives rise to an invariant $J_M$ of integral homology 3-spheres with values in $\Chh$. To show that $J_M$ takes values in $\Zqh$ we need an integral version of the core algebra.
This section is devoted to an integral form $\XZ$ of the core algebra $\Xh$.

In order to construct $\XZ$ we first introduce  Lusztig's integral form $\UZ$  and De Concini-Procesi's integral form $\VZ$.
 Then we construct $\XZ$ so that $(\UZ,\XZ,\VZ)$ form an {\em  integral dilatation triple} corresponding to the topological dilatation triple $(\Uh, \Xh,\Vh)$.

Lusztig introduced  $\UZ$ in connection with his discovery (independently with Kashiwara) of canonical bases.
De Concini and Procesi introduced $\VZ$   in connection with their study of geometric aspects of quantized enveloping algebras. For the study of the integrality of quantum invariants,  Lusztig's integral form
$\UZ$ is too big: it does not have necessary integrality
properties. For example, the quantum Killing form $\la x, y\ra$ with
$x, y \in \UZ$ belongs to $\BQ(v^{1/2})$ but not to $\BZ[v^{\pm 1/2}]$
in general. On the other hand De Concini-Procesi's form $\VZ$ is too small in the
sense that completed tensor powers of $\VZ$ do not contain the
universal invariant of general bottom tangles.  (Recently, however,
Suzuki \cite{Suzuki1,Suzuki2} proved that, for $\mathfrak{g}=sl_2$,
the universal invariant of {\em ribbon} and {\em boundary} bottom
tangles is contained in completed tensor powers of $\VZ$.)  Our integral form $\XZ$ is the perfect middle ground since it is big enough to contain quantum link invariants and small enough to have the necessary integrality.
We believe that $\XZ$ is the right integral form for the study of quantum invariants of links and 3-manifolds.

We
will show that  De Concini-Procesi's  $\VZ$ is ``almost'' dual to Lusztig's $\UZ$ under the quantum Killing form, see the precise statement in Proposition~\ref{r.ortho3}.
 This fact can be interpreted as an integral version of the duality of Drinfel'd and  Gavarini \cite{Drinfeld,Gavarini}.
 Using the duality we then show that the even part  of $\VZ$ is invariant under the adjoint action of $\UZ$, an important result which will be used frequently later. We then show that the twist forms have nice integrality on $\XZ$.

\subsection{Dilatation of based free modules} \label{sec.dilita}
Let $\tA$ be the extension ring of $\cA=\BZ[ v ^{\pm 1}]$ obtained by adjoining all $\sqrt{\phi_n(q)}$, $n=1,2,\dots$, to
$\cA$.
 Here $\phi_n(q)$ is the $n$-th cyclotomic polynomial and $q=v^2$. One reason why working over $\tA$ is not too much a sacrifice is the following.

\begin{lemma}\label{r.60}
One has $\tA \cap \BQ(q)= \Zq$.
\end{lemma}
\begin{proof}Since  $\sqrt {\phi_k(q)}$ is integral over $\Zq$, $\tA$ is integral over $\Zq$.
Hence $\tA \cap \BQ(q)= \Zq$.
\end{proof}

Suppose $V$ is {\em  based free $\cA$-module}, i.e. a free $\cA$-module
equipped with a preferred  base $\{ e(i) \mid i\in I\}$.
Assume $a: I \to \cA$ is a function such that for every $i\in I$, $a(i)$ is a product of cyclotomic polynomials in $q$. In particular, $a(i) \neq 0$ and $\sqrt {a(i)} \in \tA$. The based free $\cA$-module
 $V(a)\sub V$, with preferred base
$\{ a(i) e(i) \mid i\in I\}$, is called a {\em dilatation of $V$}, with dilatation factors $a(i)$.
Let $V(\sqrt a)$ be the based free $\tA$-module with preferred base $\{\sqrt{ a(i)}\, e(i) \mid i\in I\}$.
We call $(V, V(\sqrt a), V(a))$ a {\em dilatation triple} determined by the based free $\cA$-module $V$ and the function $a$.

We will introduce the Lusztig integral form $\UZ$, the integral core algebra $\XZ$, and the De Concini-Procesi integral form $\VZ$ so that $(\UZ,\XZ,\VZ)$ is a dilatation triple.

\subsection{Lusztig's integral form $\UZ$} \label{sec:Lus}
Let $\UZ$ be the $\modA $-subalgebra of $\Uq$ generated by all
$E_\alpha ^{(n)}, F_\alpha ^{(n)}, K_\alpha ^{\pm 1}$, with $\alpha \in  \Pi $ and $n \in
\BN$.   Set
$\UZ^*=\UZ\cap \Uq^*$ for $*=-,0,+$.

 Let us collect some well-known facts about $\UZ$. Recall that $E^{(\bn)}$ and $F^{(\bn)}$, defined for $ \bn \in \BN^t$
  in Section \ref{s.PBW}, depend on the choice of a longest reduced sequence.

\begin{proposition} \label{prop.basis} Fix a longest reduced sequence $\modi $.

(a) The $\cA$-algebra $\UZ$ is a Hopf subalgebra of $\Uq$, and  satisfies the
triangular decomposition
$$\UZ^-\otimes\UZ^0\otimes\UZ^+\overset{\cong}{\longrightarrow}\UZ,\quad x\otimes y\otimes z\mapsto
 xyz.$$

Moreover, $\UZ$ is stable under the action of  $T_\al^{\pm1}, \al \in \Pi$.

(b)
The set $\{F^{(\bn)}\mid  \bn \in \BN^t\}$ is a free $\cA$-basis of the $\cA$-module $\UZ^-$. Similarly, $\{E^{(\bn)}\mid  \bn \in \BN^t\}$ is a free $\cA$-basis of $\UZ^+$.

(c) The Cartan part $\UZ^0$ is the $\cA$-subalgebra of $\Uq^0$ generated by $K_\al^{\pm1}, \frac{(
K^2_\al;q_\al)_n}{(q_\al;q_\al)_n}, \alpha \in \Pi, n\in \BN
$.

(d) The algebra $\UZ$ is stable under $\ibar, \tau$, and $ \varphi$. Moreover $\UZ^-$ is stable under $\ibar$ and $\tau$.
\end{proposition}
\begin{proof}
Parts (a)--(c) are proved in  \cite{Lusztig01} and \cite[Proposition 41.1.3]{Lusztig}. Part (d) can be proved by noticing  that  each of $\ibar,\tau,\varphi$ maps each of the generators  $E_\alpha ^{(n)}, F_\alpha ^{(n)}, K_\alpha ^{\pm 1}$ of $\UZ$ into $\UZ$, and  each of $\ibar$ and $\tau$ maps each of the generators $F^{(n)}_\al$ of $\UZ^-$ into $\UZ^-$.
\end{proof}

We will consider $\UZ^-, \UZ^+$ as based free $\cA$-modules with preferred bases described in Proposition~\ref{prop.basis}(b). Later we will find a preferred base for the Cartan part $\UZ^0$.

Let $\UZ^\ev = \UZ \cap \Uq^\ev$ be the even part of $\UZ$. From the triangulation of $\UZ$ we have the following {\em even triangulation} of $\UZ$ and $\UZ^\ev$:
\begin{align}
\UZ^{\ev,-}\otimes\UZ^0\otimes\UZ^+\overset{\cong}{\longrightarrow}\UZ,\quad x\otimes y\otimes z\mapsto
 xyz\\
 \UZ^{\ev,-}\otimes\UZ^{\ev,0}\otimes\UZ^+\overset{\cong}{\longrightarrow}\UZ^\ev,\quad x\otimes y\otimes z\mapsto
 xyz.
\end{align}
Here
$\UZ^{\ev,0} = \UZ^\ev \cap \Uq^{\ev, 0}$, with $\Uq^{0}=\BC(v)[K_\al ^{\pm 2}, \ \al \in \Pi]$,  and $\UZ^{\ev,-} = \UZ^\ev \cap \Uq^{\ev,-} = \varphi(\UZ^+)$.

From Proposition \ref{prop.basis}(b) and $\UZ^{\ev,-}=\varphi(\UZ^+)$, we have the following.
\begin{proposition}\label{r.ba3}
The set $\{F^{(\bn)} K_{\bn} \mid  \bn \in \BN^t\}$ is a free $\cA$-basis of the $\cA$-module $\UZ^{\ev,-}$.
\end{proposition}
We will consider $\UZ^{\ev, -}$ as a based free $\cA$-module with the above preferred basis.

\subsection{De Concini-Procesi integral form $\VA$}
\label{sec:algebra-cv}

Let $\VA $  be the smallest $\modA $-subalgebra of $\UZ$ which is
invariant under the action of the braid group and contains
$(1-q_\al) E_\alpha $, $(1-q_\al) F_\alpha $ and $K_\alpha ^{\pm 1}$ for $\alpha \in \Pi $.
For $*=0,+,-$, set $\VA ^*=\VA \cap \Uq^*$.

\begin{remark}
In the original definition, De Concini and Procesi \cite[Section 12]{DeConcini-Procesi} used the ground ring $\BQ[v^{\pm 1}]$ instead of $\cA=\BZ[v^{\pm 1}]$. Our $\VZ$ is denoted by $A$ in \cite{DeConcini-Procesi}
\end{remark}

Fix  a longest reduced sequence $\bi$.
For  $ \bn=(n_1,\dots,n_t) \in \BN^t$, let
\be
\label{eq.qqn1}
 (q;q)_{\bn} = \prod_{j=1}^t (q_{\gamma_j}; q_{\gamma_j})_{n_j}.
 \ee
Note that $(q;q)_{\bn}$ depends on $\bi$ since $\gamma_j=\gamma_j(\bi)$ depends on $\bi$.

\begin{proposition} Fix a longest reduced sequence $\modi $.
  \label{r11}

(a) The $\cA$-algebra $\VA $ is a  Hopf subalgebra of $\UZ$.

(b)  We have $\VA ^0=\modA [K_1 ^{\pm 1}, \dots, K_\ell ^{\pm 1} ]$ and the triangular
  decomposition
  \begin{gather*}
    \VA ^-\otimes \VA ^0\otimes \VA ^+\congto\VA ,\quad x\otimes y\otimes z\mapsto xyz.
  \end{gather*}

  (c) The set $\{ (q;q)_\bn \, F^{(\bn)}\mid  \bn \in \BN^t\}$ is a free $\cA$-basis of the $\cA$-module $\VA^-$.
Similarly, $\{ (q;q)_\bn \, E^{(\bn)}\mid  \bn \in \BN^t\}$ is a free $\cA$-basis of $\VA^+$.
\end{proposition}

\begin{proof} The proofs for the case when
 $\modA =\modZ [v^{\pm 1}]$ is replaced by $\BQ[v^{\pm
1}]$, were given in \cite[Section 12]{DeConcini-Procesi}.  The proofs there remain valid for
$\modA $. Note that in \cite{DeConcini-Procesi}, our $\VZ$ is denoted by $A$.
\end{proof}

The even part $\VAev:=\VA \cap \Uqv$ is an $\modA $-subalgebra of $\VA $.
 From the triangular decomposition of $\VA$, we have the following {\em even triangular decompositions}
\begin{gather}
  \label{e84}
  \VA ^{\ev,-}\otimes \VA ^{\ev,0}\otimes \VA ^+\congto\VA^\ev,\quad\xyzxyz, \\
  \VA ^{\ev,-}\otimes \VA ^{0}\otimes \VA ^+\congto\VA,\quad\xyzxyz,
\end{gather}
where $\VA ^{\ev,0}:=\VA \cap \Uq^{\ev,0}=\modA [K_1 ^{\pm 2}, \dots, K_\ell ^{\pm 2} ]$ and
$\VA ^{\ev,-}:=\VA \cap \Uq^{\ev,-}=\varphi  (\VA ^+)$.

From Proposition \ref{prop.basis}(b) and $\Uq^{\ev,-}=\varphi(\Uq^+)$, we have the following.
\begin{proposition}\label{r.ba3a}
The set $\{(q;q)_\bn F^{(\bn)} K_{\bn} \mid  \bn \in \BN^t\}$ is a free $\cA$-basis of the $\cA$-module $\VA^{\ev,-}$.
\end{proposition}
We will consider $\VA^{\ev, -}$ as a based free $\cA$-module with the above preferred basis. Then $\VA^{\ev,-}$ is a dilatation of $\UZ^{\ev,-}$.
Similarly, we consider $\VA^+$ as a based free $\cA$-module with preferred base given in Proposition \ref{r11}.
Then $\VA^+$ is a dilatation of $\UZ^+$.

\subsection{Preferred bases for $\UZ^0$ and $\VZ^0$} \label{sec:pre-bas}
We will equip $\UZ^0$ and $\VA^0$ with preferred $\cA$-bases such that $\VA^0$ is a dilatation of $\UZ^0$.
Recall that $K_j=K_{\al_j}$, $ q_j = q_{\al_j}$.

For $ \bn=(n_1,\dots,n_\ell) \in \BN^\ell$ and $\bode=(\delta_1, \dots, \delta_\ell) \in \{0,1\}^\ell$ let

\begin{align}
\label{eq.Qn14}
Q^\ev(\bn)  & := \prod_{j=1}^\ell  \frac{ K_j^{-2 \lfloor \frac {n_j}2 \rfloor}\left(q_j^{-\lfloor\frac{ {n_j-1}}2 \rfloor}K_j^2;q_j\right)_{n_j}}{(q_j;q_j)_{n_j}}, \\   Q(\bn,\bode) &:= Q^\ev(\bn) \prod_{j=1}^\ell K_j^{\delta_j }\\
\label{eq.qqn2}
(q;q)_\bn&:=\prod_{j=1}^\ell (q_j;q_j)_{n_j}.
\end{align}

\begin{proposition}
 \label{r.bases5}
(a)  The sets
$\{ Q^\ev(\bn) \mid \bn \in \BN^\ell\}$ and $\{ (q;q)_\bn Q^\ev(\bn) \mid \bn \in \BN^\ell\}$ are respectively $\cA$-bases of $\UZ^{\ev,0}$ and $\VA^{\ev,0}$.

(b) The sets  $\{ Q(\bn,\bode) \mid \bn \in \BN^\ell, \bode \in \{0,1\}^\ell\}$ and
$\{ (q;q)_\bn Q(\bn,\bode) \mid \bn \in \BN^\ell, \bode \in \{0,1\}^\ell\}$ are respectively $\cA$-bases of $\UZ^{\ev}$ and $\VA^{\ev}$.
\end{proposition}

Since $\UZ^0,\VZ^0$ are $\cA$-subalgebras of the commutative algebra $\BQ(v)[K_1^{\pm1}, \dots, K_\ell^{\pm 1}]$, the proof is not difficult though involves some calculation.
We give a proof of Proposition \ref{r.bases5} in Appendix \ref{sec:u0}.
\begin{remark}
In \cite{Lusztig02}, Lusztig gave a similar, but different, basis of $\UZ^0$. Our basis can be obtained from Lusztig by an upper triangular matrix, and hence a proof of the proposition can be obtained this way.
We chose the basis in Proposition \ref{r.bases5} instead of Lusztig's one for orthogonality reason.
\end{remark}

\subsection{Preferred bases of $\UZ$ and $\VA$}\label{sec:basesUZ}
Recall that we have defined $(q;q)_\bn$ in two cases depending on the length of $\bn$, see \eqref{eq.qqn1} and \eqref{eq.qqn2}: either  $\bn=(n_1,\dots,n_t) \in \BN^t$, in which case,
$$ (q;q)_\bn= \prod_{j=1}^t (q_{\gamma_j};q_{\gamma_j})_{n_j},$$
or  $\bn=(n_1,\dots,n_\ell)\in \BN^\ell$, then
$$ (q;q)_\bn= \prod_{j=1}^\ell (q_{\al_j};q_{\al_j})_{n_j}.$$
The first one depends on a longest reduced sequence since $\gamma_j$ does, while the second one does not.

 Introduce another $(q;q)_\bn$, with length of $\bn$ equal $2t+\ell$.  For $\bn=(\bn_1, \bn_2, \bn_3) \in \BN^{t+ \ell +t}$, where $\bn_1, \bn_3 \in \BN^t$ and $\bn_2 \in \BN^\ell$,  define
 \be
 \label{eq.bbe5}
  (q;q)_\bn := (q;q)_{\bn_1} \, (q;q)_{\bn_2}\, (q;q)_{\bn_3}.
  \ee
 Further if $\bode \in \{ 0,1\}^\ell$, let
  \be
  \label{eq.bbe}
  \bbe^\ev (\bn):= F^{(\bn_1)} K_{\bn_1}\, Q^\ev(\bn_2) E^{(\bn_3)}, \quad \bbe (\bn,\bode) :=  F^{(\bn_1)} K_{\bn_1}\, Q(\bn_2,\bode) E^{(\bn_3)}.
  \ee

\begin{proposition}
\label{r.basesUZ}
(a) The set
$$\{ \bbe(\bn,\bode) \mid \bn \in \BN^{t+\ell+t}, \bode \in \{0,1\}^\ell \}$$
and its dilated set
 $$\{ (q;q)_\bn \, \bbe(\bn,\bode) \mid \bn \in \BN^{t+\ell+t}, \bode \in \{0,1\}^\ell \}$$
are respectively   $\cA$-bases of $\UZ$ and $\VZ$.

(b) The set $$\{ \bbe^\ev(\bn) \mid \bn \in \BN^{t+\ell+t} \}$$
and its dilated set
 $$\{ (q;q)_\bn\, \bbe^\ev(\bn) \mid \bn \in \BN^{t+\ell+t} \}$$
 are respectively   $\cA$-bases of $\UZ^\ev$ and  $\VA^\ev$.

\end{proposition}

\begin{proof}
The proposition follows from the even triangular decompositions of $\UZ$ and $\VA$, together with the bases of $\UZ^{\ev,-}$, $\UZ^0$, $\UZ^+$ and
$\VZ^{\ev,-}$, $\VZ^0$, $\VZ^+$ in Propositions \ref{prop.basis}, \ref{r11}, and \ref{r.bases5}.
\end{proof}

We will consider $\UZ,\UZ^\ev,\VA,\VA^\ev$ as based free $\cA$-modules with the preferred bases described in the above proposition. Then $\VA$ is a dilatation of $\UZ$, and $\VA^\ev$ is a dilatation of $\UZ^\ev$.

\subsection{Relation between $\VA$ and $\Vh$} %
\begin{proposition}
\label{r.VZVh}

 (a) One has $ \VAev \subset \VA \subset \Vh$.

(b) Moreover, $\Vh$ is the topological closure (in the $h$-adic topology of $\Uh$) of the $\Ch$-span of $\VAev$.
Consequently,  $\Vh$ is also the topological closure (in the $h$-adic topology of $\Uh$) of $\VA$.
\end{proposition}
\begin{proof}

(a) It is clear that $ \VAev \subset \VA$. Let us prove $\VA \subset \Vh$.

Fix a longest reduced sequence $\bi$.
By Proposition \ref{r.phu1}, $\Vh$ is the topological closure of the $\Ch$-subalgebra generated by $hH_\al, hF_\gamma, hE_\gamma$, with $\al\in \Pi, \gamma \in \Phi_+$.

For every $\gamma \in \Phi_+$,
there is a unit $u$ in $\Ch$ such that $1-q_\gamma= h u$, and
\be  \label{eq.bh2}
(1-q_\gamma) F_\gamma = u (h F_\gamma) \in \Vh.
\ee
 Similarly, $(1-q_\gamma)E_\gamma \in \Vh$. We already have $K_\al^{\pm 1}\in \Vh$. Since $(1-q_\gamma)F_\gamma, (1-q_\gamma)E_\gamma, K_\al^{\pm1}$ generate $\VA$ as $\cA$-algebra and $\Vh$ is an $\cA$-algebra, we have $\VA \subset \Vh$.

(b) Let $\Vh'$ be the topological closure
 of the $\Ch$-span of $\VAev$. We have to show that $\Vh'= \Vh$. From part (a) we now have that $\Vh' \subset \Vh$. It remains to show $\Vh \subset \Vh'$.
 It is easy to see that $\Vh'$ is a $\Ch$-algebra.

 Since $K_\al^2 \in \VA^\ev$ and
 $$ h H_\al = \log (K_\al^2) = -\sum_{n=1}^\infty \frac{(1-K_\al^2)^n}{n},$$
 we have $hH_\al \in \Vh'$ for any $\al \in \Pi.$ It follows that $K_\al^{\pm 1} = \exp(\pm hH_\al/2) \in \Vh'$.

 From \eqref{eq.bh2},
 $$hF_\gamma = u^{-1} (1-q_\gamma) (F_\gamma K_\gamma) K_{\gamma}^{-1}  \in \Vh', \quad hE_\gamma = u^{-1} (1-q_\gamma) E_\gamma \in \Vh'.$$

 Thus, $h H_\al, h F_\gamma, h E_\gamma$ are in $\Vh'$ for any $\al\in \Pi, \gamma \in \Phi_+$. Since $\Vh$ is the topological closure of the $\Ch$-algebra generated by $hH_\al, hF_\gamma, hE_\gamma$, we have $\Vh \subset \Vh'$.
 This completes the  proof of the proposition.
\end{proof}

\begin{corollary}\label{r.Vhbraid} The algebra
   $\Vh$ is stable under the braid group action, i.e. $T_\al^{\pm 1} (\Vh) \subset \Vh$ for any $\al \in \Pi$.

\end{corollary}

\begin{proof}  Since $\VA$ is invariant under the braid group actions, and $\Vh$ is the topological closure of the $\Ch$-span of $\VA$,  $\Vh$ is also invariant under the braid group actions.
\end{proof}

\begin{remark}Using Corollary \ref{r.Vhbraid}
 one can easily prove that  $\Vh$ is the smallest $\Ch$-subalgebra of $\Uh$ which

(i) contains $hE_\al, h F_\al, hH_\al, \al \in \Pi$,

(ii) is stable under the action of the braid group.

(iii) is closed in the $h$-adic topology of $\Uh$.

\end{remark}

\subsection{Stability of $\VZ$ under $\ibar, \tau,\tphi$} By Proposition \ref{prop.basis}, $\UZ$ is stable under $\ibar, \tau$, and $\varphi$.
\begin{proposition}
\label{r.phistab}
The algebra $\VZ$ is stable under  each of  $\tau$,  $\tphi$, and  $\ibar$.
\end{proposition}
\begin{proof} Recall that $\VZ$ is the smallest $\cA$-subalgebra of $\UZ$ containing $(1-q_\al)E_\al, (1-q_\al)F_\al, K_\al$ for $\al\in \Pi$, and is stable under the action of the braid group.
Let $f$ be one of $\tau,\varphi, \ibar$.

{\em Claim 1.}
 $f(\VZ)$ is stable under the braid group action.

{\em Proof of Claim 1.}  (i) The case $f=\tau$.
By  \cite[Formula 8.14.10]{Jantzen}, $\tau T_\al = T_\al^{-1} \tau$ for every $\al\in \Pi$. Since $T_\al$ generate the braid group, we conclude that, like $\VZ$,   $\tau(\VZ)$ is also stable under the braid group.

(ii) The case $f= \varphi$. Recall that $S$ is the antipode.
By Proposition \ref{r.phi5},
$ \varphi =  S \btau =\btau S$, where $\btau=\ibar \tau \omega$ is a $\BC$-anti-automorphism of $\Uh$.
Our $\btau$ is the same $\kappa$ in \cite{DeConcini-Procesi}, where it was observed that $\btau$ commutes with the action of the braid group, i.e.
$\btau T_\al = T_\al \btau$ for $\al \in \Pi$. It follows that $\btau(\VZ)$ is stable under the braid group. Since $\varphi(\VZ)= \btau S(\VZ)= \btau(\VZ)$, $\varphi(\VZ)$ is stable under the braid group.

(iii) The case $f=\ibar$.
Checking on the generators, one has $\ibar= \kappa \tau \omega$.

 By Formula 8.14.9 of \cite{Jantzen}, if $x\in \Uq$ is $Y$-homogeneous, then $T_\al(\omega(x)) \sim \omega T_\al(x)$, where $x \sim y $ means $x = u y$ for some unit $u \in \cA$. As $\VZ$ has an $\cA$-basis consisting of $Y$-homogeneous elements (see Proposition \ref{r.basesUZ}), we conclude that $\omega(\VZ)$ is stable under the braid group. The results of (i) and (ii) show that $\ibar(\VZ)= \kappa \tau \omega(\VZ)$ is stable under the braid group.

 This completes the proof of Claim 1.

{\em Claim 2.} One has $\VZ \subset f(\VZ)$.

{\em Proof of Claim 2.}
Using explicit formulas of $f^{-1}$  in Section \ref{sec.8147}, one sees that each of $f^{-1}((1-q_\al)E_\al), f^{-1}((1-q_\al)F_\al), f^{-1}(  K_\al)$ is in  $\VZ$. It follows that each of $(1-q_\al)E_\al, (1-q_\al)F_\al, K_\al$ is in $f(\VZ)$. Together with Claim 1, this implies $f(\VZ)$ is an algebra stable under the braid group and contains $f^{-1}((1-q_\al)E_\al), f^{-1}((1-q_\al)F_\al), f^{-1}(  K_\al)$. Hence $f(\VZ) \supset \VZ$.
This completes the proof of Claim 2.

Since $\tau$ and $\ibar$ are involutions and $\varphi^2(x)= K_{-2\rho} x K_{2\rho}$ (by Proposition \ref{r.phi5}), we have  $f^2(\VZ)=\VZ$.
 Applying $f$ to $\VZ \subset f(\VZ)$, we get $f(\VZ) \subset f^2(\VZ)= \VZ$.
 Hence, $\VZ= f(\VZ)$.
\end{proof}

\subsection{Simply-connected version of  $\UZ$} \label{sec:simply} Recall that the simply connected version $\brU_q$ is obtained from $\Uq$ by replacing the Cartan part $\Uq^0= \BC(v)[K^{\pm 1}_1, \dots,K^{\pm 1}_\ell]$ with the bigger $\brU_q^0= \BC(v)[\brK^{\pm 1}_1, \dots,\brK^{\pm 1}_\ell]$.
We introduce an analog of  Lusztig's integral form for $\brU_q$ here.

The $\BC(v)$-algebra homomorphism
$\briota: \Uq^0 \to \brU_q^0$, defined by $\briota(K_\al) = \brK_\al$, $\al\in\Pi$,
 is a Hopf algebra homomorphism. Let
 $$\brUA^0 := \briota(\UZ^0),\quad \brUA^{\ev,0} := \briota(\UZ^{\ev,0}).$$
 Then $\brUA^0, \brUA^{\ev,0}$ are $\cA$-Hopf-subalgebra of $\brU_q^{\ev,0}$.
 Define
 $$\brUA:= \brUA^0 \UZ, \quad \brUA^\ev:= \brUA^{\ev,0} \UZ^\ev.$$

 For $ \bm\in \BN^\ell, \bode=(\delta_1, \dots, \delta_\ell)\in \{0,1\}^\ell$, define
 \begin{align}
  \brQ^\ev(\bm) &:= \briota(Q^\ev(\bm)), \quad \brQ (\bm,\bode) := \briota(Q(\bm,\bode))   \notag
 \end{align}
and furthermore for $ \bn=(\bn_1, \bn_2, \bn_3)\in \BN^{t+\ell+t}$ define
\begin{align}
\label{eq.brbe}  \brbe^\ev(\bn)&:= F^{(\bn_3)} K_{\bn_3}\, \brQ^\ev(\bn_2) \, E^{(\bn_1)}, \quad  \brbe(\bn,\bode):=\brbe^\ev(\bn) \prod_{j=1}^\ell \brK_{\al_j}^{\delta_j}.
 \end{align}

\begin{proposition} \label{r.8aa}

 (a) $\brUA$ is an $\cA$-Hopf-subalgebra of $\brU_q$, and $\brUA^\ev$ is an $\cA$-subalgebra of $\brUA$. We also have the following even triangular decompositions
\begin{gather}
  \label{e84s}
  \UZ^{\ev,-}\otimes \brUA ^{\ev,0}\otimes \UZ ^+\congto\brUA^\ev,\quad\xyzxyz.
 \\ \UZ ^{\ev,-}\otimes \brUA ^{0}\otimes \UZ ^+\congto\brUA,\quad\xyzxyz,
\end{gather}

(b) The sets  $\{ \brbe(\bn,\bode) \mid \bn\in \BN^{t+\ell+t} , \bode \in \{0,1\}^\ell\}$ and $\{ \brbe^\ev(\bn) \mid \bn\in \BN^{t+\ell+t}\}$  are respectively $\cA$-bases of $\brUA$ and $\brUA^\ev$.

(c) One has $\brUA \tri \brUA^\ev \subset \brUA^\ev$. Consequently, $\UZ \tri \brUA^\ev \subset \brUA^\ev$.
\end{proposition}
\begin{proof}
(a) As an $\cA$-module,  $\UZ^0$ is spanned  by
$$f_{\al,m,n,k}:= \frac{K_\al^m (q_\al^n K_\al^2;q_\al)_k}{(q_\al;q_\al)_k},$$
 with $\al\in \Pi$, $ m,n, \in \BZ$ and  $k\in \BN$. Hence $\brUA^0= \briota(\UZ^0)$ is $\cA$-spanned by $\breve f_{\al,m,n,k}:=\briota(f_{\al,m,n,k}) $. If  $x\in \UZ$ is $Y$-homogeneous, then, using \eqref{eq.8157} which describes the commutation  between $\brK_\al$ and $x$,
 \be
 \label{eq.brcom}
  \breve f_{\al,m,n,k}\,  x =  v^{m(|x|,\bral)} x \breve f_{\al,m,n',k},
  \ee
 where $n'= n + (|x|,\al)/d_\al \in \BZ$. Hence, $\UZ$ commutes with $\brUA^0$ in the sense that
 $\UZ \brUA^0= \brUA^0 \UZ$. Since both $\UZ$ and $\brUA^0$ are $\cA$-Hopf-subalgebras of $\Uh$ and they commute in the above sense,  $\brUA= \brUA^0 \UZ$ is an $\cA$-Hopf-subalgebra of $\brU_q$.

 Identity \eqref{eq.brcom} also shows that each of $\brUA^{\ev,0}, \brUA^0$  commutes with each of $\UZ^-, \UZ^+, \UZ^0$. Hence, $\brUA^\ev= \brUA^{\ev,0} \UZ^\ev$ is an $\cA$-subalgebra of $\brUA$. The triangular decompositions  for  $\brUA^\ev$ and $\brUA$ follows from those  of  and $\UZ^\ev$ and $\UZ$.

 (b) Combining the base $\{ F^{(\bn_3)} K_{\bn_3} \}$ of $\UZ^{\ev,-}$ (see Proposition \ref{r.ba3}), $\{ \brQ^\ev(\bn_2) \}$ of $\brUA^{\ev,0}$ (by Proposition \ref{r.bases5} and isomorphism $\briota$), $\{ E^{(\bn_1)}\}$ of $\UZ^+$ (see Proposition \ref{prop.basis}), and the even triangular decompositions of $\brUA$ and $\brUA^\ev$, we get the bases of $\brUA$ and $\brUA^\ev$ as described.

 (c) Since $\brUA$ contains $E_\al^{(n)}, F_\al^{(n)}, K_\al^{\pm 1}$, which generate $\UZ$, we have $\UZ \subset \brUA$. Let us prove $\brUA \tri \brUA^\ev \subset \brUA^\ev$.
 From the triangular decomposition of $\brUA,\brUA^\ev, \brU_q$, we see that $\brUA^\ev = \brUA \cap \brU_q^\ev$.

 Since $\brUA$ is a Hopf algebra, we have $\brUA \tri \brUA^\ev \subset \brUA$.
 By Lemma \ref{eq.evenD3a},
 $$\brUA \tri \brUA^\ev \subset \brU_q \tri \brU_q^\ev \subset \brU_q^\ev.$$
  Hence
 $ \brUA \tri \brUA^\ev \subset \brUA \cap \brU_q^\ev = \brUA^\ev.$
 This finishes the proof of the proposition.
 \end{proof}

\subsection{Integral duality with respect to quantum Killing form}
\label{sec:duality-between-uzev}

Recall that $\{ \bbe^\ev(\bn) \mid \bn \in \BN^{t+\ell+t} \}$ is an $\cA$-basis of $\UZ^\ev$ (Proposition \ref{r.basesUZ}), and
$\{ \brbe^\ev(\bn) \mid \bn \in \BN^{t+\ell+t} \}$ is an  $\cA$-basis of $\brUZ^\ev$ (Proposition \ref{r.8aa}).
We will show that these two bases are orthogonal with each other with respect to the quantum Killing form.

 Recall that we  defined $(q;q)_\bn= (q;q)_{\bn_1} (q;q)_{\bn_2} (q;q)_{\bn_3}$, see Section \ref{sec:basesUZ}.

\def\bp{{\mathbf p}}
\begin{proposition}\label{r.ortho3}
(a) For $\bn,\bm\in \BN^{t+\ell+t}$, there exists a unit $u(\bn) \in \cA$ such that
$$ \la \bbe^\ev(\bn), \brbe^\ev(\bm) \ra = \delta_{\bn,\bm}\, \frac{  u(\bn)}{ (q;q)_\bn}.$$

(b) The $\cA$-module $\VA^\ev$ is the $\cA$-dual of $\ \brUA^\ev$ in $\Uq^\ev$ with respect to the quantum Killing form, i.e.
$$ \VA^\ev = \{ x \in \Uq^\ev \mid \la x, y\ra \in \cA \ \forall  y \in \brUA^\ev \}.$$
\end{proposition}
\begin{proof} Define the following units in $\cA$. For $\bm = (m_1,\dots,m_t)\in \BN^t$ and $\bk=(k_1,\dots,k_\ell)\in \BN^\ell$ let
$$ u_1(\bm)= \prod_{j=1}^t v_{\gamma_j}^{m_j^2 }, \quad u_2(\bk)=\prod_{j=1}^\ell q_{\al_j}^{- \lfloor(k_j+1)/2 \rfloor^2}.$$
For $\bn=(\bn_1,\bn_2,\bn_3) \in \BN^{t+\ell+t}$, let
$$u(\bn) = q^{(\rho,|E_{\bn_3}|)}u_1(\bn_1) u_2(\bn_2) u_1(\bn_3).$$
(a) We will use the following lemma whose proof will be given in Appendix \ref{sec:u0}.
\begin{lemma} \label{r.117}
For $\bk,\bk'\in \BN^\ell$, one has
\be
\label{eq.117}  \quad \la Q^\ev(\bk), \brQ^\ev(\bk')\ra = \delta _{\bk,\bk'} u_2({\bk})/(q;q)_\bk.
\ee

\end{lemma}

For $\bp\in \BN^t$, using the definition of $E_\bp, F_\bp$ in Section \ref{Rmatrix}, we have
$$  F^{(\bp)} \ot E^{(\bp)}   = \frac{ u_1(\bp)}{    (q;q)_\bp} \left( F_\bp \ot E_\bp\right).$$
Suppose $\bn=(\bn_1,\bn_2,\bn_3)$ and $\bm=(\bm_1,\bm_2,\bm_3)$ are in $\BN^{t+\ell+t}$.  Using the definition of  $\bbe^\ev(\bn)$ and $\brbe^\ev(\bn)$ from  \eqref{eq.bbe} and \eqref{eq.brbe}, the triangular property of the quantum Killing form,
and Formulas \eqref{e57a} and \eqref{eq.117},
\begin{align*}
\la \bbe^\ev(\bn), \brbe^\ev(\bm) \ra & = \la F^{(\bn_1)} K_{\bn_1}, E^{(\bm_1)} \ra \la Q^\ev(\bn_2), \brQ^\ev(\bm_2)  \ra
\la  E^{(\bn_3)} , F^{(\bm_3)} K_{\bm_3}\ra
\\
&= \frac{\delta_{\bn_1,\bm_1}\, u_1({\bn_1}) }{(q;q)_{\bn_1}}\,  \frac{\delta_{\bn_2,\bm_2} \,u_2({\bn_2}) }{(q;q)_{\bn_2}}\,  \frac{\delta_{\bn_3,\bm_3}\, q^{(\rho,|E_{\bn_3}|)}\, \,  u_1({\bn_3}) }{(q;q)_{\bn_3}} = \frac{\delta_{\bn,\bm} \, u(\bn) }{(q;q)_\bn}.
\end{align*}
(b) By Proposition \ref{r.basesUZ}, $\{ (q;q)_\bn\, \bbe^\ev(\bn)\mid  \bn \in \BN^{t+\ell+t} \}$ is an $\cA$-basis of $\VA^\ev$ and a $\BC(v)$-basis of $\Uq^\ev$, and by Proposition \ref{r.8aa}, $\{ \brbe^\ev(\bn) \mid \bn \in \BN^{t+\ell+t} \}$ is an $\cA$-basis of $\brUA^\ev$. Part (b) follows from the orthogonality of part (a).
\end{proof}

\begin{remark}
From the orthogonality of Proposition \ref{r.ortho3}, we can show that
\be
\label{eq.claspUZ}
\bc = \sum _{\bn \in \BN^{t+\ell+t}} \frac{(q;q)_\bn }{u(\bn)} \brbe^\ev(\bn) \ot \bbe^\ev(\bn).
\ee
\end{remark}

\subsection{Invariance of $\VA ^\ev$ under adjoint action of $\UZ$}
\label{sec:stability-ev-under}
The adjoint action makes $\UZ$ a $\UZ$-module. The following  result, showing that $\VZ^\ev$ is a $\UZ$-submodule of $\UZ$, is important for us and will be used frequently.

\begin{theorem}
  \label{thm:8}
  We have  $\brUA \tri \VA^\ev  \subset \Vev$. In particular, $\UA \tri \VA^\ev  \subset \Vev$, i.e. $\Vev$ is $\UZ$-ad-stable.
\end{theorem}

\begin{proof} By Proposition \ref{r.8aa},  $\brUA \tri \brUA^\ev \subset \brUA^\ev$, and by Proposition \ref{r.ortho3}, $\VA^\ev$ is the $\cA$-dual of $\brUA^\ev$ with respect to the quantum Killing form. Besides, the quantum Killing form is ad-invariant. Hence, one also has
$\brUA \trr  \Vev \subset \Vev$, as the following argument shows. Recall that we already have $\brUA \tri \Uq^\ev \subset \Uq^\ev$ (see Lemma \ref{eq.evenD3a}).
Suppose $a \in \brUA$, $ x\in \Vev$. We will show $ a\tri x \in \Vev$. We have
\begin{align*}
a\tri x \in \Vev & \Leftrightarrow  \la a\tri x , y \ra \in \cA \quad \forall y \in \brUA^\ev  &\text{by duality, Proposition \ref{r.ortho3}}\\
 & \Leftrightarrow  \la x , S(a) \trr y \ra \in \cA \quad \forall y \in \brUA^\ev  &\text{by ad-invariance, Proposition  \eqref{r.inva}(b)}.\\
\end{align*}
Since $ S(a) \trr y \in  \brUA^\ev$, the last statement $\la x , S(a) \trr y \ra \in \cA$ holds true by Proposition \ref{r.ortho3}. Thus we have proved that $\UZ \trr  \Vev \subset \Vev$.
\no{
To see $\brUA \tri \VA^\ev  \subset \Vev$, by the definition of
$\brUA$, it suffices to show that $\brUA^0\tri\VA^\ev\subset\VA^\ev$,
which follows from
\begin{gather*}
  \frac{\breve{K}_\alpha^m(q_\alpha^n \breve{K}_\alpha^2; q_\alpha)_k}{(q_\alpha; q_\alpha)_k}\tri x
  =
  \frac{v_\alpha^{mk_\alpha}(q_\alpha^{n+k_\alpha}; q_\alpha)_k}{(q_\alpha; q_\alpha)_k}x
  \in \VA^\ev
\end{gather*}
for $\alpha\in\Pi$, $m,n\in\BZ$, $k\ge0$ and $x\in\VA^\ev$ homogeneous
with $|x|=\sum_{\beta\in\Pi}k_\beta \beta$, $k_\beta\in\BZ$.
}
\end{proof}

\begin{remark}
  \label{r41}
  We do {\em not} have $\UZ\trr\VA \subset \VA $ in general.  For example, when
  $\g=A_2$ and $\al \neq \beta\in \Pi$,
  \begin{gather*}
    E_\alpha \trr K_\beta = (v-1)K_\beta E_\alpha \not\in \VA.
  \end{gather*}

  However, when $\g=A_1$, we do have $\UZ\trr\VA \subset \VA$, as
  it easily follows from \cite[Proposition 3.2]{Suzuki1}, where a more
  refined statement is given.
\end{remark}

\subsection{Extension from $\cA$ to $\tA$: Stability principle} Recall that $\tA$ is obtained from $\cA$ by adjoining all square roots  $ \sqrt{\phi_k(q)},  k=1,2,\dots$ of cyclotomic polynomials $\phi_k(q)$.

Suppose $V$ is a based free $\cA$-module with preferred base $\{ e(i) \mid i \in I\}$ and $a:I \to \cA$ is a function such that for every $i\in I$, $a(i)$ is a product of cyclotomic polynomials in $q$. We already defined the dilatation triple $(V, V(\sqrt a), V(a))$ in Section \ref{sec.dilita}.
 Recall that $V(a)$ is the free $\cA$-module with base $\{a(i)  e(i) \mid i \in I\}$, and $V(\sqrt a)$ is the free $\tA$-module with base $\{\sqrt {a(i)}\,   e(i) \mid i \in I\}$.

 For any $\cA$-module homomorphism $f: V_1 \to V_2$ we also use the same notation $f$ to denote the linear extension $f\ot \id : V_1 \ot_\cA \tA \to V_2 \ot_\cA \tA$, which is an $\tA$-module homomorphism.

 \begin{proposition}[Stability principle] \label{r.closed} Let $(V_1, V_1(\sqrt {a_1}), V_1(a_1))$ and $(V_2, V_2(\sqrt {a_2}), V_2(a_2))$ be two dilatation triples, and $f:V_1\to V_2$ be an $\cA$-module homomorphism.
If $f(V_1(a_1)) \subset V_2(a_2)$, then $f(V_1(\sqrt {a_1}) ) \subset V_2(\sqrt {a_2})$.
\end{proposition}

\begin{proof}First we prove the following.\\
{\it Claim.}
Suppose $a,b,c \in \cA$, where $b, c$ are products of  cyclotomic polynomials $\phi_k(q)$. If $ab/c \in \cA$ then $ a \sqrt{ b/c} \in \tA$.\\
{\it Proof of Claim.} Since $\cA$ is a unique factorization domain,
 one can assume that $b$ and $c$ are co-prime. Then $a$ must be divisible by $c$, $a= a' c$ with $a' \in \cA$. Then
 $ a \sqrt{ b/c} = a' \sqrt{b c'} \in \tA$, which proves the claim.

 The proof of  the proposition is now parallel to that in the topological case (Proposition \ref{r.topdilat}).
Using the bases $\{ e_1(i)\mid i\in I_1 \}$ and $\{ e_2(i) \mid i \in I_2\}$ of $V_1$ and $V_2$, we can write
\begin{align*}
f(e_1(i) )=  \sum_{k \in I_2}f_i^k  \, e_2(k)
\end{align*}
where $f_i^k =0$ except for a finite number of $k$ (when $i$ is fixed) and
$f_i^k \in \cA$.

Multiplying by $a_1(i)$  and $\sqrt{a_1(i)}$, we get
\begin{align}
\label{eq.zg3}
f\big( a_1(i) e_1(i) \big)& =  \sum_{k\in I_2 }f_i^k  \,  \frac{ a_1(i) }{a_2(k)}   \, \big(a_2(k)  e_2(k) \big)  \\
\label{eq.zg2} f\left ( \sqrt{a_1(i)} e_1(i) \right)& =  \sum_{k\in I_2 }f_i^k  \, \sqrt{ \frac{  a_1(i) }{ a_2(k)}}   \, \left( \sqrt{ a_2(k)}\,  e_2(k) \right).
\end{align}
Since  $f(V_1(a_1)) \subset V_2(a_2)$, \eqref{eq.zg3} implies that
$f_i^k  \,  \frac{ a_1(i) }{a_2(k)}  \in \cA$, which, together with $f_i^k \in \cA$ and the Claim, implies that
$$f_i^k  \, \sqrt{ \frac{  a_1(i) }{ a_2(k)}}  \in \tA.$$
Now \eqref{eq.zg2} shows that $f(V_1(\sqrt {a_1}) ) \subset V_2(\sqrt {a_2})$.
\end{proof}

\subsection{The integral core subalgebra $\XZ$} \label{sec:XZ}
By Proposition \ref{r.basesUZ}, we can consider
$\UZ$  as a based free $\cA$-module with the preferred base
$
\{ \bbe(\bn,\bode) \mid \bn \in \BN^{t+ \ell+t}, \bode \in \{0,1\}^\ell \}
$.

  Let $a: \BN^{t+ \ell+t} \times \{0,1\}^\ell \to \cA$ be the function defined by
 $a(\bn,\bode)= (q;q)_\bn$, where $(q;q)_\bn$ is defined by \eqref{eq.bbe5}.
 We will consider the dilatation triple $(\UZ,\UZ(\sqrt a), \UZ(a))$.
 By Proposition \ref{r.basesUZ}, $\UZ(a)$ is $\VZ$.

 Let $\XZ$ be $\UZ(\sqrt a)$, which by definition is the
 free $\tA$-module with basis
 \be
 \label{eq.basesXZ}
 \{\sqrt{(q;q)_\bn}\, \bbe(\bn,\bode) \mid \bn \in \BN^{t+ \ell+t}, \bode \in \{0,1\}^\ell \}.
 \ee

The even part $ \sXZev$ of $\sXZ$ is defined to be the $\tA$-submodule spanned by
\be
\label{eq.basisXZe}
\{\sqrt {(q;q)_\bn  }\, \bbe^\ev(\bn) \mid \bn \in \BN^{t+ \ell+t} \}.
\ee
Then $\XZ^\ev= \XZ \cap (\UZ^\ev \ot_\cA \tA)$, and $(\UZ^\ev, \XZ^\ev, \VZ^\ev)$ is a dilatation triple.

\begin{theorem}
\label{r.sXZ}
(a) The $\tA$-module $\sXZ$ is an  $\tA$-Hopf-subalgebra of $\UZ \ot_\cA \tA$.

(b)  The $\tA$-module $\sXZev$ is an $\tA$-subalgebra of $\UZ^\ev \ot_\cA \tA$. Besides, $\sXZev$ is
\begin{itemize}
\item[(i)] $\UZ$-ad-stable,
\item[(ii)] stable under the action of the braid group, and
\item[(iii)] stable under  $\ibar$ and $\tphi$.
\end{itemize}

(c) The core algebra $\Xh$ is the $\sqrt h$-adic completion of the $\Chh$-span of $\sXZ^\ev$ (or $\sXZ$) in $\UUH$.
\end{theorem}
\begin{proof} (a) Let us show that $\Delta (\sXZ) \subset \sXZ \ot \sXZ$. Since $(\UZ,\XZ,\VZ)$ is a dilatation triple, $(\UZ \ot \UZ, \XZ \ot \XZ, \VZ \ot \VZ)$ is also a dilatation triple.
We have
$\Delta (\UA) \subset \UA \ot \UA$ and $\Delta(\VA) \subset \VA \ot \VA$. By the stability principle (Proposition \ref{r.closed}), we have $\Delta (\sXZ) \subset \sXZ \ot \sXZ$, i.e. $\sXZ$ is an $\tA$-coalgebra.

Similarly, applying the stability principle to all the operations of a Hopf algebra, we conclude that $\sXZ$ is an $\tA$-Hopf-subalgebra of $\UZ \ot_\cA \tA$.

(b) Because $\VA^\ev$ is an $\cA$-subalgebra of $\UA^\ev$, the stability principle for the dilatation triple $(\UZ^\ev, \XZ^\ev, \VZ^\ev)$ shows that $\sXZev$ is an $\tA$-algebra.

By Theorem  \ref{thm:8},  $\VA^\ev$ is $\UZ$-ad-stable;  and by Proposition \ref{r.phistab}, $\VA^\ev$ is stable under $\ibar$, $\tphi$. Since $\UZ^\ev$ is $\UZ$-ad-stable is stable under $\ibar$ and  $\tphi$ (by Proposition \ref{prop.basis}), the stability principle proves that  $\sXZev$ is (i) $\UZ$-ad-stable,
(ii) stable under the action of the braid groups, and
 (iii) stable under  $\ibar$ and $\tphi$.

(c) Each element of the basis \eqref{eq.basesXZ} of $\sXZ$ is in $\Xh$. Hence $\sXZ \subset \Xh$.
On the other hand, the $\tA$-basis \eqref{eq.basisXZe} of $\sXZev$ is also
a topological basis of $\Xh$. Hence  $\Xh$ is the $\sqrt h$-adic completion of the $\Chh$-span of $\sXZ^\ev$ in $\UUH$.
\end{proof}
\begin{corollary}
\label{r.corealgebra2}
 (a) The core algebra $\Xh$ is stable under the action of the braid group.

 (b) The core algebra $\Xh$  is a smallest $\sqrt h$-adically completed topological $\Chh$-subalgebra of $\UUH$ which  (i) is closed in the
 $\sqrt h$-adic topology, (ii) contains $\sqrt h E_\al, \sqrt h F_\al, \sqrt h H_\al$ for each $\al \in \Pi$, and (iii) is
 invariant under the action of the braid group.
\end{corollary}
\begin{proof}
(a) Since $\Xh$ is the $\sqrt h$-adic completion of the $\Chh$-span of $\sXZ^\ev$, which is stable under the action of the braid group, $\Xh$ is also  stable under the action of the braid group.

(b)
Let $\Xh'$ be the smallest completed subalgebra of $\UUH$ satisfying (i), (ii), and (iii). Since $\Xh$ satisfies (i), (ii), and (iii),  we have $\Xh' \subset \Xh$.

For each $\gamma\in \Phi_+$, $E_\gamma$ and $F_\gamma$ are obtained from $E_\al, F_\al, \al \in \Pi$  by actions of the braid group.
 Thus $\Xh'$ contains all $\sqrt h E_\gamma$, $\sqrt h F_\gamma$, $\gamma \in \Phi_+$ and $\sqrt hH_\al, \al \in \Pi$, which generate $\Xh$ as an algebra (after $h$-adic completion). It follows that $\Xh\subset\Xh'$. Hence $\Xh=\Xh'$.
 \end{proof}

\begin{remark}
\label{sec.XA}
The disadvantage of $\sXZ$ is its ground ring is $\tA$, not $\cA$. Let us define
$$\XA = \sXZ \cap \UZ.$$
Then $\XA$ is an $\cA$-algebra. However, $\XA$ is not an $\cA$-Hopf algebra in the usual sense, since
$$\Delta(\XA) \not \subset \XA \ot_\cA \XA.$$

Let us define a new tensor product
 \be
 \label{eq.XAsq}
 (\XA)^{\boxtimes n} := \sXZ^{\ot n} \cap \UZ^{\ot n}, \quad (\XA^\ev)^{\boxtimes n} := (\sXZev)^{\ot n} \cap (\UZ^\ev)^{\ot n}.
 \ee
Then we have
$$\Delta(\XA) = \Delta (\XZ \cap \UZ)\subset (\XZ \ot \XZ)  \cap (\UZ \ot \UZ) = \XA \boxtimes \XA.$$
Hence $\XA$, with this new tensor power, is a Hopf algebra, which is a Hopf subalgebra of both $\sXZ$ and $\UZ$.

What we will prove later implies that if $T$ is an $n$-component bottom tangle with 0 linking matrix, then
$$ J_T \in \varprojlim_k (\XA^\ev)^{\boxtimes n}/(( q;q)_k).$$
However, we will not use $\XA$ in this paper.
\end{remark}

\subsection{Integrality of twist forms $\cT_\pm$ on $\sXZ^\ev$} Recall that we have twist forms $\cT_\pm:\Xh \to \Chh$. By Theorem \ref{r.sXZ}, $\XZ \subset \Xh$.

The embedding $\cA \hookrightarrow \Ch$ by $v= \exp(h/2)$ extends to an embedding $\tA \to \Chh$.
Although there are many extensions, it is easy to see that the image of the extended embedding does not depend on the extension, because the two roots of $\phi_k(q)$ are inverse (with respect to addition) of each other.

\begin{proposition}\label{r.integ6}
One has $\cT_\pm(\sXZ^\ev) \subset \tA$.
\end{proposition}
The proof of this proposition will occupy the rest of this section (subsections \ref{sec:27}--\ref{sec:30}.)

\def\brXZ{\breve{\mathbf X}_\BZ}
\subsubsection{Integrality on the Cartan part}
\label{sec:27}
\begin{lemma}
\label{r.Hopfsub}
(a) The Cartan part $\sXZ^{\ev,0}$ of $\sXZe$ is a  $\tA$-Hopf-subalgebra of $\sXZ$.

(b) Suppose  $x, y \in \sXZ^{\ev,0}$ and $\lambda \in X$. Then  $\la x,  y \ra \in \tA$ and $\la x,  K_{2 \lambda} \ra \in \tA$.

\end{lemma}\begin{proof} (a) Since $\sXZ^{\ev,0}$ is an $\tA$-subalgebra of the commutative co-commutative Hopf algebra $\sXZ^0$, we need to check that
$
 \Delta (\sXZ^{\ev,0}) \subset \sXZ^{\ev,0} \otimes \sXZ^{\ev,0}.
 $
 This  follows from the fact that $\sXZ^0$ is an $\cA$-Hopf algebra, and $\Delta(K_\al^2) = K_\al^2 \ot K_\al^2$.

(b) Recall that $\briota\colon\;\Uq^0\to\brU_q^0$ is the algebra homomorphism %
defined by $\briota(K_\al) = \brK_\al$. Recall that
$(\UZ^{\ev,0}, \XZ^{\ev,0}, \VZ^{\ev,0})$ is a dilatation triple. We have $\brUZ^{\ev,0}= \briota(\UZ^{\ev,0})$. Define $\brXZ^{\ev,0}= \briota(\XZ^{\ev,0})$ and
 $\brVA^{\ev,0} = \briota(\VZ^{\ev,0})$. Then $(\brUZ^{\ev,0}, \brXZ^{\ev,0}, \brVZ^{\ev,0})$ is also a dilatation triple. Then  $\XZ^{\ev,0}$ and $\brXZ^{\ev,0}$ are free $\tA$-modules with respectively bases
 \be
 \label{eq.ba7}
 \{ \sqrt {(q;q)_\bn}\, Q(\bn) \mid  \bn \in \BN^\ell \}, \quad \{ \sqrt {(q;q)_\bn}\, \brQ(\bn) \mid  \bn \in \BN^\ell \}.
 \ee

Since the inclusion $\UZ^{\ev,0} \hookrightarrow \brUA^{\ev,0}$ maps $\VA^{\ev,0}$ into $\brVA^{\ev,0}$, the stability principle (Proposition~\ref{r.closed}) shows that $\sXZ^{\ev,0} \subset \sXZX^{\ev,0}$. In particular $y \in \sXZX^{\ev,0}$.

The orthogonality \eqref{eq.117} and bases \eqref{eq.ba7} show that if $x\in \sXZ^{\ev,0}$ and $y\in  \sXZX^{\ev,0}$ then $\la x, y \ra \in \tA$. Since $K_{2 \lambda} \in \sXZX^{\ev,0}$, we also have $\la x, K_{2 \lambda} \ra \in \tA$.
\end{proof}

\subsubsection{Diagonal part of the ribbon element} \label{sec:28}

The diagonal part $\br_0$ of the ribbon element
(see Section \ref{sec:universal-r-matrix}) is given by
 $$ \br_0 =   K_{-2\rho}  \exp(- h \sum_{\al\in \Pi }H_\al  \brH_{\al}/d_\al).$$
For $\al\in \Pi$ let the $\al$-part of $\sXZ^{\ev,0}$ be $ \sXZ^{\ev,0,\al} := \sXZ^{\ev,0} \cap \tA[K_\al^{\pm 2}]$.
\begin{lemma} \label{r.al0} (a) Each $\sXZ^{\ev,0,\al}$ is an $\tA$-Hopf-subalgebra of  $\sXZ^{\ev,0}$
and $
\sXZ^{\ev,0} = \bigotimes_{\al\in \Pi} \sXZ^{\ev,0,\al}
$.

(b)
For any $\al\in \Pi$,
$\la \br_0 ,  \sXZ^{\ev,0,\al} \ra \in
\tA.$
\end{lemma}
\begin{proof}

By definition, $\sXZ^{\ev,0}$ has $\tA$-basis $\{\sqrt{(q;q)_\bn} \, Q^\ev(\bn) \mid \bn \in \BN^\ell \}$, where $Q^\ev(\bn)  = \prod_{j=1}^\ell Q(\al_j;n_j)$, with
$$Q(\al;n)  =  K_\al^{-2 \lfloor n/2 \rfloor} \frac{(q_\al ^{- \lfloor(n-1)/2   \rfloor} K_\al^2;q_\al)_n}{(q_\al;q_\al)_n} .$$

It follows that $\sXZ^{\ev,0,\al}$ is the $\tA$-module spanned by $\sqrt{(q_\al;q_\al)_n} Q(\al;n)$, and
$
\sXZ^{\ev,0} = \bigotimes_{\al\in \Pi} \sXZ^{\ev,0,\al}$.
 Because $\sXZ^{\ev,0}$ is an $\tA$-Hopf-algebra (Lemma \ref{r.Hopfsub}),  $\sXZ^{\ev,0,\al}$ is an $\tA$-Hopf-subalgebra of $\sXZ^{\ev,0}$.

 (b) We need to show that  for every $n \in \BN$,
$\la \br_0 , \sqrt{(q_\al;q_\al)_n} Q(\al;n) \ra \in
\tA$.
Fix such an $n$.

Let $\cI$ be
 the ideal of $\BZ[q_\al^{\pm 1}, K_\al^{\pm 2}]$ generated by elements of the form $(q^m K_\al^2;q_\al)_n$, $m \in \BN$. Then ${(q_\al;q_\al)_n} Q(\al;n)\in \cI$.
By \eqref{eq.989}
$$\la \br_0, K_{2\al}^k \ra = q_\al^{-k^2 +k}.$$
 With $z= K_{2 \a}$, the  $\BZ[q^{\pm1}_\al]$-linear map $\cL_*: \BZ[q_\al^{\pm 1}, K_\al^{\pm 2}] \to \BZ[q_\al^{\pm 1}]$ given by
$\cL_*(K_{2\al}^k)=\la \br_0, K_{2\al}^k \ra$ is equal to the map $\cL_{-x^2+x}:\BZ[q_\al^{\pm 1},z^{\pm 1}] \to \BZ[q_\al^{\pm 1}]$ of \cite{BCL}. By \cite[Theorem 2.2]{BCL}, for any $f\in \cI$,
$$ \cL_{*}(f ) \in \frac{(q_\al;q_\al)_n}{(q_\al;q_\al)_{\lfloor n/2 \rfloor}} \BZ[q_\al^{\pm1}].$$
As ${(q_\al;q_\al)_n} Q(\al;n)\in \cI$, one has
 \be  \label{eq.mk1}
 \la    \br_0 , (q_\al;q_\al)_n \, Q(\al;n) \ra = \cL_{*}((q_\al;q_\al)_n \, Q(\al;n)   ) \in \frac{(q_\al;q_\al)_n}{(q_\al;q_\al)_{\lfloor n/2 \rfloor}} \BZ[q_\al^{\pm1}].
 \ee
We have
$$ \left( \frac{\sqrt{ (q_\al;q_\al)_n} }{(q_\al;q_\al)_{\lfloor n/2 \rfloor}} \right)^2 = \left( \frac{ (q_\al;q_\al)_n }{(q_\al;q_\al)_{\lfloor n/2 \rfloor}\, (q_\al;q_\al)_{\lfloor n/2 \rfloor}} \right) \in \BZ[q_\al^{\pm1}],$$
where the last inclusion follows from the integrality of the quantum binomial coefficients. Hence, from
\eqref{eq.mk1},
$$ \la \br_0^\al , \sqrt {(q_\al;q_\al)_n}\,  \, Q(\al;n) \ra  \in \frac{\sqrt{ (q_\al;q_\al)_n} }{(q_\al;q_\al)_{\lfloor n/2 \rfloor}}  \BZ[q_\al^{\pm1}] \in \tA.$$
This completes the proof of the lemma.
\end{proof}

\begin{remark}
Theorem \cite[2.2]{BCL}, used in the proof of the above lemma, is one of the main technical results of \cite{BCL} and is difficult to prove.  Its  proof uses Andrews' generalization of the Rogers-Ramanujan identity. Actually, only a special case of Theorem \cite[2.2]{BCL} is used here. This special case can be proved using other methods.
\end{remark}

\subsubsection{Integrality of $\br_0$} \label{sec:29}
\begin{lemma}
\label{r.r50}
Suppose $x \in \sXZ^{\ev,0}$, then $\la \br_0 , x \ra \in \tA$.
\end{lemma}%

\begin{proof} We first prove the following claim.

 {\em Claim.}  If $\la \br_0 , x \ra \in \tA$ for  all $x\in \sH_1$ and for all  $x\in \sH_2$,
where $\sH_1, \sH_2$ are $\tA$-Hopf-subalgebras of $\sXZ^{\ev,0}$, then $\la \br_0 , x \ra \in \tA$ for all $x \in \sH_1 \sH_2$.

{\em Proof of Claim.} Suppose $x\in \sH_1, y \in \sH_2$.
Using the Hopf dual property of the quantum Killing on the Cartan part \eqref{eq.sd10}, we have
$$ \la \br_0, xy \ra  = \la \Delta(\br_0) ,   x\ot y \ra.$$
A simple calculation shows that $\Delta(\br_0)= (\br_0 \ot \br_0) \cD^{-2}$, where $\cD$ is the diagonal part of the $R$-matrix,
$$\cD^{-2}=\exp(-h \sum _\al H_\al \ot \brH_\al/d_\al).$$
Writing $\cD^{-2}= \sum \delta_1 \ot \delta_2$, we have
\begin{align}
\la \br_0, xy \ra & = \sum \la \br_0 \delta_1, x \ra \la \br_0 \delta_2, y \ra  \notag \\
&= \sum \la \br_0, x_{(1)} \ra \la  \delta_1, x_{(2)} \ra \la \br_0, y_{(1)} \ra \la  \delta_2, y_{(2)} \ra  \notag\\
&= \sum \la \br_0, x_{(1)} \ra  \la \br_0, y_{(1)} \ra \la  x_{(2)}, y_{(2)} \ra, \label{eq.xd1}
\end{align}
where in the last identity we use the fact that $\sum \la \delta_1, x\ra  \la \delta_2, y\ra = \la x, y \ra$, which is easy to prove. (Note that on the Cartan part $\Xh^0$, the quantum Killing form is the dual of $\cD^{-2}$, which is the Cartan part of the clasp element $\modc$.)

Since $x_{(1)}\in \sH_1$ and  $y_{(1)} \in \sH_2$, we have $\la \br_0, x_{(1)} \ra  \la \br_0, y_{(1)} \ra \in\tA$.
By Lemma  \ref{r.Hopfsub}(b),  $\la  x_{(2)}, y_{(2)} \ra \in \tA$. Hence, \eqref{eq.xd1} shows that $\la \br_0, xy\ra \in \tA$. This proves the claim.

By Lemma \ref{r.al0}, $\sXZ^{\ev,0} = \bigotimes_{\al\in \Pi} \sXZ^{\ev,0,\al}$, each $\sXZ^{\ev,0,\al}$ is a Hopf-subalgebra of $\sXZ^{\ev,0}$, and $\la \br_0, \sXZ^{\ev,0,\al} \ra \subset \tA$. Hence from the claim we have $\la \br_0, \sXZ^{\ev,0} \ra \subset \tA$.
\end{proof}

\subsubsection{Proof of Proposition \ref{r.integ6}}\label{sec:30}
\begin{proof} We have to show that for every $x\in \sXZe$, $\cT_\pm (x) \in \tA$.
 First  we will show $\cT_+(x)\in \tA$.

It is enough to consider the case when $x= \sqrt {(q;q)_\bn}\, \bbe^\ev(\bn)$, where $\bn=(\bn_1, \bk, \bn_3)\in \BN^{t+ \ell + t}$, since $\sXZe$ is $\tA$-spanned by elements of this form.
 By the triangular property \eqref{e89}  of the quantum Killing form and \eqref{eq.989},
 $$ \cT_+(x) =  \delta_{\bn_1, \bn_3}  q^{(\rho, | E_{\bn_1} |)} \,  \la  \br_0, \sqrt { (q;q)_\bk } Q(\bk) \ra \in \tA.$$
 Here the last inclusion follows from Lemma \ref{r.r50}. This proves the statement for $\cT_+$.

By  Theorem \ref{r.sXZ}, $\sXZe$ is $\tphi$-stable.
 By  \eqref{eq.cTminus}, we have
 $\cT_- (x) =  \cT_+(\tphi(x)) \in \tA$.
 \end{proof}

\subsection{More on integrality of $\br_0$}

\begin{lemma}
\label{r.halfinteg.new}
 Suppose $y\in \sXZ^{\ev,0}$. Then
  $\la \br_0^{\pm 1} , K_{2\rho}\,y \ra \in v^{(\rho,\rho)}\tA $.
\end{lemma}

\newcommand{\Kt}{K_{2\rho}}
\newcommand{\calD}{\mathcal{D}}
\newcommand{\bmu}{\boldsymbol{\mu}}
\newcommand{\mD}{\bmu(\calD^{-1})}
\begin{proof}
Since $\sXZ^{\ev,0}$ is a Hopf-algebra (Lemma \ref{r.Hopfsub}), we have $\Delta(y)= \sum y_{(1)} \ot y_{(2)}$ with $y_{(1)}, y_{(2)} \in \sXZ^{\ev,0}$.
Using \eqref{eq.xd1} then \eqref{eq.989a}, we have
$$  \la \br_0 , K_{2\rho} y \ra = \sum \la \br_0, K_{2\rho} \ra  \la \br_0, y_{(1)} \ra \la  K_{2\rho}, y_{(2)} \ra= v^{(\rho,\rho)}\sum
\la\br_0, y_{(1)} \ra \,\la  K_{2\rho}, y_{(2)} \ra,$$
where we use $\la \br_0, K_{2\rho} \ra= v^{(\rho,\rho)}$, which follows from an easy calculation.
 The second factor $\la\br_0, y_{(1)} \ra$ is in $\tA$ by Lemma \ref{r.r50}. The last factor $\la  K_{2\rho}, y_{(2)} \ra$ is in $\tA$. Thus, we have $ \la \br_0 , K_{2\rho} y \ra \in v^{(\rho,\rho)} \tA $.

Using \eqref{eq.cTminus}, the fact that $\sXZ^{\ev,0}$ is $\varphi$-stable, and the above case for $\br_0$, we have
 $$ \la \br_0^{- 1} , K_{2\rho} y \ra =  \la \br_0, \varphi(K_{2\rho} y )\ra = \la \br_0, K_{2\rho}\varphi( y )\ra \in v^{(\rho,\rho)} \tA.$$
 This completes the proof of the lemma.
\end{proof}

\def\V{\VZ}

\np
\section{Gradings}
\label{sec:nonc-grad}

In Section \ref{sec:Ygrading}, we  defined the $Y$-grading and the
$Y/2Y$-grading on
$\Uq$.  In this section we define a grading of $\Uq$  by a group $G$, which is a (possibly noncommutative)
central $\modZ /2\modZ $-extension
of $Y\times (Y/2Y)$, thus refining both the two gradings by $Y$ and $Y/2Y$.
This grading is extended to the tensor powers of $\Uq$.

The reason for the introduction of the $G$-grading is the following.
The integral core $\XZ$ will be enough for us to show that the invariant $J_M$ of integral homology 3-spheres, a priori belonging to $\Chh$, is in
$$ \varprojlim_k \BZ[v^{\pm 1}]/((q;q)_k).$$
But we want to show that $J_M$ belongs to a smaller ring, namely $\Zqh= \varprojlim_k \BZ[q^{\pm 1}]/((q;q)_k)$, and the $G$-grading will be helpful in the proof. In the section \ref{sec:integralM} we will show that  quantum link invariants of algebraically split bottom tangles belong to a certain homogeneous part
of this $G$-grading.

\subsection{The groups $G$ and $\Gv$} \label{sec:grad1}

Let $G$ denote the group generated by the elements $\dv$, $\dK_\alpha $,
$\de_\alpha $ ($\alpha \in \Pi $) with the relations
\begin{gather*}
  \dv\quad \text{central},\quad
  \dv^2 = \dK_\alpha ^2 = 1,\quad
  \dK_\alpha \dK_\beta =\dK_\beta \dK_\alpha ,\\
  \dK_\alpha \de_\beta =\dv^{(\alpha ,\beta )}\de_\beta \dK_\alpha ,\quad
  \de_\alpha \de_\beta =\dv^{(\alpha ,\beta )}\de_\beta \de_\alpha .
\end{gather*}
Let $\Gv$ be the subgroup of $G$ generated by $\dv$, $\de_\alpha $ ($\alpha \in \Pi $).

\begin{remark}
  \label{r24}
  The groups $G$ and $\Gv$ are abelian if and only if $\mathfrak g$ is of type $A_1$ or $B_n$ ($n\ge 2$).
\end{remark}

Define a homomorphism $G\rightarrow Y$, $g\mapsto|g|$, by
\begin{gather*}
  |\dv|=|\dK_\alpha |=0,\quad |\de_\alpha |=\alpha \quad (\alpha \in \Pi ).
\end{gather*}

For $\gamma =\sum_{i}m_i\alpha _i\in Y$, set
\begin{gather*}
  \dK_\gamma =\prod_i\dK_{\alpha _i}^{m_i},\quad
  \de_\gamma =\prod_i\de_{\alpha _i}^{m_i}=\de_{\alpha _1}^{m_1}\cdots \de_{\alpha _l}^{m_l}.
\end{gather*}
Note that $\de_\gamma $ depends on the order of the simple roots
$\alpha _1,\ldots ,\alpha _l\in \Pi $.

One can easily verify the following commutation rules:
\begin{gather}
  \label{e28}
  g\dK_\lambda =\dv^{(|g|,\lambda )}\dK_\lambda g\quad \text{for $g\in G$, $\lambda \in Y$},\\
  \label{e29}
  gg'=\dv^{(|g|,|g'|)}g'g\quad \text{for $g,g'\in \Gv$}.
\end{gather}

Let $N$ be the subgroup of $G$ generated by $\dv$. Then $N$ has order $2$ and  is a
subgroup of the center of $G$.  Note that $G/N\cong Y\times  (Y/2Y)$ and
$\Gv/N\cong Y$.

\subsubsection{Tensor products of $G$ and $\Gv$}

By $G\otimes G=G\otimes _NG$, we mean the ``tensor product over $N$'' of two
copies of $G$, i.e.,
\begin{gather*}
  G\otimes G:= (G\times G)/( (\dv x, y)\sim (x,\dv y)).
\end{gather*}
Similarly, we can define $G\otimes \Gv$, $\Gv\otimes \Gv$, etc., which are
subgroups of $G\otimes G$.  Denote by $x\otimes y$ the element in $G\otimes G$
represented by $(x,y)$.  Thus we have $\dv x\otimes y=x\otimes \dv y$.

Similarly, we can also define the tensor powers $G^{\otimes n}=G\otimes \cdots \otimes G$,
$(\Gv)^{\otimes n}=\Gv\otimes \cdots \otimes \Gv\subset G^{\otimes n}$ (each with $n$ tensorands).
Define a homomorphism $\iota _n\col N\rightarrow G^{\otimes n}$ by
\begin{gather*}
  \iota _n(\dv^k)=\dv^k\otimes 1^{\otimes (n-1)},\quad k=0,1.
\end{gather*}
We have
\begin{gather*}
  G^{\otimes n}/\iota _n(N)\cong Y^n\times (Y/2Y)^n,\quad
  (\Gv)^{\otimes n}/\iota _n(N)\cong Y^n.
\end{gather*}
For $n=0$, we set
\begin{gather*}
  G^{\otimes 0}=(\Gv)^{\otimes 0}=N.
\end{gather*}

\subsection{$G$-grading of $\Uq$}
\label{sec:g1g.u}
By a $G$-grading of $\Uq$ we mean a direct sum decomposition of
$\mathbb{C}(q)$-vector spaces
$$ \Uq = \bigoplus_{g \in G} \left [ \Uq \right]_g$$
such that $1\in[\Uq]_1$ and $\left [ \Uq \right]_g\, \left [ \Uq
  \right]_{g'} \subset \left [ \Uq \right]_{gg'}$ for $g,g'\in G$. If  $x\in \left [ \Uq \right]_g$, we write $\deg_G(x)=g$.

\begin{proposition}
  \label{thm:2} There is a unique $G$-grading on $\Uq$ such that
  \begin{gather*}
  \deg_G(v) = \dv,\quad
  \deg_G(K_{\pm \alpha })=\dK_\alpha ,\quad
  \deg_G(E_\alpha )=\dv^{d_\alpha }\de_\alpha ,\quad
  \deg_G(F_\alpha )=\de_\alpha ^{-1}\dK_\alpha .
\end{gather*}
\end{proposition}

\begin{proof} Since $v^{\pm 1}, K_\al, E_\al, F_\al$ generates the $\BC(q)$-algebra $\Uq$, the uniqueness is clear. Let us prove the existence of the $G$-grading.

  Let $\cUq$ denote the free $\BC (q)$-algebra generated by the elements
  $\tilde  v$, $\tilde  v^{-1}$, $\tilde  K_\alpha $, $\tilde  K_\alpha ^{-1}$,
  $\tilde  E_\alpha $, $\tilde  F_\alpha $.  We can define a $G$-grading of $\cUq$
  by
  \begin{gather*}
    \deg_G(\tilde  v^{\pm 1}) = \dv,\quad
    \deg_G(\tilde  K_\alpha ^{\pm 1})=\dK_\alpha ,\quad
    \deg_G(\tilde  E_\alpha )=\dv^{d_\alpha }\de_\alpha ,\quad
    \deg_G(\tilde  F_\alpha )=\de_\alpha ^{-1}\dK_\alpha .
  \end{gather*}
  The kernel of
  the obvious homomorphism $\cUq\rightarrow \Uq$ is the two-sided ideal in
  $\cUq$ generated by the defining relations of the $\BC (q)$-algebra
  $\Uq$:
  \begin{gather*}
    \tilde  v\tilde  v^{-1}=\tilde  v^{-1}\tilde  v=1,\quad \tilde
    v^2=q,\quad \tilde  v\ \text{central},
    \\
    \tilde  K_\alpha \tilde  K_\alpha ^{-1}=\tilde  K_\alpha ^{-1}\tilde  K_\alpha =1,\quad
    \tilde  K_\alpha \tilde  K_\beta =\tilde  K_\beta \tilde  K_\alpha ,\\
    \tilde  K_\alpha \tilde  E_\beta \tilde  K_\alpha ^{-1}
    =\tilde  v^{(\alpha ,\beta )}\tilde   E_\beta ,\quad
    \tilde  K_\alpha \tilde  F_\beta \tilde  K_\alpha ^{-1}
    =\tilde  v^{-(\alpha ,\beta )}\tilde  F_\beta ,\\
    \tilde  E_\alpha \tilde  F_\beta -\tilde  F_\beta \tilde  E_\alpha
    =\delta _{\alpha ,\beta }(q^{d_\alpha }-1)^{-1}\tilde  v^{d_\alpha }(\tilde  K_\alpha -\tilde
    K_\alpha ^{-1}),\\
    \sum_{s=0}^{1-a_{\alpha \beta }}(-1)^s \bbb{1-a_{\alpha \beta }}{s}_\alpha ^{\tilde {}}
    \tilde  E_\alpha ^{1-a_{\alpha \beta }-s}\tilde  E_\beta \tilde  E_\alpha ^s=0\quad (\alpha \neq\beta ),\\
    \sum_{s=0}^{1-a_{\alpha \beta }}(-1)^s \bbb{1-a_{\alpha \beta }}{s}_\alpha ^{\tilde {}}
    \tilde  F_\alpha ^{1-a_{\alpha \beta }-s}\tilde  F_\beta \tilde  F_\alpha ^s=0\quad (\alpha \neq\beta ).
  \end{gather*}
  Here, for $n,s\ge 0$,  $\bbb ns^{\tilde {}}_\alpha $ is obtained from
  $\bbb ns_\alpha \in \modZ [v_\alpha ,v_\alpha ^{-1}]$ by replacing $v_\alpha ^{\pm 1}$ by $\tilde
  v^{\pm d_\alpha }$.  Since the above relations are
  homogeneous in the $G$-grading of $\cUq$, the assertion holds.
\end{proof}
From the definition, we have
$$ \Uq^\ev = \bigoplus_{g \in G^\ev} \left [\Uq \right]_g.$$
We say that $x\in \Uq$ is {\em $G$-homogeneous}, and write $\dot x =g$, if $x\in [\Uq]_g$ for some $g\in G$. Similarly, we say
$x\in \Uq$ is {\em $G^\ev$-homogeneous} if $x\in [\Uq]_g$ for some $g\in G^\ev$.

\subsubsection{The $G^{\otimes m}$-grading of $\Uq^{\otimes m}$}
\label{sec:gm.grading.um}
For $m\ge 1$, $\Uq^{\otimes m}$ is $G^{\otimes m}$-graded:
\begin{equation*}
  \Uq^{\otimes m}=\bigoplus_{g\in G^{\otimes m}} [\Uq^{\otimes m}]_g,
\end{equation*}
where, for $g=g_1\otimes \cdots \otimes g_m\in G^{\otimes m}$ ($g_i\in G$), we set
\begin{equation*}
  [\Uq^{\otimes m}]_g = [\Uq]_{g_1}\otimes _{\BC (v)}\cdots \otimes _{\BC (v)}[\Uq]_{g_m}\subset \Uq^{\otimes m}.
\end{equation*}
Note that $\BC (v)=\Uq^{\otimes 0}$ is $N(=G^{\otimes 0})$-graded:
$[\BC (v)]_{\dv^k}=v^k\BC (q)$, $k=0,1$.
We extend the $N$-grading of $\BC(v)$ to a $G$-grading by setting
$$ [\BC(v)]_g = \begin{cases} g\,  \BC (q) \quad &\text{if } g= 1 \ \text{or } g=v \\
0 \quad &\text{otherwise }.
\end{cases}
$$

\subsubsection{Total $G$-grading of $\Uq^{\otimes m}$ and $G$-grading preserving map}
\label{sec:gradTotal}

For $g\in G$ and $m\ge 0$, set
\begin{gather*}
  [\Uq^{\otimes m}]_g:=
  \sum_{g_1,\ldots ,g_m\in G;\; g_1\cdots g_m=g}[\Uq^{\otimes m}]_{g_1\otimes \cdots \otimes g_m}.
\end{gather*}
This gives a $G$-grading of the $\mathbb{C}(q)$-module $\Uq^{\otimes m}$ for each $m\ge 0$.
(If
$m=0$, we have $[\Uq^{\otimes 0}]_{\dv^k}=[\BC (v)]_{\dv^k}=v^k\BC (v)$ for
$k=0,1$, and $[\Uq^{\otimes 0}]_g=0$ for $g\in G\setminus \{1,\dv\}$.)

A $\mathbb{C}[[h]]$-module map $f: \Uh^{\ho n} \to \Uh^{\ho m}$ is said {\em to preserve the $G$-grading} if for every $g\in G$, $f( [\UZ^{\ot n}]_g) \subset  [\UZ^{\ot m}]_g$. Here
$$ [\UZ^{\ot n}]_g = [\Uq^{\ot n}]_g \cap \UZ^{\ot n}.$$
\subsection{Multiplication, unit, and counit}
\label{sec:multiplication-unit}

From the definition of the $G$-grading, we have the following.
\begin{proposition}\label{r.bh11}
Each of $\boldmu, \boldsymbol \eta, \boldep$ preserves the $G$-grading, i.e.
$$ \boldmu([\Uq^{\ot 2}]_g) \subset [\Uq]_g, \quad \boldsymbol \eta([\BC(v)]_g) \subset  [\Uq]_g, \quad \boldep ( [\Uq]_g) \subset [\BC(v)]_g.$$
\end{proposition}

\no{
Define a map $\dmu\col G\otimes G\rightarrow G$ by
$
  \dmu(g\otimes g')=gg'%
$.

Then, for $g,g'\in G$  we have
\begin{gather*}
  {\boldsymbol \mu} ([\Uq^{\otimes 2}]_{g\otimes g'})\subset \uqq{\dmu(g\otimes g')}.
  \end{gather*}

Set $\dot\eta =\iota _1\col N\rightarrow G$.
We have
\begin{gather*}
  {\boldsymbol \eta} ([\BC (v)]_{\dv^k})\subset [\Uq]_{\dot\eta (\dv^k)}.
\end{gather*}

For the counit, we have for $g\in G$
\begin{gather*}
  {\boldsymbol \epsilon}([\Uq]_g)\subset
  \begin{cases}
  v^{i}\modQ(q)\quad \text{if $g=\dv^i$, $i=0,1$},\\
  0\quad \text{otherwise},
  \end{cases}
\end{gather*}
hence
\begin{gather*}
  {\boldsymbol \epsilon}([\Uq]_g)\subset [\mathbb{Q}(v)]_g.
\end{gather*}

}

\subsection{Bar involution $\ibar$ and mirror automorphism $\varphi$}From the definition one has immediately the following.
\begin{lemma}\label{r.ibar}
The bar involution $\ibar: \Uh \to \Uh$ preserves the $G$-grading.
\end{lemma}
 Let $\dtphi:G \to G$ be the automorphism given by $\dtphi(\dv)=
 \dv$, $\dtphi(\dK_\al)= \dK_{\al}$, $\dva(\de_\al) = v^{d_\al}\de_{\al}^{-1}$. From   %
the definition of $\dtphi$
one has the following.
\begin{lemma} \label{lem.dva}
$g\in G$, we have $\tphi(\uqq{g})\subset \uqq{\dtphi(g)}$.
\end{lemma}

\subsection{Antipode}
\label{sec:antipode}

Define a function $\dS\col G\rightarrow G$ by
\begin{gather*}
  \dS(g\dK_\gamma )=\dK_{\gamma +|g|}g = \dv^{(|g|,\gamma )}g\dK_{\gamma +|g|}
\end{gather*}
for $g\in \Gv,\gamma \in Y$.
One can easily verify that $\dS$ is an involutive anti-automorphism.

\begin{lemma}
  \label{thm:6}
  For $g\in G$, we have $S(\uqq{g})\subset \uqq{\dS(g)}$. In particular, if $y=S(x)$ where $x$ is $G^\ev$-homogeneous, then
  \be
  \label{eq.Sgrade}
   \dy = \dx \dK_{|x|} = \dK_{|x|} \dx .
   \ee
   Here $\dy=\deg_G(y)$ and $\dx=\deg_G(x)$.
\end{lemma}

\begin{proof}
  It is easy to check that if $x=v,K_\alpha ,E_\alpha ,F_\alpha $, then $S(x)$ is
  homogeneous of degree $\dS(g)$.  If $x,y\in \Uq$ are homogeneous of
  degrees $\dx,\dy\in G$, respectively, then $S(xy)=S(y)S(x)$ is homogeneous of
  degree $\dS(\dy)\dS(\dx)=\dS(\dx\dy)$.
  Hence, by induction, we deduce that, for each monomial $x$ in the
  generators, $S(x)$ is homogeneous of degree $\dS(\dx)$.  This
  completes the proof.
\end{proof}

\subsection{Braid group action}
\label{sec:symmetries-t}

Define a function $\dT_\alpha \col G\rightarrow G$ by
\begin{gather*}
  \dT_\alpha (g\dK_\gamma )=\dv_\alpha ^{\hf r(r+1)} \de_\alpha ^r g\dK_{s_\alpha (\gamma )}\quad \text{where } r=-(|g|,\alpha )/d_\al,
\end{gather*}
for $g\in \Gv$, $\gamma \in Y$.  Note that $\dT_\alpha $ is an
involutive automorphism of $G$,
satisfying $\dT_\alpha (\Gv)\subset \Gv$.

\begin{lemma}
  \label{r13}
  If $g\in G$, then we have
  \begin{gather*}
    T_\alpha (\uqq{g})\subset \uqq{\dT_\alpha (g)}.
  \end{gather*}
\end{lemma}

\begin{proof}
  It suffices to check that for each generator $x$ of $\Uq$ we have
  $T_\alpha (x)\in \uqq{\dT_\alpha (\deg_G(x))}$, which follows from the
  definitions.
\end{proof}

\subsection{Quasi-R-matrix}
\label{sec:elements-dth}

For $\lambda \in Y$, set
\begin{equation*}
  \dth_\lambda  = \de_\lambda ^{-1}\dK_\lambda \otimes \de_\lambda \in G^{\ot2}.
\end{equation*}
We have $\dth_0=1\otimes 1$.  Note that $\dth_\lambda $ does not depend
on the order of the simple roots $\alpha _1,\ldots ,\alpha _\ell$.

\begin{lemma}
  \label{lem:1}
  For $\lambda ,\mu \in Y$, we have
  \begin{equation*}
    \dth_\lambda \dth_\mu =\dth_{\lambda +\mu }.
  \end{equation*}
\end{lemma}

\begin{proof}
  \[
  \begin{split}
    \dth_\lambda \dth_\mu
    &=(\de_\lambda ^{-1}\dK_\lambda \otimes \de_\lambda )(\de_\mu ^{-1}\dK_\mu \otimes \de_\mu )\\
    &=\de_\lambda ^{-1}\dK_\lambda \de_\mu ^{-1}\dK_\mu \otimes \de_\lambda \de_\mu \\
    &=\de_\mu ^{-1}\de_\lambda ^{-1}\dK_\lambda \dK_\mu \otimes \de_\lambda \de_\mu \\
    &=(\de_\lambda \de_\mu )^{-1}\dK_{\lambda +\mu }\otimes \de_\lambda \de_\mu \\
    &=\de_{\lambda +\mu }^{-1}\dK_{\lambda +\mu }\otimes \de_{\lambda +\mu }\\
    &=\dth_{\lambda +\mu }.
  \end{split}
  \]
\end{proof}

The automorphism $\dT_\alpha \col G\rightarrow G$ induces an automorphism
\begin{equation*}
  \dT_\alpha ^{\otimes 2}\col G^{\otimes 2}\rightarrow G^{\otimes 2},\quad
  g_1\otimes g_2\mapsto\dT_\alpha (g_1)\otimes \dT_\alpha (g_2).
\end{equation*}

\begin{lemma}
  \label{thm:3}
  If $\alpha \in \Pi $ and $\lambda \in Y$, then we have
  \begin{equation*}
    \dT_\alpha ^{\otimes 2}(\dth_\lambda ) = \dth_{s_\alpha (\lambda )}.
  \end{equation*}
\end{lemma}

\begin{proof}
  We have
  \[
  \begin{split}
    \dT_\alpha ^{\otimes 2}(\dth_\lambda )
    =&\dT_\alpha ^{\otimes 2}(\de_\lambda ^{-1}\dK_\lambda \otimes \de_\lambda )\\
    =&\dT_\alpha (\de_\lambda ^{-1})\dT_\alpha (\dK_\lambda )\otimes \dT_\alpha (\de_\lambda )\\
    =&\dT_\alpha (\de_\lambda )^{-1}\dK_{s_\alpha (\lambda )}\otimes \dT_\alpha (\de_\lambda ).
  \end{split}
  \]
  Hence it suffices to show that
  \begin{equation}
    \label{eq:2}
    \dT_\alpha (\de_\lambda )^{-1}\otimes \dT_\alpha (\de_\lambda )=\de_{s_\alpha (\lambda )}^{-1}\otimes \de_{s_\alpha (\lambda )},
  \end{equation}
  which can be verified by using the fact that there is $k\in \{0,1\}$
  such that $\dT_\alpha (\de_\lambda )=\dv^k\de_{s_\alpha (\lambda )}$.
\end{proof}

Recall that $\Theta$ is the quasi-$R$-matrix and its definition is given in Section \ref{Rmatrix}.
For $\gamma \in Y_+$, let $\Theta _\gamma \in \Uq^{\otimes 2}$ denote the weight $(-\gamma ,\gamma )$-part of
$\Theta $, so that we have $\Theta =\sum_{\gamma \in Y_+}\Theta _\gamma $.  Similarly, let
$\bar\Theta _\gamma $ denote the weight $(-\gamma ,\gamma )$-part of $\bar\Theta =\Theta ^{-1}$.

\begin{lemma}
  \label{thm:4}
  For $\gamma \in Y_+$, we have $\Theta _\gamma ,\bar\Theta _\gamma \in [\Uq^{\otimes 2}]_{\dth_\gamma }$.
\end{lemma}

\begin{proof} Suppose $\bi=(i_1,\ldots,i_t)$ is a longest reduced sequence.
  Note that
  \begin{gather*}
    \Theta _\gamma =\sum_{\modm =(m_1,\ldots ,m_t)\in \modZ ^t,\;|E_\modm(\bi)| =\gamma }
    \Theta ^{[t]}_{m_t}\cdots \Theta ^{[1]}_{m_1},
  \end{gather*}
  where we set
  \begin{gather*}
    \Theta ^{[i]}_{n}:=(T_{\alpha _{j_1}}\cdots T_{\alpha _{j_{i-1}}})^{\otimes 2}(
    (-1)^n v_{\alpha _{j_i}}^{-\hf n(n-1)} F_{\alpha _{j_i}}^{(n)}\otimes \bar E_{\alpha _{j_i}}^n).
  \end{gather*}

  For each $\alpha \in \Pi $, we have
  \begin{gather*}
    (-1)^n v_\alpha ^{-\hf n(n-1)} F_\alpha ^{(n)}\otimes \bar E_\alpha ^n\in [\Uq^{\otimes 2}]_{\dth_{n\alpha }}.
  \end{gather*}
  By Lemma \ref{thm:3}, we deduce that $\Theta ^{[i]}_{n}\in \Uq^{\otimes 2}$ is
  homogeneous of degree
  \begin{gather*}
    (\dT_{\alpha _{j_1}}\cdots \dT_{\alpha _{j_{i-1}}})^{\otimes 2}(\dth_{n\alpha _{j_i}})
    =\dth_{ns_{\alpha _{j_1}}\cdots s_{\alpha _{j_{i-1}}}(\alpha _{j_i})}.
  \end{gather*}
  Hence it follows that $\Theta _\gamma $ is homogeneous of degree $\dth_\gamma $.
  The case of $\bar\Theta _\gamma $ follows from $\Theta^{-1} = (\ibar\ot\ibar)(\Theta)$ and Lemma \ref{r.ibar} which says $\ibar$ preserves the $G$-grading.
\end{proof}

\begin{corollary} \label{cor.918}
 Fix a longest reduced sequence $\bi$. For $\bm \in \BN^t, \gamma \in Y$,
\begin{align}
 E_\bm \ot K_\bm F_\bm,\  E'_\bm \ot K_\bm F'_\bm & \in [\UZ^\ev \ot \UZ^\ev]_{\elm \ot \elm^{-1}}\subset [\UZ^\ev\ot\UZ^\ev]_1.
 \label{eq.nh1}\\
  F_\bm \, K_\bm\, K_{2\gamma} E_\bm \in \left[ \UZ \right]_1.  \label{eq.nh2}
 \end{align}
Here $\lambda_\bm = |E_\bm|= |E'_\bm|$.
\end{corollary}

\begin{proof}
 We have $\Theta= \sum_\bm F_\bm \otimes E_\bm$ and $\Theta^{-1} = \sum_\bm F'_\bm \otimes E'_\bm$.
Hence, \eqref{eq.nh1} follows from Lemma \ref{thm:4}. In turn,  \eqref{eq.nh2} follows from \eqref{eq.nh1}, because
 $\dK_{2 \gamma}=1$.
\end{proof}
\subsection{Twist forms} Recall that we have defined $\cT_\pm: \UZ^\ev \to \BQ(v)$, see Section \ref{twistonUq}.

\begin{proposition}\label{r.918a}
Both maps $\cT_\pm: \UZ^\ev   \to \BQ(v)$ preserve the $G$-grading, i.e.
$\cT_\pm \left([\UZ^\ev ]_g \right)
\subset [\BQ(v)]_g.$
\end{proposition}

\begin{proof} (a) First we consider the case of $\cT_+$.
The set
$$\{    F_\bm \, K_\bm\, K_{2\gamma} E_\bn \mid \bn,\bm \in \BN^t, \gamma\in Y  \}$$
 is a $\BQ(v)$-basis of  $\UZ^\ev \ot_\cA \BQ(v)$. Hence,
 $$\{   v^\delta F_\bm \, K_\bm\, K_{2\gamma} E_\bn \mid \bn,\bm \in \BN^t, \gamma\in Y , \delta \in \{0,1\} \}$$
  is a $\BQ(q)$-basis of $\UZ^\ev \ot_\cA \BQ(v)$. Each element of this basis is $G$-homogeneous.
 By \eqref{eq.ba12},
 $$ \cT_+(v^\delta F_\bm \, K_\bm\, K_{2\gamma} E_\bn) = \delta_{\bn,\bm}\,
 v^\delta q^{(\gamma,\rho) - (\gamma,\gamma)/2} \in v^{\delta} \Zq= [\cA]_{\dv^\delta}.
 $$
 By  Corollary \ref{cor.918}, the $G$-grading of $v^\delta F_\bm \, K_\bm\, K_{2\gamma} E_\bm$ is $\dv^\delta$. Hence, we have
 \be
 \label{eq.gr1}
 \cT_+
 \left([\UZ^\ev \ot_\cA \BQ(v)]_g \right)
\subset [\BQ(v)]_g .
 \ee
(b) Now consider $\cT_-$.
  Using \eqref{eq.cTminus}, Lemma \ref{lem.dva}, and \eqref{eq.gr1}, we have
 $$ \cT_- \left([\UZ^\ev \ot_\cA \BQ(v)]_g \right) = \cT_+ \left(\varphi([\UZ^\ev \ot_\cA \BQ(v)]_g) \right) \subset \cT_+ \left( [ \UZ^\ev \ot_\cA \BQ(v)]_{\dtphi(g)} \right) \subset [\BQ(v)]_{\dtphi(g)}= [\BQ(v)]_{g},$$
 where the last identity follows from the fact that for the involution $\dtphi$, we have $\dtphi(1) =1$ and $\dtphi(v)=v$, and for any $g \not \in \{1,v\}$, we have $[\BQ(v)]_g=0$.
\end{proof}

\subsection{Coproduct}
\label{sec:coproduct}

\begin{lemma}
\label{r.coprod}
Suppose $x\in \Uq$ is $G^\ev$-homogeneous. There exists a presentation
  $$ \Delta(x)= \sum x_{(1)} \ot x_{(2)}$$
  such that each for each  $x_{(1)}\ot x_{(2)}$,

   (i) $x_{(1)}$ is $G$-homogeneous

   (ii) $x_{(2)}$ and $x_{(1)} K_{|x_{(2)}|}$ are
  $G^\ev$-homogeneous,
  and
  \be
  \dx _{(1)} \dK_{|x_{(2)}|} \dx_{(2)}=\dx =  \dx _{(2)} \dK_{|x_{(2)}|} \dx_{(1)}. \label{eq.evenD}
  \ee
  \end{lemma}

\begin{remark}
A presentation of $\D(x)$ as in Proposition Lemma \ref{r.coprod} is called a {\em $G$-good presentation}. When $x$ is $G^\ev$-homogeneous, we always use a $G$-good presentation for $\D(x)$.
\end{remark}

\begin{proof} Suppose $x,y$ are $G^\ev$-homogeneous.
If $\Delta(x) = \sum_i x'_i \ot x''_i$ and $\Delta(y) = \sum_j y'_j \ot y''_j$ are $G$-good presentation of $x$ and $y$ respectively, then it is easy to check that $\sum_{i,j} x'_i y'_j \ot x''_i y''_j$ is
a $G$-good presentation of $\Delta(xy)$. Hence, one needs only to check the statement for $x$ equal to  generators  $K_{2\al}, E_\al, F_\al K_\al$ of $\Uq^\ev$. For each of these generators, the defining formulas of $\Delta$ show that the statement holds.
\end{proof}

\subsection{Adjoint action}
\label{sec:adjoint-action}

Define a map
\begin{equation*}
  \dot\ad\col G\otimes \Gv\rightarrow \Gv
\end{equation*}
by
\begin{equation*}
  \dot\ad(g\dK_{\lambda }\otimes g') = \dv^{(\lambda ,|g'|)}gg'
\end{equation*}
for $\lambda \in Y$, $g,g'\in \Gv$.  Note that for $g,g'\in \Gv$ we have
$\dot\ad(g\otimes g')=gg'$.

\begin{lemma}
  \label{r18}
  For $g,g'\in \Gv$ and $\gamma \in Y$, we have
  \begin{gather*}
    \ad([\Uq\otimes \Uqv]_{g\dK_\gamma \otimes g'})\subset \uqvq{\dot\ad(g\dK_\gamma \otimes g')}.
  \end{gather*}
  In particular, if $z= x \tri y$ where both $x,y$ are $G^\ev$-homogeneous, then $z$ is $G^\ev$-homogeneous with
  \begin{gather*}
    \dot z = \dx \dy.
  \end{gather*}
\end{lemma}

\begin{proof} Suppose $x,y$ are $G^\ev$-homogeneous, and $\gamma \in Y$. Choose a $G$-good presentation  $\Delta(x) = \sum x_{(1)} \ot x_{(2)}$ (see Section \ref{sec:coproduct}).
By definition,
$$ (x K_\gamma) \tri y = \sum x_{(1)} K_\gamma y S( x_{(2)} K_\gamma) = \sum x_{(1)} K_\gamma \, y K_{\gamma}^{-1} \,  S( x_{(2)}) =   \sum v^{(\gamma,|y|)}\, x_{(1)} \, y  \,  S( x_{(2)}).$$
A term of the last sum is in $[\Uq]_u$, where
\begin{align*}
u & = \dv^{(\gamma,|y|)}\, \dx_{(1)} \dy \,  \dS( \dx_{(2)}) \\
&= \dv^{(\gamma,|y|)}\, \dx_{(1)}\,  \dy  \, \dK_{|x_{(2)}|}\,  \dx_{(2)}\quad \text{by Lemma \ref{thm:6}} \\
&= \dv^{(\gamma,|y|)}\, \dx_{(1)} \, \dK_{|x_{(2)}|}\,  \dx_{(2)} \,  \dy\\
&= \dv^{(\gamma,|y|)}\, \dx \,  \dy \quad \text{by \eqref{eq.evenD}}.
\end{align*}
  Hence we have the assertion.
\end{proof}

\np
\section{Integral values of $J_M$}\label{sec:integralM}
By Theorem \ref{r.JMdef}, the core subalgebra $\Xh$, constructed in Section \ref{sec:corealgebra}, gives rise to an invariant $J_M$ of integral homology 3-spheres. A priori, $J_M \in \Chh$.
The main result of this section is to show that $J_M\in \Zqh$ for any integral homology 3-sphere $M$. To prove this fact we will construct a family of $\cA$-submodules $\tK_n \subset \Xh^{\ho n}$ satisfying Conditions (AL1) and (AL2) of Theorem \ref{thm.construction}, with $\tK_0= \Zqh$. Then by Theorem  \ref{thm.construction}, $J_M \in \tK_0= \Zqh$.

\subsection{Module
$\tK_n$}
\label{sec:Kn1}
For $n \ge 0$  let $\left[(\UZ^\ev)^{\ot n}\right]_1$ denote the $G$-grading 1 part of $(\UZ^\ev)^{\ot n}$.
Define
\begin{align}
 \modK_n & := (\sXZe)^{\ot n} \cap \left[(\UZ^\ev)^{\ot n}\right]_1
 \subset \left( \Xh^{\ho n} \cap \Uh^{\ho n} \right).
\end{align}

In the notation of \eqref{eq.XAsq}, $\modK_n = [\XA^{\boxtimes n}]_1$. For example, $\modK_0= \Zq$.
Define filtrations on $\modK_n$ by
$$ \cF_k(\modK_n) := (q;q)_k \modK_n \subset \left( h^k \Xh^{\ho n} \cap h^k  \Uh^{\ho n} \right).$$

Let $\tK_n$ be the completion of $\modK_n$ by the filtrations $\cF_k(\modK_n)$ in $\Uh^{\ho n}$, i.e.
\be
\notag
\tK_n :=\left \{ x=\sum_{k=0}^\infty  x_k  \ \bigg| \ x_k \in \cF_k(\modK_n)\right \} \subset \left(  \Xh^{\ho n} \cap  \Uh^{\ho n} \right).
\ee
Since $\modK_0=\Zq$, we have  $\tK_0=  \Zqh$.
Each $\tK_n$ has the structure of a complete $\Zqh$-module.

\begin{proposition} \label{r.integ10} The family $(\tK_n)$ satisfies Condition (AL2) of Theorem \ref{thm.construction}. Namely,
if  $\ve_1, \dots, \ve_n \in \{ \pm1\}$ and $x\in \tK_n$ then
 $$ (\cT_{\ve_1} \ho \dots \ho \cT_{\ve_n}) (x) \in \tK_0= \Zqh.$$
\end{proposition}
\begin{proof} By definition, $x$ has a presentation $x = \sum_{k=0}^\infty (q;q)_k x_k$, where $x_k \in \modK_n \subset \Xh$.
Since $\cT_\pm$ are continuous on the $h$-adic topology of $\Xh^{\ho n}$, we have
 \be
 \label{eq.e118}
 (\cT_{\ve_1} \ho \dots \ho \cT_{\ve_n})(x)=   \sum_{k=0}^\infty (q;q)_k \, (\cT_{\ve_1} \ho \dots \ho \cT_{\ve_n}) (x_k) \in \Chh.
 \ee
Since $x_k \in \modK_n \subset (\XZ^\ev)^{\ot n}$, by Proposition \ref{r.integ6}, $(\cT_{\ve_1} \ho \dots \ho \cT_{\ve_n}) (x_k) \subset \tA$.

Since $x_k \in  \left[(\UZ^\ev)^{\ot n}\right]_1$, by Proposition \ref{r.918a}, $(\cT_{\ve_1} \ho \dots \ho \cT_{\ve_n}) (x_k) \subset  [\BQ(v)]_1= \BQ(q).$
Hence,
$$ (\cT_{\ve_1} \ho \dots \ho \cT_{\ve_n}) (x_k) \in \tA \cap \BQ(q)= \Zq,$$
where the last identity is Lemma \ref{r.60}. From \eqref{eq.e118} we have $(\cT_{\ve_1} \ho \dots \ho \cT_{\ve_n})(x) \in \Zqh$.
\end{proof}

\subsection{Finer version of
$\tK_n$} \label{sec:Kn2}
 We will show that for an $n$-component bottom tangle $T$ with 0 linking matrix, $J_T \in \tK_n$. Then Proposition \ref{r.integ10} will show that $J_M\in \Zqh$ for any integral homology 3-spheres.

For the purpose of proving that $J_M$ recovers the Witten-Reshetikhin-Turaev invariant, we want $J_T$ to belong to smaller subsets of $\tK_n$, which we will describe here.

Suppose $\cU$ is  an $\cA$-Hopf-subalgebra of $\UZ$.
Define (with convention $\Ur^{\ot 0} = \cA$)
$$ \modK_n(\cU):= \modK_n \cap \cU^{\ot n}, \quad \cF_k(\modK_n(\cU)):= \cF_k(\modK_n) \cap \cU^{\ot n}.$$

Let $\tK_n(\cU)$ be completion of $\modK_n(\Ur)$ with respect to the filtration $(\cF_k(\modK_n))$ in $\Uh$, i.e.
\be
\notag
\tK_n(\cU) := \left \{ x=\sum_{k=0}^\infty x_k \mid  x_k \in \cF_k(\modK_n(\cU)) \right\}.
\ee

Since $\cF_k(\modK_n(\cU)) \subset \cF_k(\modK_n) $
we have
$\tK_n(\cU) \subset \tK_n \subset \Xh^{\ho n}$. %
We always  have $\tK_0(\cU)= \tK_0= \Zqh$.

\subsection{Values of universal invariant of algebraically split bottom tangle} Throughout we fix a longest reduced sequence $\bi$.

Recall that $\Gamma = \bc \cD^2$ is the quasi-clasp element, see Section \ref{Rosso-form}.  By Lemma \ref{r5},
$$ \Gamma = \sum_{\bn\in \BN^{2t}} \Gamma_1(\bn) \ot \Gamma_2(\bn),$$
where for $\bn=(\bn_1, \bn_2)\in \BN^t \times \BN^t$,
\be
\label{eq.Gn}
\Gamma_1(\bn)= q^{-(\rho- |E_{\bn_1}|,|E_{\bn_2}|) } F_{\bn_1} K_{\bn_1}^{-1} E_{\bn_2}, \quad \Gamma_2(\bn)= F_{\bn_2} K_{\bn_2}^{-1} E_{\bn_1}.
\ee
\begin{proposition}\label{r.Jvalues1}
 Suppose $\cU$ is an $\cA$-Hopf-subalgebra of $\UZ$ such that $K_\al \in \Ur$ for all $\al \in \Pi$,
 and $F_\bm \ot E_\bm, F'_\bm \ot E'_\bm \in  \cU \ot \cU$ for all $\bm \in \BN^t$.

 Then the family $(\tK_n(\cU))$ satisfies Condition (AL1) of Theorem \ref{thm.construction}. Namely, the following statements hold.

(i) $1_{\Ch} \in \tK_0(\cU)$, $1_{\Uh} \in \tK_1(\cU)$, and $x\ot y \in \tK_{n+m}(\cU) $ whenever $x\in\tK_n(\cU)$ and $y \in  \tK_m(\cU) $.

(ii) Each of  ${\boldsymbol \mu} $, ${\boldsymbol{\psi}} ^{\pm1}$,  $\uD$, $\bS$  is $(\tK_n(\cU))$-admissible.

(iii) The Borromean element $\modb $ belongs to $\tK_3(\cU)$.
\end{proposition}

Note that under the assumption of Proposition \ref{r.Jvalues1}, we have
$\Gamma_1(\bn) \ot \Gamma_2(\bn) \in \cU \ot \cU$ for all $\bn \in \BN^{2t}$.

Before embarking on the proof of the proposition, let us record some consequences.
\begin{theorem}\label{r.Jvalues}  Suppose $\cU$ is an $\cA$-Hopf-subalgebra of $\UZ$ satisfying the assumption of Proposition \ref{r.Jvalues1}. Then
(a) For any $n$-component bottom tangle $T$ with 0 linking matrix, $J_T \in \tK_n(\cU)$.
In particular, $J_T \in \tK_n$.

(b) For any integral homology 3-sphere $M$, $J_M\in \Zqh$.
\end{theorem}

\begin{proof}
(a) By Proposition \ref{r26}, (i)--(iii) of Proposition \ref{r.Jvalues1} imply that  $J_T \in  \tK_n(\cU) \subset \tK_n$.

(b)  By Propositions \ref{r.Jvalues1} and \ref{r.integ10}, $(\tK_n(\cU))$ satisfies both conditions (AL1) and (AL2) of Theorem~\ref{thm.construction}. By Theorem~\ref{thm.construction}, $J_M \in \tK_0(\cU)= \Zqh$.
\end{proof}

\def\tQ{\tilde \cQ}
The remaining part of the section is devoted to the proof of Proposition  \ref{r.Jvalues1}.
Statement (i) of Proposition \ref{r.Jvalues1} follows trivially from the definitions. We will prove (ii) and (iii) in this section.
We fix $\cU$ satisfying the assumptions of Proposition \ref{r.Jvalues1}.

\begin{remark}
(a) One can relax the assumption of the Proposition \ref{r.Jvalues1}, requiring only that $K_\al \in \Ur$ for all $\al\in \Pi$ and both $\Theta$ are in the topological closure of $\Ur \ot \Ur$ (in the $h$-adic topology of $\Uh\ho \Uh$).

(b) Almost identical proof shows that
Theorem \ref{r.Jvalues}  holds true if
$\cU$ is an  $\tA$-Hopf-subalgebra of $\XZ$ instead of $\cU \subset \UZ$.
\end{remark}
\subsection{Quasi-$R$-matrix}\label{sec:Kn4}
Recall that $\Theta = \sum_{\bn \in \BN^t} F_\bn \ot E_\bn$, see Section \ref{sec:universal-r-matrix}.
For a multiindex $\bn=(n_1,\dots,n_k) \in \BN^k$ let $\max(\bn) = \max_{j}(n_j)$ and
\be
\label{eq.oorder}
o(\bn):= (q;q)_{\lfloor \max(\bn)/2 \rfloor} \in \BZ[q^{\pm1}].
\ee
\begin{lemma} For each $\bn\in \BN^t$, we have
\begin{align}
 \label{eq.eprime}   E_{\bn},\  E'_{\bn} \ & \in \  o(\bn)\,   \UZ^\ev  \\
  \label{eq.bg2}   \ K_\bn \, F_{\bn} \ot E_{\bn},\ K_\bn \, F'_{\bn} \ot E'_{\bn}
  \ & \in \ o(\bn) (\sXZe \ot \UZe) .
 \end{align}
\end{lemma}
\begin{proof}
We write $x \sim y$ if $x =uy$ with $u$ a unit in $\cA$.
From the definition of $E_\bn$ (see Section \ref{sec:universal-r-matrix}),
$$E_\bn \sim (q;q)_{\bn} E^{(\bn)} \in (q;q)_{\bn} \UZe \subset o(\bn) \UZe.$$
Recall that $E'_\bn= \ibar(E_\bn), F'_\bn= \ibar(F_\bn)$. Since $\ibar$ preserves the even part (Proposition \ref{eq.evenD3}), and $\ibar$ leaves both $\UZ$  and $\XZ$ stable (Proposition \ref{prop.basis} and  Theorem \ref{r.sXZ}), $\ibar$ leaves both $\UZ^\ev$  and $\XZ^\ev$ stable. Hence, we have
$$ E'_\bn= \ibar(E_\bn) \subset o(\bn) \ibar(\UZe) = o(\bn)\UZe,$$
which proves \eqref{eq.eprime}. Let us now prove  \eqref{eq.bg2}.
We have
\begin{align*}
 K_\bn \, F_{\bn} \ot E_{\bn}    %
& \sim (q;q)_{\bn} \, F^{(\bn)}  K_\bn
\ot E^{(\bn)} \sim \sqrt {(q;q)_{\bn} }  \Big(\sqrt {(q;q)_{\bn}} \,  F^{(\bn)}  K_\bn  \Big ) \ot E^{(\bn)} \\
& \in \sqrt {(q;q)_{\bn} } \, \sXZe \ot \UZe \subset   o(\bn)\,  \sXZe \ot \UZe.
\end{align*}
Applying $\ibar$, we get
$$K_\bn^{-1} \, F_{\bn}' \ot E_{\bn}'\in \ o(\bn) (\sXZe \ot \UZe).$$
 Since $K_{\bn}^2\in \sXZe$, we also have $K_\bn \, F_{\bn}' \ot E_{\bn}'\in \ o(\bn) (\sXZe \ot \UZe)$.
\end{proof}

\subsection{On $\cF_k(\modK_n)$}

\begin{lemma}
For any $k,n \in \BN$, one has
\begin{gather}
\label{eq.sosanh}
\cF_k(\modK_n) =(q;q)_k (\XZ^\ev)^{\ot n} \cap \left[ (\UZ^\ev)^{\ot n}  \right]_1 = (q;q)_k (\XZ^\ev)^{\ot n} \cap \left[ (\UZ)^{\ot n}  \right]_1\\
\label{eq.sosanh1} (q;q)_k (\XZ)^{\ot n} \cap  ((\UZ^\ev)^{\ot n} \ot_\cA \tA) = (q;q)_k (\XZ^\ev)^{\ot n}.
\end{gather}

\end{lemma}
\begin{proof}
  The preferred basis \eqref{eq.basesXZ} of $\XZ$  is a dilatation of the preferred basis of $\UZ$ (described in Proposition \ref{r.basesUZ}). The basis of  $\UZ$
 gives rise in a natural way to an  $\cA$-basis $\{ e(i) \mid i \in I\}$ of $\UZ^{\ot n}$. By construction,   there is a function $a: I \to \tA$ such that $\{ a(i) e(i) \mid i\in I\}$ is an $\tA$-basis of $\XZ^{\ot n}$. Besides, there is subset $I^\ev \subset I$ such that $\{ e(i) \mid i \in I^\ev\}$ is an $\cA$-basis of $(\UZ^\ev)^{\ot n}$ and  $\{ a(i) e(i) \mid i \in I\}$ is an $\tA$-basis of $(\XZ^\ev)^{\ot n}$.

Using the $\tA$-basis $\{ e(i) \mid i\in I\}$, every $x\in (\UZ^{\ot n}\ot_\cA \tA)$ has unique presentation
$$x = \sum _{I\in I} x_i e(i), \quad x_i \in \tA.$$
Then
\begin{itemize}
\item[(a)] $x\in \UZ^{\ot n}$ if and only if $x_i \in \cA$ for all $i\in I$.
\item[(b)] $x \in (\UZ^\ev)^{\ot n}$ if and only if $x_i \in \cA$ for all $i\in I$ and $x_i=0$ for $i\not \in I^\ev$.
    \item[(c)] $x \in ((\UZ^\ev)^{\ot n} \ot_\cA \tA)$ if and only if %
      $x_i=0$ for $i\not \in I^\ev$.
    \item[(d)] $x \in (q;q)_k (\XZ)^{\ot n}$ if and only if $x_i \in (q;q)_k a(i) \tA$ for all $i \in I$.
    \item[(e)] $x \in (q;q)_k (\XZ^\ev)^{\ot n}$ if and only if $x_i \in (q;q)_k a(i) \tA$ for all $i \in I$ and $x_i=0$ for $i\not \in I^\ev$.
\end{itemize}
Proof of \eqref{eq.sosanh1}. By (c) and (d),  $x \in (q;q)_k (\XZ)^{\ot n} \cap  ((\UZ^\ev)^{\ot n} \ot_\cA \tA)$ if and only if $x_i \in (q;q)_k a(i) \tA$ for all $i \in I$ and $x_i=0$ for $i\not \in I^\ev$, which, by (e), is equivalent to $x \in (q;q)_k (\XZ^\ev)^{\ot n}$. Hence we have \eqref{eq.sosanh1}.

Proof of \eqref{eq.sosanh}. Since $\cF_k(\modK_n) = (q;q)_k \Big( (\XZ^\ev)^{\ot n} \cap \left[ (\UZ^\ev)^{\ot n}  \right]_1\Big)$, we have
$$ \cF_k(\modK_n) \subset (q;q)_k (\XZ^\ev)^{\ot n} \cap \left[ (\UZ^\ev)^{\ot n}  \right]_1 \subset (q;q)_k (\XZ^\ev)^{\ot n} \cap \left[ (\UZ)^{\ot n}  \right]_1.$$
It remains to prove the last term is a subset of the first, i.e. if
 $y\in (q;q)_k (\XZ^\ev)^{\ot n} \cap \left[ (\UZ)^{\ot n}  \right]_1$, then  $x:= y/
(q;q)_k $   belongs to $ (\XZ^\ev)^{\ot n} \cap \left[ (\UZ^\ev)^{\ot n}  \right]_1$.
By definition,
$$x \in (\XZ^\ev)^{\ot n} \cap \frac1{(q;q)_k}\left[ (\UZ)^{\ot n}  \right]_1,$$ and we need to show
$x \in \left[ (\UZ^\ev)^{\ot n}  \right]_1$. Since  both $y$ and $ (q;q)_k$ have $G$-grading 1, $x= y/(q;q)_k$ is an element of $(\Uq)^{\ot n}$ has $G$-grading 1. It remains to show that $x\in (\UZ^\ev)^{\ot n} $.

Because $x\in (\XZ^\ev)^{\ot n}$, (e) implies  $x_i \in \tA$  and $x_i=0$ if $i\not\in I^\ev$. Because $x\in \frac1{(q;q)_k} (\UZ)^{\ot n}$, (a) implies $x_i \in \BQ(v)$. It follows that
$x_i \in \tA \cap \BQ(v)= \cA$ and $x_i=0$ if $i\not\in I^\ev$. By (b),  $x\in (\UZ^\ev)^{\ot n} $.
\end{proof}

\subsection{Admissibility decomposition}
Suppose  $f\col (\Uh) ^{\ho a}\rightarrow (\Uh) ^{\ho b}$ is a  $\BC [[h]]$-module homomorphism. We also use $f$ to denote its natural extension $f\col (\UUh) ^{\ho a}\rightarrow (\UUh) ^{\ho b}$, where $\UUh= \Uh \ho_{\Ch} \Chh$.

Recall that $f$ {\em preserves the $G^\ev$-grading} if for every $g\in G^\ev$,
   $$ f\left(\left[(\UZ^\ev)^{\ot a}\right]_g\right) \subset \left[(\UZ^\ev)^{\ot b}\right]_g,$$
  and  $f$ is $(\tK_n(\cU))$-admissible if for every $i, j\in \BN$,
   $$ f_{(i,j)}(\tK_{i+a+j}(\cU)) \subset \tK_{i+b+j}(\cU),$$
   where $f_{(i,j)} = \id^{\ho i} \ho f \ho \id^{\ho j}$.

The following definition is useful in showing a map is $(\tK_n(\cU))$-admissible.

\begin{definition}
  \label{r351}
  Suppose  $f\col (\Uh) ^{\ho a}\rightarrow (\Uh) ^{\ho b}$ is a  $\BC [[h]]$-module homomorphism. An {\em admissibility
  decomposition} for $f$ is a decomposition $f=\sum_{p\in P_f}f_p$
   as an $h$-adically
  converging sum  of $\BC [[h]]$-module
  homomorphisms $f_p\col (\Uh) ^{\ho a}\rightarrow (\Uh) ^{\ho b}$ over a set $P_f$, satisfying the
  following conditions (A)--(C).
  \begin{enumerate}
  \item[(A)]
  For $p\in P_f$, $f_p$ preserves the $G^\ev$-gradings.

  \item[(B)] For $p\in P_f$, $f_p\left( \cU^{\ot a}\right) \subset \cU^{\ot b}$.

  \item[(C)] There are $m_p\in\BN$ for $p\in P_f$ such that $\lim_{p\in P_f} m_p=\infty$ and for each $p\in P_f$ we have
  \be
  f_p( (\sXZe)^{\ot a} )\  \subset \ (q;q)_{m_p}  (\sXZe)^{\ot b}.
  \notag
  \ee
  \end{enumerate}
\end{definition}
Here, $\lim_{p\in P_f} m_p=\infty$ means if $n \ge 0$, then $m_p \ge n$ for all but a finite number of $p\in P_f$. By definition, if $P_f$ is finite, then we always have
 $\lim_{p\in P_f} k_p=\infty$.

\begin{lemma}
  \label{r36}
  If $f\col (\Uh)^{\ho a}\rightarrow(\Uh)^{\ho b}$ has an admissibility decomposition then $f$
  is $(\tK_n(\cU))$-admissible.
\end{lemma}

\begin{proof} Recall that $\cF_k(\modK_n(\cU)) = \cF_k(\modK_n) \cap \Ur^{\ot n}$.
From \eqref{eq.sosanh},
\be
\label{eq.bh2a}
\cF_k(\modK_n(\cU))= (q;q)_k (\XZ^\ev)^{\ot n} \cap \left[ (\UZ^\ev)^{\ot n}  \right]_1 \cap \cU^{\ot n}.
\ee

 Let $f= \sum_{p \in P} f_p$ be an admissibility decomposition of $f$.
Suppose $x \in \tK_{i+a+j}(\cU)$ with presentation
$$ x = \sum  x_k, \quad x_k \in \cF_k(\modK_{i+a+j}(\cU)).$$
Then, with the $h$-adic topology, we have
$$ f_{(i,j)}(x) = \sum_{k,p} (f_p)_{(i,j)}(x_k).$$
We look at each term of the right hand side.
Since $x_k \in \left[ (\UZ^\ev)^{\ot i+a+j}  \right]_1$,  (A) implies that
\be
\label{eq.bh1}
(f_p)_{(i,j)}(x_k) \in \left[ (\UZ^\ev)^{\ot i+b+j}  \right]_1.
\ee
Since $x_k \in \cU^{\ot i+a+j}$, Condition (B) implies that
\be
\label{eq.bh1a}
(f_p)_{(i,j)}(x_k) \in \cU^{\ot i+a+j}.
\ee

We have $x_k = \ooo_k \, y_k$ with $y_k \in(\sXZe)^{\ot i+a+j}$.
By Condition (C),
\be
(f_p)_{(i,j)}(x_k) = \ooo_k(f_p)_{(i,j)}(y_k)  \in \ooo_k \ooo_{m_p} (\sXZe)^{\ot i+b+j} \subset \ooo_{m(k,p)}\, (\sXZe)^{\ot i+b+j},
\notag
\ee
where $m(k,p) = \max(k, m_p)$. Together with \eqref{eq.bh1}, \eqref{eq.bh1a}, and \eqref{eq.bh2a}, this implies
$$
(f_p)_{(i,j)}(x_k) \in \cF_{m(k,p)}(\modK_{i+b+j}).
$$

Condition (C) implies that
$$
 \lim_{(k,p)\in \BN \times P_f} m({k,p})=\infty.
 $$
Hence,
$ f_{(i,j)}(x) = \sum_{k,p} (f_p)_{(i,j)}(x_k)$
belongs to $\tK_{i+b+j}$.
\end{proof}

\begin{remark}
  \label{r34}
  It is not difficult to show that the set of $(\tK_n)$-admissible maps
  are closed under composition and tensor product.  Thus there is a
  monoidal category whose objects are nonnegative integers and whose
  morphisms from $m$ to $n$ are $(\tK_n)$-admissible $\BC [[h]]$-module
  homomorphisms from $\Uhx m$ to $\Uhx n$.
  According to
  Lemma \ref{r.psi}, this category is braided with $\psi _{1,1}={\boldsymbol{\psi}} $,
  and contains a braided Hopf algebra structure.
\end{remark}

 \subsection{Admissibility of
 ${\boldsymbol \mu} $}
 \begin{lemma}\label{r.mu}The multiplication $\boldmu: \Uh \ho \Uh \to \Uh$
  is $(\tK_n)$- admissible.
 \end{lemma}
\begin{proof}  We show that the trivial decomposition, $P_{\boldmu}=\{0\}$ and $\boldmu_0=\boldmu$, is an admissibility decomposition for $\boldmu$.

(A)  The fact that $\boldmu$ preserves the $G^\ev$-grading is part of
 Proposition \ref{r.bh11}.

(B) Since $\cU$ is a subalgebra of $\UZ$, $\boldmu(\cU\ot \cU) \subset \cU$.

(C)  Since $\sXZe$ is an $\tA$-algebra, we have
$ \quad {\boldsymbol \mu} (\sXZe \ot \sXZe) \subset \sXZe,$ which proves~(C).

By Lemma \ref{r36}, ${\boldsymbol \mu} $  is $(\tK_n)$- admissible.
\end{proof}

\subsection{Admissibility of ${\boldsymbol{\psi}} $}

\begin{lemma}  \label{r.psi}
Each of
 $\boldsymbol \psi^{\pm 1}$ %
 is  $(\tK_n)$-admissible.
 \end{lemma}
 \begin{proof} First  consider  $\boldpsi$.
 Using \eqref{501} and \eqref{e64},
we obtain ${\boldsymbol{\psi}} = \sum_{\modm\in P_{\boldsymbol{\psi}} } {\boldsymbol{\psi}} _\modm $, where $P_{\boldsymbol{\psi}} =\BN^t$
and
\begin{gather}
\label{e38a}
  {\boldsymbol{\psi}} _\modm (x\otimes y)= v^{(|y|+\lambda _\modm ,|x|-\lambda _\modm )} (E'_\modm \trr y)\otimes (F'_\modm \trr x),
\end{gather}
with $\lambda_\bm = |E'_\bm|$. We will show this is an admissibility decomposition of $\boldpsi$.

(A) Suppose $x,y\in \UZ^\ev$ are $G^\ev$-homogeneous. By Lemma \ref{thm:4},
$$E'_\bm \ot F'_\bm \in [\UZ \ot \UZ]_{\elm \ot \elm^{-1}\dK_{\bm} }.$$
From  \eqref{e38a} and  Lemma \ref{r18},
we have ${\boldsymbol{\psi}} _\modm (x\otimes y)\in [(\UZ^\ev)^{\otimes 2}]_u$, where
\begin{gather*}
  \begin{split}
    u&=\dv^{(|y|+\lambda _\modm ,|x|-\lambda _\modm )}\,  \dad(\de_{\lambda _\modm }\ot  \dy)\,
    \dad  (\de_{\lambda _\modm }^{-1}\dK_{\lambda _\modm }\ot  \dx)\\
    &=\dv^{(|y|+\lambda _\modm ,|x|-\lambda _\modm )+(\lambda _\modm , |x|)}
    \de_{\lambda _\modm }\dy
    \, \de_{\lambda _\modm }^{-1}\dx\\
    &=\dv^{(|y|+\lambda _\modm ,|x|-\lambda _\modm )+(\lambda _\modm ,|x|)+(\lambda _\modm ,|y|)+(|x|,|y|)}
    \dx\dy\\
    &=\dx\dy.
  \end{split}
\end{gather*}
This shows that $\boldpsi_\bm$ preserves the $G^\ev$-grading.

(B)  By assumptions on $\Ur$,  $E'_\bm \ot F'_\bm \in \Ur \ot \Ur$ and $\Ur$ is a Hopf algebra. Now \eqref{e38a} shows that $\boldpsi_\bm (\Ur \ot \Ur) \subset \Ur \ot \Ur$. This proves (B).

(C)
By \eqref{eq.eprime}, $E'_\bm \ot F'_\bm \in o(\bm) \UZ \ot \UZ$.  Hence, from \eqref{e38a},
$$  {\boldsymbol{\psi}} _\modm (\sXZe\otimes \sXZe) \subset o(\bm) (\UZ \tri \sXZe) \ot (\UZ \tri \sXZe) \subset o(\bm) \sXZe \ot \sXZe,$$
where for the last inclusion we use Theorem  \ref{r.sXZ}(a), which in particular says $\sXZe$ is $\UZ$-stable.
This establishes   (C) of Lemma \ref{r36}.

 By Lemma \ref{r36},
 $\boldpsi$ is $(\tK_n(\Ur))$-admissible.

Now consider the case $\boldpsi^{-1}$. By computation, we obtain ${\boldsymbol{\psi}} ^{-1}= \sum_{\bm\in\BN^t } ({\boldsymbol{\psi}} ^{-1})_\modm $, where
\begin{gather*}
  ({\boldsymbol{\psi}} ^{-1})_\modm (x\otimes y) =  v^{-(|x|,|y|)} (F_\modm \trr y)\otimes (E_\modm \trr x),
\end{gather*}
for homogeneous $x,y\in \Uh$.  The proof is similar to the case of ${\boldsymbol{\psi}} $.
\end{proof}

\begin{remark}
One can check that ${\boldsymbol{\psi}} ^{-1} = (\tphi \ho \tphi) {\boldsymbol{\psi}} (\tphi^{-1}\ho \tphi^{-1})$. Hence, the admissibility of ${\boldsymbol{\psi}} ^{-1}$ can also be derived from that of
${\boldsymbol{\psi}}$.
\end{remark}

\subsection{Admissibility of \underline{$\Delta $}}
\label{sec:admiss-underl}
\begin{lemma} \label{r.uD}
The braided co-product $\uD$
is $(\tK_n(\Ur))$-admissible.
\end{lemma}
\begin{proof} Suppose $x\in \UZ^\ev$ is $G^\ev$-homogeneous. %
By a simple calculation, we have
\be
\uD=\sum_{\modm \in \BN^t}\uD_\modm , \label{eq.decompD}
\ee
where, with  $\lambda_\bm:= |E'_\bm|$, %
\begin{gather}
  \label{e35}
  \uD_\modm (x)= \sum  v^{-(|x_{(2)}|, \lambda_\bm)} \Big(E'_\modm \trr x_{(2)}\Big)\otimes  \Big(K_{\modm}\, F'_\modm \Big)\, \Big(K_{|x_{(2)}|} \, x_{(1)}\Big).
\end{gather}

(A) By Corollary \ref{cor.918}, $E'_\modm \otimes  K_{\modm}\, F'_\modm \in [\UZ^\ev \ot \UZ^\ev]_{\de_\bm \ot \de_\bm^{-1}}$. We will use a $G$-good presentation $\Delta(x) = \sum x_{(1)} \ot x_{(2)}$ (see Section \ref{sec:coproduct}).
From Lemma \ref{r18}, each summand of the right hand side of \eqref{e35} is in $[(\UZe)^{\otimes 2}]_u$, where
\begin{gather*}
  \begin{split}
      u&= \dv^{-(|x_{(2)}|, \lambda_\bm) }\, \elm \,  \dx_{(2)} \,   \elm^{-1} \,  \dK_{|x_{(2)}|}\,  \dx_{(1)}\\
      &=\dx_{(2)} \,   \dK_{|x_{(2)}|}\,  \dx_{(1)} = \dx.
  \end{split}
\end{gather*}
Here the last identity is \eqref{eq.evenD}.  Thus, $\uD_\bm$ preserves the $G^\ev$-grading.

(B) Since $K_\al \in \Ur$ and $E'_\bm \ot K_\bm F'_\bm \in \Ur\ot \Ur$, \eqref{e35} shows that  $\uD_\bm(\Ur) \subset\Ur \ot \Ur$.

(C) Let $x\in \sXZe$. By an argument similar to the proof of
  Lemma \ref{r.coprod}, we see that $$x_{(2)} \ot K_{|x_{(2)}|} x_{(1)} \in \sXZe \ot \sXZe.$$
By  \eqref{eq.bg2},
$$ E'_\bm \ot K_\bm F'_\bm \in \ o(\bm)\,  \UZe \ot \sXZe.$$

Hence, from \eqref{e35},
\begin{gather*}
  \uD_\modm (x) \in  o(\bm)\,\, (\UZe \tri \sXZe)\, \sXZe \subset
  o(\bm)\, \sXZe,
\end{gather*}
where for the last inclusion we again use the fact the $\sXZe$ is $\UZ$-stable (Theorem \ref{r.sXZ}). This shows (C) of Lemma \ref{r36} holds.
By Lemma \ref{r36}, $\uD$ is $(\tK_n(\Ur))$-admissible.
\end{proof}

\subsection{Admissibility of \underline{$S$}}

\begin{lemma} \label{r.bS}
The braided antipode
$\bS$
is $(\tK_n(\Ur))$-admissible.
\end{lemma}
\begin{proof}
By computation, we obtain $\bS=\sum_{\modm \in \BN^t}\bS_\modm$,  where
\be
\label{eq.decompS}
\bS_\modm (x)=   S^{-1}(E_\modm \trr x)F_\modm K_{-|x|}
\ee
for $Y$-homogeneous $x\in \Uh$. We will assume $x$ is $G^\ev$-homogeneous.

(A) By Lemma \ref{thm:6}, we have $\bS_\modm (x)\in [\UZ]_g$, where
\begin{gather*}
  \begin{split}
    g&=\dS^{-1}(\dad (\elm  \ot \dx))\, \elm^{-1}\klm\dK_{-|\dx|}
    =\dS^{-1}(\elm\dx)\elm^{-1}\klm\dK_{-|\dx|}\\
    &\quad =\dK_{|\dx|}\dx \klm\elm\elm^{-1}\klm\dK_{-|\dx|}
    =\dx.
  \end{split}
\end{gather*}

(B)  Since $K_\al \in \Ur$ and $E_\bm \ot F_\bm \in \Ur \ot \Ur$,  \eqref{eq.decompS} shows that  $\bS_\modm(\Ur) \subset \Ur$.

(C) We rewrite \eqref{eq.decompS} as
  \be
  \label{e31a}
  \bS_\modm (x)=  v^{-(|x|, |E_\bm|)} S^{-1} (E_\modm \trr x) K_{-|E_\bm| -|x|}   \, K_{\bm}   F_\modm.
  \ee
By \eqref{eq.bg2}, $E_\bm \ot K_\bm F_\bm \in o(\bm) \, (\UZ \ot  \sXZe)$.
Since $\sXZe$ is $\UZ$-stable,
$$ E_\modm \trr x \ot  K_{\bm}   F_\modm \in o(\bm) \, (\UZ\tri \sXZe \ot  \sXZe) \subset  o(\bm) \, (\sXZe \ot  \sXZe).$$

Hence, from \eqref{e31a} we have
$$ \bS_\modm (x) \in o(\bm) \, \sXZe,$$
which proves property (C).

By Lemma \ref{r36}, $\bS$ is $(\tK_n(\Ur))$-admissible.
\end{proof}

Thus, statement (ii) of Proposition \ref{r.Jvalues1} holds.

\subsection{Borromean tangle} The goal now is to establish  (iii) of Proposition \ref{r.Jvalues1}. Namely, we will show that $\modb \in \tK_3$, where $\modb$ is the universal invariant of the Borromean bottom tangle.

First we recall Formula \eqref{e42} which expresses  $\modb$ through the clasp element $\bc$ using braided commutator. With
$ \bc = \sum_{\bn\in \BN^{2t}} [\Ga_1(\bn) \ot \Ga_2(\bn)] \cD^{-2},$ Formula \eqref{e42} says
$$ \modb = \sum_{\bn,\bm \in \BN^{2t}} \modb _{\bn,\bm},$$
where
for $\bn, \bm\in \BN^{2t}$,
\be
\label{eq.bnm}
\modb_{\bn,\bm}:= (\id^{\ho 2}\ho \Upsilon)\Big ([\Ga_1(\bn) \ot \Ga_1(\bm) \ot \Ga_2(\bm) \ot \Ga_2(\bn)]\cD_{14}^{-2} \cD_{23}^{-2} \Big).
\ee
Here if $x=\sum x' \ot x''$ then $x_{14}= \sum x'\ot 1 \ot 1 \ot x''$, $x_{23}= \sum 1 \ot x'\ot x''\ot 1$.

\begin{lemma}\label{r.Borrom}
For $\bn, \bm \in \BN^{2t}$ one has
$\modb_{\bn,\bm} \in
o(\bn,\bm)\tK_3(\cU).$
Consequently, $\modb\in\tK_3(\cU)$. \\
\rm{(Recall that $o(\bn,\bm)=(q;q)_{\lfloor\max(\bn,\bm)/2\rfloor}$.)}
\end{lemma}
The remaining part of this section is devoted to the proof of Lemma \ref{r.Borrom}.

\subsubsection{Quasi-clasp element} Recall that $\Ga_1(\bn), \Ga_2(\bn)$ are  given by~\eqref{eq.Gn}, for $\bn\in \BN^{2t}$.

\begin{lemma} \label{r.Gamma}
Suppose $\bn=(\bn_1,\bn_2) \in \BN^{t}\times \BN^t$. Then
\begin{align} \label{eq.Ga1}
\Gamma_1(\bn) \ot \Gamma_2(\bn) & \in \modK_2 = (\sXZe)^{\ot 2} \cap [(\UZe)^{\ot 2}]_1
\\
\label{eq.Ga2} \Gamma_1(\bn) \ot \Gamma_2(\bn) & \in o(\bn) \sXZe \ot \UZ^\ev.
\end{align}
\end{lemma}
\begin{proof} We write $x\sim y$ if $x=uy$ with $u$ a unit in $\cA$.
   Note that $ { \sqrt{(q;q)_{\bn_1}}} \, {F^{(\bn_1)} ,  { \sqrt{(q;q)_{\bn_2}}} \,E^{(\bn_2)} }$  are in $\sXZe$, as they are among the preferred basis elements.
Using the definition \eqref{eq.Gn} of $\Gamma_1(\bn), \Gamma_2(\bn)$, we have
\be
\label{eq.e313}
\Gamma_1(\bn) \ot \Gamma_2(\bn) \sim   (q;q)_{\bn_1} (q;q)_{\bn_2} {F^{(\bn_1)} K_{\modn_1}^{-1} E^{(\bn_2)} } \ot {F^{(\bn_2)} K_{\modn_2}^{-1} E^{(\bn_1)} } \in (\sXZe)^{\ot 2}.
\ee
From  Corollary \ref{cor.918}, $\Gamma_1(\bn) \ot \Gamma_2(\bn)$ , which is in $(\UZe)^{\ot 2}$, has $G$-grading equal to
$$ (\de_{\bn_1}^{-1} \de_{\bn_2})  (\de_{\bn_2}^{-1} \de_{\bn_1})= 1.   $$
This shows $ \Gamma_1(\bn) \ot \Gamma_2(\bn)  \in  (\sXZe)^{\ot 2} \cap [(\UZe)^{\ot 2}]_1 = \modK_2$. This proves  \eqref{eq.Ga1}.

Because $\sqrt{(q;q)_{\bn_1} (q;q)_{\bn_2}} {F^{(\bn_1)} K_{\modn_1}^{-1} E^{(\bn_2)} } \in \sXZ^\ev$ and ${F^{(\bn_2)} K_{\modn_2}^{-1} E^{(\bn_1)} }\in \UZ^\ev$, from \eqref{eq.e313}, we have
$$ \Gamma_1(\bn) \ot \Gamma_2(\bn)  \in \sqrt{(q;q)_{\bn_1} (q;q)_{\bn_2}} \, (\sXZe)\ot \UZe  \subset o(\bn)
 (\sXZe)\ot \UZe.$$
 This proves \eqref{eq.Ga2}.
\end{proof}

\subsubsection{
Decomposition of $\modb_{\bn,\bm}$} \label{sec:decompo}
Recall that $\cD= \exp(\frac h 2 \sum_{\al \in \Pi} H_\al \otimes \brH_\al/d_\al)$ is the diagonal part of the $R$-matrix.
We will freely use the following well-known properties of $\cD$:
\begin{gather*}
  (\Delta \otimes 1)(\cD)=\cD_{13}\cD_{23},\quad
  ({\boldsymbol \epsilon} \otimes 1)(\cD)=1,\quad
  (S\otimes 1)(\cD)=\cD^{-1},
\end{gather*}
where $\cD_{13}=\sum \cD_1\otimes 1\otimes \cD_2,\cD_{23}=1\otimes \cD\in \Uh^{\ho3}$.
In the sequel we set
\begin{gather*}
  \cD^{-2}=\sum \d_1\otimes \d_2 = \sum \d'_1\otimes \d'_2.
\end{gather*}

Recall
  \eqref{eq.bnm}
$$ \modb_{\bn,\bm}= ( \id^{\ot 2} \ot \Upsilon) \Big ([\Gamma_1(\bn) \otimes \Gamma_1(\bm) \otimes \Ga_2(\bm)  \otimes \Ga_2(\bn)] \cD_{14}^{-2} \cD_{23}^{-2} \Big).$$
By \eqref{e41}, $\Upsilon$ is the composition of four maps:
$$ \Upsilon = \boldmu \circ (\ad \ho \id) \circ (\id \ho \bS \ho \id ) \circ (\id \ho \uD).$$
Using the above decomposition, one gets
\be
\label{eq.deco}
\modb_{\bn,\bm} =  \fmu \circ \fad_\bm \circ \fbS  \circ  \fuD \Big( \Ga_1(\bn)\ot \Ga_2(\bn)  \Big)  ,
\ee
where
 \begin{align}
\fuD & : \Uh^{\ho2} \to \Uh^{\ho 3}, \
\fuD (x)= \left[ (\id \ho \uD)(x \cD^{-2}) \right] \cD_{12}^2 \cD_{13}^2  \label{eq.fuDdef}\\
\fbS &: \Uh^{\ho 3} \to \Uh^{\ho 3}, \ \fbS(x) = \left[ (\id \ho \bS \ho \id)(x \cD^{-2}_{12}) \right] \cD_{12}^{-2}   \label{eq.fbSdef}\\
\fad_\bm&: \Uh^{\ho 3} \to \Uh^{\ho 4},  \fad_\bm (x)= %
(\id^{\ho 2} \ho \ad \ho \id) \Big([ x_1  \ot \Ga_1(\bm)  \ot \Ga_2(\bm) \ot x_2 \ot x_3] \cD^{-2}_{23} \cD_{14}^2   \Big)
\cD_{13}^{-2}
\label{eq.faddef}\\
\fmu&: \Uh^{\ho 4} \to \Uh^{\ho 3}, \ \fmu(x) = (\id^{\ho 2} \ho \boldmu) (x \cD_{13}^2 \cD_{14}^{-2}).
\label{eq.fmudef}
\end{align}

Similarly, using \eqref{e41a} instead of \eqref{e41}, we have
\be
\label{eq.deco1}
\modb_{\bn,\bm} =  \tfmu \circ \tfad_\bn  \circ \fbS  \circ \fuD \Big( \Ga_1(\bm)\ot \Ga_2(\bm)  \Big)  ,
\ee
where $\fuD, \fbS$ are as above, and
 \begin{align}
\tfad_\bn&: \Uh^{\ho 3} \to \Uh^{\ho 4}, \ \tfad_\bn (x)= \left[(\id^{\ho 3} \ho \uad^r ) \Big([ \Ga_1(\bn)  \ot x_1  \ot  x_2 \ot x_3 \ot \Ga_2(\bn) ]  \cD^{-2}_{24}\cD_{15}^{-2}    \Big)\right] \cD_{24}^{2}
\label{eq.tfaddef}\\
\tfmu&: \Uh^{\ho 4} \to \Uh^{\ho 3}, \ \fmu(x) = (\id^{\ho 2} \ho \boldmu) (x \cD_{23}^2 \cD_{24}^{-2}). \label{eq.tfmudef}
\end{align}

We will prove that each of $\fuD,\fbS,\fmu $ is $(\tK_n)$-admissible, while each of ${\fad_\bn}, {\tfad_\bn}$ maps $\tK_3$ to ${o(\bn)} \tK_4$.
From here Lemma \ref{r.Borrom} will follow easily.

\subsubsection{Extended adjoint action} To study the maps $\fuD,\fbS,{\fad_\bm}, {\tfad_\bn}$, we need the following extended adjoint action.
For $a\in \UUh= \Uh \ho_{\Ch}\Chh$ and $Y$-homogeneous  $x,y\in \UUh$ define
\begin{align}
 a \btri (y\ot x) &:= \left[ (\id \ot \ad_a)\left ( ( y\ot x) \cD^{2}\right)\right] \cD^{-2} \notag\\
 &= y K_{2|a_{(2)}|} \ot a_{(1)} x S(a_{(2)}).   \label{eq.eadj}
 \end{align}
It is easy to check that  $(a\ot x\ot y) \to a \btri (x \ot y)$ gives rise to an action of  $\UUH$ on $\UUH \ho \UUH$.

 \begin{lemma}\label{r.ediv} (a) Suppose $a,x,y \in \UZ^\ev$ are $G^\ev$-homogeneous, then
 \begin{align}
 a \btri (y\ot x) \in [\UZ^\ev\ot \UZ^\ev]_{\dy \ot \dot a \dx} \\
 S^{-1}a \btri (y\ot x) \in [\UZ^\ev \ot \UZ^\ev]_{\dy \ot  \dx \dot a}
 \end{align}

(b) One has $\Ur \btri (\Ur \ot \Ur) \subset \Ur \ot \Ur$.

(c) One has
 $$ \UZ \btri (\sXZe \ot \sXZe) \subset \sXZe \ot \sXZe.$$

\end{lemma}

\begin{proof} (a)
 The right hand side of \eqref{eq.eadj} shows that $a \btri (y\ot x)$ has $G^{\ot 2}$-grading equal to
$$\dy \ot \dot a_{(1)}  \dot x \dot S(\dot a_{(2)}) = \dy \ot \dot a_{(1)}  \dot x \dot a_{(2)} \dK_{|a_{(2)}|}= \dy \ot \dot a_{(1)}  \dot a_{(2)} \dK_{|a_{(2)}|}  \dot x=  \dy \ot \dot a \dot x,$$
where we use $\dot a_{(1)}  \dot a_{(2)} \dK_{|a_{(2)}|}=\dot a$ from \eqref{eq.evenD}.
This shows the first identity. The second one is proved similarly.

(b) By assumptions on $\Ur$,  $K_\al^{\pm 1}\in \Ur$ and $\Ur$ is a Hopf algebra. Hence, (b) follows from \eqref{eq.eadj}.

(b) Suppose $a\in \UZ, x, y\in \sXZe$, we need to show that  $a \btri (y\ot x) \in \sXZe \ot \sXZe$.
Because $ (ab)\btri (y\ot x) = a \btri ( b \btri (y \ot x))$, it is sufficient to consider the case when $a$ is one of the generator $E_\al^{(n)}, F_\al^{(n)}, K_\al^{\pm1}$ of $\UZ$, where $\al\in \Pi$ and $n \in \BN$. The cases $a= K_\al^{\pm1}$ are trivial.

For $a= E_\al^{(n)}$, a  calculation by induction on $n$ shows that
\begin{align*}
E_\al^{(n)} \btri (y \otimes x ) & =  \sum_{j=0}^n (-1)^n\, v_\al^{2jn + \binom {n+1}2} \,  \, y \, \left(K_{\al}^{2};q_\al\right)_{n-j} \otimes E_\al^{(n-j)} \,  \left ( E_\al^{(j)} \tri x \right) \\
& =  \sum_{j=0}^n (-1)^n\, v_\al^{2jn + \binom {n+1}2}  \left[ y \, \frac{\left(K_{\al}^{2};q_\al\right)_{n-j}}{ \sqrt{(q_\al;q_\al)_{n-j}}}\right] \otimes   \left[ \sqrt{(q_\al;q_\al)_{n-j}}\, E_\al^{(n-j)} \right]    \left[  E_\al^{(j)} \tri x \right]
\end{align*}
The right hand side belongs to $\sXZe \ot \sXZe$, since each factor in square brackets is in~$\sXZe$.

The case $a= F_\al^{(n)}$ can be handled by a similar calculation, or can be derived from the already proved case $a= \tphi(F_\al)= K_\al^{-1} E_\al$, using
$$ (\tphi\ot\tphi) \left( a \btri (y\ot x) \right) = \tphi(a) \btri (\tphi( y)\ot \tphi(x)),$$
which follows from the fact that $\tphi$ commutes with $S, \Delta$ and $\tphi(K_\al)= K_\al$.
\end{proof}

\subsubsection{The map $f^{\uD}$}
\begin{lemma}\label{r.fuD}
The map $\fuD: \Uh^{\ho 2} \to \Uh^{\ho 3}$ is $(\tK_n(\Ur))$-admissible.
\end{lemma}
\begin{proof}Using the definition \eqref{eq.fuDdef} and the decomposition \eqref{eq.decompD} of $\uD$
we have $\fuD= \sum_{\bu \in \BN^t} \fuD_\bu$, where
  $$ \fuD_\bu (y\ot x) = [\sum y\delta_1 \ot \uD_\bu(x \delta_2)] \cD_{12}^2 \cD_{13}^2.$$

We will show that  $\fuD= \sum_{\bu \in \BN^t} \fuD_\bu$ is an admissible decomposition. Using  definitions, we have
  \begin{gather*}
    \begin{split}
     \sum y\delta_1 \ot \uD_\bu(x \delta_2) &=\sum y\d_1\otimes E'_\modu \trr (x_{(2)}(\d_2)_{(2)})\otimes K_{\lambda_\modu +|x_{(2)}|}F'_\modu x_{(1)}(\d_2)_{(1)}\\
      &=\sum y\d_1\d_1'\otimes E'_\modu \trr (x_{(2)}\d_2)\otimes K_{\lambda_\modu +|x_{(2)}|}F'_\modu x_{(1)}\d_2'\\
      &=\sum y\d_1\d_1'\otimes (E'_\modu )_{(1)} x_{(2)}\d_2S((E'_\modu )_{(2)})\otimes K_{\lambda_\modu +|x_{(2)}|}F'_\modu x_{(1)}\d_2'\\
      &=\sum y K_{2|(E'_\modu )_{(2)}|}\d_1\d_1'\otimes (E'_\modu )_{(1)} x_{(2)}S((E'_\modu )_{(2)})\d_2\otimes K_{\lambda_\modu +|x_{(2)}|}F'_\modu x_{(1)}\d_2'\\
      &=\left(\sum yK_{2|(E'_\modu )_{(2)}|}\otimes (E'_\modu )_{(1)} x_{(2)}S((E'_\modu )_{(2)})\otimes K_{\lambda_\modu +|x_{(2)}|}F'_\modu x_{(1)}\right)\cD^{-2}_{12}\cD^{-2}_{13}\\
      &= \sum \left(v^{(|x_{(2)}|,\lambda_\bu)} E'_\bu \btri( y \ot x_{(2)}) \ot (K_\bu F'_\bu) K_{|x_{(2)}|} x_{(1)} \right)\cD^{-2}_{12}\cD^{-2}_{13}.
    \end{split}
  \end{gather*}
  This shows that
  \begin{align} \label{eq.ba1}
  \fuD_\bu (y\ot x)& = \sum  v^{(|x_{(2)}|,\lambda_\bu)} \Big( E'_\bu \btri( y \ot x_{(2)})\Big) \ot \Big(K_\bu F'_\bu\Big) \Big(K_{|x_{(2)}|} x_{(1)}\Big)
  \end{align}
(A) Suppose $x, y \in \UZ^\ev$ are $G^\ev$-homogeneous. By $G$-good presentation (see Section \ref{sec:coproduct}) and Lemma \ref{r.ediv}, all the factors in parentheses on the right hand side of \eqref{eq.ba1} are in $\UZe$.

  From \eqref{eq.ba1} and Lemma \ref{r.ediv}, $\fuD_\bu (y\ot x)\in [(\UZ^\ev)^{\ot 3}]_{g}$, where
  $$ g= \dv^{(|x_{(2)}|,\lambda_\bu)}\dy \, \de_{\lambda _\bu} \dx_{(2)} \de_{\lambda _\bu}^{-1} \dK_{|x_{(2)}|} \dx_{(1)} = \dy \,  \dx_{(2)}  \dK_{|x_{(2)}|} \dx_{(1)}=\dy \dx, $$
 with the  last equality obtained from \eqref{eq.evenD}. This shows $\fuD_\bu$ preserves the $G^\ev$-grading.

(B) Suppose $x,y\in \Ur$. By assumptions on $\Ur$, $K_\al \in \Ur$ and  $E'_\bu \ot K_\bu F'_\bu \in \Ur \ot \Ur$. Now Lemma~\ref{r.ediv} shows that the right hand side of \eqref{eq.ba1} is in $\Ur^{\ot 3}$. Thus, $\fuD_\bu(\Ur^{\ot 2} ) \subset \Ur^{\ot 3}$.

(C)  By \eqref{eq.bg2}, $ {E'_\bu} \ot K_\bu F'_\bu \in o(\bu)(\UZ^\ev \ot \sXZe)$.  Lemma \ref{r.ediv} and \eqref{eq.ba1} show that
  $$ \fuD_\bu \left( (\sXZe)^{\ot 2}\right) \in  o(\bu)   \, (\sXZe)^{\ot 3}.$$

 By Lemma \ref{r36}, $\fuD$ is $(\tK_n)$-admissible.
\end{proof}

\subsubsection{The map $\fbS$}
\begin{lemma} \label{r.fbS}
The map $\fbS: \Uh^{\ho 3} \to \Uh^{\ho 3}$ is $(\tK_n(\Ur))$-admissible.
\end{lemma}

\begin{proof}Using the definition \eqref{eq.fbSdef} and the decomposition \eqref{eq.decompS} of $\bS$,
 we have $\fbS= \sum_{\bu \in \BN^t} \fbS_\bu$, where
  $$ \fbS_\bu(y \ot x \ot z)= \left[\sum y\delta_1 \ot \bS_\bu(x\d_2) \ot z\right] \cD_{12}^{-2}= (y \ot 1 \ot z) (\sum \d_1\otimes\bS_\bu(x\d_2) \ot 1) \cD_{12}^{-2}.$$
  Using  the definitions, we have
  \begin{gather*}
    \begin{split}
      (1\otimes\bS_\bu)(\sum \d_1\otimes x\d_2)
      &=\sum \d_1\otimes\bS_\bu(x\d_2)\\
      &=\sum \d_1\otimes S^{-1}(E_\bu\trr (x\d_2))F_\bu K_{-|x|}\\
      &=\sum \d_1\otimes S^{-1}((E_\bu)_{(1)}
      (x\d_2)S((E_\bu)_{(2)}))F_\bu K_{-|x|}\\
      &=\sum
      \d_1\otimes(E_\bu)_{(2)}S^{-1}(\d_2)S^{-1}(x)S^{-1}((E_\bu)_{(1)})F_\bu K_{-|x|}\\
      &=\sum
      K_{2(|x|+|(E_\bu)_{(1)}|+|F_\bu|)}\d_1\otimes(E_\bu)_{(2)}S^{-1}(x)S^{-1}((E_\bu)_{(1)})F_\bu K_{-|x|}S^{-1}(\d_2)\\
      &=\left(\sum
      K_{2(|x|+|(E_\bu)_{(1)}|+|F_\bu|)}\otimes(E_\bu)_{(2)}S^{-1}(x)S^{-1}((E_\bu)_{(1)})F_\bu K_{-|x|}\right)\cD^2\\
      &=  \sum \big[  ( 1  \ot K_{-|x|} F_\bu) \,   (S^{-1} \ot S^{-1})\left(  E_\bu \btri (K_{ -2|x| }\ot x)     \right)\big] \cD^2.
    \end{split}
  \end{gather*}
  It follows that
  \begin{align} \label{eq.mj2}
  \fbS_\bu(y \ot x \ot z) &= \sum (y \ot 1 \ot 1) \Big( \big[  ( 1  \ot K_{-|x|} F_\bu) \,   (S^{-1} \ot S^{-1})\left(  E_\bu \btri (K_{ -2|x| }\ot x)     \right)\big] \ot z \Big).
  \end{align}

 Assume that $x,y,z \in \UZ^\ev$ are
  $G^\ev$-homogeneous.
By Lemma \ref{thm:4},
$$F_\bu \ot E_\bu \in [\UZ^\ev\ot \UZ^\ev]_{\de_{\lambda_\bu}^{-1} \dK_{\lambda_\bu} \ot \de_{\lambda_\bu}}.$$ Hence from Lemma \ref{r.ediv}(a),
$\fbS_\bu(y \ot x \ot z) \in
[(\UZ^\ev)^{\ot 3}]_g$,  where
$$g=    \dy \ot   \dK_{|x|} \de_{\lambda_\bu}^{-1} \dK_{\lambda_\bu}  \dS^{-1}(\de_{\lambda_\bu} \dx) \ot \dot z= \dy\ot  \dx \ot \dot z, $$
where the last equality follows from a simple calculation. Thus,
\be
\label{eq.mj5}
\fbS_\bu(y \ot x \ot z) \in
[(\UZ^\ev)^{\ot 3}]_{\dy\ot  \dx \ot \dot z}
\ee

(A) From \eqref{eq.mj5}, $\fbS_\bu$ preserves the $G^\ev$-grading.

(B) Assume that $x,y,z \in \Ur$.
Since $K_\al \in \Ur$, $F_\bu \ot E_\bu \in \Ur \ot \Ur$,
Lemma \ref{r.ediv}(b) shows that the right hand side of \eqref{eq.mj2} is in $\Ur^{\ot 3}$.

(C)   By \eqref{eq.bg2}, $F_\bu \ot E_\bu \in  o(\bu) \sXZ\ot \UZ$. Lemma \ref{r.ediv} and  \eqref{eq.mj2}
show that
 $$ \fbS_\bu((\sXZe)^{\ot 3}) \subset o(\bu) (\sXZ)^{\ot 3}.$$
On the other hand, \eqref{eq.mj5} shows that  $$\fbS_\bu((\sXZe)^{\ot 3}) \subset ((\UZ^\ev)^{\ot 3} \ot_\cA\tA).$$
 Because
 $$ o(\bu) (\sXZ)^{\ot 3} \cap ((\UZ^\ev)^{\ot 3} \ot_\cA\tA) =  o(\bu) (\sXZe)^{\ot 3}$$
by \eqref{eq.sosanh1}, we have  $\fbS_\bu((\sXZe)^{\ot 3}) \subset o(\bu) (\sXZe)^{\ot 3}$.
\end{proof}

\subsubsection{The maps $\fad$ and $\tfad$}
\begin{lemma} \label{r.fad}
For  $f=\fad_\bm$ or $f=\tfad_\bm$, one  has
$ f(\modK_3(\Ur)) \subset \cF_{\lfloor  \max\bm/2\rfloor}(\modK_4(\Ur)).$

\end{lemma}

\begin{proof} Assume $x\ot y \ot z\in \modK_3(\Ur)= (\sXZ^\ev)^{\ot 3}\cap [(\UZ^\ev)^{\ot 3}]_1 \cap \Ur^{\ot 3}$.
First assume $f=\fad_\bm$.
 Recall that
 \begin{align*}
 \fad_\bm (x_1 \ot x_2 \ot x_3) & = \left[\sum (\id^{\ot 2} \ot \ad \ot \id) \left( x_1 S(\d_1) \ot \Ga_1(\bm) \d_1'  \ot \Ga_2(\bm)\d'_2 \ot x_2 \d_2 \ot x_3     \right)\right] \cD_{13}^{-2}\\
  & = (x_1\ot \Ga_1(\bm) \ot 1 \ot x_3) \left[\sum (\id^{\ot 2} \ot \ad ) \left( S(\d_1) \ot  \d_1'  \ot \Ga_2(\bm)\d'_2 \ot x_2 \d_2 \right) \ot 1    \right] \cD_{13}^{-2}.
 \end{align*}

  We have
  \begin{gather*}
    \begin{split}
      (\id\otimes \id\otimes\ad)\left(\sum S(\d_1)\otimes \d_1'\otimes x\d_2'\otimes y\d_2\right)
      &=\sum S(\d_1)\otimes \d_1'\otimes (x\d_2')\trr  (y\d_2)\\
      &=\sum S(\d_1)\otimes \d_1'\otimes x_{(1)}(\d_2')_{(1)}y\d_2S((\d_2')_{(2)})S(x_{(2)})\\
      &=\sum S(\d_1)\otimes K_{-2|y|}\otimes x_{(1)}y\d_2S(x_{(2)})\\
      &=\sum K_{2|x_{(2)}|}S(\d_1)\otimes K_{-2|y|}\otimes x_{(1)}yS(x_{(2)})\d_2\\
      &=\left[\sum K_{2|x_{(2)}|}\otimes K_{-2|y|}\otimes x_{(1)}yS(x_{(2)})\right]\cD^2_{13}\\
      &= \left[( x \btri y)_{13} (1 \ot  K_{-2|y|}\otimes 1)    \right]\cD^2_{13}.
    \end{split}
  \end{gather*}
  It follows that
    \begin{align} \label{eq.mj6}
  \fad_\bm (x_1 \ot x_2 \ot x_3) &= \left[\big(x_1\ot \Ga_1(\bm) \ot 1 \big) \,\left ( \Ga_2(\bm) \btri x_2\right)_{13}\, \big (1 \ot  K_{-2|x_2|}\otimes 1\big)\right]  \ot x_3.
  \end{align}
Since  $\Ga_1(\bm) \ot \Ga_2(\bm) \in  [(\UZ^\ev)^{\ot 2}]_1$, Lemma \ref{r.ediv}(a) shows that
$$\fad_\bm (x_1 \ot x_2 \ot x_3)\in [(\UZe)^{\ot 4}]_g,$$
where $g= \dx_1 \dot \Gamma_1 (\bm) \dot \Gamma_2 (\bm) \dx_2 \dx_3 = \dx_1 \dx_2 \dx_3=1$. Thus, the right hand side of \eqref{eq.mj6} is in $[(\UZe)^{\ot 4}]_1$.

Since $x\ot y \ot z\in \Ur^{\ot 3}$, Lemma \eqref{r.ediv}(b) shows that the right hand side of \eqref{eq.mj6} is in $\Ur^{\ot 4}$.

Since $\Ga_1(\bm) \ot \Ga_2(\bm) \in {o(\bm)} \sXZe \ot \UZ^\ev$ by \eqref{eq.Ga2}, Lemma \eqref{r.ediv}(c) shows that
$$\fad_\bm (x_1 \ot x_2 \ot x_3) \in  {o(\bm)}(\sXZe)^{\ot 4}.$$
Hence
$$\fad_\bm (x_1 \ot x_2 \ot x_3) \in {o(\bm)}(\sXZe)^{\ot 4} \cap [(\UZe)^{\ot 4}]_1  \cap \Ur^{\ot 4} = \cF_{\lfloor  \max\bm/2\rfloor}(\modK_4(\Ur)),$$
which proves the statement for $f= \fad_\bm$.

The proof for $f=\tfad_\bm$ is similar: Using the definition and Formula \eqref{e40} for $\underline{\ad^r}$, one gets
 \be \label{eq.ks1}
 \tfad_\bm(x_1 \ot x_2 \ot x_3)=  \Big( \Gamma_1(\bm) \ot x_1 \ot x_2 \ot 1\Big) \Big( K_{2|x_3| + 2 \lambda_\bm}
\ot (S \ot \id)( S^{-1}(\Ga_2(\bm) \btri x_3)) \Big)_{24}.\ee
By Lemma \ref{r.ediv}(a), the right hand side of \eqref{eq.ks1} is in $(\UZ^\ev)^{\ot 4}$ having $G$-grading equal to
$$ \dot \Ga_1(\bm)\,  \dx_1 \dx_2 \dx_3  \dot \Ga_2(\bm) = \dot \Ga_1(\bm)  \dot \Ga_2(\bm)=1.$$
Again, Lemma \eqref{r.ediv}(b) shows that the right hand side of \eqref{eq.ks1} is in $\Ur^{\ot 4}$, and
Lemma \eqref{r.ediv}(c) shows that it is in ${o(\bm)}(\sXZe)^{\ot 4}$. Hence $\tfad_\bm(x_1 \ot x_2 \ot x_3) \in
\cF_{\lfloor  \max\bm/2\rfloor}\modK_4(\Ur)$.
\end{proof}
\subsubsection{The maps $\fmu$ and $\tfmu$}
\begin{lemma}\label{r.fmu}
Both  $\fmu$ and $\tfmu$ are $(\tK_n(\Ur))$-admissible.
\end{lemma}
\begin{proof} By definition
\begin{align*}
\fmu(x_1 \ot x_2 \ot x_3 \ot x_4)& = (\id^{\ot 2} \ot \boldmu)(\sum x_1 \d_1S(\d_1')\otimes x_2 \ot x_3 \d'_2 \ot x_4 \d_2) \\
&= (x_1 \ot x_2 \ot 1)\, \left[ (\id^{\ot 2} \ot \boldmu)(\sum \d_1S(\d_1')\otimes 1 \ot x_3 \d'_2 \ot x_4 \d_2)\right].
\end{align*}
  We have
  \begin{gather}
    \begin{split}
      (\id\otimes{\boldsymbol \mu})(\sum \d_1S(\d_1')\otimes x\d_2'\otimes y\d_2)
      &=\sum \d_1S(\d_1')\otimes x\d_2'y\d_2\\
      &=\sum \d_1S(\d_1')K_{2|y|}\otimes xy\d_2'\d_2\\
      &= K_{2|y|}\otimes xy.
    \end{split}
    \notag
  \end{gather}
 It follows that $\fmu$ has a very simple expression
 \begin{align*}
 \fmu(x_1 \ot x_2 \ot x_3 \ot x_4)& =  (x_1 \ot x_2 \ot 1) ( K_{2|x_4|}\otimes 1 \ot x_3 x_4)\\&
 = x_1 K_{2|x_4|}\otimes   x_2 \ot x_3 x_4.
 \end{align*}
 The trivial decomposition for $\fmu$ is admissible. Hence,  $\fmu$ is $(\tK_n(\Ur))$-admissible.

 Similarly, a simple computation shows that
 $$ \tfmu(x_1 \ot x_2 \ot x_3 \ot x_4)=  x_1  \ot x_2 K_{2|x_4|} \ot x_3 x_4. $$
 The trivial decomposition for $\tfmu$ is admissible. Hence,  $\tfmu$ is $(\tK_n(\Ur))$-admissible.
\end{proof}

\subsubsection{Proof of Lemma \ref{r.Borrom}}
\begin{proof} First suppose $\max(\bm) \ge \max(\bn)$. By \eqref{eq.deco}
$$
\modb_{\bn,\bm} =  \fmu \circ \fad_\bm \circ \fbS  \circ  \fuD \Big( \Ga_1(\bn)\ot \Ga_2(\bn)  \Big)  ,
$$
By \eqref{eq.Ga1}, $\Ga_1(\bn)\ot \Ga_2(\bn)\in \modK_2$. Lemmas \ref{r.fuD}, \ref{r.fbS}, \ref{r.fad}, and \ref{r.fmu} show that  $\modb_{\bn,\bm} \in o(\bm) \tK_3$.

Suppose  $\max(\bn) > \max(\bm)$. Using \eqref{eq.deco1} instead of  \eqref{eq.deco},
 we have $\modb_{\bn,\bm} \in o(\bn) \tK_3$.

Hence $ \modb_{\bn,\bm} \in o(\bn,\bm) \tK_3$. As a consequence, $\modb = \sum \modb_{\bn,\bm} \in \tK_3$.
\end{proof}

\subsection{Proof of Proposition \ref{r.Jvalues1}} As noted, statement (i) follows trivially from the definition of $\tK(\Ur)$. Statement (ii) follows from Lemmas \ref{r.mu}, \ref{r.psi}, \ref{r.uD}, \ref{r.bS}.
Finally, statement (iii) is Lemma \ref{r.Borrom}.
 \qed

\subsection{Integrality of the quantum link invariant}
\label{integrality2}

In \cite{Le_Duke}, the second author proved that, for a framed link $L=L_1\cup\dots\cup L_n$ in $S^3$, the quantum $\mathfrak g$ link invariant $J_L(V_{\lambda_1},\dots,V_{\lambda_n})$, up to multiplication by a fractional power of $q$, is contained in $\BZ[q,q^{-1}]$.  Here we sketch an alternative proof using Theorem \ref{r.Jvalues} of the following special case for algebraically split framed links.

\begin{theorem}[\cite{Le_Duke}]
  \label{r1}
  Let $L=L_1\cup\dots\cup L_n$ be an algebraically split $0$-framed link in $S^3$.  Let $\lambda_1,\dots,\lambda_n\in X_+$ be dominant integral weights.
  Then we have
  \begin{gather*}
    J_L(V_{\lambda_1},\dots,V_{\lambda_n})\in q^p\BZ[q,q^{-1}].
  \end{gather*}
  where $p=(2\rho,\lambda_1+\dots+\lambda_n)$.
\end{theorem}

It is much easier to prove
\begin{gather}
  \label{eee1}
  J_L(V_{\lambda_1},\dots,V_{\lambda_n})\in q^p\BZ[v,v^{-1}],
\end{gather}
and the difficult part of the proof is to show that the normalized invariant
$q^{-p}J_L(V_{\lambda_1},\dots,V_{\lambda_n})\in\BZ[v,v^{-1}]$ is contained in
$\BZ[q,q^{-1}]$.  In \cite{Le_Duke}, a result of Andersen \cite{An} on quantum
groups at roots of unity is involved in the proof.  The main idea of the proof below is implicitly the use of
the $G$-grading of the quantum group $\Uq$ as $\BC(q)$-module described in
Section \ref{sec:nonc-grad}.

\begin{proof}[Sketch proof of Theorem \ref{r1}]
  Let $T$ be an algebraically split $0$-framed bottom tangle such that the
  closure link of $T$ is $L$.  Recall that the quantum invariant
  $J_L(V_{\lambda_1},\dots,V_{\lambda_n})$ can be defined by using quantum
  traces
  \begin{gather}
    \label{e2}
    J_L(V_{\lambda_1},\dots,V_{\lambda_n})=(\tr_q^{V_{\lambda_1}}\ot\dots\ot\tr_q^{V_{\lambda_n}})(J_T).
  \end{gather}

  It is not difficult to prove that
  for
$1\le i\le n$, $\lambda\in X_+$, we have
  \begin{gather}
    \label{eee2}
    (\id^{\ot i-1}\ot \tr_q^{V_\lambda}\ot \id^{\ot n-i})(\tK_n)\subset q^{(2\rho,\lambda)}\tK_{n-1}.
  \end{gather}
  Using \eqref{eee2}, one can prove that
  \begin{gather}
    \label{eee3}
    (\tr_q^{V_{\lambda_1}}\ot\dots\ot \tr_q^{V_{\lambda_n}})(\tK_n)\subset
    q^p\tK_{0} =q^p\Zqh.
  \end{gather}
  Hence, using \eqref{e2}, \eqref{eee3} and Theorem \ref{r.Jvalues}(a), we have
  \begin{gather*}
      J_L(V_{\lambda_1},\dots,V_{\lambda_n})
      \in(\tr_q^{V_{\lambda_1}}\ot\dots\ot\tr_q^{V_{\lambda_n}})(\tK_n)
      \subset q^p\Zqh,
  \end{gather*}
  which, combined with \eqref{eee1}, yields
  $J_L(V_{\lambda_1},\dots,V_{\lambda_n})\in q^p\BZ[q,q^{-1}]$ since we have
  $\BZ[v,v^{-1}]\cap\Zqh=\BZ[q,q^{-1}]$.
\end{proof}

\np

\renewcommand{\cD}{D}
\def\cD{D}

\section{Recovering the Witten-Reshetikhin-Turaev invariant}
\label{sec.WRT}
In Section \ref{sec:integralM} we showed that $J_M\in \Zqh$, where $J_M$ is the invariant (associated to a simple Lie algebra $\fg$) of an integral homology 3-sphere $M$. Hence we can evaluate $J_M$ at any root of unity. Here we show that by evaluating of $J_M$ at a root of unity we recover the Witten-Reshetikhin-Turaev invariant. We also prove Theorem \ref{r40a} and Proposition \ref{r10} of Introduction.

\subsection{Introduction}\label{sec.WRT01}

Recall that $\fg$ is a simple Lie algebra and
$\ZZ$ is the set of all roots of unity. Suppose $\zeta \in \ZZ$ and
$M$ is closed oriented $3$-manifold.
Traditionally the Witten-Reshetikhin-Turaev (WRT) invariant (cf. \cite{RT2,BK}) $\tau_M^\fg(\xi;\zeta)\in \BC$ is defined when $\zeta$ is a root  of unity of order
  $2\cD dk $ with $k> h^\vee$, where $h^\vee$ is the dual Coxeter number, $d\in\{1,2,3\}$  is defined as in Section \ref{sec.UhUq}, and $\cD=|X/Y|$. Here $\xi= \zeta^{2D}$.
  In this case, $k-h^\vee$ is called the level of the theory. The definition of  $\tau_M^\fg(\xi;\zeta)$ can be extended to a bigger set $\ZZ'_\fg$ which is more than all roots
  of unity of order divisible by $2d\cD$, see subsection \ref{WRT}. For values of $d,D,h^\vee$ of simple Lie algebras, see Table \ref{tab:1} in Section \ref{sec.UhUq}.

This section is devoted to the proofs of the following theorem and its generalizations.
\begin{theorem}
Suppose $M$ is an integral homology $3$-sphere and $\zeta\in \ZZ'_\fg$. Then
$$ \tau_M^\fg(\xi;\zeta) = J_M \big|_{q= \xi}.$$
\label{thm030}
\end{theorem}

\begin{remark}  (a) Although $\xi$ is determined by $\zeta$, we use the notation $\tau_M(\xi;\zeta)$ since in many cases, $\tau_M^\fg(\xi;\zeta)$  depends only on $\xi$, but not a $2D$-th root $\zeta$ of $\xi$. In that case, we write $\tau_M(\xi)$ instead of $\tau_M(\xi;\zeta)$. The set $\ZZ_\fg$ in Section \ref{sec:intro} is defined by $\ZZ_\fg =\{\zeta^{2D} \mid \zeta \in \ZZ'_\fg\}$.

  (b)
 The theorem implies that for an {\em integral homology $3$-sphere}, $\tau_M^\fg(\xi;\zeta)$  depends only on $\xi$, but not a $2D$-th root $\zeta$ of $\xi$. This does not hold true for general $3$-manifolds.
\end{remark}
In subsections \ref{WRT1} and \ref{WRT} we recall the definition of the WRT invariant and define the set $\ZZ'_\fg$.
Subsection \ref{proof01} contains the proof of a stronger version of Theorem \ref{thm030}, based on results proved in later subsections. To prove the main results we introduce an integral form $\Ur$ of $\Uq$ which is sandwiched between Lusztig's integral form $\UZ$ and De Concini-Procesi's integral form $\VZ$. For  $\fg = sl_2$, the algebra
$\cU$ was considered by the first author
\cite{H:integralform,H:unified}.
A large part of the proof is devoted to the determination of the center of a certain completion of $\cU$.
For this part we use, among other things, integral  bases of $\UZ$-modules, the quantum Harish-Chandra isomorphism, and Chevalley's theorem in invariant theory.
In Section \ref{sec:Drin}, we give a geometric interpretation of Drinfel'd's construction of central elements.

\subsection{Finite-rank $\Uh$-modules} \label{sec:UhMod}
Suppose $V$ is a topologically free $\Uh$-module.
For $\mu \in X$ the {\em  weight $\mu$ subspace} of $V$ is defined by
$$ V_{[\mu]} = \{ e \in V \mid H_\al(e) = (\al,\mu) e \quad \forall \al \in \Pi \},$$
and $\mu\in X$ is called a {\em weight of $V$} if $V_{[\mu]}\neq 0$.
We call  $V$ a {\em highest weight module} if $V$ is generated by a non-zero
element $1_\mu\in V_{[\mu]}$ for some $\mu\in X$ such that $E_\alpha 1_\mu=0$ for $\alpha\in\Pi$.  Then $1_\mu$ is called a {\em highest
  weight vector} of $V$, and $\mu$ the {\em highest weight}.

By {\em a finite-rank
   $\Uh$-module}, we mean a $\Uh$-module which is
   (topologically) free of finite rank as a $\mathbb{C}[[h]]$-module.
The theory of finite-rank $\Uh$-modules is well-known and is parallel to
that of finite-dimensional $\fg$-modules, see e.g. \cite{CP,Jantzen,Lusztig}:
Every finite-rank  $\Uh$-module is the direct sum of irreducible finite-rank $\Uh$-modules. For every dominant integral weight $\lambda \in X_+:= \{ \sum _{\al \in \Pi} k_\al \bral \mid k_\al \in \BN\}$, there exists a unique finite-rank irreducible $\Uh$-module with highest weight $\lambda$, and every finite-rank irreducible $\Uh$-module is one of $V_\lambda$. The  Grothendieck ring of finite-rank  $\Uh$-modules is naturally isomorphic to that of  finite-dimensional $\fg$-modules.

\subsection{Link invariants and symmetries at roots of unity} \label{WRT1}
\subsubsection{Invariants of colored links}

Suppose $L$ is the closure link of  a framed bottom tangle $T$, with  $m$  components.  Let $V_1,\ldots,V_m$ be  finite-rank $\Uh$-modules.
 Recall that the  quantum link invariant \cite{RT1} can be defined by
 $$ J_L(V_1,\ldots,V_m) = (\tr_q^{V_1}\otimes\cdots \otimes\tr_q^{V_m})(J_T) \in \BC[[h]].$$
  Actually, $J_L(V_1,\ldots,V_m)$ belongs to a subring $\BZ[v^{\pm 1/\cD}]$ of $\BC[[h]]$, where $\cD=|X/Y|$,
 see \cite{Le_Duke}.  ($\cD$ is also equal to the determinant of the Cartan matrix.) We say that $V_j$ is the color of the $j$-th component, and consider $J_L(V_1,\ldots,V_m)$ as an invariant of colored links, which  is a generalization of the famous Jones polynomial \cite{Jones}.

Let $U$ be the trivial knot with
0 framing. For a finite-rank $\Uh$-modules $V$, $\dim_q(V):=J_U(V)$ is called  as the {\em quantum dimension} of $V$.
It is known that for $\lambda\in X_+$,
\be
 \dim_q(V_\lambda)= \frac{\sum_{w\in \fW} \sgn(w) v^{-(2(\lambda+\rho), w(\rho))}}{\sum_{w\in \fW} \sgn(w) v^{-(2(\rho), w(\rho))}}=
 q^{-(\lambda,\rho)}\prod_{\al \in \Phi_+}\frac{q^{(\lambda+ \rho, \al)}-1}{q^{(\rho, \al)}-1}.
\label{eq.unknot}
\ee
Here $\fW$ is the Weyl group and $\sgn(w)$ is the sign of $w$ as a linear transformation.

One has $\max_{\al\in \Phi_+} (\rho, \al) = d(h^\vee-1)$, where $h^\vee$ is the dual Coxeter number of $\fg$. Hence, if $\xi$ is root of unity with
\be
\ord(\xi) > d(h^\vee-1), \label{eq.epower}
\ee
then the denominator
of the right hand side of \eqref{eq.unknot} is not $0$ under the evaluation $q= \xi$. For this reason we often make the assumption \eqref{eq.epower}.

\subsubsection{Evaluation at a root of unity}
Throughout we  fix  a  root  of unity $\zeta\in \BC$. Let $\xi=\zeta^{2 \cD}$ and $r=\ord(\zeta^{2\cD})$.

For $f\in \BC[v^{\pm1/\cD}]$ let $ \ev_{v^{1/\cD}=\zeta}(f)$ be the value of $f$ at $v^{1/\cD}=\zeta$.
Note that if $v^{1/2\cD}=\zeta$, then $q= \xi$. If $f\in \BC[q^{\pm 1}]$, then $ \ev_{v^{1/\cD}=\zeta}(f)$ is the value of $f$ at $q=\xi$.

Suppose  $f,g \in\BC[v^{\pm 1/\cD}]$. If
$ \ev_{v^{1/\cD}=\zeta}(f)= \ev_{v^{1/\cD}=\zeta}(g)$, then
 we say $f=g$ at $\zeta$ and write
 $$f \eqz g.$$

We say that $\mu \in X$ is a {\em $\zeta$-period} if for every link $L$,  $\ez(J_L)$ does not change when the color of
a component changes from $V_\lambda$ to $V_{\lambda+ \mu}$ for arbitrary $\lambda \in X_+$ such that $\lambda + \mu \in \X_+$ (the colors of other components remain unchanged).

The set of all $\zeta$-periods is a subgroup of $X$. It turns out that if  $\ord(\xi) > d(h^\vee-1)$, then the group of $\zeta$-periods has finite index in $X$: in \cite{Le_Duke} it was proved that
the group of $\zeta$-periods contains $2rY$, which, in turn, contains  $(2r\cD)X$ (because $\cD X \subset Y$).

When $\ord(\xi)\le d(h^\vee-1)$, the behavior of $\ez(J_L)$ is quite different. For example, when $\zeta=1$, from \eqref{eq.unknot}  and the Weyl dimension formula, one can see that $\dim_q(V_\lambda)$ is the dimension
of the classical $\fg$-module of highest weight $\lambda$. When $\zeta=1$, the action of the ribbon element on any $V_\lambda$ is the identity, and the braiding action ${\boldsymbol{\psi}}$ is trivial on any pair
of $\Uq$-modules. Hence, we have the following.

\begin{proposition}
For any framed oriented link $L$ with $m$ ordered components and $\mu_1,\ldots, \mu_m \in X_+$,
$$ \ev_{v^{1/\cD} =1}(J_L(V_{\mu_1}, \ldots, V_{\mu_m}))= \prod_{j=1}^m \dim (V_{\mu_j}).$$
Here $\dim (V_{\mu_j})$ is the dimension of the irreducible $\fg$-module  with highest weight $\mu_j$.
\label{pro.color1}
\end{proposition}

\subsection{The WRT invariant of 3-manifolds}
\label{WRT}
Here we recall the definition of the WRT invariant.
\subsubsection{$3$-manifolds and Kirby moves} Suppose $L$ is a framed link in the standard $3$-sphere $S^3$. Surgery along $L$ yields an oriented $3$-manifold $M = M(L)$. Surgeries along
two framed links $L$ and $L'$ give the same $3$-manifold if and only $L$ and $L'$ are related by a finite sequence of Kirby moves:
handle slide move and stabilization move, see e.g. \cite{Kirby,KM}.
If one can find an invariant of unoriented framed links which is invariant under the two Kirby moves, then the link invariant descends to an invariant of $3$-manifolds.

\subsubsection{Kirby color} \label{sec.Kirby}
Let $\cB:=\cA\otimes_\BZ \BC= \BC[v^{\pm 1}]$. We call any $\cB$-linear combination of $V_\lambda, \lambda \in X_+$,  {\em a color}.
By linear extension we can define $J_L(V_1,\ldots,V_m)\in \BC[v^{\pm 1/\cD}]$ when each $V_j$ is a color.

A color $\Omega$ is called a {\em handle-slide color at level $v^{1/\cD}=\zeta$} if

(i)
$\ez \left( J_L(\Omega, \ldots, \Omega)\right) $ is an invariant of non-oriented links, and

(ii)  $\ez \left( J_L(\Omega, \ldots, \Omega)\right) $ is invariant under the handle slide move.

Let  $U_\pm$ be the unknot  with framing $\pm 1$. A handle-slide color is called {\em a Kirby color} (at level $v^{1/\cD}=\zeta$) if it satisfies the {\em non-degeneracy condition}
\be J_{U_\pm}(\Omega) \neqz 0.
\label{em2}
\ee
Suppose $\Omega$ is a Kirby color at level $v^{1/\cD}=\zeta$, and $M= M(L)$ is the $3$-manifold obtained by surgery on $S^3$ along a framed link $L$. Then
\be
\tau_{M}(\Omega) := \ev_\zeta \left( \frac{J_L(\Omega,\ldots, \Omega)}{ (J_{U_\pm}(\Omega)) ^{\sigma_+} \, (J_{U_\pm}(\Omega)) ^{\sigma_-}}\right)
\label{n4011}
\ee
 is invariant under both Kirby moves, and hence defines an invariant of $M$. Here $\sigma_+$ (resp.  $\sigma_-$) is the number of positive (resp. negative) eigenvalues of the linking matrix of $L$.

\subsubsection{Strong Kirby color} All the known Kirby colors satisfy a stronger condition on the invariance under the handle slide move as described below.

A {\em root color} is any $\cB$-linear combinations of $V_\lambda$ with $\lambda\in Y\cap X_+$.
A handle-slide color $\Omega$ at level $v^{1/\cD}=\zeta$ is a {\em strong handle-slide color} if it satisfies the following: Suppose the first component of $L_1$ is
colored by $\Omega$ and other components are  colored by {\em arbitrary root colors} $V_1, \ldots, V_m$. Then a handle slide of any other component over the first component
does not change the value of the quantum link invariant, evaluated at $v^{1/D}=\zeta$, i.e. if $L_2$ is the resulting link after the handle slide, then
\be  J_{L_1} (\Omega, V_1,\ldots, V_m)\eqz  J_{L_2} (\Omega, V_1,\ldots, V_m)
\label{em1}.
\ee
A non-degenerate strong handle-slide color is called a {\em strong Kirby color}.

\subsubsection{Strong Kirby color exists}\label{sec:st.Kirby}

Let $P_\zeta$ be the following half-open parallelepiped, which is a domain of translations of  $X$ by elements of the lattice $(2r\cD)X$,
$$ P_\zeta := \left\{ \lambda =\sum_{i=1}^\ell k_i \bral_i \in X_+ \mid    0\le k_i < 2r\cD \right \}.$$
Let

$$\Omega^\fg(\zeta):= \sum _{\lambda \in P_\zeta}\dim_q(V_\lambda) \, V_\lambda, \quad \Omega^{P\fg}(\zeta)= \sum _{\lambda \in P_\zeta\cap Y}\dim_q(V_\lambda) \, V_\lambda.$$

In \cite{Le:quantum}, it was proved that both $\Omega^\fg(\zeta)$ and $\Omega^{P\fg}(\zeta)$ are handle-slide colors at level $v^{1/\cD}=\zeta$ if $\ord(\zeta^{2\cD}) > d(h^\vee-1)$. Actually, the proof there shows that $\Omega^\fg(\zeta)$ and $\Omega^{P\fg}(\zeta)$ are strong handle-slide colors at level $v^{1/\cD}=\zeta$. Hence, assuming $\ord(\zeta^{2\cD}) > d(h^\vee-1)$, $\Omega^\fg(\zeta)$ (resp. $\Omega^{P\fg}(\zeta)$) is a strong Kirby color at $v^{1/\cD}=\zeta$ if and only
 $\Omega^\fg(\zeta)$ (resp. $\Omega^{P\fg}(\zeta)$) is non-degenerate at $v^{1/\cD}=\zeta$. There are many cases of $v^{1/\cD}=\zeta$ when both $\ord(\zeta^{2 \cD})$ and  $ \Omega^\fg(\zeta)$ are strong Kirby colors, and there are many cases when one of them is not. Let $\ZZ'_\fg$ (resp. $\ZZ'_{P\fg}$) be the set of all roots of unity $\zeta$ such that $\Omega^\fg(\zeta)$ (resp. $\Omega^{P\fg}(\zeta)$) is a strong Kirby color.

For $\zeta\in \ZZ'_\fg$ the $\fg$ WRT invariant of an oriented closed 3-manifold $M$  is defined by
$$ \tau_M^\fg(\xi;\zeta) = \tau_M(\Omega^\fg(\zeta)).$$
Similarly, for $\zeta\in \ZZ'_{P\fg}$ the $P\fg$ WRT invariant of an oriented closed 3-manifold $M$ is defined by
$$ \tau_M^{P\fg}(\xi;\zeta) = \tau_M(\Omega^{P\fg}(\zeta)).$$

\begin{proposition} Suppose $\zeta$     %
is a root of unity
with  $\ord(\zeta^{2\cD}) > d(h^\vee -1)$. Then $\zeta \in \ZZ'_\fg \cup \ZZ'_{P\fg}$. More specifically,
if $\ord(\zeta^{2 \cD})$ is odd then  $\zeta \in \ZZ'_{P\fg}$ and if  $\ord(\zeta^{2 \cD})$ is even then  $\zeta \in \ZZ'_{\fg}$.
\label{pKirby}
\end{proposition}

 We will give a  proof of the proposition in Appendix \ref{sec:ppKirby}. Actually, in Appendix we will describe precisely the sets  $\ZZ'_\fg$ and $\ZZ'_{P\fg}$
(for $\ord(\zeta^{2\cD}) > d(h^\vee -1)$).

The  proposition shows that  $\ZZ'_\fg \cup \ZZ'_{P\fg}$ is all $\ZZ$ except for a finite number of elements. This means,  $\tau_M^\fg(\xi;\zeta)$ or $\tau_M^{P\fg}(\xi;\zeta)$ can always be defined except for a finite number of $\zeta$.

\begin{remark} (1) If  $\ord(\zeta)$ is divisible by $2d\cD$, the proposition had been well-known, since in this case a modular category, and hence a Topological Quantum Field Theory (TQFT), can be constructed, see e.g. \cite{BK}. The rigorous construction of the WRT invariant and the corresponding TQFT was first
given by Reshetikhin and Turaev  \cite{RT2} for $\fg=sl_2$.
The construction of TQFT for higher rank Lie algebras (see e.g. \cite{BK,Turaev}) uses Andersen's theory of tilting modules \cite{AP}. In \cite{Le:quantum}, the WRT invariant was constructed without
 TQFT (and no tilting modules theory).
 Here we are  interested only in the invariants of $3$-manifolds, but not the stronger structure -- TQFT. We don't know if a modular category -- the basis ground of a TQFT -- can be constructed for every root $\zeta$ of unity with $\ord(\zeta^{2\cD}) > d(h^\vee -1)$. At least for $\fg=sl_{n}$, if the order of $\zeta$ is $2\pmod 4$ and $n$ is even, then according to \cite{Bruguieres}, the corresponding pre-modular category is not modularizable.

(2)  In general, different strong Kirby colors  give different $3$-manifold invariants. The invariant corresponding to $\Omega^{P\fg}$, called the projective version of the WRT invariant, was first defined in \cite{KM} for $\fg=sl_2$,
then in \cite{KT} for $sl_n$, and then in \cite{Le:quantum} for general Lie algebras.
When both $\Omega^\fg(\zeta)$ and $\Omega^{P\fg}(\zeta)$ are non-degenerate, the relation between the two invariants $\tau_M(\Omega^\fg)$ and $\tau_M(\Omega^{P\fg})$ is simple if
$\ord(\zeta^{2\cD})$ is co-prime with $d\cD$, but in general  the relation is more complicated, see \cite{Le:quantum}.

(3) It is clear that in the definition of $\Omega^\fg(\zeta)$ and $\Omega^{P\fg}(\zeta)$, instead of $P_\zeta$ one can take any fundamental domain of any group of $\zeta$-periods which has
finite index in $Y$.

\end{remark}

\subsubsection{Dependence on $\xi=\zeta^{2\cD}$}

When components of a framed link $L$ are colored by $\Omega^{P\fg}(\zeta)$, $J_L$ takes values in $\BC[q^{\pm1}] \subset \BC[q^{\pm1/2\cD}]$, see \cite{Le:quantum}.
Hence, the $P\fg$ WRT invariant $\tau_M^{P\fg}(\xi;\zeta)$, if defined,  depend only on $\xi= \zeta^{2\cD}$, but not on any choice of a $2D$-th root $\zeta$ of $\xi$.

The $\fg$ WRT invariant $\tau_M^{\fg}(\xi;\zeta)$ does depend on a choice of a $2D$-th root $\zeta$ of $\xi$, even in the case $\fg=sl_2$.
We will see that when $M$ is an integral homology $3$-sphere, the $\fg$ WRT invariant of $M$  depends only on $\xi= \zeta^{2\cD}$, but not on any choice of a $2D$-th root $\zeta$ of $\xi$.
However, there are cases when $\zeta^{2\cD}=\xi= (\zeta')^{2 \cD}$, but $\zeta\in \ZZ'_\fg$ and $\zeta' \not \in \ZZ'_\fg$. For example, suppose $\fg= sl_2$ and $\xi= \exp(2pi/(2k+1))$, a root of unity of odd order.
Then $\zeta = \exp(2pi/(8k+4))$ and $\zeta'= i \zeta$ are both $4$-th roots of $\xi$ (in this case $2 \cD=4$). But $\zeta \in \ZZ'_\fg$ and $\zeta' \not \in \ZZ'_\fg$.

\subsubsection{Trivial color at  $\zeta=1$ and the case when $\ord(\zeta) \le d(h^\vee -1)$}
\begin{proposition} Let $\Omega=\Ch$ be the trivial $\Uh$-module. Then $\Omega$
 is a strong Kirby color at level $\zeta =1$ and $\tau_M(\Omega)=1$.
\label{p.color1a}
\end{proposition}
This follows immediately from Proposition \ref{pro.color1} and the defining formula \eqref{n4011} of $\tau_M(\Omega)$.

It is not true that the trivial color  is a strong Kirby color for all $\zeta$ with $\ord(\zeta^{2\cD}) \le d(h^\vee -1)$. For example, if $\fg= sl_6$ and $\ord(\zeta^{2\cD})=4$, then  the trivial color is not a strong Kirby color.
One can prove that if $n\equiv 0, \pm 1 \pmod r$, then  the trivial color is a strong Kirby color for $sl_n$ at level $\zeta$ with $r= \ord(\zeta^{2\cD})$.
\begin{remark}
If $\ord(\zeta)= 2 d\cD k$, then the level of the corresponding TQFT is $k-h^\vee$. Hence, if the level is non-negative as assumed by physics, we automatically have $\ord(\zeta^{2\cD}) > d(h^\vee -1)$.
\end{remark}

\subsection{Stronger version of Theorem \ref{thm030}} Proposition \ref{pKirby} shows that strong Kirby colors exist at every level $\zeta$, if the order of $\zeta$ is big  enough.
Although different Kirby colors at level $\zeta$
might define different $3$-manifold invariants, we have the following result for {\em integral homology $3$-spheres}, which is more general than Theorem \ref{thm030}.

\begin{theorem} Suppose $\Omega$ is a strong Kirby color at level $v^{1/\cD}=\zeta$ and  $M$ is an integral homology $3$-sphere. Then
\begin{gather*}
  \label{r31}
  \tau_{M}(\Omega)\  = \ \ez(J_M)\ =\  \ev_{q=\xi}(J_M).
\end{gather*}
\label{thm03}
\end{theorem}

\begin{remark} There is no restriction on the order of $\zeta$ on the
  right hand side of \ref{r31}. We do not know how to directly define the WRT invariant with  $\ord(\zeta^{2\cD }) \le d(h^\vee -1)$.
\end{remark}

The remaining part of this section is devoted to a proof of this theorem. Throughout we fix a root of unity $\zeta$
and a strong Kirby color $\Omega$ at level $\zeta$. Let $\xi= \zeta^{2D}$ and $r=\ord(\xi)$.

\newcommand{\ff}{{\mathfrak f}}

\subsection{Reduction of Theorem \ref{thm03} to Proposition \ref{p.tech1}}\label{proof01}
 Here we reduce  Theorem \ref{thm03} to
 Proposition  \ref{p.tech1}, which will be proved later.
\subsubsection{Twisted colors $\Omega_\pm$}
Suppose the $j$-th component of a link $L$ is colored by $V= V_\lambda$,
 and $L'$ is obtained from $L$ by increasing the framing of the $j$-th component by 1, then it is known that
 \be
  J_{L'}(\ldots, V, \ldots) =   \ff_\lambda \,  J_L (\ldots, V, \ldots), \quad \text{where\ } \quad \ff_\lambda = q^{(\lambda,\lambda+ 2\rho)/2}= \frac{ \tr_q^V(\br^{-1})}{\dim_q V}.
  \label{n40101}
  \ee
For example, if $U_\pm$ is the unknot with framing $\pm 1$, then
$$J_{U_\pm}(V_\lambda) = \ff_\lambda^{\pm1} \, \dim_q( V_\lambda) = J_U\left (\ff_\lambda^{\pm1} \,  V_\lambda  \right)\,.$$
By definition $\Omega$ is a finite sum
$ \Omega= \sum c_\lambda V_\lambda$,
where $c_\lambda \in \cB=\BC[v^{\pm1}]$.
Define the pair $\Omega_\pm$ by
$$ \Omega_{\pm}= \sum \frac{   \ez\left( c_\lambda \ff_\lambda^{\pm 1}\right) } { \ez\left( J_{U_\pm }(\Omega)   \right) } \,  V_\lambda ,$$
which are $\BC$-linear combinations of finite-rank  irreducible $\Uh$-modules.

 Suppose a distinguished component of $L$ has framing $\ve=\pm 1$ and color $\Omega$, and $L'$ is the same link with the
distinguished component having framing $0$ and color $\Omega_\ve$. Then from \eqref{n40101} and the definition of $\Omega_\ve$ one has
\be
 J_L( \ldots, \Omega, \ldots) \eqz J_{U_\ve}(\Omega) \, J_{L'} (\ldots, \Omega_\ve, \ldots).
 \label{e032}
\ee

\subsubsection{Reduction of  Theorem \ref{thm03}}  Here we reduce Theorem \ref{thm03} to the following.

\begin{proposition}\label{p.tech1}
Let $\Omega$ be a strong Kirby color.
Suppose $T$ is an algebraically split $0$-framed bottom tangle $T$ with $m$ ordered components and $(\ve_1,\ldots, \ve_m)\in \{\pm 1\}^m$. Then
$$
\left(\tr_q^{\Omega_{\ve_1}} \otimes \cdots \otimes \tr_q^{\Omega_{\ve_m}}\right)
(J_T) \ \eqz \   \big( \cT_{\ve_1}
\otimes \cdots \otimes \cT_{\ve_m} \big) \, (J_T) .$$
\end{proposition}
\begin{proof}
[Proof of Theorem \ref{thm03} assuming Proposition \ref{p.tech1}] Suppose $T$ is an $m$-component bottom tangle,  $\ve_1,\dots,\ve_m\in \{\pm1\}$, and $M=M(T,\ve_1,\dots,\ve_m)$. This means, if
$L$ is the closure link of $T$ and $L'$ is the same $L$ with framing of the $i$-th component switched to $\ve_i$, then $M$  is obtained from $S^3$ by surgery along $L'$. Every integral homology 3-sphere can be obtained in this way. By construction,
\begin{align*}
J_M& = (\cT_{\ve_1}
\otimes \cdots \otimes \cT_{\ve_m} \big) \, (J_T).\\
\end{align*}
From  \eqref{e032}  and the definition \eqref{n4011} of $\tau_M(\Omega)$, we have
$$
\tau_{M}(\Omega)  =
\ev_\zeta \Big(  \left(\tr_q^{\Omega_{\ve_1}} \otimes \cdots \otimes \tr_q^{\Omega_{\ve_m}}\right)
(J_T)\Big).
$$
By Proposition \ref{p.tech1}, we have  $\tau_{M}(\Omega) = \ez(J_M)$.
 This proves Theorem \ref{thm03}. \end{proof}
The rest of this section is devoted to a proof of Proposition \ref{p.tech1}.

\subsection{Integral form $\cU$ of $\Uq$}\label{sec:Ur}
 Besides the integral form $\UZ$ (of Lusztig) and $\VZ$ (of De Concini-Procesi), we need another integral form $\cU$ of $\Uq$, with $\VZ \subset \cU\subset \UZ$. Let
$$ \cU := \UZ^-\VZ =\UZ^- \VZ^0 \VZ^+=\UZ^{\ev,-}\VZ^0\VZ^+$$
and
\begin{gather*}
  \cU^\ev:= \cU\cap\cU_{\BZ}^\ev=\UZ^{\ev,-}\VZ^{\ev,0}\VZ^+.
\end{gather*}

\begin{theorem} \label{r.Ur}
(a) The $\cA$-module $\cU$ is an $\cA$-Hopf-subalgebra of $\UZ$.

(b) Each of $\cU$ and $ \cU^\ev$ is stable under $\ibar$ and $\tau$.

(c) There are even triangular decompositions
\begin{gather*}
  \UZ^{ev,-}\otimes \VZ ^0\otimes \VZ ^+ \congto \Ur ,\quad x\otimes y\otimes z\mapsto xyz \\
  \UZ^{ev,-}\otimes \VZ ^{\ev,0}\otimes \VZ ^+ \congto \Ur^\ev ,\quad x\otimes y\otimes z\mapsto xyz.
  \end{gather*}

  (d) For any longest reduced sequence, the sets
  \begin{align*}
  \{ F_\bm K_\bm K_\gamma E_\bn \mid \bn, \bm \in \BN^t, \gamma \in Y \} \\
  \{ F_\bm K_\bm K_{2\gamma} E_\bn \mid \bn, \bm \in \BN^t, \gamma \in Y \}
  \end{align*}
  are respectively $\cA$-bases of $\Ur$ and $\Ur^\ev$.

  (e) The Hopf algebra $\cU$ satisfies the assumptions of Theorem \ref{r.Jvalues}, i.e. $K_\al^{\pm1} \in \Ur$ for $\al\in \Pi$, $F_\bn \ot E_\bn, F'_\bn \ot E'_\bn \in \Ur \ot \Ur$ for $\bn \in \BN^t$.

(f) One has $\cT_\pm (\Ur^\ev) \subset \cA=\Zvv$.

(g) For any $n \ge 0$, one has $(\UZ^\ev)^{\ot n} \cap \Ur^{\ot n} = (\Ur^\ev)^{\ot n}$.

\end{theorem}
\begin{proof} (a) We have the following statement whose easy proof is dropped.

{\it Claim.} If $\scH_1,\scH_2$ are $\cA$-Hopf-subalgebras of a Hopf algebra $\scH$ such that  $\scH_2 \scH_1 \subset \scH_1 \scH_2$, then $\scH_1\scH_2$ is an $\cA$-Hopf-subalgebra of $\scH$.

We will apply the claim to $\scH_1=\UZ^- \VZ^0$ and $\scH_2=\VZ$.
By checking the explicit formulas of the co-products and the antipodes of $F_\al^{(n)}, K_\al$, $\al \in \Pi$, $n\in \BN$, which generates the $\cA$-algebra $\scH_1=\UZ^-\VZ^0$, we see that $\scH_1$ is an $\cA$-Hopf-subalgebra of $\UZ$. Since $\scH_2$ is also an $\cA$-Hopf-subalgebra of $\UZ$, it remains to show $\scH_2 \scH_1 \subset \scH_1 \scH_2$.

  Given $x,y$ in any Hopf algebra,we have
 $
    xy
    = \sum y_{(2)}(S^{-1}(y_{(1)})\trr x).
$ Hence, since $\scH_1$ is a Hopf algebra,
   and $\scH_1 \tri \VZ^\ev \subset \VZ^\ev$ (Theorem \ref{thm:8}),
  \begin{gather} \label{e8z}
    \VZ^\ev \scH_1 \subset  \scH_1 (\scH_1 \tri  \VZ^\ev) \subset \scH_1 \VZ^\ev.
  \end{gather}

Because $\VZ= \VZ^\ev \VZ^0$ and $\VZ^0 \scH_1= \VZ^0 \UZ^- \VZ^0= \UZ^- \VZ^0= \scH_1$, we have
$$ \scH_2 \scH_1 = \VZ \scH_1 = \VZ^\ev \VZ^0 \scH_1 = \VZ^\ev \scH_1 \subset \scH_1 \VZ^\ev \subset \scH_1 \scH_2,$$
where we used \eqref{e8z}. By the above claim, $\scH_1 \scH_2$ is an $\cA$-Hopf-subalgebra of $\UZ$.

(b) Let $f= \ibar$ or $f=\tau$.
By Propositions \ref{prop.basis} and \ref{r.phistab}, $f(\UZ^-) = \UZ^-\subset \UZ^- \VZ= \Ur$ and $f(\VZ)= \VZ\subset \UZ^- \VZ= \Ur $. Hence
$f(\Ur)=f(\UZ^- \VZ) \subset \Ur$.

By Proposition \ref{eq.evenD3}, $f(\Uq^\ev) \subset \Uq^\ev$. Hence
$$f(\Ur^\ev)= f(\Ur \cap \Uq^\ev)\subset f(\Ur) \cap f(\Uq^\ev) \subset \Ur \cap \Uq^\ev = \Ur^\ev.$$

(c) The even triangular decompositions of $\UZ$ (see Section \ref{sec:Lus}) imply the even triangular decompositions of $\Ur$.

(d) Since $F_\bm \sim  F^{(\bm)}$ and $E_\bn \sim (q;q)_\bn E^{(\bn)}$, where $a\sim b$ means $a=ub$ with $u$ a unit in $\cA$, Propositions \ref{r.ba3} and \ref{r11} show that $\{F_\bm K_\bm\}$ and $\{ E_\bn\}$ are respectively $\cA$-bases of $\UZ^{\ev,-}$ and $\VZ^+$. It is clear that $\{ K_\gamma \mid \gamma \in Y\}$ and $\{ K_{2\gamma} \mid \gamma \in Y\}$ are respectively $\cA$-bases of $\VZ^0$ and $\VZ^{\ev,0}$. Combining these bases using the even triangular decompositions, we get the desired bases of $\Ur$ and $ \Ur^\ev$.

(e) Since $K_\al^{\pm1}, F_\bn, E_\bm$ are among the basis elements described in (d), we have $K_\al^{\pm1} \in \Ur$ and $F_\bn \ot E_\bn\in \Ur \ot \Ur$.
Since $\Ur$ is stable under $\ibar$ and $F'_\bn=\ibar (F_\bn)$, $ E'_\bm= \ibar(E_\bm)$, we also have $F'_\bn \ot E'_\bn\in \Ur \ot \Ur$.

(f)  Applying $\cT_+$ to a basis element of $\Ur^\ev$  in (d), using \eqref{eq.989} and \eqref{eq.989a},
 \be
 \label{eq.evenT}
 \cT_+(F_\bm K_\bm K_{2\gamma} E_\bn)= \delta_{\bn,\bm} q^{(\rho, |E_\bn|)} q^{(\gamma,\rho) - (\gamma,\gamma)/2} \in \Zq \subset \cA.
 \ee
It follows that $\cT_+(\Ur^\ev) \subset \cA$.

Let us now show $\cT_-(\Ur^\ev) \subset \cA$. By \cite[Section 6.20]{Jantzen}, for any $x,y \in \Uq$, one has
 $$
 \la \omega S (x), \omega S (y)  \ra = \la y, x \ra.
 $$
 Because $\omega S (\br^{-1})= \br^{-1}$, and
 by \eqref{eq.qkcentral},
 $\la x, \br^{-1} \ra= \la \br^{-1},x \ra = \cT_-(x)$, we have
 $$\cT_-( x) = \cT_- (\omega S (x)),$$
 which is the same as
 $ \cT_-( x) = \cT_-\left( (\omega S)^{-1} (x)\right).
 $
 Hence,
 \begin{align*}
 \cT_-(\Ur^\ev) &= \cT_-\left( (\omega S)^{-1} (\Ur^\ev)\right) = \cT_+( \varphi \circ (\omega S)^{-1} (\Ur^\ev) )  \quad \text{by  \eqref{eq.cTminus}  }  \\
 &=  \cT_+ (\ibar \tau (\Ur^\ev))  \quad \text{because  $\varphi=  \ibar \tau\omega S $ by Proposition \ref{r.phi5} }  \\
  & \subset \cT_+(\Ur^\ev)  \subset\cA,
 \end{align*}
 where we have used part (b) which says $\ibar \tau (\Ur^\ev) \subset \Ur^\ev$.

 (g) It is clear that $(\Ur^\ev)^{\ot n} \subset (\UZ^\ev)^{\ot n} \cap \Ur^{\ot n}$. Let us prove the converse inclusion.

 The $\cA$-basis of $\Ur$  described in (d) is also a  $\BC(v)$-basis of $\Uq$. This basis  generates in a natural way an $\cA$-basis $\{ e(i) \mid i\in I\}$ of $\Ur^{\ot n}$, which is also a $\BC(v)$-basis of  $\Uq^{\ot n}$. There is a subset $I^\ev \subset I$ such that $\{ e(i) \mid i\in I^\ev\}$ is an $\cA$-basis of $(\Ur^\ev)^{\ot n}$ and at the same time a $\BC(v)$-basis of $(\Uq^\ev)^{\ot n}$. Using these bases, one can easily show that $(\Ur^\ev)^{\ot n} = (\Uq^\ev)^{\ot n} \cap \Ur^{\ot n}$. Hence,
 $$ (\UZ^\ev)^{\ot n} \cap \Ur^{\ot n} \subset (\Uq^\ev)^{\ot n} \cap \Ur^{\ot n} = (\Ur^\ev)^{\ot n},$$
 which is the converse inclusion.
 The proof is complete.
\end{proof}
Theorems \ref{r.Ur}(d) and  \ref{r.Jvalues} give the following.
\begin{corollary}
\label{r.Jvalues2}
If $T$ is an $n$-component bottom tangle with 0 linking matrix, then $J_T \in \tK_n(\cU)$.
\end{corollary}

\begin{remark}
(a)  For the case $\fg = sl_2$, the algebra
$\cU$ was considered by the first author
\cite{H:integralform,H:unified}.

(b)  The algebra $\Ur$ is not balanced between $E_\al $ and $F_\al$, and   $\varphi(\cU)  \neq \cU$.
\end{remark}

\subsection{Complexification of $\tK_m(\Ur)$} \label{sec:Kmprime}
To accommodate the complex coefficients appearing in the definition of $\Omega_\pm$, we often extend the ground ring from $\cA=\BZ[v^{\pm1}]$ to $\cB= \BC[v^{\pm1}]$. Let
$$ \Cvh := \varprojlim_{k} \BC[v^{\pm 1}]/(q;q)_k = \varprojlim_{k} \BC[v]/(q;q)_k.$$

By \eqref{eq.bh2a},
 \begin{align*}
 \cF_k(\modK_m(\Ur))& = (q;q)_k (\sXZe)^{\ot m} \cap \left[ \UZ^{\ot m}\right]_1 \cap \Ur^{\ot m}\\
 & \subset \ (q;q)_k (\sXZe)^{\ot m}  \cap (\Ur^\ev)^{\ot m}  \quad \text{by  Theorem \ref{r.Ur}(g)}.
 \end{align*}
Let
 $$\cF_k(\modK_m') := \Big( (q;q)_k (\sXZ^\ev)^{\ot m}  \cap (\Ur^\ev)^{\ot m}\Big )
  \ot_{\cA}  \cB  \subset h^k (\CXh)^{\ho m} \cap h^k \Uh^{\ho m}.$$
 Define the completion
\be
\label{eq.CKn}
\CtK_m=\left \{
x=\sum_{k=0}^\infty x_k \mid   x_k \in \cF_k(\CK_m) \right\} \subset (\CXh)^{\ho m} \cap (\Uh)^{\ho m}.
\ee
Then $\widetilde \modK_m(\cU) \subset \CtK_m$, and $\CtK_0= \Cvh$. We will work with $\tK'_n$ instead of $\tK_n(\Ur)$.

\subsection{Integral basis of $V_\lambda$}
\label{sec:IntegBasis} For $\lambda \in X_+$
recall that $V_\lambda$ is the finite-rank $\Uh$-module of highest weight $\lambda$.
Let  $1_\lambda \in V_\lambda$ be a highest weight element. It is known that the $\UZ$-module
 $\UZ\cdot 1_\lambda$ is a free $\cA$-module of rank equal to the rank of $V_\lambda$ over $\Ch$. Besides, there is an $\cA$-basis of $\UZ\cdot 1_\lambda$ consisting of weight elements, see e.g. \cite{CP}.
We call such a basis an {\em integral basis} of $V_\lambda$. For example, the {\em canonical basis} of Kashiwara and Lusztig \cite{Kashiwara,Lusztig} is such an integral basis. An integral basis of $V_\lambda$ is also a topological basis of $V_\lambda$.

Recall that $\brUA= \brUA^0 \UZ$ and $\brU_q=\brU_q^0 \Uq$ are  respectively  the simply-connected versions of $\UZ$ and $\Uq$, see Section \ref{sec:simply}.
 For $\lambda \in X_+$ we have the quantum trace map $\tr_q^{V_\lambda}: \Uh \to \Ch$.
This maps extends to $\tr_q^{V_\lambda}: \Uh[h^{-1}] \to \Ch[h^{-1}]$. In particular, if $ x\in \brU_q$, then
 one can define $\tr_q^{V_\lambda}(x) \in \Ch[h^{-1}]$.

\begin{lemma} \label{r.911}
Suppose  $\lambda$ is a dominant weight, $\lambda \in X_+$.

(a) If $x \in  \UZ$  then the
$\tr_q^{V_\lambda}(x)\in  \cA$.

(b) If   $x \in  \brU_q$ then  $\tr_q^{V_\lambda}(x) \in \BQ(v^{\pm 1/D})$.

(c) If $x \in  \brUA$ and
$\lambda \in Y$ then $\tr_q^{V_\lambda}(x)\in  \cA$.

(d) If $x\in \sXZ$ then $\tr_q^{V_\lambda}(x) \in \tA$.

\end{lemma}
\begin{proof} Fix an integral basis of $V_\lambda$. Using the basis,  each $x\in \Uh$ acts on $V_\lambda$ by a matrix with entries in $\BC[[h]]$, called the matrix of $x$.

 (a) If  $x\in \UZ$ then its matrix has entries in $\cA$. It follows that $\tr_q^{V_\lambda}(x) = \tr^{V_\lambda}(x K_{-2\rho}) \in \cA$.

(b) As a $\BQ(v)$-algebra, $\brU_q$ is generated by $\Uq$ and $\brK_\al, \al \in \Pi$. Since $\Uq= \UZ \ot_\cA \BC(v)$, the matrix of  $x\in \Uq$ has entries in $\BC(v)$.
For an element $e$ of weight $\mu$, we have $\brK_\al(e)= v^{(\bral, \mu)} e$. Note that $(\bral, \mu)\in \frac 1 D \BZ$. It follows that the matrix of $\brK_\al(e)$ has  entries in $\BQ(v^{\pm 1/D})$.
Hence the matrix of every $x\in \brU_q$ has  entries in $\BC(v^{\pm 1/D})$, and $\tr_q^{V_\lambda}(x) \in \BC(v^{\pm 1/D})$.

(c) As  $\cA$-algebra, $\brUA$ is generated by $\UZ$ and
$ f(\brK_\al;n,k) := \brK_\al^n \frac{(\brK_\al^2 ;q_\al)_k}{(q_\al;q_\al)_k}, \quad n\in \BZ, k\in \BN, \al \in \Pi.$

When $\lambda\in Y$, all the weights of $V_\lambda$ are in $Y$. From the orthogonality between simple roots and fundamental weights we have $(\bral, \mu) \in d_\al \BZ$ for every $\al\in \Pi$ and $y\in Y$. Hence
$$f(v^{(\bral,\mu)};n,k) =v^{n (\bral,\mu)} \frac{(q_\al^{(\bral,\mu)/d_\al} ;q_\al)_k}{(q_\al;q_\al)_k} \in \cA.$$
Suppose $e\in V_\lambda$ has weight $\mu \in Y$. Then
$$ f(\brK_\al;n,k) (e) = f(v^{(\bral,\mu)};n,k) \, e \in \cA e.$$
Thus, the matrix  of $f(\brK_\al;n,k)$ on $V_\lambda$ has entries in $\cA$. We conclude that the matrix of every $x \in \brUA$ has entries in $\cA$, and
 $\tr_q^{V_\lambda}(x)\in\cA$.

(d)
Because $\sXZ \subset \UZ \ot_{\cA} \tA$, by part (a)  we have
$\tr_q^{V_\lambda}(x) \in \tA$.
\end{proof}

\subsection{Quantum traces associated to $\Omega_\pm$} \label{n5013} %
Define
$$  \tT_\pm: \Uh \to \Ch \quad \text{ by } \  \tT_\pm(x) = \tr_q^{\Omega_\pm}(x).$$
Note that $\tT_\pm$, being quantum traces, are ad-invariant.
Since  $\Omega_\pm$  are $\BC$-linear combination of $V_\lambda$,
Lemma~\ref{r.911} shows that $\tT_\pm$
restricts to a $\cB$-linear map from $ \UZ\ot_\cA \cB$ to  $\cB=\Cv$.

Recall that $(\tK'_{n})^\inv$ denotes the set of elements in $\tK'_n$ which are  $\UZ$-ad-invariants.

\begin{proposition} \label{r.5011} Suppose $f$ is one of $\cT_\pm,\tT_\pm$. Then $f$ is $(\CK_m)^\inv$-admissible in the sense that
 for  $m\ge j  \ge 1$,
\begin{align*}
\big(\id^{\otimes j-1} \otimes\,  f \otimes \id^{\otimes m-j}\big) \left( \left(\CtK_{m} \right)^{\inv} \right)\subset \left(\CtK_{m-1} \right)^{\inv}.
 \end{align*}
\end{proposition}
\begin{proof} Recall that  $\cT_\pm,\tT_\pm$ are ad-invariant. By Proposition \ref{r.inva}(d) it is enough to prove
\begin{align*}
\big(\id^{\otimes j-1} \otimes\,  f \otimes \id^{\otimes m-j}\big)   \left(\CtK_{m} \right) \subset \CtK_{m-1} ,
 \end{align*}
which, in turn, will follow from
\begin{align}
 \label{eq.as02}
 \big(\id^{\otimes j-1} \otimes\,  f \otimes \id^{\otimes m-j}\big) \left ( \cF_k ( \CK_m  ) \right) &\subset \cF_{ k} ( \CK_{m-1}) .
\end{align}

Let us prove \eqref{eq.as02} for $f=\cT_\pm$.
By Proposition \ref{r.integ6},
  $$ \big(\id^{\otimes j-1} \otimes\,  \cT_\pm \otimes \id^{\otimes m-j}\big) \Big(  (q;q)_k (\sXZ^\ev)^{\ot m} \Big)  \subset  (q;q)_k (\sXZ^\ev)^{\ot m-1}.$$
  By Proposition \ref{r.Ur}(f),
   $$ \big(\id^{\otimes j-1} \otimes\,  \cT_\pm \otimes \id^{\otimes m-j}\big) \Big(  (\Ur^\ev)^{\ot m} \Big)  \subset  (\Ur^\ev)^{\ot m-1}.$$
   Because $\cF_k(\modK_m')= ((q;q)_k (\sXZ^\ev)^{\ot m} \cap (\Ur^\ev)^{\ot m})\ot _\cA \cB$,   we have
   $$ \big(\id^{\otimes j-1} \otimes\,  \cT_\pm \otimes \id^{\otimes m-j}\big) \cF_k(\modK_m') \subset \cF_k(\modK_{m-1}').$$

Let us now prove \eqref{eq.as02} for $f=\tT_\pm$.
Because $\Omega_\pm$ is a $\BC$-linear combination of $V_\lambda$, by Lemma \ref{r.911}(d),
$\tT_\pm(\sXZe) \subset  \tA\ot_{\cA} \cB$. Hence,
\be
\label{eq.as04}
\big(\id^{\otimes j-1} \otimes\,  \tcD_{\pm} \otimes \id^{\otimes m-j}\big)\, \left ( (q;q)_k (\sXZe)^{\ot m} \right ) \, \subset \, (q;q)_k \left( (\sXZe)^{\ot m-1} \ot_{\cA} \cB \right).
\ee
From Lemma \ref{r.911}(a), $\tT_\pm (\Ur)\subset \cB$, and hence
$$
\big(\id^{\otimes j-1} \otimes\,  \tcD_{\pm} \otimes \id^{\otimes m-j}\big) \, \left ( ( \Ur^\ev)^{\ot m} \right ) \, \subset \, (\Ur^\ev)^{\ot m-1} \ot_{\cA} \cB,
$$
which, together with \eqref{eq.as04}, proves \eqref{eq.as02}.
\end{proof}

\subsection{Actions of Weyl group on $\Uh^0$ and Chevalley theorem}\label{sec:ActionW}
 The Weyl group acts on the Cartan part $\Uh^0$ by {\em  algebra automorphisms} given  by $w(H_\lambda) = H_{w(\lambda)}$. Then $w(K_\al)= K_{w(\al)}$, and
$\fW $ restricts and extends to actions on the Cartan parts $\UZ^0$, $\VA^0$, and $\Xh^0$.

 We say an element $x \in \Uh^0$ is {\em $\fW $-invariant} if $w(x)= x$ for every $w\in \fW $, and   $x$ is {\em $\fW $-skew-invariant} if $w(x)= \sgn(w)x$ for every $w\in \fW $. As usual, if $\fW $ acts on $V$ we denote by $V^\fW $
 the subset of $\fW $-invariant elements.

\no{

The quantum Killing form, when restricted to $\UZ^0$, is $\fW $-invariant in the sense then
\be
\label{eq.Winv}
 \la w(x), w(y) \ra = \la x, y\ra \quad \text{
for any $w\in \fW $, $x,y \in \UZ^0$.}
\ee
}

By Chevalley's theorem \cite{Chevalley}, there are $\ell$ homogeneous polynomials $e_1,\dots, e_\ell \in \BZ[H_1,\dots,H_\ell]$ such that $(\BC[H_1,\dots,H_\ell])^\fW = \BC[e_1,\dots,e_\ell]$, the  polynomial ring freely generated by $\ell$ elements $e_1, \dots,e_\ell$.

Suppose the degree of $e_i$ is $k_i$. Since $\exp(hH_\al)= K_\al^2$, we have
 \be
 \label{eq.hy00}
 \tilde e_i:= \exp \left( h^{k_i} e_i\right)  \in \BZ[K_{1}^{\pm2}, \dots, K_\ell^{\pm 2}]^\fW \subset (\VA^{\ev,0})^\fW .
 \ee

\begin{proposition} (a) One has
\begin{align}
\label{eq.hy01}
(\Uh^0)^\fW &=  \BC[e_1,\dots,e_\ell][[h]]%
\\
\label{eq.hy02}
(\Xh^0)^\fW &= \BC[h^{k_1/2}e_1,\dots,h^{k_\ell/2} e_\ell][[\sqrt h]]
\\
\label{eq.hy03}
(\Vh^0)^\fW &=  \overline{\BC[h^{k_1}e_1,\dots, h^{k_\ell}e_\ell][[h]]}
\\
\label{eq.hy03a}
&
=\overline{\BC[\tilde e_1,\dots, \tilde e_\ell][[h]]}.
\end{align}
Here the overline in \eqref{eq.hy03} and \eqref{eq.hy03a} denotes the topological closure in the $h$-adic topology of $\Uh$.

\end{proposition}
\begin{proof} We have
$$(\Uh^0)^\fW = (\BC[H_1,\dots,H_\ell][[h]])^\fW = (\BC[H_1,\dots,H_\ell])^\fW [[h]] = \BC[e_1,\dots,e_\ell][[h]],$$
which proves \eqref{eq.hy01}. Similarly, using
\begin{align*}
(\Xh^0)^\fW &= (\BC[h^{1/2}H_1,\dots,h^{1/2}H_\ell][[h]])^\fW \\
(\Vh^0)^\fW &= (\BC[hH_1,\dots,hH_\ell][[h]])^\fW \end{align*}
we get \eqref{eq.hy02} and \eqref{eq.hy03}.
We have
$$ \tilde e_i -1 = h^{k_i} e_i + h (\Vh^0)^\fW .$$
 It follows that
$$ \BC[\tilde e_1,\dots,\tilde e_\ell][[h]]= \BC[h^{k_1}e_1,\dots, h^{k_\ell}e_\ell][[h]],$$
from which one has \eqref{eq.hy03a}.
\end{proof}

\subsection{The Harish-Chandra isomorphism, center of $\Uh$}\label{sec.HC1}

Let $\Zc(\Uh)$ be the center of $\Uh$, which is known to be $\Uh^\inv$, the ad-invariant subset of $\Uh$.
For any subset $V \subset \Uh$  denote  $\Zc(V)= V \cap \Zc(\Uh)$,  the set of central elements in $V$.

 Let  $p_0: {\Uh}\to \Uh^0$ be the projection
corresponding to the triangular decomposition. This means, if
$x = x_- \, x_0 \, x_+$, where $x_- \in \Uh^-, x_+ \in \Uh^+$ and $x_0\in \Uh^0$, then $p_0(x)=\boldsymbol\epsilon  (x_-)\,  \boldsymbol\epsilon (x_+) \, x_0$. Here $\boldsymbol\epsilon$ is the co-unit.

For $\mu \in X$, define the algebra homomorphism $\sh_\mu: \Uh^0 \to \Uh^0$ by $\sh_\mu(H_\al)= H_\al + (\al,\mu)$. Then
$\sh_\mu(K_\al) = v^{(\mu,\al)} \, K_\al$. Since $v^{(\mu,\al)} = \la K_{-2\mu}, K_\al \ra $, we have
\be
\label{eq.shift}
\sh_\mu(K_\al) = \la K_{-2\mu}, K_\al \ra \, K_\al.
\ee
The {\em  Harish-Chandra map} is the $\BC[[h]]$-module
  homomorphism
$$\cH = \sh_{-\rho} \circ \, p_0: {\Uh} \to \Uh^0
=\BC[H_1,\ldots,H_\ell][[h]].$$ The restriction of $\cH$ to the
$Y$-degree $0$ part of $\Uh$, denoted $\cH$ by abuse of notation, is a
$\BC[[h]]$-algebra homomorphism, called the {\em Harish-Chandra
homomorphism}.

 One has the following description of the center (see e.g. \cite{CP,Rosso}).
\begin{proposition}
  \label{r.HCiso}The restriction of $\cH$ on the center $\Zc(\Uh)$ is an algebra isomorphism from $ \Zc(\Uh)$ to
$ (\Uh^0)^\fW = \BC[H_1, \ldots, H_\ell]^\fW [[h]]$.

\end{proposition}

\begin{remark} Suppose  $\sH \subset \Uh$ is any subring satisfying the triangular decomposition (like $\UZ$ or $\Vh$).
 From definition
\be
\label{eq.equal}
\chi(\Zc(\sH)) \subset (\sH^0)^\fW.
\ee
For $\sH= \Uh$, we have equality in \eqref{eq.equal} by Proposition \ref{r.HCiso}. But in general, the left hand side is strictly smaller than the right hand side.
For example, one can show that
$$ \chi(\Zc(\UZ)) \neq (\UZ^0)^\fW .$$

Over the ground ring $\cA$, the determination of the image of the Harish-Chandra map is difficult.  Later we will determine $\chi(\Zc(\sH))$ for two cases, $\sH= \VA^\ev$, which is defined over $\cA$, and $\sH= \Xh$, which is defined over $\Chh$. In both cases, the duality with respect to the quantum Killing form will play an important role.
\end{remark}

 \subsection{From $\Uh$-modules to central elements} \label{sec:Drin}
 In the classical case, the center of the enveloping algebra of $\fg$ is isomorphic to  the ring of $\fg$-modules via the character map. We will recall (and modify) here the corresponding fact in the quantized case.

 For a dominant weight $\lambda\in X_+$, recall that $V_\lambda$ is the irreducible $\Uh$-module  of highest weight $\lambda$. Since the map $\tr_q^{V_\lambda}: \Uh \to \Ch$ is
 ad-invariant and the clasp element $\modc$ is ad-invariant, by Proposition \ref{r.inva}(d) the element
 \be
 z_\lambda := \left( \tr_q^{V_\lambda} \ho \id\right)(\modc)
 \notag
 \ee
 is in $(\Uh)^\inv = \Zc(\Uh)$. This construction of central elements was sketched in \cite{Drinfeld}, and studied in details in \cite{JL2,Baumann}.
  Our approach gives a geometric meaning of $z_\lambda$ as it shows that   $z_\lambda= J_T$, where $T$ is  the open Hopf link bottom tangle depicted in
Figure \ref{fig:openHopf}, with the closed component  colored by $V_\lambda$.  \FIG{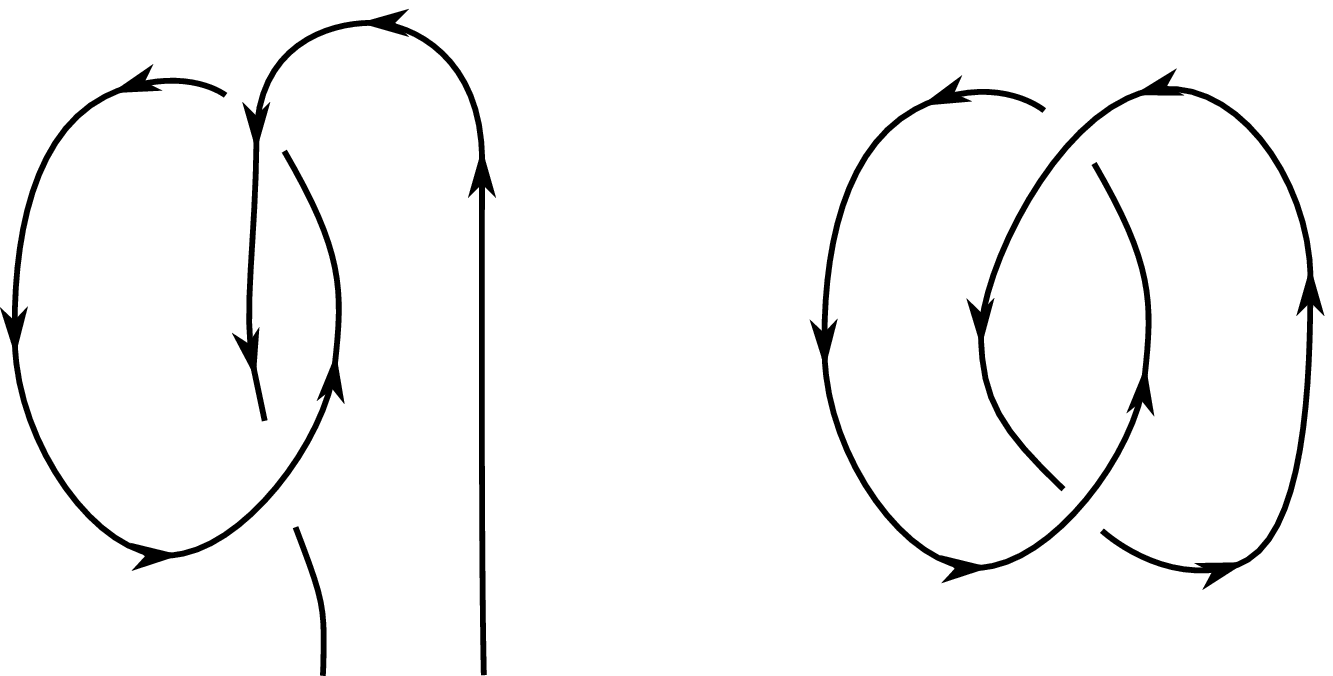}{The open Hopf link (left) and the Hopf link}{height=26mm}
Let us summarize some more or less well-known properties of $z_\lambda$, see \cite{Baumann,Caldero,JL2}.

\begin{proposition}
\label{r.centt}
Suppose $\lambda,\lambda' \in X_+$.

(a) For every $x\in \brU_q$,
\be
\label{eq.cent}
 \tr_q^{V_\lambda}(x) = \la z_\lambda, x \ra .
 \ee

 (b) One has
\be
\label{40008}
\cH (z_\lambda) = \sum_{\mu \in X} \dim (V_\lambda)_{[\mu]} K_{-2 \mu}
=
\frac
{
\sum_{w\in \fW}
\sgn(w) \, K_{-2 w(\lambda +\rho)}
}{\sum_{w\in \fW}  \sgn(w) \, K_{-2 w(\rho)}}.
\ee

 (c) If  $L$ is  the Hopf link, see Figure \ref{fig:openHopf}, then
\be  J_{\Hopf}(V_\lambda, V_{\lambda'}) = \la z_\lambda, z_{\lambda'} \ra =  \tr_q^{V_\lambda}(z_{\lambda'}),
\label{eHopf}
\ee

(d) One has $z_\lambda \in \breve{\mathbf U}_q^\ev$, and if $\lambda \in Y$, then $z_\lambda \in \Uq^\ev$.

\end{proposition}

\begin{proof}

(a)
Recall that  the quantum Killing form is the dual to $\modc= \sum \modc_1 \ot \modc_2$, and $x = \sum \la \modc_2, x\ra \modc_1$.
We  have
\begin{align*}
\la z_\lambda, x \ra  &=  \la \sum \tr_q^{V_\lambda}(\modc_1) \modc_2 , x \ra  \\
&=  \sum \tr_q^{V_\lambda}(\modc_1) \la  \modc_2 , x \ra =  \sum \tr_q^{V_\lambda}(\la  \modc_2 , x \ra \modc_1) = \tr_q^{V_\lambda}(x).
\end{align*}

(b) In \cite[Chapter 6]{Jantzen}, it is proved that if $\lambda \in X_+\cap \frac12 Y$, then
\be
\label{eq.charac}
\cH (z_\lambda) = \sum_{\mu \in X} \dim \left( (V_\lambda)_{[\mu]}\right) K_{-2 \mu},
\ee
where $\dim \left( (V_\lambda)_{[\mu]}\right)$ is the rank of the weight $\mu$ submodule.
Actually, the simple proof in \cite[Chapter 6]{Jantzen} works for all $\lambda \in X_+$.
 The second equality of  \eqref{40008} is   the famous Weyl character formula, see e.g. \cite{Humphreys}.

 (c) Let $T$ be  the open Hopf link bottom tangle depicted in
Figure \ref{fig:openHopf}, with the closed component  colored by $V_\lambda$. Then $J_T= z_\lambda$. We have
$$ J_L(V_\lambda,V_{\lambda'})= \tr_q^{V_{\lambda'}}(J_T)= \la z_{\lambda'}, J_T\ra = \la z_{\lambda'}, z_{\lambda}\ra= \la z_{\lambda}, z_{\lambda'}\ra.$$

(d) %
 Joseph and Letzter \cite[Section 6.10]{JL:separation} (see \cite[Proposition 5]{Baumann} for another proof) showed that %
 $z_\lambda \in
 \brU_q \tri K_{-2\lambda}$. Since $K_{-2\lambda}\in \brU_q^\ev$, we have $z_\lambda\in \brU_q \tri \brU_q^\ev \subset  \brU_q^\ev$, by Lemma \ref{eq.evenD3a}. If $\lambda\in Y$, then $K_{-2\lambda}\in \Uq^\ev$, hence
 $z_\lambda\in \Uq^\ev$ again by Lemma \ref{eq.evenD3a}.
\end{proof}

Note that the right hand side of \eqref{40008}
makes sense, and is in $(\Uq^0)^\fW $, for {\em any}  $\lambda\in X$ not necessarily in  $\X_+ \cap \frac12 Y$. For {\em any}  $\lambda\in X$, define $z_\lambda\in \Zc(\brU_q)$ by
$$ z_\lambda = \chi^{-1} \left (  \sum_{\mu \in X} \dim (V_\lambda)_{[\mu]} K_{-2 \mu}   \right) .$$

If $\lambda+\rho$ and $\lambda'+\rho$ are in the same $\fW $-orbit, then by \eqref{40008}, $z_\lambda= z_{\lambda'}$. On the other hand, if $\lambda+\rho$ is fixed by a non-trivial element of
the Weyl group, then $z_\lambda=0$.

When $\lambda$ is in the root lattice,  $\lambda \in Y$, the right hand side of \eqref{40008} is in $\cA[K_{\al_1}^{\pm 2}, \ldots, K_{\al_\ell}^{\pm 2}]^\fW $.
Actually, the theory of invariant polynomials says that the right hand side of \eqref{40008}, with $\lambda\in Y$, gives all $\cA[K_{\al_1}^{\pm 2}, \ldots, K_{\al_\ell}^{\pm 2}]^{W}$, see e.g. \cite[Section 2.3]{Macdonald}.
Hence, we have the following statement.

\begin{proposition} The Harish-Chandra homomorphism maps the $\cA$-span of $\{z_\al, \al \in Y\} $ isomorphically onto $\cA[K_{\al_1}^{\pm 2}, \ldots, K_{\al_\ell}^{\pm 2}]^\fW $.
\label{50010}
\end{proposition}

\subsection{Center of $\VA^\ev$}

\begin{lemma}
\label{6001}
Suppose $\beta \in Y$. Then
$z_\beta \in \VA^\ev$.
\end{lemma}
\begin{proof}
By Proposition \ref{r.ortho3}, $\VA^\ev$ is the $\cA$-dual of
$\brUA^\ev$ with respect to the quantum Killing form, i.e.
$$ \VA^\ev = \{ x \in \Uq^\ev\mid \la x, y \ra \in \cA \quad \forall y \in \brUA^\ev \}.$$
Since $z_\beta \in \Uq^\ev$ by Proposition \ref{r.centt}, it is sufficient to show that for any $y \in \brUA^\ev$, $\la z_\beta,y \ra \in \cA$.

We can assume that $\beta$ is a dominant weight, $\beta \in X_+\cap Y$. By Proposition \ref{r.centt}
$$ \la z_\beta,y \ra = \tr_q^{V_\beta}(y)\in \cA,$$
 where the inclusion follows from  Lemma \ref{r.911}. This shows $z_\beta \in \VA^\ev$.
\end{proof}

\begin{proposition} \label{40003}
(a) One has
$$ \Zc(\VA^\ev) =  \Zc(\Ur^\ev) = \cA\text{\rm -span of } \{z_\al\mid \al \in Y\}.$$
(b)
The Harish-Chandra homomorphism maps $\Zc(\VA^\ev)$ isomorphically onto $(\VA^{\ev,0})^\fW $, i.e.
\be
\label{eq.hy05}
\chi(\Zc(\VA^\ev)) = (\VA^{\ev,0})^\fW =\cA[K_{\al_1}^{\pm 2}, \ldots, K_{\al_\ell}^{\pm 2}]^{\fW}.
\ee

\end{proposition}

\begin{proof} (a) Let us prove the following inclusions
\be  \Zc(\VA^\ev) \subset  \Zc(\Ur^\ev) \subset \cA\text{\rm -span of } \{z_\al\mid \al \in Y\} \subset  \Zc(\VA^\ev),
\label{n1011}
\ee
which implies that all the terms are the same and proves part (a).

The first inclusion is obvious, since $ \VA^\ev \subset \Ur^\ev$, while the third is  Lemma \ref{6001}.

Because the $\Ur^{\ev,0}=\cA[K_{\al_1}^{\pm 2}, \ldots, K_{\al_\ell}^{\pm 2}]$, one has $\cH(\Zc(\Ur^\ev)) \subset \cA[K_{\al_1}^{\pm 2}, \ldots, K_{\al_\ell}^{\pm 2}]^{\fW }$.
 Hence, by Proposition \ref{50010} we have
 $
 \Zc(\Ur^\ev)  \subset \cA\text{\rm -span of } \{z_\al\mid \al \in Y\}
 $,
 which is the second inclusion in \eqref{n1011}. This proves (a).

 (b)  follows from (a) and Proposition \ref{50010}.
\end{proof}

\begin{proposition}
\label{r.fu1}
The Harish-Chandra map $\chi$ is an isomorphism between $\Zc(\Vh)$ and $(\Vh^0)^\fW $.
\end{proposition}
\begin{proof} Since $\chi(\Zc(\Vh)) \subset (\Vh^0)^\fW $,  it remains to show $(\Vh^0)^\fW \subset  \chi(\Zc(\Vh))$. By \eqref{eq.hy03}
 \begin{align*}
 (\Vh^0)^\fW =  \overline  { \BC[\tilde e_1,\dots,\tilde e_\ell] [[h]]    } .
  \end{align*}
By \eqref{eq.hy00} and \eqref{eq.hy05},
 $$\tilde e_i \in (\VA^{\ev,0})^\fW = \chi(\Zc(\VA^\ev)) \subset  \chi(\Zc(\Vh)).$$
 Hence  $ (\Vh^0)^\fW \subset \chi(\Zc(\Vh))$. This completes the proof of the proposition.
\end{proof}

\subsection{Center of $\Xh$}
\begin{proposition}
\label{r.isomor}
The Harish-Chandra map $\chi$ is an isomorphism between $\Zc(\Xh)$ and $(\Xh^0)^\fW $.
\end{proposition}
\begin{proof} By the definition, $\chi(\Zc(\Xh))\subset (\Xh^0)^\fW $. We need to show that $\chi^{-1}((\Xh^0)^\fW ) \subset \Zc(\Xh)$. Because $\chi^{-1}((\Xh^0)^\fW )$ consists of central elements, one needs only to show $\chi^{-1}((\Xh^0)^\fW ) \subset \Xh$. We will use the stability principle of dilatation triples.

From \eqref{eq.hy01}, \eqref{eq.hy02}, and \eqref{eq.hy03}, the triple $(\Uh^0)^\fW , (\Xh^0)^\fW , (\Vh^0)^\fW $ form a topological dilatation triple (see Section \ref{sec:topdil}).

The triple $\Uh,\Xh,\Vh$ also form a topological dilatation  triple (see Section \ref{sec:sX}).
Since $\chi^{-1} \left( (\Uh^0)^\fW \right) \subset \Uh$ and $\chi^{-1} \left( (\Vh^0)^\fW \right) \subset \Vh$ by Proposition \ref{r.fu1}, one also has
$\chi^{-1} \left( (\Xh^0)^\fW \right) \subset \Xh$, by the  stability principle (Proposition \ref{r.topdilat}).
\end{proof}

\subsection{Quantum Killing form and Harish-Chandra homomorphism} \label{sec.DDD}
Since $\chi(x), \chi(y)$ determine $x,y$ for central  $x, y\in \UZ$, one should be able to calculate  $\la x, y\ra$ in terms of $\chi(x), \chi(y)$.

Let $\DDD$ be the denominator of the right hand side of \eqref{40008}, i.e.
$$\DDD:= \sum_{w \in \fW} \sgn(w) K_{-w(2\rho)}.$$
By the Weyl denominator formula,
\be
\label{eq.DDD}  \DDD= \prod_{\al \in \Phi_+}(K_\al^{-1} - K_{\al}) = K_{2\rho} \prod_{\al \in \Phi_+}(K_{\al}^{-2}-1) \in K_{2\rho} \VA^\ev.
\ee

 Let us define
$$ \quad \ddd :=  \la K_{-2\rho}, \DDD \ra =\prod_{\al \in \Phi_+}(v_\al^{-1} - v_{\al}).$$

From the formula for the quantum dimension \eqref{eq.unknot}, we have
\be
\label{eq.dimq}
{\ddd} \dim_q(V_\lambda) = \la K_{-2\rho-2\lambda}, \DDD \ra .
\ee

Here is a formula expressing $\la x, y\ra$ in terms of $\chi(x), \chi(y)$.
\begin{proposition}\label{r.KHC}
Suppose $x\in \Zc(\overline \Xh)$, and $y=z_\lambda$, $\lambda\in Y$.
Then
\be
\label{eq.KHC}
|\fW | {\ddd}\,  \la x, y \ra = \la \DDD \,\chi(x), \DDD \,\chi(y) \ra
\ee
\end{proposition}
\begin{proof}

As $x$ is central, it acts on $V_\lambda$ by $c(\lambda,x) \, \id$, where $c(\lambda,x) \in \Ch$. Recall that $1_\lambda$ is the highest weight vector of $V_\lambda$. We have
$K_\al \cdot 1_\lambda = v^{(\al,\lambda)} 1_\lambda = \la K_\al, K_{-2\lambda} \ra \, 1_\lambda$. Hence for every $z\in \Uh^0$,
\be
\label{eq.uua}
z \cdot 1_\lambda = \la x, K_{-2\lambda} \ra \, 1_\lambda
\ee
 Since the highest weight vector $1_\lambda$ is killed by all $E_\al, \al \in \Pi$, we have
\begin{align*}
 x \cdot 1_\lambda & = p_0(x) \cdot 1_\lambda = \sh_{\rho} \chi(x) \cdot 1_\lambda  \\
 &= \la \sh_{\rho} \chi(x), K_{-2 \lambda} \ra \, 1_\lambda \quad \text{by   \eqref{eq.uua} }
\end{align*}
Thus, $c(\lambda,x) = \la \sh_{\rho} \chi(x), K_{-2 \lambda} \ra$. Further, by \eqref{eq.shift},
\begin{align*}
c(\lambda,x) & = \la \sh_{\rho} \chi(x), K_{-2 \lambda} \ra = \big\la \la K_{-2 \rho}, \chi(x) \ra\, \chi(x),  K_{-2 \lambda}\big \ra \\
 &=  \la K_{-2 \rho}, \chi(x) \ra\,\la  \chi(x),  K_{-2 \lambda}\ra =  \la K_{-2 \rho}, \chi(x) \ra\,\la  K_{-2 \lambda}, \chi(x)  \ra = \la K_{-2 \rho-2\lambda}, \chi(x) \ra \\
 &= \la K_{-2 \rho-2\lambda}, \frac{\DDD \,\chi(x)}{\DDD} \ra = \frac{\la K_{-2 \rho-2\lambda}, \DDD \chi(x)\ra}{\la K_{-2 \rho-2\lambda}, \DDD\ra} = \frac{\la K_{-2 \rho-2\lambda}, \DDD \chi(x)\ra}{ \ddd \dim_q(V_\lambda)}\\
 &= \frac{1}{ \ddd \dim_q(V_\lambda)} \left\la \frac{1}{|\fW |}\sum_{w\in \fW} \sgn(w) K_{-2 w(\lambda +\rho}), \DDD \chi (x) \right\ra\\
 &= \frac{1}{ |\fW |\ddd \dim_q(V_\lambda)} \left\la \DDD\chi(z_\lambda), \DDD \chi (x) \right\ra.
\end{align*}
Here the last equality on line three follows from \eqref{eq.dimq}, and the equality on the line four follows from the fact that $\DDD \chi (x)$ is $\fW $-skew-invariant and the quantum Killing form is $\fW $-invariant on $\Xh^0$.

 Using \eqref{eq.cent} and the fact that $x= c(\lambda,x)\id$ on $V_\lambda$,
 \begin{align*}
 \la x, z_\lambda \ra& = \tr_q^{V_\lambda}(x) = c(\lambda,x) \dim_q(V_\lambda) = \frac{1}{ |\fW |\ddd } \left\la \DDD\chi(z_\lambda), \DDD \chi (x) \right\ra,
 \end{align*}
 where for the last equality we used the value of $c(\lambda,x)$ calculated above.
\end{proof}

\begin{remark}
It is not difficult to show that Proposition \ref{r.KHC} holds for every $y \in \Zc(\Xh)$.
\end{remark}

\subsection{Center of $\tK_1'$} Recall that $\CtK_1$ is the set of all elements of the form
$$ x= \sum x_k, \quad    x_k \in \cF_k(\CK_1).$$
One might expect that every central element of $\CtK_1$ has the same form with $x_k$ central. We don't know if this is true. We have here a weaker statement which is enough for our purpose. In our presentation, $x_k$ is central, but might not be in $\cF_k(\CK_1)$. However, $x_k$ still has enough integrality.

\begin{lemma} \label{40006} Suppose $x\in \Zc(\CtK_1)$. There are central elements $x_k \in \Zc(\CsX)$ such that

(a) one has $|\fW | x= \sum_{k=0}^\infty (q;q)_k\, x_k$,

(b) for every $k \ge 0$, $(q;q)_k x_k$ belongs to $\Zc(\VA^\ev \ot_\cA \cB)$,

(c) for every $k \ge 0$, one has $  \cT_\pm( x_k)  \in \frac1{\ddd}{\BC[v^{\pm 1}]}%
$, and

(d) for every $k \ge 0$, one has $ \tT_\pm( x_k )  \in \frac1{\ddd}{\BC[v^{\pm 1}]}%
 $.

\end{lemma}
\begin{proof}
(a) Recall that
$\cF_k(\CK_1)= \Big( (q;q)_k (\sXZ^\ev) \cap (\Ur^\ev)\Big )
  \ot_{\cA}  \cB.$
Hence $x$ has a presentation
\be \label{eq.xn1}
 x =  \sum_{k=0}^\infty (q;q)_k\, x_k',
 \ee
where $x_k' \in \sXZe \ot_{\cA}  \cB$ and $(q;q)_k x_k'\in \Ur^\ev \ot_{\cA}  \cB$.

Let  $y_k = \sum _{w\in \fW} w(\chi(x'_k))$, which is $\fW $-invariant. Then $y_k \in (\CsX^0)^\fW $. By Proposition \ref{r.isomor}, $x_k:= \chi^{-1}(y_k)$ is central and belongs to $ \Zc(\CsX)$.

Using the  $\fW $-invariance of $\chi(x)$ and \eqref{eq.xn1}, and using  $\fW $-invariance of $\chi(x)$,
$$
 |\fW |\, \chi( x) =  \sum _{w\in \fW} w(\chi(x))=    \sum_{k=0}^\infty (q;q)_k \, \sum _{w\in \fW} w(\chi(x_k')) = \sum_{k=0}^\infty (q;q)_k \, y_k.
 $$
Applying $\chi^{-1}$ to the above, we get the form required in (a):
$|\fW | x= \sum_{k=0}^\infty (q;q)_k\, x_k.$

(b)
Since $(q;q)_k\, x_k'\in \Ur^\ev\ot_{\cA}  \cB $ and $\Ur^{\ev,0}= \VA^{\ev,0}$, one has
$$(q;q)_k \,y_k=(q;q)_k  \sum _{w\in \fW} w(\chi(x'_k))\in \VA^{\ev,0}\ot_{\cA}  \cB.$$
 By Proposition \ref{40003}, $(q;q)_k\,x_k = \chi^{-1}((q;q)_k\,y_k)\in \Zc(\VA^\ev)\ot_{\cA}  \cB$. %

(c)
Because $\VA^\ev \subset \Ur^\ev$, we have $\cT_\pm (\VA^\ev)\subset \cA$, by Theorem \ref{r.Ur}(f). From (b), we have $$(q;q)_k \cT_\pm (x_k) \in \cA\ot_\cA \cB=\cB,$$
 or
\be
\label{eq.jh1}
  \cT_\pm (x_k) \in  \frac 1{(q;q)_k }\cB.
\ee
A simple calculation shows that $\chi(\br)= v^{(\rho,\rho)} K_{2\rho} \br_0$. Since $\sXZ^{\ev,0}$ is an $\tA$-Hopf-algebra (Lemma~\ref{r.Hopfsub}), we have
 $$ \Delta (K_{2\rho}\sXZ^{\ev,0}) \subset K_{2\rho}\sXZ^{\ev,0} \ot K_{2\rho}\sXZ^{\ev,0}.$$
 Since $\DDD\in K_{2\rho} \VA^{\ev,0}$, we have $\DDD y_k \in K_{2\rho}\sXZ^{\ev,0}$. Hence $\Delta(\DDD y_k) = \sum K_{2\rho} y_k' \ot K_{2\rho} y_k''$, where $y'_k,y_k''\in \sXZ^{\ev,0}\ot_{\cA}  \cB$.
 Since $\DDD K_{\pm 2\rho} \in \sXZ^{\ev,0}$, we have $\Delta (\DDD K_{\pm 2\rho})= \sum a_1 \ot a_2$ with $a_1,a_2 \in \sXZ^{\ev,0}$.
Using \eqref{eq.KHC}, we have
\begin{align*}
  \ddd\cT_\pm(x_k) & =    \ddd\la \br^{\pm 1},  x_k \ra =  \la \DDD \chi(\br^{\pm 1}), \DDD  \chi(x_k) \ra= v^{(\rho,\rho)}\la \DDD K_{\pm 2\rho} \br^{\pm 1}_0, \DDD y_k \ra  \\
  &= v^{(\rho,\rho)}\sum  \la \DDD K_{\pm 2\rho}, K_{ 2\rho} y_k' \ra \, \la  \br^{\pm 1}_0,  K_{ 2\rho}y_k'' \ra \quad \text{by \eqref{eq.sd10}}  \\
  &= v^{(\rho,\rho)}\sum  \la a_1, K_{\pm 2\rho}   \ra  \, \la a_2,  y_k' \ra \, \la  \br^{\pm 1}_0,  K_{ 2\rho} y_k'' \ra \quad \text{again by \eqref{eq.sd10}}.
 \end{align*}
The first two factors $\la a_1, K_{\pm 2\rho}\ra$ and $\la a_2,  y_k' \ra$ are in $\tB$ by Lemma \ref{r.Hopfsub}, where $\tB=\tA \ot_{\cA}  \cB$.
The third factor $\la  \br^{\pm 1}_0,  K_{ 2\rho} y_k'' \ra $ is in $v^{(\rho,\rho)} \tB$ by Lemma \ref{r.halfinteg.new}.
Hence $ \ddd\cT_\pm(x_k) \in v^{2(\rho,\rho)} \tB =\tB$. Together with \eqref{eq.jh1},
$$   \ddd\cT_\pm(x_k) \in \BC(v) \cap \tB= \cB.$$

(d)  By definition, $\Omega_\pm =\sum c^{\pm}_{\lambda} V_\lambda$, where the sum is finite and $c^\pm_\lambda \in \BC$.
We have
$$ \ddd\, \tT_\pm (x_k)=
\sum c^{\pm}_{\lambda} \ddd \,\tr_q^{V_\lambda}(x_k)
 .$$
Hence, to show that $\ddd\, \tT_\pm (x_k) \in \cB$,
it is enough to show that for any $\lambda \in X_+$,
$
\ddd \, \tr^{V_\lambda}(x_k) \in \cB.
$
Using \eqref{eq.cent} and \eqref{eq.KHC}, we have
\begin{align*}
|\fW | \ddd \, \tr_q^{V_\lambda}(x_k) &= |\fW |\ddd \, \la z_\lambda,x_k \ra = \la \DDD \chi(z_\lambda), \DDD \chi(x_k) \ra \\
&=  \left\la \sum_{w\in \fW} \sgn(w) K_{-2 w(\lambda + \rho)}, \DDD  y_k \right\ra \qquad \text{by \eqref{40008}} \\
&= \sum_{w\in \fW} \sgn(w) \la  K_{-2 w(\lambda + \rho)}, \DDD  y_k \ra \\
& = \sum_{w\in \fW} \sgn(w) \la  K_{-2 w(\lambda + \rho)}, \DDD \ra  \la  K_{-2 w(\lambda + \rho)}, y_k \ra.
\end{align*}
The second factor $ \la  K_{-2 w(\lambda + \rho)}, y_k \ra$ is in $\tB$ by Lemma \ref{r.Hopfsub}. As for the first factor, for any $\mu \in X$,
$$ \la  K_{2 \mu}, \DDD \ra = \la  K_{2 \mu}, \prod_{\al \in \Phi_+}(K_\al - K_\al^{-1}) \ra = \prod_{\al \in \Phi_+} (\la  K_{2 \mu} , K_\al\ra - \la  K_{2 \mu} , K_{-\al}\ra)=
\prod_{\al \in \Phi_+} (v^{-(\mu,\al)} - v^{(\mu,\al)}) \in \Cv.$$
Hence, $\ddd\, \tr_q^{V_\lambda}(x_k)\in \tB$.

On the other hand, since $(q;q)\, x_k \in \VA^\ev\ot_{\cA}  \cB$, we have $\la z_\lambda, (q;q) x_k \ra \in \cB$. Hence
$$ \ddd \, \tr_q^{V_\lambda}(x_k)\in  \tB \cap \BC(v) = \Cv.$$
This completes the proof of the lemma.
\end{proof}

\subsection{Comparing $\cT$ and $\tcD$}
\begin{proposition} Suppose $\Omega$ is a strong Kirby color at level $\zeta$, $x \in (\CtK_m)^\inv$, and $\ve_j=\pm 1$ for $j=1, \ldots,m$. Then
$$  \Big( \bigotimes_{j=1}^m \tcD_{\ve_j}\Big) \big(x \big) \eqz \Big( \bigotimes_{j=1}^m \cT_{\ve_j}\Big) \big(x \big).
$$
\label{60011}
\end{proposition}
\begin{proof}

We proceed in three steps.

 Step 1: $m=1$ and
$x \in (\VA^\ev\ot_{\cA}  \cB)^\inv= \Zc(\VA^\ev\ot_{\cA}  \cB)$. By Proposition~\ref{40003}, $x$ is a $\cB$-linear combination of   $z_\lambda$, $\lambda \in Y$. We can assume that $x= z_\lambda$ for some $\lambda\in X_+\cap Y$.

Let $L_1$ be the disjoint union of  $U_{-\ve}$ and $U_{\ve}$, where the first is colored by $V_\lambda$ and the second by $\Omega$. Sliding the first component over the second, from $L_1$
we get a link $L_2$, which is the Hopf link where
 the first component has framing $0$ and the second has framing $\ve$, see Figure \ref{fig:sliding}.
 \FIG{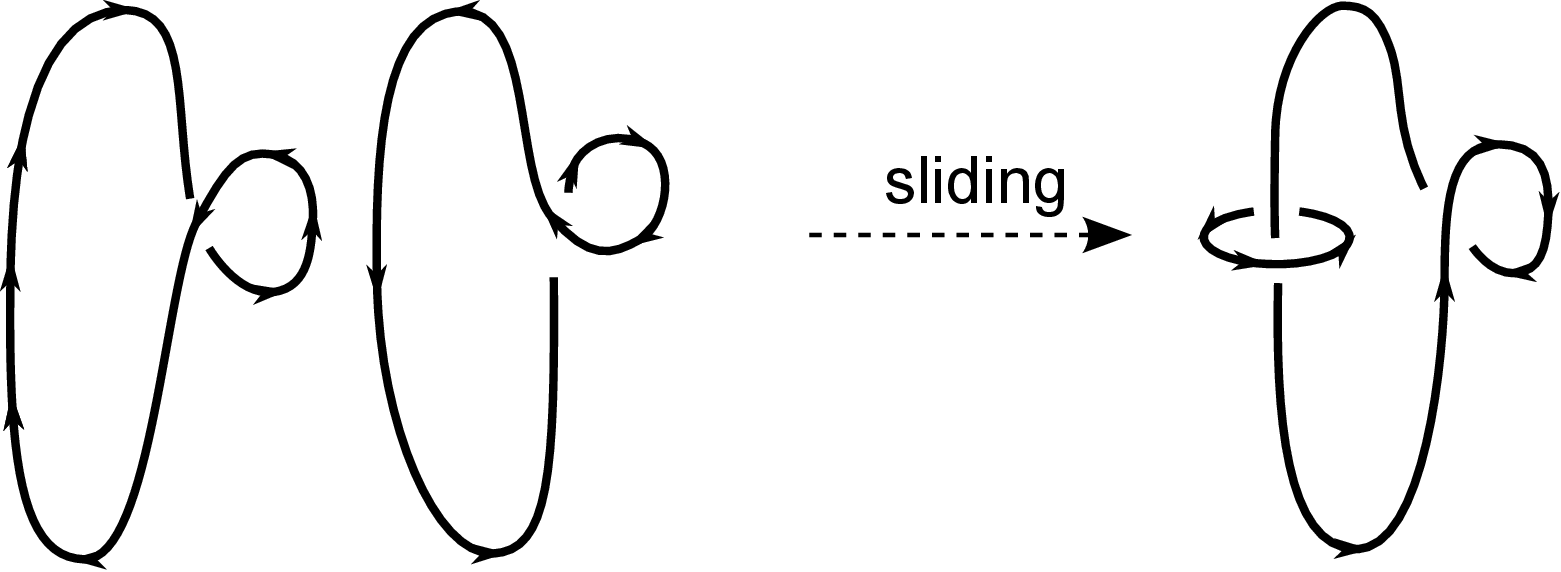}{Links $L_1$ (left) and $L_2$, which is obtained from $L_1$ by sliding. Here $\ve=-1$}{height=3cm}
  From the strong handle slide invariance \eqref{em1} we get
\be
  J_{L_1}(V_\lambda,\Omega)  \eqz  J_{L_2}(V_\lambda,\Omega).
 \label{e031a}
\ee

 Let us rewrite the left hand side and the right hand side of \eqref{e031a}.
 \begin{align*}\text{LHS of \eqref{e031a}} &= J_{U_{-\ve}}(V_\lambda) \, J_{U_{\ve}}(\Omega) = \tr_q^{V_\lambda}(\br^{\ve}) \,  J_{U_{\ve}}(\Omega) \\
 &= \la z_\lambda, \br^{\ve} \ra  J_{U_{\ve}}(\Omega) = \cT_{\ve}(z_\lambda)\, J_{U_{\ve}}(\Omega) .
 \end{align*}

 Let $L_0$ be the Hopf link with $0$ framing on both components. Then
\begin{align*}
 \text{RHS of \eqref{e031a}} = J_{L_2}(V_\lambda,\Omega )
& =  J_{U_{\ve}}(\Omega) \, J_{L_0} (V_\lambda, \Omega_\ve)  \quad \text{ by \eqref{e032}} \\
&\eqz   J_{U_{\ve}}(\Omega)  \,  \tr_q^{\Omega_\ve } (z_\lambda) \quad \text{ by \eqref{eHopf}} \\
&=  J_{U_{\ve}}(\Omega) \tT_\ve(z_\lambda).
\end{align*}

Comparing the left hand side and the right hand side of \eqref{e031a} we get
 $\cT_{\ve}(z_\lambda) \eqz \tcD_\ve(z_\lambda)$.

Step 2: $m=1$, and $x$ is an arbitrary element of $(\CtK_1)^\inv= Z (\CtK_1)$. Let
$ x = \sum_{k=0}^\infty (q;q)_k\, x_k$
be the presentation of $x$ as described in
 Lemma \ref{40006}. Since
  $x_k \in \Zc (\Xh)$ and all $\cT_\pm, \tT_\pm$ %
  are continuous in the $h$-adic topology of $\Xh$,
  \begin{align*}
  \cT_\pm (x) =  \sum_{k=0}^\infty (q;q)_k \cT_\pm  (x_k) \\
  \tT_\pm (x) =  \sum_{k=0}^\infty (q;q)_k \tT_\pm  (x_k).
  \end{align*}
Both right hand sides are in $\frac1\ddd\hCv$ because $\cT_\pm  (x_k) , \tT_\pm  (x_k) \in \frac1\ddd\Cv$
by Lemma \ref{40006}.
Since $(q;q)_k \eqz 0$ if  $k \ge r$ and $\ddd \not \eqz 0$,  we have
\begin{align*}
\cT_\pm (x) \eqz \sum_{k=0}^{r-1} (q;q)_k \cT_\pm(x_k)\eqz \cT_\pm \left( \sum_{k=0}^{r-1} (q;q)_k x_k  \right)\\
\tT_\pm (x) \eqz \sum_{k=0}^{r-1} (q;q)_k \tT_\pm(x_k)\eqz \tT_\pm \left( \sum_{k=0}^{r-1} (q;q)_k x_k  \right)
\end{align*}
By Lemma \ref{40006}(b), the elements in the big parentheses are in $\Zc(\VA^\ev\ot_\cA \cB)$. Hence, by the result of Step 1, we have $\cT_\pm(x)\eqz \tT_\pm(x)$.

Step  3: general case. Define $a_k$ (for
 $k=0,1,\ldots,m$) and $b_k$ (for
 $k=1,\ldots,m$)  as follows:
$$ a_k= \left( \bigotimes_{j=1}^k \tcD_{\ve_j} \,  \otimes \,    \bigotimes_{j=k+1}^m \cT_{\ve_j} \right) (x), \quad b_k = \left(  \bigotimes_{j=1}^{k-1} \tcD_{\ve_j} \,  \otimes\id\otimes  \,  \bigotimes_{j=k+1}^m \cT_{\ve_j} \right) (x).$$

Then
\be a_{k-1}= \cT_{\ve_k} (b_k), \quad \text{and} \quad a_{k}= \tcD_{\ve_k} (b_k).
\label{n6101}
\ee

By
 Proposition  \ref{r.5011},  $b_k \in (\CtK_1)^\inv$.
By Step 2,

$$  \tcD_{\ve_k} (b_k) \eqz  \cT_{\ve_k} (b_k).$$

Using \eqref{n6101}, the above identity becomes
$
 a_{k-1}\eqz a_{k} .
$
Since this holds true for $k=1,2,\ldots,m$, we have
$ a_0 \eqz a_m,$ which is the statement of the proposition.
\end{proof}
\subsection{Proof of Proposition \ref{p.tech1}} By Theorem \ref{r.Jvalues}, if $T$ is an algebraically split $m$-component bottom tangle, then $J_T\in \tK_m(\Ur)\subset \CtK_m$. Hence Proposition \ref{p.tech1} follows from Proposition \ref{60011}.
This also completes the proof of Theorems \ref{thm03} and \ref{thm030}.

\subsection{Proof of Theorem \ref{r40a}}
\label{sec:r40a}
The existence of invariant $J_M=J_M^\fg \in \Zqh$ is established by Theorem \ref{r.Jvalues}.  Theorem \ref{thm03} shows that  $\ev_\xi(J^\fg_M)= \tau_M^\fg(\xi)$. The uniqueness of $J_M$ follows from (i) every element of $\Zqh$ is determined by its values at infinitely many roots of 1  of prime power orders (see Section \ref{sec.Habiro}), and (ii) $\ZZ'_{P\fg}$ contains infinitely many such roots of unity (by Proposition \ref{pKirby}). This completes the proof of Theorem \ref{r40a}.

\subsection{The case $\zeta=1$, proof of Proposition \ref{r10}}
\label{sec:r10}
Let $\Omega$ be the trivial $\Uh$-module $\BC[[h]]$. By   Proposition \ref{p.color1a}, $\Omega$ is a strong Kirby color, and $\tau_M(\Omega)=1$.
 By Theorem \ref{thm03}, we have
$\ev_1(J_M)=1$. This completes the proof of Proposition \ref{r10}.

Proposition \ref{r10} can also be proved using the theory of finite type invariants of
integral homology $3$-spheres as follows. Note that $\ev_1(J_M)$ is the constant coefficient of the Taylor expansion of $J_M$ at $q=1$, which is a finite type invariant of order $0$ (see for example \cite{KLO}). Hence $\ev_1(J_M)$ is  constant on
the set of integral homology $3$-spheres. For $M=S^3$, $\ev_1(J_M)= 1$.  Hence $\ev_1(J_M)= 1$ for any integral homology $3$-sphere $M$.
\np
\appendix

\section{Another Proof of Proposition \ref{r.41a}}\label{sec:r.41a} In the main text we take Proposition \ref{r.41a} from work of Drinfel'd  \cite{Drinfeld} and Gavarini \cite{Gavarini}. Here we give an independent proof.

Each of  $\Uh^{\le 0} := (\Uh^0 \Uh^-)\,\hat{}$ and  $\Uh^{\ge 0} := (\Uh^0 \Uh^+)\,\hat{}$, where $(\ )\,\hat{}$ denotes the $h$-adic completion, is a Hopf subalgebra of $\Uh$, and
 $\cR \in \Uh^{\le 0} \ho \Uh ^{\ge 0}$. Let $A_L\subset\Uh^{\le0}$ and $A_R\subset\Uh^{\ge0}$ are respectively the left image (see Section \ref{sec.leftmodule}) and the right image  of $\cR \in \Uh^{\le 0} \ho \Uh ^{\ge 0}$. Here the right image is the obvious counterpart of the left image and can be formally defined so that  $\sigma_{21}(A_R)$ is the left image of $\sigma_{21}(\cR)$, where $\sigma_{21}: \Uh^{\le 0} \ho \Uh ^{\ge 0} \to \Uh^{\ge 0} \ho \Uh ^{\le 0}$ is the isomorphism given by $\sigma_{21}(x\ot y) = y \ot x$.

 Explicitly, $A_L$ and $A_R$ are defined as follows.
For $\bn=(\bn_1, \bn_2) \in \BN^t \times \BN^\ell$ let
$$ \cR'(\bn)= F^{(\bn_1)} H^{\bn_2}, \quad \cR''(\bn)= E^{(\bn_1)} \brH^{\bn_2}.$$
Then $\{ \cR'(\bn) \mid \bn \in \BN^{t+\ell} \}$ is a topological basis of $\Uh^{\le 0}$, and $\{ \cR''(\bn) \mid \bn \in \BN^{t+\ell} \}$ is a topological basis of $\Uh^{\ge 0}$. From \eqref{501}, there are units $f(\bn)$ in $\Ch$ such that
$$ \cR = \sum_{\bn \in \BN^{t+\ell}} f(\bn) \, h^{\| \bn\| } \cR'(\bn) \ot \cR''(\bn).$$
Then $A_L$  and $A_R$ are respectively  the topological closures (in $\Uh$) of the $\Ch$-span of
 \be
 \label{eq.basisALR}
 \{ h^{\| \bn\| } \cR'(\bn) \mid \bn \in \BN^{t+\ell} \}\  \text{ and } \  \{ h^{\| \bn\| } \cR''(\bn) \mid \bn \in \BN^{t+\ell} \}.
 \ee

For $\Ch$-submodules $\sH_1, \sH_2 \subset \Uh$, let $\overline{\sH_1 \ot \sH_2}$, called the {\em closed tensor product}, be the topological closure of $\sH_1 \ot \sH_2$ in the $h$-adic topology of $\Uh\ho \Uh$.
\begin{proposition} For each of $A= A_L, A_R$ one has
$$\boldmu (\overline {A \ot A}) \subset A, \quad \Delta(A) \subset \overline {A \ot A}, \quad S(A) \subset A.$$

This means, each of $A_L,A_R$ is a Hopf algebra in the category where the completed tensor product is replaced by the closed tensor product.
\end{proposition}
\begin{remark}
When the ground ring is a field, the fact that both $A_L, A_R$ are Hopf subalgebras is proved in \cite{Radford}. Here we modify the proof in \cite{Radford} for the case when the ground ring is $\Ch$.
\end{remark}
\begin{proof} We prove the proposition for $A=A_L$ since the case $A=A_R$ is quite analogous.

Let $\tcR'(\bn) = f(\bn) \, h^{\| \bn\| } \cR'(\bn)$. Then $\cR = \sum _{\bn} \tcR'(\bn) \ot \cR''(\bn)$.
Using the defining relation  $(\Delta\ot \id)(\cR)= \cR_{13} \cR_{23}$, we have
 \be
\label{eq.ne1} \sum_\bn \Delta(\tcR'(\bn)) \ot \cR''(\bn) = \sum_{\bk, \bm} \tcR'(\bm) \ot \tcR'(\bk)  \ot \cR''(\bm) \cR''(\bk).
 \ee

 Since $\{ \cR''(\bn)\}$ is a topological basis of $\Uh^{\ge 0}$, there are structure constants $f_{\bm,\bk}^\bn \in \Ch$ such that
$$ \cR''(\bm) \cR''(\bk) = \sum_{\bn} f_{\bm,\bk}^\bn \cR''(\bn),$$
and the right hand side converges. Using the above in \eqref{eq.ne1}, we have
$$ \Delta(\tcR'(\bn)) = \sum_{\bm,\bk} f_{\bm,\bk}^\bn \tcR'(\bm) \ot \tcR'(\bk),$$
with the right hand side convergent in the $h$-adic topology of $\Uh \ho \Uh$. This proves
$\Delta(A_L) \subset \overline {A_L \ot A_L}$. Actually, we just proved that the co-product in $A_L$ is dual to the product in $\Uh^{\ge 0}$.

Similarly, using $(\id \ot \Delta) (\cR) = \cR_{13} \cR_{12}$, one can easily prove that the product in $A_L$ is dual to the co-product in $\Uh^{\ge 0}$, i.e.
$$ \tcR'(\bm) \,  \tcR'(\bk) = \sum_{\bn} f^{\bm,\bk}_\bn \tcR'(\bn), \quad \text{where } \Delta(\cR''(\bn)) = \sum_{\bm,\bk} f^{\bm,\bk}_\bn \cR''(\bm) \ot \cR''(\bk).$$
This proves that $\boldmu(\overline {A_L \ot A_L}) \subset A_L$.

Next we consider the antipode. We have $(S \ho \id)(\cR) = (\id \ho S^{-1}) (\cR) = \cR^{-1}$. Let $A'_L$ be the left image of $\cR^{-1}$.

Since $S^{-1}$ is an $\Ch$-module automorphism of $\Uh^{\ge 0}$, Identity $(\id \ho S^{-1}) (\cR) = \cR^{-1}$ shows that $A'_L= A_L$.

 Identity $(S\ho \id) (\cR) = \cR^{-1}$ shows that $A'_L= S(A_L)$. Thus, we have $A_L= S(A_L)$.
\end{proof}

Proposition \ref{r.41a} follows immediately from the following.
\begin{proposition}
(a) One has $\overline{A_R A_L} = \overline{A_L A_R}$. It follows that $\overline{A_L A_R}$ is a  Hopf algebra with closed tensor products.

(b) One has $\overline{A_R A_L}= \Vh$.
\end{proposition}
\begin{proof}
We use the following identity in a ribbon Hopf algebra: for every $y \in \Uh$ one has
\begin{align}
\label{eq.ne2}  \cR (y \ot 1) & = \sum_{(y)} (y_{(2)} \ot y_{(1)}) \cR (1 \ot S( y_{(3)})) \\
\label{eq.ne21}  (y \ot 1)  \cR & = \sum_{(y)} (1 \ot S( y_{(1)})) \cR   (y_{(2)} \ot y_{(3)})  \\
\label{eq.ne22}  \cR (1 \ot y) & = \sum_{(y)} (y_{(3)} \ot y_{(2)}) \cR ( S^{-1}( y_{(1)} \ot 1)) \\
\label{eq.ne23}  (1 \ot y) \cR & = \sum_{(y)}  ( S^{-1}( y_{(1)} \ot 1)) \cR (y_{(1)} \ot y_{(2)}),
\end{align}
which are Identities (6)--(9) in  \cite{Radford}.
Suppose $x \in A_L$ and $y\in A_R$. We will show that $xy \in \overline{A_R A_L}$. This will imply that $\overline{A_L A_R} \subset \overline{A_R A_L}$. We only need the fact that $A_R$ is a co-algebra in the closed category: $\Delta(A_R) \subset \overline{A_R \ot A_R} \subset \Uh^{\ge 0} \ho \Uh^{\ge 0}$.

Since $x\in A_L$, we have a presentation
$$ x = \sum _{\bn \in \BN^{t+\ell}} x_\bn \tcR'(\bn), \quad x_\bn \in \Ch\ \forall \bn \in \BN^{t+\ell}.$$
Let $p: \Uh^{\ge 0} \to \Ch$ be the unique $\Ch$-module homomorphism such that $p(\cR''(\bn)) = x_\bn$. Then
$ x = \sum_\bn \tcR'(\bn) \, p(\cR''(n))$. Hence
\begin{align*}
 xy &= \sum_{\bn} \tcR'(\bn) y \, p(\cR''(n)) \\
    &= \sum_{\bn} y_{(2)}\tcR'(\bn) \, p(   y_{(1)}  \, \cR''(n) \, S(y_{(3)})) \in \overline{A_R A_L}.
\end{align*}

Similarly, one can prove $\overline{A_RA_L} \subset \overline{A_L A_R}$, and conclude that $\overline{A_LA_R}=\overline{A_RA_L}$.

(b) The two sets $\{ h^{\|\bn\|} H^{\bn} \mid \bn\in \BN^\ell \}$ and $\{h^{\|\bn\|} \breve H^{\bn} \mid \bn\in \BN^\ell \}$ span the same $\Ch$-subspace of $\Uh^0$.
Using spanning sets \eqref{eq.basisALR}, we see that $\overline{A_LA_R}$ is the topological closure of the $\Ch$-span of $$\{ h^{\| \bn_1\| + \| \bn_2\| + \| \bn_3\| } F^{(\bn_1)} H^{\bn_2} E^{(\bn_3)} \mid \bn_1,\bn_3\in \BN^t, \bn_2 \in \BN^\ell\}.$$
Comparing this set with  the formal basis \eqref{eq.hbasis} of $\Vh$, one can easily show that $\Vh= \overline{A_L A_R}$.
\end{proof}

\section{Integral duality} \label{sec:u0}
\subsection{Decomposition of $\UZ^{\ev,0}$} \label{sec:as12}
Recall that
\begin{align*}      %
\Uq^{\ev,0}  = \BC(v)[K_\al ^{\pm 2}, \ \al \in \Pi], \quad
\VA^{\ev,0}   = \cA[K_\al ^{\pm 2}, \ \al \in \Pi].
\end{align*}
For simple root $\al\in\Pi$,
 {\em the even $\al$-part of $\UZ^{0}$} is defined to be $\cI_\al:= \BC(v)[K_\al ^{\pm 2}] \cap \UZ^{0}$.
 Note that $\cI_\al$ is an $\cA$-Hopf-subalgebra of $\BQ(v)[K_\al^{\pm2}]$.
From  Proposition \ref{prop.basis}, $\cI_\al$ is $\cA$-spanned by
 \be
 \label{eq.inbas}
 \{\frac{K_\al^{2m}(q_\al^n K_{\al}^2 ; q_\al)_k}{(q_\al;q_\al)_k}  \mid m,n \in \BZ, k \in \BN \},
 \ee
 and there is an isomorphism
 \be
\label{eq.iso1}
 \bigotimes_{\al\in \Pi} \cI_\al \overset \cong \longrightarrow \UZ^{\ev,0}, \quad \bigotimes_{\al\in \Pi} a_\al \to \prod_{\al} a_\al.
 \ee

 Hence, if one can find  $\cA$-bases for $\cI_\al$, then one can combine them together using \eqref{eq.iso1} to get an $\cA$-basis for $\UZ^{\ev,0}$.

 Similarly, let $\VZ^{\ev,0} \cap \BC(v)[K_\al ^{\pm 2}] = \cA[K_\al^{\pm2}]$ be the {\em even $\al$-part} of $\VZ^{\ev,0}= \cA[K_\al ^{\pm 2}, \ \al \in \Pi]$. The analog of \eqref{eq.iso1} is much easier for $\VZ^{\ev,0}$, since in this case it is
\be
  \label{eq.iso2}\bigotimes_{\al\in \Pi}  \cA[K_\al^{\pm2}]  \overset \cong \longrightarrow \VZ^{\ev,0}= \cA[K_\al ^{\pm 2}, \ \al \in \Pi], \quad \bigotimes_{\al\in \Pi}a_\al \to \prod_{\al} a_\al.
  \ee

\subsection{Bases for $\cI_\al$ and $\cA[x^{\pm 1}]$}\label{sec:u1}
 Fix $\al \in \Pi$, and denote by $x= K_\al^2$, and $y= \brK_\al^2$. The even $\al$-part of $\VZ^{\ev,0}$ is $\cA[x^{\pm1}]$, and  $\cI_\al$, the even $\al$-part of $\UZ^{\ev,0}$, is now an $\cA$-submodule of
 $\BQ(v)[x^{\pm 1}]$. The quantum Killing form restricts to the
 $\BQ(v)$-bilinear form
\be
\label{eq.222}
 \la ., . \ra : \BQvK \ot \BQ(v)[y^{\pm 1}] \to \BQ(v) \quad \text{given by} \quad
 \la x^m, y^{n} \ra= q_\al^{-mn}.
 \ee

Let $\briota: \BQvK \to \BQ(v)[y^{\pm 1}]$ be the $\BQ(v)$-algebra map defined by $\briota(x)=y$. For $n\in \BN$, let
\begin{align}
\label{eq.Qnn}
Q'(\al;n) &:= x^{-\lfloor \frac n2  \rfloor}(q_\al^{-\lfloor \frac {n-1}2 \rfloor} \, x;q_\al)_n
, \quad &\brQ'(\al,n)&:= \briota(Q'(\al;n))
\\
 Q(\al;n) & := \frac{Q'(\al;n)}{(q_\al;q_\al)_n}
  , \quad &\brQ(\al,n)&:= \briota(Q(\al;n))
  .
  \notag
\end{align}

 We will consider  $\cA[x^{\pm 1}] \subset \BC[H_\al][[h]]$ by setting  $x = \exp(hH_\al)$.

\begin{proposition} \label{r.orthog}
(a) The $\cA$-module $\cI_\al$ is the $\cA$-dual of  $\cA[y^{\pm1}]$   with respect to the form \eqref{eq.222} in the sense that
\begin{align*}
 \cI_\al = \{ f(x) \in \BQvK \mid  \la f(x), g(y) \ra \in \cA \quad \forall\ g(y) \in \BQ(v)[y^{\pm 1}] \}.
 \end{align*}

(b) The set $\{{Q'(\al;n)} \mid n \in \BN\}$ is an $\cA$-basis of $\cA[x^{\pm1}]$.

(c)
One has the following orthogonality
\be
\label{eq.ortho1}
 \la Q(\al;n), Q'(\al;m)\ra = \delta_{m,n}\,   q_\al^{-\lfloor (n+1)/2 \rfloor ^2}.
 \ee

(d) The set $\{{Q(\al;n)} \mid n \in \BN\}$ is an $\cA$-basis of $\cI_\al$.

\end{proposition}
\begin{proof} (a) In Section \ref{sec:as12}, $\cI_\al$
is the $\cA$-submodule
of $\BQvK$ spanned by the set \eqref{eq.inbas} with $K_\al^2$ replaced by $x$. This set spans the module of polynomial with $q$-integral values: By \cite[Proposition 2.6]{BCL},
$\cI_\al$ is exactly the set of all
 Laurent polynomials $f(x)\in \BQ(v)[x^{\pm 1}]$ such that $f(q_\al^k)\in \cA=\BZ[v^{\pm 1}]$ for every $k\in \BZ$.

For $f(x)\in \BQvK$, $g(y) \in \BQ(v)[y^{\pm 1}]$, and $k\in \BZ$, from \eqref{eq.222},
\be
\label{eq.sf3}
\la f(x), y^{k}\ra = f(q_\al^{-k}), \quad \la x^{k}, g(y)\ra =  g(q_\al^{-k}).
\ee

Suppose now $f(x) \in \BQ(v)[x^{\pm 1}]$. Since $\{ y^k \mid k\in \BZ\}$ is an $\cA$-basis of $\cA[y^{\pm 1}]$,
\begin{align*}
 f(x) \  \text{is in $\cA$-dual of $\cA[y^{\pm 1}]$} \quad   & \Longleftrightarrow  \quad \la f(x), y^k \ra \in \cA  \ \forall k\in \BZ\\
&  \Longleftrightarrow  \quad  f(q_\al^{-k}) \in \cA \  \forall k\in \BZ   \Longleftrightarrow  f(x) \in \cI_\al.
\end{align*}
This proves part (a).

(b)  The bijective map $j: \BN \to \BZ$ given by
$ j(n) = (-1) ^{n+1} \lfloor \frac{n+1} 2 \rfloor$
defines an order on $\BZ$, by $j(0) \leu j(1) \leu j(2)\leu \dots$. This order looks as follows:
$$ 0 \leu 1 \leu -1 \leu  2 \leu -2 \leu 3 \leu -3 \dots$$
We define an order on the set of monomials $\{ x^n \mid n \in \BZ\}$ by $ x^n \leu x^m$ if $ n \leu m$.
Using this order, one can define the leading term of a non-zero Laurent polynomial $f(x) \in \BQ(v)[x^{\pm 1}]$. One can easily calculate the leading term of $Q'(\al;n)$,
\be
\label{eq.211a}
Q'(\al;n)= (-1)^n x^{j(n)} + \text{ lower order terms}.
\ee
It follows that $\{Q'(\al;n) \mid n\in \BN \}$ is an $\cA$-basis of $\cA[x^{\pm 1}]$.

(c)  Suppose $m< n$. By \eqref{eq.sf3},
$$ \la Q'(\al;n), y^{j(m)} \ra = Q'(\al;n)\vert_{x= q_\al^{-j(m)}}= 0,$$
since $x= q_\al^{-j(m)}$ annihilates one of the factors of $Q'(\al;n)$ when $m <n$. By expanding $\brQ'(\al;m)$ using
 \eqref{eq.211a}, we have
$$ \la Q'(\al;n), \brQ'(\al;m)  \ra = 0 \quad \text{ if $m<n$}.$$

Similarly, one also has $ \la Q'(\al;n), \brQ'(\al;m)  \ra = 0$ if $m>n$. It remains to consider the case $m=n$. Using \eqref{eq.211a}, we have
$$ \la Q'(\al;n), \brQ'(\al;n)  \ra =  \la  Q'(\al;n), (-1)^n y^{j(n)}  \ra = (-1)^n \, Q'(\al;n)|_{x=q_\al^{-j(n)}} =q_\al^{-\lfloor (n+1)/2 \rfloor ^2}\, (q_\al;q_\al)_n,$$
where the last identity follows from an easy calculation. This proves part (c).

(d) By part (b), $\{ \brQ'(\al;n)\mid n \in \BN\}$ is an $\cA$-basis of $\cA[y^{\pm 1}]$.
Because $\cI_\al$ is the $\cA$-dual of $\cA[y^{\pm 1}]$ with respect to the form \eqref{eq.222}, the orthogonality \eqref{eq.ortho1}  shows that $\{ Q(\al;n) \mid n \in \BN\}$ is an $\cA$-basis of $\cI_\al$. This proves part (d).
\end{proof}

\subsection{Proof of Proposition \ref{r.bases5}}
\begin{proof}
(a) The  definition  \eqref{eq.Qn14} means that, for $\bn= (n_1,\dots, n_\ell)\in \BN^t$,
\be
\label{eq.117a}
Q^\ev(\bn) := \prod_{j=1}^\ell Q(\al_j;n_j)|_{x= K_j^2},\quad (q;q)_\bn\, Q^\ev(\bn)= \prod_{j=1}^\ell Q'(\al_j;n_j)|_{x= K_j^2}.
\ee
By Proposition \ref{r.orthog}(d), $\{ Q(\al_j;n)|_{x= K_j^2} \mid n \in \BN\}$ is an $\cA$-basis of $\cI_{\al_j}$.
Hence the isomorphism \eqref{eq.iso1} shows that $\{ Q^\ev(\bn) \mid \bn \in \BN^\ell\}$ is an $\cA$-basis of $\UZ^{\ev,0}$.

Similarly, Proposition \ref{r.orthog}(b) and isomorphism \eqref{eq.iso2} shows that $\{ (q;q)_\bn Q^\ev(\bn) \mid \bn \in \BN^\ell\}$ is an $\cA$-basis of $\VA^{\ev,0}$.

(b) Let $K^{\bode}= \prod_j {K_{\al_j}^{\delta_j}}$ for $\bode=(\delta_1,\dots,\delta_\ell)$. We have
$$\VZ^0 = \bigoplus_{\bode \in \{ 0,1\}^\ell}  K^{\bode} \VZ^{\ev,0}, \quad  \UZ^0 = \bigoplus_{\bode \in \{ 0,1\}^\ell}  K^{\bode} \UZ^{\ev,0},$$
where the first identity is obvious and the second follows from Proposition \ref{prop.basis}.
 Hence (b) follows from (a). This completes the proof of Proposition \ref{r.bases5}.
\end{proof}

\subsection{Proof of Lemma \ref{r.117}}
\begin{proof}
For $\al,\beta\in \Pi$, $\la K^2_\al, \brK^2_\beta \ra = \delta_{\al,\beta}\, q_\al$. Hence, with $Q^\ev(\bn), \brQ^\ev(\bm)$ as in \eqref{eq.117a},
$$ \la Q^\ev(\bn), \brQ^\ev(\bm)\ra = \prod_{j=1}^\ell\la Q(\al_j;n_j), \brQ(\al_j;m_j) \ra = \delta_{\bn,\bm} \prod_{j=1}^\ell q_j^{- \lfloor(n_j+1)/2 \rfloor^2}, $$
where the last identity follows from Proposition \ref{r.orthog}(c). This proves Lemma \ref{r.117}.
\end{proof}

\section{On the existence of the WRT invariant}

Here we prove Proposition \ref{pKirby} on the existence of strong Kirby colors at every level $\zeta$ such that
$\ord(\zeta^{2\cD }) > d (h^\vee -1)$. We also determine when $\zeta \in \ZZ'_\fg$ and when  $\zeta\in \ZZ'_{P\fg}$, if $\ord(\zeta^{2\cD }) > d (h^\vee -1)$.
\subsection{Criterion for non-vanishing of Gauss sums} Suppose $\mathfrak A $ is a free abelian group of rank $\ell$ and
$\phi: \mathfrak A  \times \mathfrak A  \to \BZ$ is a symmetric $\BZ$-bilinear form. Assume further $\phi$ is even in the sense that $\phi(x,x)\in 2 \BZ$ for every $x \in \mathfrak A$.

The {\em quadratic Gauss sum associated to $\phi$ at level $m\in \BN$}
is defined by
$$\mathfrak G_\phi(m):= \sum_{x\in \mathfrak A /m \mathfrak A } \exp\left (\pi i \frac{\phi(x,x)}{m}\right).$$
Let $\mathfrak A_\phi^*$ be the $\BZ$-dual of $\mathfrak A$ with respect to $\phi$, and
$$ \ker_\phi(m) := \{ x\in \mathfrak A  \mid \phi(x,y) \in m \BZ\ \forall y\in\mathfrak A \} = m\mathfrak A ^*_\phi \cap \mathfrak A .$$

We have the following  well-known criterion for the vanishing of $\mathfrak G_\phi(m)$, see \cite[Lemma 1]{Deloup}.
\begin{lemma} \label{l656}
 (a)  If $m$ is odd, then $\mathfrak G_\phi(m) \neq 0$.

(b)  $\mathfrak G_\phi(m) \neq 0$ if and only for every $x \in \ker_\phi(m)$ one has
$ \frac{1}{2m} \phi(x,x) \in \BZ.$

\end{lemma}
\begin{lemma}\label{l822}
For every $x \in \ker_\phi(m)$, $ \frac{1}{2m} \phi(x,x) \in \frac12\BZ.$
\end{lemma}
\begin{proof} Because $x \in m\mathfrak A ^*_\phi$, one has $\phi(x,x) \in m\BZ$. Hence $ \frac{1}{2m} \phi(x,x) \in \frac12\BZ$.
\end{proof}

\subsection{Gauss sums on weight lattice}
Recall that $X,Y$ are respectively the weight lattice  and the root lattice  in ${\mathfrak h}^*_\BR$, which is equipped with the invariant inner product.
The $\BZ$-dual $X^*$ of $X$ is $\BZ$-spanned by $\al/d_\al, \al \in \Pi$.
\begin{lemma}For $y \in X^*$, we have $(y,y)\in \Ztwo:= \{ a/b \mid a,b \in \BZ,\ b \text{ odd}\}$.
\label{l823}
\end{lemma}
\begin{proof} Suppose $y = \sum k_i \al_i /d_i$. Then
$$(y,y) = \sum_i k_i^2 \frac{(\al_i,\al_i)}{d_i^2}+ 2 \sum_{i<j}\frac{(\al_i,\al_j)}{d_i d_j}= \sum_i k_i^2 \frac{2}{d_i}+ \sum_{i<j}\frac{2(\al_i, \al_j)/d_j}{d_i}\in \frac{2}{d}\BZ.$$
Since $d=1,2$ or $3$, we see that $(y,y) \in \Ztwo$.
\end{proof}
\begin{lemma} \label{l825}

Suppose $\zeta$ is a root of $1$ of order $s$. Let $r= s/\gcd(s,2\cD )$ be the order $\xi=\zeta^{2\cD }$.

  (a) Suppose $r$ is odd. Then $\mathfrak G^{\P\fg}(\zeta)\neq 0$, where

$$ \mathfrak G^{\P\fg}(\zeta):= \sum_{\lambda \in P_\zeta \cap Y} \zeta^ {D (\lambda, \lambda + 2\rho)}= \sum_{\lambda \in P_\zeta \cap Y} \xi^ {(\lambda, \lambda + 2\rho)/2} .$$

(b) Suppose $r$  is even. Then  $\mathfrak G^{\fg}(\zeta)\neq 0$, where
$$ \mathfrak G^{\fg}(\zeta):= \sum_{\lambda \in P_\zeta} \zeta^ {D (\lambda, \lambda + 2\rho)}.$$
\end{lemma}
\begin{proof} After a Galois transformation of the from $\zeta \to \zeta^k$ with $\gcd(k,s) =1$ we can assume that $\zeta = \exp(2 \pi i/s)$.

(a) The following is the well-known completing the square trick:
\begin{align*} \mathfrak G^{\P\fg}(\zeta)& = \sum_{\lambda \in P_\zeta \cap Y} \xi^ {\frac12 (\lambda, \lambda + 2\rho(r+1))} \quad \text{ since } \ord(\xi)= r\\
&= \xi^{-  \frac{(r+1)^2 }{2}(\rho,\rho) } \sum_{\lambda \in P_\zeta \cap Y} \xi^ {\frac12 (\lambda +  (r+1)\rho , \lambda + (r+1)\rho)}\\
&= \xi^{-  \frac{(r+1)^2 }{2}(\rho,\rho) }  \sum_{\lambda \in P_\zeta \cap Y} \xi^ {\frac12 (\lambda  , \lambda )}.
\end{align*}
Here the last identity follows because $2\rho \in Y$ and hence $(r+1)\rho \in Y$ since $r+1$ is even, and because the shift $\lambda \to \lambda + \beta$ does not change the Gauss sum
for any $\beta \in Y$.

The expression $\xi^{\frac12(\lambda, \lambda)}, \lambda \in Y$ is invariant under the translations by vectors in both $rY$ and $2rX$. Hence
\begin{align*}
 \mathfrak G^{\P\fg}(\zeta) & = \xi^{-  \frac{(r+1)^2 }{2}(\rho,\rho) }  \sum_{\lambda \in P_\zeta \cap Y} \xi^ {(\lambda  , \lambda)/2 }\\
 &= \xi^{-  \frac{(r+1)^2 }{2}(\rho,\rho) }  \frac{\vol(2rX)}{\vol(rY)} \sum_{\lambda \in Y/rY } \xi^ {(\lambda  , \lambda )/2}
 \end{align*}

By Lemma \ref{l656}(a) with $\mathfrak A  =Y$, $\phi(x,y) = (x,y)$, and $m=r$, the right hand side is non-zero.

(b) Again using the completing the square trick we get

\begin{align} \mathfrak G^{\fg}(\zeta)= \zeta^{-D (\rho,\rho)}\sum_{\lambda \in P_\zeta} \zeta^ {D (\lambda, \lambda)}& = \zeta^{-D (\rho,\rho)}\sum_{\lambda \in X/2rD X} \exp\left( \frac{\pi i}{s} {2\cD  (\lambda, \lambda)}\right) \notag\\
&= \zeta^{-D (\rho,\rho)}\left( \frac{2\cD r}{s}\right)^\ell \, \sum_{\lambda \in X/sX} \exp\left( \frac{\pi i}{s} {2\cD  (\lambda, \lambda)}\right).
\label{e771}
\end{align}

Note that $\frac{s}{\gcd(s,2\cD )}$ is even if and only if
\be \frac{s}{4D}\in \Ztwo.\label{e823}
\ee

Apply Lemma \ref{l656}(b) with $\mathfrak A  = X$, $\phi(x,y) = 2\cD (x,y)$, and $m=s$. Then $s \mathfrak A _\phi^* = \frac{s}{2\cD } X^*$.
Suppose $x\in \ker_\phi(s) = s \mathfrak A _b^* \cap \mathfrak A  \subset s \mathfrak A _b^*$.
Then $x = \frac{s}{2\cD } y$ with $y \in X^*$.

We have

$$ \frac{1}{2s} \phi(x,x)= \frac{s}{4D} (y,y) \in \Ztwo,$$
where the last inclusion follows from \eqref{e823} and Lemma \ref{l823}.
  From Lemma \ref{l822}  we have
$$ \frac{1}{2s} \phi(x,x) \in \frac12 \BZ \cap \BZ_{(2)} = \BZ.$$
By Lemma \ref{l656}(b), the right hand side of \eqref{e771} is non-zero.
\end{proof}

\subsection{Proof of Proposition \ref{pKirby}} \label{sec:ppKirby}
\begin{proof}[Proof of Proposition \ref{pKirby}] By \cite[Proposition 2.3 \& Theorem 3.3]{Le:quantum},   $\Omega^\fg(\zeta)$ and $\Omega^{P\fg}(\zeta)$ are strong handle-slide colors.
Although the formulation in \cite{Le:quantum} only says that $\Omega^\fg(\zeta)$ and $\Omega^{P\fg}(\zeta)$ are handle-slide colors, the proofs there actually show that $\Omega^\fg(\zeta)$ and $\Omega^{P\fg}(\zeta)$ are strong handle-slide colors.

It remains to show that $J_{U_\pm}(\Omega^\fg(\zeta))\neq 0$ if $r$ is even, and  $J_{U_\pm}(\Omega^{P\fg}(\zeta)) \neq 0$ if $r$ is odd.

  From  \cite[Section 2.3]{Le:quantum}, with the assumption  $\ord(\zeta^{2\cD } )> d(h^\vee -1)$, we have
\be
 J_{U_+}(\Omega^\fg(\zeta))  \overset {(\zeta)}{=} \frac{\mathfrak G^\fg(\zeta)}{\prod_{\al \in \Phi_+}(1- \xi^{(\al,\rho)})}, \quad  J_{U_+}(\Omega^{P\fg}(\zeta))  \overset {(\zeta)}{=} \frac{\mathfrak G^{P\fg}(\zeta)}{\prod_{\al \in \Phi_+}(1- \xi^{(\al,\rho)})} \label{e825}.
 \ee
Besides, $J_{U_-}(\Omega^\fg(\zeta))$ and $J_{U_-}(\Omega^{P\fg}(\zeta)) $ are respectively  the complex conjugates of $J_{U_+}(\Omega^\fg(\zeta))$ and $J_{U_+}(\Omega^{P\fg}(\zeta)) $.
By Lemma \ref{l825}, if $\ord(\zeta^{2 \cD})$ is even then $J_{U_+}(\Omega^\fg(\zeta))\neq 0$, and if  $\ord(\zeta^{2 \cD})$ is odd then $J_{U_+}(\Omega^{P\fg}(\zeta))\neq 0$. This completes the proof of Proposition \ref{pKirby}.
\end{proof}

\subsection{The sets $\ZZ'_\fg$ and $\ZZ'_{P\fg}$ for each simple Lie algebra} %

\begin{proposition}\null (a) One has $\mathfrak G^{\fg}(\zeta)=0$ in and only in the following cases:

\begin{itemize} \item
 $\fg= A_\ell$ with  $\ell$  odd and $\ord(\zeta)  \equiv 2 \pmod 4$.

\item $\fg= B_\ell$ with $\ell$ odd and $\ord(\zeta) \equiv 2 \pmod 4$.

\item $\fg= B_\ell$ with $\ell \equiv 2 \pmod 4$ and $\ord(\zeta)  \equiv 4 \pmod 8$.

\item $\fg= C_\ell$ and $\ord(\zeta)  \equiv 4 \pmod 8$.

\item $\fg= D_\ell$ with $\ell$ odd and $\ord(\zeta) \equiv 2 \pmod 4$.

\item $\fg= D_\ell$ with $\ell \equiv 2 \pmod 4$ and $\ord(\zeta)  \equiv 4 \pmod 8$.

\item $\fg= E_7$  and $\ord(\zeta)  \equiv 2 \pmod 4$.
\end{itemize}

(b) In particular, if $\ord(\zeta)$ is odd or $\ord(\zeta)$ is divisible by $2d\cD$, then $\mathfrak G^{\fg}(\zeta)\neq 0$.
\label{pro.Gauss}
\end{proposition}

The proof is a careful, tedious, but not difficult check of the vanishing of the Gaussian sum using Lemma \ref{l656} and the explicit description of the weight lattice for each simple Lie algebra, and we drop the details.

\begin{corollary} Suppose $\zeta \in \ZZ$ with $\ord(\zeta^{2\cD}) > d(h^\vee -1)$. Then $\zeta \in \ZZ'_\fg$ if and only if
 $\zeta$ satisfies the condition of
Proposition \ref{pro.Gauss}(a).
\label{cor.Gauss}
\end{corollary}

Similarly, using Lemma \ref{l656}, one can prove the following.
\begin{proposition}\null Let $r= \ord(\xi)= \ord(\zeta^{2\cD})$.

 (a) One has $\mathfrak G^{P\fg}(\zeta)=0$ in and only in the following cases:

\begin{itemize} \item
 $\fg= A_\ell$ and  $\ord_2(r) = \ord_2 (\ell+1) \ge 1$.

\item $\fg=B_\ell$  and $r  \equiv 2 \pmod 4$.

\item $\fg= C_\ell$, $r$ even and  $ r\ell \equiv 4 \pmod 8$.

\item $\fg= D_\ell$, $r$ even and  $ r\ell \equiv 4 \pmod 8$.

\item $\fg= E_7$  and $r  \equiv 2 \pmod 4$.

\end{itemize}
Here $\ord_2(n)$ is the order of $2$ in the prime decomposition of the integer $n$.

(b) In particular, if $r$ is co-prime with $2^{\ord_2(\cD)}$, then $\mathfrak G^{P\fg}(\zeta)\neq 0$.

\label{pro.Gauss2}

\end{proposition}

\begin{corollary}
 Suppose $\ord(\zeta^{2\cD}) > d(h^\vee -1)$. Then $\zeta \in \ZZ'_{P\fg}$ if and only if
 $\zeta$ satisfies the condition of
Proposition \ref{pro.Gauss2}(a).
\end{corollary}

\np

\section{Table of notations}

\begin{tabular}{|l|l|l|}
	\hline
Notations &  defined in   & remarks  \\
	\hline
$\Zqh$, $(x;q)_n$ & \ref{sec.Habiro} & \\
$\left( \Ch^I\right)_0 $ & \ref{sec:top1} & \\
$\sH, \boldmu,\boldsymbol {\eta},\Delta,\boldsymbol \epsilon,S$   &  \ref{sec:ribbon-hopf-algebra}             & Hopf algebra \\
$\cR$    &   \ref{sec:ribbon-hopf-algebra}, \ref{Rmatrix2}           &  $R$-matrix\\
$ \br$  &     \ref{sec:ribbon-hopf-algebra}, \ref{Rmatrix2}         & ribbon element \\
$\modg$   &     \ref{sec:ribbon-hopf-algebra}          & balanced element\\
$\ad(x \ot y), x \tri y$  & \ref{sec:adjaction} & adjoint action \\
$\tr_q^V$ &\ref{sec:adjaction}  & quantum trace \\
$J_T$   &   \ref{sec:uni11}            & universal invariant of bottom tangle \\
${\boldsymbol{\psi}}$, $\bD,\bS$ & \ref{transmutation} & braiding, transmutation\\
$\modc,\modc^-, C^+, C^-$  &  \ref{sec:clasp-bottom-tangle}  & clasp, $\modc = J_{C^+}$  \\
$\cT_\pm$ & \ref{sec.simple}, \ref{sec.twist1} & full twist forms \\
$J_M$ &  \ref{sec.partii}, \ref{sec.JM} & invariant of 3-manifold \\
$\Upsilon$ & \ref{sec:borr-tangle-braid}& braided commutator\\
$\modb $ & \ref{sec.partii}   & universal invariant of Borromean tangle\\
$ \sL(x \ot y)$, $\la x, y \ra$ & \ref{sec:coresub}, \ref{sec:rosso-form} & clasp form \\
$\fg,\ell,  \modh$ & \ref{sec.UhUq1} &  Lie algebra, its rank, Cartan subalgebra\\
$d, d_\al, t, \Ht(\gamma)$ & \ref{sec.UhUq1} &   \\
$X, Y$ & \ref{sec.UhUq1} & weight lattice, root lattice \\
$\Pi, \Phi, \Phi_+$ & \ref{sec.UhUq1} &  simple roots, all roots, positive roots\\
$\rho, \al_i, \bral_i$ & \ref{sec.UhUq1} &  \\
$h,v,q,v_\al , q_\al, \cA$ & \ref{sec.UhUq2} & $q=v^2= \exp(h)$, $\cA= \BZ[v^{\pm1}]$\\
$[n]_\al, \{ n\}_\al,  [n]_\al!, \{ n\}_\al!, \qbinom n k _\al$ & \ref{sec.UhUq2} & \\
$\Uh, F_\al, E_\al, H_\lambda, F_i, E_i$ & \ref{subsub.UU} & \\
$K_\lambda, \brK_\al, K_i$ & \ref{subsub.UU} & \\
$\Uq, \brU_q, \brU_q^0$ & \ref{sec.UU2} & \\
$\ibar, \varphi,\omega,\tau$ & \ref{sec.8147} & (anti) automorphisms of $\Uh$\\
$|x|$ & \ref{sec:y-grading} & $Y$-grading\\
$\Uq^\ev$, $\brU_q^\ev$  & \ref{even-grading} & even grading\\
$\Uh^\pm,\Uh^0, \Uq^\pm, \Uq^0$, $\Uq^{\ev,-}, \Uq^{\ev,0}$  & \ref{sec.tri1} & \\
$\fW, s_\al, s_i$  & \ref{503} & Weyl group, reflection\\
$T_\al$  & \ref{503} & braid group action\\
$E_\gamma, F_\gamma, E^{(\bn)}, F^{(\bn)}, K_\bn $  & \ref{sec.PBW} & \\
$\Theta, E_\bn, F_\bn, E'_\bn, F'_\bn $  & \ref{Rmatrix} & \\
$\cD, \brH_\al, \br_0$  & \ref{Rmatrix2} & \\
$\Gamma$  & \ref{Rosso-form} & quasi-clasp element\\
$\Uhh$  & \ref{sec:corealgebra} &  $\UUh := \Uh \ho _{\Ch} \Chh$\\
$\|\bn \|, \bbe_h(\bn), \Vh, \Vh^{\bar \ot n}$  & \ref{sec.Vh} & \\
$\Xh $  & \ref{sec:sX} & core subalgebra of $\UUh$\\
$\tA $  & \ref{sec.dilita} & \\

	\hline
\end{tabular}

\np

\begin{tabular}{|l|l|l|}
	\hline
Notations &  defined in   & remarks  \\
	\hline
$\UZ,\UZ^\pm, \UZ^0, \UZ^\ev, \UZ^{\ev,-}, \UZ^{\ev,0} $  & \ref{sec:Lus} & \\
$\VZ,\VZ^\pm, \VZ^0, \VZ^\ev, \VZ^{\ev,-}, \VZ^{\ev,0} $  & \ref{sec:algebra-cv} & \\
$(q;q)_\bn $  & \ref{sec:algebra-cv}, \ref{sec:pre-bas}, \ref{sec:basesUZ} & \\
$Q^\ev(\bn), Q(\bn,\bode) $  & \ref{sec:pre-bas} & \\
$\bbe^\ev(\bn), \bbe(\bn,\bode) $  & \ref{sec:basesUZ} & \\
$\brUZ, \brUZ^0, \brUZ^\ev, \brUZ^{\ev,0} $  & \ref{sec:simply} & \\
$\brbe^\ev(\bn), \brbe(\bn,\bode) $  & \ref{sec:simply} & \\
$\XZ, \XZ^\ev $  & \ref{sec:XZ} & integral core subalgebra\\
$G, G^\ev, \dv,\dK_\al, \de_\al, \dot x, [\Uq]_g $  & \ref{sec:grad1} & \\
$ [\Uq^{\ot n}]_g $  & \ref{sec:gradTotal} & $G$-gradings\\
$ \modK_n, \tK_n, \cF_k(\modK_n) $  & \ref{sec:Kn1} & \\
$ \modK_n(\cU), \tK_n(\cU), \cF_k(\modK_n(\cU)) $  & \ref{sec:Kn2} & \\
$\max(\bn), o(\bn) $  & \ref{sec:Kn4} & \\
$\ZZ, \ZZ_\fg, h^\vee, D  $  & \ref{sec.WRT01} & \\
$\dim_q(V), U  $  & \ref{WRT1} & \\
$\ev_{v^{1/D}= \zeta}(f),  f \eqzeta g  $  & \ref{WRT1} & \\
$\cB, U_\pm, \tau_M(\Omega) $  & \ref{sec.Kirby} & $\cB= \BC[v^{\pm 1}]$\\
$ \tau^\fg, \tau^{P\fg}, \ZZ'_\fg, \ZZ'_{P\fg} $  & \ref{sec:st.Kirby} & \\
$ \Omega_\pm $  & \ref{proof01} &  twisted colors\\
$\cU, \cU^\ev$  & \ref{sec:Ur} &  \\
$\modK'_m, \cF_k(\modK'_m), \tK'_m$  & \ref{sec:Kmprime} &  \\
$\tT_\pm$  & \ref{n5013} &  \\
$\Zc(\Uh), \Zc(V), \chi, \sh_\mu$  & \ref{sec.HC1} &  \\
$z_\lambda$  & \ref{sec:Drin} &  \\
$\DDD, \ddd $ & \ref{sec.DDD}  &\\
	\hline
\end{tabular}


\begin{thebibliography}{[EMSS]}
\bibitem[An]{An} H. H. Andersen, {\em Quantum group at roots of $\pm1$}, {Comm. Algebra} {\bf 24} (1996), 3269--3282.

\bibitem[AP]{AP}
H. H.  Andersen and J.  Paradowski, {\em Fusion categories arising from semisimple Lie algebras},  { Comm. Math. Phys.}   {\bf  169} (1995), no. 3, 563--588.

\bibitem[BK]{BK}B. Bakalov, A. Kirillov Jr, {\em  Lectures on tensor categories and modular functors},  University Lecture Series, {\bf 21}, American Mathematical Society, Providence, RI, 2001.

\bibitem[Bau]{Baumann} P.  Baumann, {\em  On the center of quantized enveloping algebras}, J. Algebra {\bf 203} (1998),  244--260.

\bibitem[BBlL]{BBL} A. Beliakova, C. Blanchet, and T. T. Q. L\^e, {\em Unified quantum invariants and their refinements for homology $3$-spheres with 2-torsion},  Fund. Math.  {\bf 201}  (2008),   217--239.

\bibitem[BCL]{BCL} A. Beliakova, Q. Chen, and T. T. Q. L\^e,  {\em On the integrality of Witten-Reshetikhin-Turaev $3$-manifold invariants},  Quantum Topol. {\bf 5} (2014), 99--141.

\bibitem[BL]{Beliakova-Le} A. Beliakova and T. T. Q. L\^e, {\em On the unification of quantum 3-manifold invariants}. Introductory lectures on knot theory, 1--21, Ser. Knots Everything, {\bf 46}, World Sci. Publ., Hackensack, NJ, 2012.

\bibitem[BBuL]{Beliakova-Buehler-Le} A. Beliakova, I. B\"uhler, and
  T. T. Q. L\^e, {\em A unified quantum $SO(3)$ invariant for rational homology 3-spheres},  Invent. Math. {\bf 185} (2011), no. 1, 121--174.

  \bibitem[Bl]{Blanchet} C.
  Blanchet,
{\em Hecke algebras, modular categories and 3-manifolds quantum invariants},
Topology {\bf 39} (2000), no. 1, 193--223.

\bibitem[Br]{Bruguieres}
A. Brugui\`eres,  {\em Cat\'egories pr\'emodulaires, modularisations et invariants des vari\'et\'es de dimension $3$}, {\em  Math. Ann.} {\bf 316} (2000), 215--236.

\bibitem[Cal]{Caldero}
P. Caldero {\em \'El\'ements ad-finis de certains groupes quantiques},  C. R. Acad. Sci. Paris S\`er. I Math. {\bf 316} (1993), 327--329.

\bibitem[CP]{CP} V. Chari and A. Pressley, {\em A guide to quantum groups}, Cambridge University Press, 1994.

\no{
\bibitem[CX] {CX} V. Chari and  N. Xi,  Monomial bases of quantized enveloping algebras, in  {\em ``Recent developments in quantum affine algebras and related topics (Raleigh, NC, 1998)''}, Contemp. Math. {\bf 248} (1999), 69--81.
    }

\bibitem[Che] {Chevalley}
C. Chevalley,
{\em Invariants of finite groups generated by reflections},  Amer. J. Math.
{\bf 77}
(1955), 778--782.

\bibitem[DKP]{DKP} C. De Concini, V. G. Kac, and C. Procesi, {\em Quantum coadjoint action},
J. Amer. Math. Soc. {\bf 5} (1992), 151--189.

\no{
\bibitem[DK]{DK} C. De Concini and V. Kac,  {\em Representation of quantum groups at roots of 1}, { in Progress in Mathematics}, {\bf 92} (1990), 471--509.
    }
\bibitem[DP]{DeConcini-Procesi} C. De Concini and C. Procesi,  {\em Quantum group}, in
 { ``D-modules, representation theory, and quantum groups (Venice, 1992)''},  31--140, Lecture Notes in Math. {\bf 1565}, Springer, Berlin, 1993.




 \bibitem[De]{Deloup} F. Deloup,   {\em Linking form, Reciprocity for Gauss sums  and invariants of $3$-manifolds},  {  Trans. Amer. Math. Soc.} {\bf 351} (1999), 1859--1918.


 \bibitem[Dr]{Drinfeld} V. Drinfeld, {\em Quantum groups},    { Proceedings of the International Congress of Mathematicians}, {\bf 2} (Berkeley, Calif., 1986),  798--820, Amer. Math. Soc., Providence, RI, 1987.

\bibitem[Gav]{Gavarini} F. Gavarini, {\em The quantum duality principle},  Ann. Inst. Fourier {\bf 52} (2002),  809--834.

\bibitem[Gou]{Goussarov}
M. Goussarov,
{\em Finite type invariants and $n$-equivalence of $3$-manifolds},
{\it C. R. Acad. Sci. Paris S\`er. I Math.} {\bf 329} (1999), no. 6, 517--522.

\bibitem[Ha1]{H:claspers}
K. Habiro,
{\em Claspers and finite type invariants of links},
{ Geom. Topol.} {\bf 4} (2000), 1--83.

\bibitem[Ha2]{H:rims2001}
K. Habiro,
{\em On the quantum $sl_2$ invariants of knots and integral homology spheres},
{ Invariants of knots and $3$-manifolds (Kyoto 2001)},
Geometry and Topology Monographs, 4, (2002), 55-68.

\bibitem[Ha3]{H:cyclotomic}
K. Habiro,
{\em Cyclotomic completions of polynomial rings},
{ Publ. Res. Inst. Math. Sci.} {\bf 40} (2004), 1127--1146.

\bibitem[Ha4]{H:bottom}
K. Habiro,
{\em Bottom tangles and universal invariants},  { Algebr. Geom. Topol.}  {\bf 6}  (2006), 1113--1214 (electronic).

\bibitem[Ha5] {H:integralform}
K. Habiro,
{\em An integral form of the quantized enveloping algebra of $sl_2$ and its
completions},  { J. Pure Appl. Algebra} {\bf 211} (2007), 265--292.

\bibitem[Ha6]{H:kirby}
K. Habiro,
{\em Refined Kirby calculus for integral homology spheres},
{ Geom. Topol.} {\bf 10} (2006), 1285--1317.

\bibitem[Ha7] {H:unified} K. Habiro, {\em A unified Witten-Reshetikhin-Turaev invariant for integral homology
spheres},   { Invent. Math.}  {\bf 171}  (2008),   1--81.

\bibitem[He]{Hennings} M. A. Hennings, {\em Invariants of links and
  $3$-manifolds obtained from Hopf algebras},   J. London Math. Soc.
  {\bf 54} (1996), 594--624.

\bibitem[Ho] {Hoste}
J Hoste, {\em A formula for Casson's invariant}, Trans. Amer. Math. Soc. 297 (1986) 547--562.

\bibitem[Hum] {Humphreys}
J. Humphreys, {\em  Introduction to Lie algebras and representation theory}, Graduate Texts in Mathematics {\bf 9},  Springer-Verlag, New York-Berlin, 1972.

\bibitem[Ja]{Jantzen} J.C. Jantzen,
        {\em Lecture on quantum groups},
        Graduate Studies in Mathematics, vol {\bf 6}, AMS  1995.

        \bibitem[JL1]{JL} A. Joseph and G. Letzter, {\em Local finiteness of
the adjoint action for quantized enveloping algebras},  J. Algebra
{\bf 153}  (1992),   289--318.

        \bibitem[JL2]{JL:separation} A. Joseph and G. Letzter, {\em
Separation of variables for quantized enveloping algebras},
Amer. J. Math. {\bf 116} (1994), 127--177.

        \bibitem[JL3]{JL2} A. Joseph and G. Letzter, {\em Rosso's form and quantized Kac Moody algebra},  Mat. Zeitschrift, {\bf 222} (1996), 543--571.





\bibitem[Jo]{Jones}
V. F. R. Jones,
{\em Hecke algebra representations of braid groups and link polynomials},
{ Ann. Math.} {\bf 126} (1987), 335--388.

\bibitem[Kash]{Kashiwara}
M. Kashiwara,
{\em On crystal bases of the $Q$-analogue of universal enveloping algebras}, {
Duke Math. J.}  {\bf 63} (1991), 465--516.

\bibitem[Kass]{Kassel} C. Kassel, {\em Quantum groups}, Graduate Texts in Mathematics, {\bf 155}, Springer-Verlag, New York, 1995.

\bibitem[Kau]{Kauffman} L.H. Kauffman, {\em Gauss codes, quantum groups and
  ribbon Hopf algebras},   Rev. Math. Phys. {\bf 5} (1993) 735--773.

\bibitem[KR]{Kauffman-Radford} L.H. Kauffman and D.E. Radford,
  {\em Invariants of $3$-manifolds derived from finite dimensional Hopf
  algebras},   J. Knot Theory Ramifications {\bf 4} (1995) 131--162.

\bibitem[Ke]{Kerler} T. Kerler, {\em Genealogy of nonperturbative
  quantum-invariants of $3$-manifolds: The surgical family},   {
  ``Geometry and physics'', Lecture Notes in Pure and Applied
  Physics} {\bf 184}, Marcel Dekker (1996), 503--547.

\bibitem[Ki]{Kirby}
R. Kirby,
{\em A calculus for framed links in $S\sp{3}$},
{ Invent. Math.} {\bf 45} (1978), no. 1, 35--56.

\bibitem[KM]{KM}
R. Kirby and P. Melvin,
{\em The $3$-manifold invariants of Witten and Reshetikhin-Turaev for $sl(2,C)$},
{ Invent. Math.} {\bf 105}  (1991), no. 3, 473--545.

\no{
\bibitem[KiRe]{KR}
A. N. Kirillov, N. Reshetikhin,  $q$-Weyl group and a multiplicative formula for universal $R$-matrices,  {\em Comm. Math. Phys.}  {\bf 134}  (1990),   421--431.
}

\bibitem[KT]{KT}
T. Kohno, T. Takata,  {\em Level-rank duality for Witten's $3$-manifold invariant},   { Adv. Studies Pure Math.},
{\bf 24} (1996): Progress in Algebraic Combinatorics, 243--264.

\bibitem[KlS]{KlS}  A. Klimyk and K.  Schm\"udgen, {\em Quantum groups and their representations}, Texts and Monographs in Physics. Springer-Verlag, Berlin, 1997.

\no{
\bibitem[KoS]{KS}
L. Korogodski, Y. Soibelman, {\em  Algebras of functions on quantum groups, Part I},  Mathematical Surveys and Monographs {\bf 56},  American Mathematical Society, Providence, RI, 1998.
}

\bibitem[KLO]{KLO}
T. Kuriya, T. Ohtsuki, and T. T. Q. Le,  {\em The perturbative invariants of rational homology
$3$-spheres can be recovered from the LMO invariant}, { J. Topology,}  {\bf 5} (2012), 458--484.

\bibitem[La1]{Lawrence:90}
R. J. Lawrence,
{\em A universal link invariant},
{ The interface of mathematics and particle physics (Oxford,
  1988)}, Inst. Math. Appl. Conf. Ser. New Ser., 24, 151--156, Oxford
University Press, 1990.

\bibitem[La2]{Lawrence:integrality}
R. J. Lawrence,
{\em Asymptotic expansions of Witten-Reshetikhin-Turaev invariants for some simple $3$-manifolds},
{ J. Math. Phys.} {\bf 36} (1995), 6106--6129.

\no{
\bibitem[LZ]{Lawrence-Zagier}
R. Lawrence and D. Zagier,
Modular forms and quantum invariants of $3$-manifolds.
{\it Asian J. Math.} {\bf 3} (1999), 93--107.
}

\bibitem[Le1]{Le:universal}
T. T. Q. L\^e,
{\em An invariant of integral homology $3$-spheres which is universal for all finite type invariants},
{\it Solitons, geometry, and topology: on the crossroad}, 75--100,
Amer. Math. Soc. Transl. Ser. 2, 179, Amer. Math. Soc., Providence,
RI, 1997.

\bibitem[Le2]{Le_Duke} T. T. Q. L\^e,
         {\em Integrality and symmetry of quantum link invariants},
       {  Duke Math. J.}  {\bf 102} (2000) 273--306.

\bibitem[Le3]{Le:PSUn}
T. T. Q. L\^e,
{\em On perturbative $PSU(n)$ invariants of rational homology $3$-spheres},
{ Topology} {\bf 39} (2000), no. 4, 813--849.

\bibitem[Le4]{Le:quantum} T. T. Q. L\^e,
{\em Quantum invariants of $3$-manifolds: integrality, splitting, and perturbative expansion},
{ Proceedings of the Pacific Institute for the Mathematical
Sciences Workshop ``Invariants of Three-Manifolds'' (Calgary, 1999)}.
{\it Topology Appl.}  {\bf 127} (2003), no. 1-2, 125--152.

\bibitem[Le5]{Le_Strong_Integrality} T. T. Q. L\^e, {\em Strong Integrality of Quantum Invariants of $3$-manifolds}, { Trans. Amer. Math. Soc.} {\bf 360} (2008), 2941--2963.

\bibitem[LMO]{LMO}
T. T. Q. Le, J. Murakami and T. Ohtsuki,
{\em On a universal perturbative invariant of $3$-manifolds},
{\it Topology} {\bf 37} (1998), no. 3, 539--574.

\bibitem[Lu1]{Lusztig}  G. Lusztig,
        {\em Introduction to quantum groups}, Birkhauser, 1993.
\bibitem[Lu2]{Lusztig01}  G. Lusztig, {\em Quantum groups at roots of 1},  Geom. Dedicata {\bf 35} (1990),  89--113.

\bibitem[Lu3]{Lusztig02}  G. Lusztig, {\em Canonical bases arising from quantized enveloping algebras II}, in  ``Common trends in mathematics and quantum field theories (Kyoto, 1990)'',  {\em Progr. Theoret. Phys.} Suppl. No. 102 (1990), 175--201.

\bibitem[Ly]{Lyubashenko} V. V. Lyubashenko, {\em Invariants of $3$-manifolds
  and projective representations of mapping class groups via quantum
  groups at roots of unity},  { Comm. Math. Phys.} {\bf 172} (1995) 467--516.

 \bibitem[Mac]{Macdonald}
  I. G.  Macdonald, {\em Symmetric functions and orthogonal polynomials},  University Lecture Series, {\em 12} (1988), American Mathematical Society, Providence, RI.

\bibitem[Maj1]{Majid1}
S. Majid,
{\em Algebras and Hopf algebras in braided categories},
{\it Advances in Hopf algebras}, Marcel Dekker
  Lec. Notes Pure and Applied Math., 158, (1994) 55--105.

  \bibitem[Maj2]{Majid2}    S.  Majid, {\em Foundations of Quantum Group Theory},  Cambridge University Press, 1995.

  \bibitem[Man]{Manin}Y. I.  Manin, {\em Cyclotomy and analytic geometry over F1},  {  Quanta of maths}, 385--408, Clay Math. Proc., {\bf 11} (2010), Amer. Math. Soc., Providence, RI.
\bibitem[Mar]{Mar}  M. Marcolli, {\em Cyclotomy and endomotives},  {
  $p$-Adic Numbers Ultrametric Anal. Appl.} {\bf 1} (2009), no. 3, 217--263.
\bibitem[MR]{MR}
G. Masbaum and J. D. Roberts,
{\em A simple proof of integrality of quantum invariants at prime roots of unity},
{ Math. Proc. Cambridge Philos. Soc.}  {\bf 121} (1997), no. 3, 443--454.

\bibitem[MW]{MW}
G. Masbaum and H. Wenzl,
{\em Integral modular categories and integrality of quantum invariants at roots of unity of prime order},
{\it J. Reine Angew. Math.} {\bf 505} (1998), 209--235.
\no{
Melvin, P.M., Morton, H.R.: The coloured Jones function. Commun. Math. Phys. 169, 501--520 (1995) MR1328734 (96g:57012)
}

 \bibitem[Mu]{Murakami}  H. Murakami,  {\em Quantum SO(3)-invariants dominate the SU(2)-invariant of Casson and Walker}, { Math. Proc. Camb. Phil.  Soc.} {\bf 117} (1995), 237--249.

\no{
\bibitem{Ohtsuki1}
T. Ohtsuki,
A polynomial invariant of integral homology $3$-spheres.
{\it Math. Proc. Cambridge Philos. Soc.} {\bf 117} (1995), no. 1, 83--112.
}

\bibitem[Oht1]{Ohtsuki:colored} T. Ohtsuki, {\em Colored ribbon Hopf algebras and universal invariants of framed links}, { J. Knot Theory Ramifications} {\bf 2} (1993), no. 2, 211--232.

\bibitem[Oht2]{Ohtsuki:3mfds} T. Ohtsuki, {\em Invariants of $3$-manifolds
  derived from universal invariants of framed link}, {
  Math. Proc. Camb. Phil. Soc.} {\bf 117} (1995) 259--273.

\bibitem[Oht3]{Ohtsuki1}
T. Ohtsuki,
{\em A polynomial invariant of rational homology $3$-spheres},
{ Invent. Math.} {\bf 123} (1996), no. 2, 241--257.

\bibitem[Oht4]{Ohtsuki:finite}
T. Ohtsuki,
{\em Finite type invariants of integral homology $3$-spheres},
{ J. Knot Theory Ramifications} {\bf 5} (1996), no. 1, 101--115.

\bibitem[Oht5]{Ohtsuki}T. Ohtsuki, {\em Quantum invariants.
A study of knots, $3$-manifolds, and their sets}, Series on Knots and
Everything, {\bf 29}, World Scientific Publishing Co., Inc., River
Edge, NJ, 2002.

\bibitem[Rad]{Radford} D. Radford,
{\em Minimal quasitriangular Hopf algebras},
J. Algebra {\bf 157} (1993), 285--315.

\bibitem[Res]{Reshetikhin} N. Reshetikhin, {\em Quasitriangular Hopf algebras and invariants of links},  (Russian) Algebra i Analiz {\bf 1} (1989),  169--188; translation in Leningrad Math. J. {\bf 1} (1990),  491--513.

\bibitem[RS]{Reshetikhin-SemenovTianShansky}N.  Reshetikhin, M.   Semenov-Tian-Shansky,   {\em Quantum $R$-matrices and factorization problems},   J. Geom. Phys. {\bf 5}  (1988),  533--550.

\bibitem[RT1]{RT1} N. Reshetikhin and V. Turaev, {\em Ribbon graphs
 and their invariants derived from quantum groups}, { Comm. Math.
 Phys.}, {\bf 127} (1990), pp.  262--288.

\bibitem[RT2]{RT2} N.Yu. Reshetikhin and V.G. Turaev,
{\em Invariants of $3$-manifolds via link polynomials and quantum groups}
{ Invent. Math.} {\bf 103} (1991), no. 3, 547--597.


\bibitem[Ros]{Rosso} M. Rosso,  {\em Analogues de la forme de Killing et du theoreme d'Harish-Chandra pour les groupes quantiques}, { Ann. Sci. Ecole Norm. Sup.} (4), {\bf   23}  (1990), 445--467.


\bibitem[Roz1]{Rozansky:Ohtsuki} L. Rozansky,  {\em Witten's invariants of rational homology spheres at prime values of K and trivial
connection contribution}, { Comm. Math. Phys.}  {\bf  180} (1996), 297--324.

\bibitem[Roz2]{Rozansky:integrality}
L. Rozansky,  {\em  On $p$-adic properties of the Witten-Reshetikhin-Turaev invariant},  in: { Primes and Knots. Contemp. Math.} {\bf 416}, 213--236. Am. Math. Soc., Providence, RI (2006).

\bibitem[Sa]{Sawin} S. F. Sawin, {\em Invariants of Spin three-manifolds
  from Chern-Simons theory and finite-dimensional Hopf algebras}, {
  Adv. Math.} {\bf 165} (2002), 35--70.

\bibitem[Su1]{Suzuki1} S. Suzuki, {\em On the universal $sl_2$ invariant of
  ribbon bottom tangles},  { Algebr. Geom. Topol.}  {\bf 10} (2010),
  1027--1061. (electronic)
\bibitem[Su2]{Suzuki2} S. Suzuki, {\em On the universal $sl_2$ invariant of
  boundary bottom tangles},  { Algebr. Geom. Topol.}  {\bf 12} (2012),
  997--1057.  (electronic)



\bibitem[TY]{TY}
T. Takata and Y. Yokota,
{\em The $PSU(N)$ invariants of $3$-manifolds are algebraic integers},
{ J. Knot Theory Ramifications} {\bf 8} (1999), no. 4, 521--532.


\bibitem[Ta]{Tanisaki} T. Tanisaki, {\em
Killing forms, Harish-Chandra isomorphisms, and universal $R$-matrices for quantum algebras},   Adv. Ser. Math. Phys., {\bf 16}, World Sci. Publ., River Edge, NJ, 1992,  941--961,

\bibitem[Tur]{Turaev}  V.~G. Turaev, {\it Quantum invariants of
knots and $3$-manifolds}, de Gruyter Studies in Mathematics 18, Walter
de Gruyter, Berlin New York 1994.

\bibitem[Vi]{Virelizier} A. Virelizier, {\em Kirby elements and quantum
  invariants}, { Proc. London Math. Soc.} {\bf 93} (2006) 474--514.

\bibitem[Wi]{Witten}
E. Witten,
{\em Quantum field theory and the Jones polynomial},
{\it Comm. Math. Phys.} {\bf 121}  (1989),  no. 3, 351--399.




\end{thebibliography}
\end{document}